\documentclass[letterpaper, 11pt, reqno]{amsart} 

\usepackage[utf8]{inputenc}
\usepackage[T1]{fontenc}

\usepackage[english]{babel}

\usepackage{varwidth}
\usepackage{enumitem}
\usepackage[bookmarks = true]{hyperref}
\usepackage{xcolor}
\usepackage{amsmath, amsthm, amssymb, bbm}
\usepackage[new]{old-arrows}
\usepackage{unicode-math}

\usepackage[all]{xy}
\usepackage{tikz}
\usepackage[abbrev]{amsrefs}

\theoremstyle{plain}
\newtheorem{thm}{Theorem}[section]
\newtheorem*{thm*}{Theorem}
\newtheorem{lem}[thm]{Lemma}
\newtheorem{prop}[thm]{Proposition}
\newtheorem{cor}[thm]{Corollary}

\theoremstyle{definition}
\newtheorem{rmk}[thm]{Remark}

\newtheorem{defn}[thm]{Definition}
\newtheorem{ex}[thm]{Example}
\newtheorem{conj}[thm]{Conjecture}
\newtheorem{construction}[thm]{Construction}

\numberwithin{equation}{section}

\newcommand{\overbar}[1]{\mkern 1.5mu\overline{\mkern-1.5mu#1\mkern-1.5mu}\mkern 1.5mu}

\newcommand{\ob}{\overbar}

\newcommand{\mr}{\mathrm}

\newcommand{\wt}{\widetilde}

\newcommand{\trdeg}{\mathrm{trdeg}}

\newcommand{\diag}{\operatorname{diag}}
\newcommand{\Supp}{\operatorname{Supp}}

\newcommand{\Int}{\operatorname{Int}}

\newcommand{\red}{\mathrm{red}}
\newcommand{\rk}{\mathrm{rk}}

\newcommand{\Sch}{\operatorname{Sch}}

\newcommand{\Hom}{\operatorname{Hom}}
\newcommand{\End}{\operatorname{End}}

\newcommand{\Lie}{\operatorname{Lie}}

\newcommand{\Spec}{\operatorname{Spec}}

\newcommand{\Spf}{\operatorname{Spf}}

\newcommand{\codim}{\operatorname{codim}}

\newcommand{\len}{\operatorname{len}}

\newcommand{\tensor}{\otimes}

\newcommand{\iso}{\cong}

\newcommand{\mbA}{\mathbb{A}}

\newcommand{\mbC}{\mathbb{C}}
\newcommand{\mbD}{\mathbb{D}}

\newcommand{\mbF}{\mathbb{F}}
\newcommand{\mbG}{\mathbb{G}}

\newcommand{\mbN}{\mathbb{N}}

\newcommand{\mbP}{\mathbb{P}}
\newcommand{\mbQ}{\mathbb{Q}}
\newcommand{\mbR}{\mathbb{R}}

\newcommand{\mbX}{\mathbb{X}}

\newcommand{\mbZ}{\mathbb{Z}}

\newcommand{\mcC}{\mathcal{C}}
\newcommand{\mcD}{\mathcal{D}}
\newcommand{\mcE}{\mathcal{E}}

\newcommand{\mcH}{\mathcal{H}}
\newcommand{\mcI}{\mathcal{I}}

\newcommand{\mcK}{\mathcal{K}}

\newcommand{\mcM}{\mathcal{M}}
\newcommand{\mcN}{\mathcal{N}}
\newcommand{\mcO}{\mathcal{O}}

\newcommand{\mcS}{\mathcal{S}}
\newcommand{\mcT}{\mathcal{T}}
\newcommand{\mcU}{\mathcal{U}}

\newcommand{\mcX}{\mathcal{X}}
\newcommand{\mcY}{\mathcal{Y}}
\newcommand{\mcZ}{\mathcal{Z}}

\newcommand{\mfa}{\mathfrak{a}}

\newcommand{\mfm}{\mathfrak{m}}

\newcommand{\mfp}{\mathfrak{p}}

\newcommand{\lr}{\longrightarrow}

\begin{document}

\title{$δ$-Forms on Lubin--Tate spaces}
\author{Andreas Mihatsch}
\date{September 7, 2024}

\begin{abstract}
We extend Gubler--Künnemann's theory of $δ$-forms from algebraic varieties to good Berkovich spaces. This is based on the observation that skeletons in such spaces satisfy a tropical balancing condition. Our first theorem in this formalism is a construction of Green $δ$-forms for complete intersection cycles. Our second theorem is a formula for artinian intersection numbers on formal models in terms of integrals of $δ$-forms on their generic fibers. These two results generalize known results from divisors to higher codimension. Finally, we illustrate our intersection number formula in the context of an intersection problem for Lubin--Tate spaces.
\end{abstract}

\maketitle

\tableofcontents

\section{Introduction}

There is a well-known \emph{analytic} formulation of non-archimedean Arakelov geometry for curves. It goes back to Chinburg--Rumely \cite{CR} and S.-W. Zhang \cite{SWZha_Admissible}, and has been developed in terms of Berkovich spaces by A. Thuillier \cite{Thu}.  A brief sketch of this theory is as follows. Let $(k,v)$ be an algebraically closed non-archimedean valued field and $C/k$ a proper smooth Berkovich curve. Its type (2) and (3) points form an infinite metrized graph $Γ$ which is the union of the reduction graphs $Γ(\mcC)$ of all strictly semi-stable formal models $\mcC/\mcO_k$. Intersection theory of divisors on such models may be described completely in terms of $Γ$: A formal model $\mcD\in \mr{Div}(\mcC)$ of a divisor $D\in \mr{Div}(C)$ gives rise to a Green function $g_{\mcD}$ for $D$. (This is a piecewise linear function on $C\setminus \Supp D$ with a certain kind of singularity along $\Supp D$.) The intersection number $\mcD_1\cdot \mcD_2$ of two divisors with disjoint generic fibers may be recovered as
\begin{equation}\label{eq:intro_intersection_identity}
\mcD_1 \cdot \mcD_2 = \int_{D_1} g_{\mcD_2} + \int_{Γ} g_{\mcD_1}\cdot c_1(D_2,g_{\mcD_2}).
\end{equation}
The curvature term $c_1(D_2,g_{\mcD_2})$ here is the finite sum of Dirac measures on $Γ$ obtained by applying the following Laplace operator to $g = g_{\mcD_2}$,
\begin{equation}\label{eq:intro_laplace}
d'd''g = \sum_{x\in Γ} \left(\sum_{p\, =\, [x,x+ε]\,\subseteq\, Γ} (g\vert_p)'(x)\right) \cdot δ_x.
\end{equation}
The inner sum is over all germs $p$ of isometric paths emanating from $x$. Put in words, the multiplicity of $δ_x$ in \eqref{eq:intro_laplace} is the sum of all outgoing slopes of $g$ in $x$.

A natural expectation is that this picture should generalize (in some form) from curves to higher dimensions. In the present article, we propose some ideas towards such a theory. Our approach builds on A. Ducros' results on skeletons of Berkovich spaces \cites{Duc_early, Duc_local}, takes up the theories of smooth real-valued $(p,q)$-forms developed by A. Lagerberg \cite{Lag} as well as Chambert-Loir--Ducros \cite{CLD}, and extends Gubler--Künnemann's theory of $δ$-forms \cite{GK}. The following provides a brief summary.

Let $X$ be a Berkovich space over some non-archimedean valued field $(k, v)$. The analog of the set of type (2) and (3) points $Γ$ from above is the set of Abhyankar points in $X$. It is the union of all skeletons $\Sigma \subseteq X$ in the sense of A. Ducros \cites{Duc_early, Duc_local} and hence has an ind-piecewise linear structure. Our observation is that many of these $\Sigma$ have a natural structure as \emph{tropical space}, which constitutes a generalization of the above metrized graph property.
In this context, we introduce a theory of $δ$-forms on tropical spaces, which relies on the theory of $δ$-forms on $\mbR^r$ from our previous article \cite{Mih_trop_inter}. Then we define $δ$-forms on non-archimedean spaces as certain compatible families of $δ$-forms on all their skeletons. These $δ$-forms come with a $\wedge$-product and various differential operators. In particular, there is a generalization of the Laplacian from \eqref{eq:intro_laplace}. One of our two main results is a higher codimension variant of the intersection number identity \eqref{eq:intro_intersection_identity} for certain complete intersection situations.

Our results were originally motivated by an article of Q. Li \cite{Li_LT}. There he proves an analytic expression for intersection numbers of certain quadratic cycles on Lubin--Tate spaces. His method is remarkable because it involves passage to higher level which surprisingly simplifies the problem. Our starting question was whether his proof could be understood more naturally at infinite level in the sense of Scholze--Weinstein \cites{SW_pdiv, SW_Berkeley}. This infinite level space is an analytic object, so we developed the analytic methods of the present article. In the last section, we come back to Li's intersection problem and illustrate our results in his context.

We mention some related works. The idea of a balancing condition for higher-dimensional skeletons has occurred before, for example in works of Cartwright \cite{Cartwright_tropical}, Yu \cite{Yu_balancing} and Gubler--Rabinoff--Werner \cites{GRW1}. The skeletons in those articles are constructed from strictly semi-stable formal models. Generalizations to more general skeletons were given by Gubler--Rabinoff--Werner in \cite{GRW2} and by Gubler--Jell--Rabinoff in their very recent \cite{GJR}. The latter is especially relevant for us because we work with their definition of weighted skeletons. Their work is moreover related to ours because they, too, give a generalization of Chambert-Loir--Ducros' differential forms. This relies on a definition of so-called \emph{harmonic} tropicalization maps, which are more general than the smooth ones we consider. It seems that harmonic tropicalizations could also be applied to further enhance our definition of $δ$-forms.

Another theme we take up is that of extending tropical intersection theory from $\mbR^r$ to more general situations. Such generalizations have been considered before, for example by K. Shaw \cite{Shaw_matroid} and Fran\c{c}ois--Rau \cite{FR} for matroidal fans or by Jell--Shaw--Smacka \cite{Jell_Shaw_Smacka} and Jell--Rau--Shaw \cite{Jell_Rau_Shaw} for partial compactifications of $\mbR^r$. Our theory of $δ$-forms on tropical spaces presents a different approach to this question.

We now describe our results in more detail.

\subsection{Tropical Spaces}

We define tropical spaces as triples $(X,μ,L)$ of the following kind, cf. Def. \ref{def:trop_space} for all details.
\begin{enumerate}[wide, labelindent=0pt, labelwidth=!, label=(\arabic*), topsep=4pt, itemsep=4pt]
\item $X$ denotes a piecewise linear space, purely of some dimension $n$. By our definition this means that $X$ is a Hausdorff space together with a sheaf of continuous real valued functions $Λ_X$ such that $(X,Λ_X)$ is locally isomorphic to an open subset of a polyhedral set in some $\mbR^r$ together with its piecewise linear functions.
\item $μ$ is a weight for $X$. Roughly, this means that $μ$ is the datum of a weight for every small enough $n$-dimensional polyhedron $σ\subseteq X$.
\item $L$ is a \emph{balanced sheaf of linear functions} for $(X,μ)$, meaning it is subsheaf of $\mbR$-vector spaces of $Λ_X$ that contains the locally constant functions and satisfies the following conditions:
\item For every point $x\in X$, there exist a compact polyhedral neighborhood $x\in K$ and functions $ϕ \in L(K)^r$ such that $ϕ:K\to \mbR^r$ has finite fibers. We also impose that $L$ locally has a subordinate polyhedral complex of definition, cf. Def. \ref{def:sheaf_of_linear_functions}.
\item Finally, $μ$ and $L$ are compatible in the following sense. Whenever $K\subseteq X$ is a polyhedral subset and $\mcK$ a structure of polyhedral complex for $K$ that is subordinate to $μ$ and $L$, then for every $(n-1)$-dimensional $τ\in \mcK$ that is contained in the inner of $K$ and for every $ϕ\in L(K)$ such that $ϕ\vert_τ$ is constant,
\begin{equation}\label{eq:intro_balanced}
\sum_{σ\in \mcK,\ \dim σ = n,\ τ\subseteq σ} \frac{\partial ϕ\vert_σ}{\partial n_{σ,τ}} = 0.
\end{equation}
Here, the $\{n_{σ,τ}\}_{τ\subseteq σ}$ are a family of so-called normal vectors for the inclusions $τ\subseteq σ$. Roughly, these are vectors in the affine linear space $N_σ$ spanned by $σ$ that point away from $τ$ and have length proportional to the weight on $σ$.
\end{enumerate}

For example, $(\mbR^r, μ_{\mr{std}}, \mr{Aff})$ is a tropical space where $μ_{\mr{std}}$ is the standard weight on $\mbR^r$ and $\mr{Aff}$ the sheaf of affine linear functions. More generally, every tropical cycle $(X,μ) \subseteq \mbR^r$ becomes a tropical space when endowed with the sheaf $L = \mr{Aff}\vert_X$. This essentially describes the tropical spaces with the property that $L(X)$ \emph{embeds} $X$ into some $\mbR^r$. Axioms (4) and (5) are just slightly weaker; general tropical spaces locally map finitely onto tropical cycles. In fact, this is an equivalent way to phrase (4) and (5), cf. Cor. \ref{cor:equiv_char_trop_space}.

The motivating property of tropical spaces is that they allow a formulation of the balancing and harmonicity condition. Namely let $K$, $\mcK$ and $τ$ be as in (5) above. Assume additionally that the functions $L(K)$ embed $τ$ into some $\mbR^r$ as in (4) and let $φ:K\to \mbR$ be any piecewise linear function. Then there exists $ϕ_τ\in L(K)$ such that $(φ - ϕ_τ)\vert_τ$ is constant and we call $φ$ harmonic in $τ$ if
\begin{equation}\label{eq:intro_balanced_II}
\sum_{σ\in \mcK,\ \dim σ = n,\ τ\subseteq σ} \frac{\partial (φ-ϕ_τ)\vert_σ}{\partial n_{σ,τ}} = 0.
\end{equation}
This condition is independent of the choice $ϕ_τ$ precisely by (5). A variant of \eqref{eq:intro_balanced_II} moreover gives the notion of balanced polyhedral currents, cf. Def. \ref{def:balanced_polyhedral_current}. (A polyhedral current on $X$ is a locally finite linear combination $\sum_{i\in I} α_i\wedge [σ_i, μ_i]$ where the $(σ_i, μ_i)$ are weighted polyhedra and the $α_i$ smooth $(p,q)$-forms in the sense of Lagerberg \cite{Lag}.)

Consider again the tropical space $(\mbR^r, μ_{\mr{std}}, \mr{Aff})$ which we usually denote just by $\mbR^r$. Balanced polyhedral currents on $\mbR^r$ are called $δ$-forms and are denoted by $B(\mbR^r)$. Examples of $δ$-forms on $\mbR^r$ are piecewise smooth $(p,q)$-forms in the sense of Lagerberg \cite{Lag} and tropical cycles in the sense of tropical geometry. A theory of such $δ$-forms has been worked out in \cite{Mih_trop_inter}. For example, they come with a $\wedge$-product that extends the $\wedge$-product of piecewise smooth $(p,q)$-forms and the intersection product of tropical cycles. They furthermore come with various differential operators $d'$, $d''$, $d'_P$, $d''_P$, $\partial'$ and $\partial''$ that generalize the differential operators $d'$ and $d''$ of $(p,q)$-forms and the corner locus construction of tropical geometry.

Consider again a general tropical space $(X, μ, L)$, together with a tuple $f\in L(X)^r$ and a $δ$-form $γ\in B(\mbR^r)$. We prove in §\ref{ss:tropical_delta_forms} that there is a natural way to define a pullback $f^\star(γ)$ of $γ$ to $X$. This pullback is a polyhedral current on $X$ that is characterized by a certain type of projection formula identity, cf. \eqref{eq:realization}.
\begin{defn}[cf. Def. \ref{def:delta_form_tropical_space}]\label{def:intro_delta_tropical}
A $δ$-form on $(X,μ,L)$ is a polyhedral current on $X$ that is locally of the form $f^\star(γ)$ for some $r\geq 1$, some $f\in L^r$ and some $γ\in B(\mbR^r)$. The sheaf of $δ$-forms on $X$ is denoted by $B$ or $B_X$.
\end{defn}
It is not hard to see that the $\wedge$-product and the differential operators for $δ$-forms on $\mbR^r$ transfer to analogous operators on $B_X$. The existence of pullbacks $h^*:h^{-1}B_Y\to B_X$ along linear maps $h:(X,μ_X,L_X)\to (Y,μ_Y,L_Y)$ is subtle, however. We introduce a combinatorial condition of \emph{flatness} to resolve this, cf. Def. \ref{def:flat}. It turns out that this condition is always satisfied for maps between skeletons in Berkovich spaces, cf. Thm. \ref{thm:intro_flatness} below.

\subsection{Skeletons}
Fix a non-archimedean valued field $(k,v)$. We take the additive point of view throughout our work, meaning $v:k^\times \to \mbR$. We consider $k$-analytic spaces in the sense of Berkovich \cites{Ber1, Ber2} and we always assume our spaces to be good and Hausdorff. Whenever we write $\mbG_m^r$ in the following, we really mean the torus over $k$ as $k$-analytic space.

Let $X$ be a $k$-analytic space that is in addition purely of some dimension $n$. By toric coordinates on $X$, we mean a map $f = (f_1,\ldots,f_r):X\to \mbG_m^r$ to some such torus. The tropicalization map of $f$ is its composition with taking valuation,
\begin{equation}\label{eq:intro_tropicalization_map}
t_f:X\lr \mbR^r,\ x\longmapsto (v(f_1(x)),\ldots, v(f_r(x))).
\end{equation}
Assuming that $X$ is compact, the $n$-dimensional locus $T(X,f)$ of the image $t_f(X)$ was made into a tropical cycle (away from $t_f(\partial X)$) by Chambert-Loir--Ducros \cite{CLD}. It is called the \emph{tropicalization of $X$ by $f$}.

The datum $T(X, f)$ can be refined by considering skeletons: First, Ducros \cite{Duc_local} proved that the union
$$\Sigma'(X,f) := \bigcup_{1\leq i_1, \ldots, i_n \leq r} (f_{i_1},\ldots,f_{i_n})^{-1} \Sigma(\mbG_m^n)$$
of inverse images of the standard skeleton $\Sigma(\mbG_m^n)$ is itself a piecewise linear space. Second, the $n$-dimensional locus $\Sigma(X,f)$ of $\Sigma'(X,f)$ has been endowed with weights by Gubler--Jell--Rabinoff \cite{GJR}. Third, we may give it the sheaf of linear functions generated by all components $t_{f_i}$ of \eqref{eq:intro_tropicalization_map}. The restriction of these data to the open subset $\Sigma(X\setminus \partial X, f)$ then defines a tropical space in the sense of \S1.1. Note that if $\partial X = \emptyset$, then simply $\Sigma(X,f) = \Sigma'(X,f)$, see Prop. \ref{prop:skeleton_pure_dim}.

We prove two results on skeletons and tropicalizations. The first is the following homogeneity property. It extends the analogous results of Chambert-Loir--Ducros \cite{CLD}*{Prop. 4.6.6} from divisors to higher codimension.
\begin{prop}[cf. Prop. \ref{prop:reg_sequence_statement}]\label{prop:intro_reg_sequence}
Assume that $X$ is compact and that $Z = V(f_1,\ldots,f_r)\subseteq X$ is defined by a regular sequence $f_1,\ldots,f_r\in \mcO_X(X)$. Let $g:X\to \mbG_m^s$ be a tuple of toric coordinates. Writing $U = X\setminus V(f_1\cdots f_r)$, there exists a constant $C$ such that
$$T\big(\{x\in U,\ v(f_i(x)) > C\ \forall i\},\ f\times g\big) = (C,\infty)^r\times T(Z,g)$$
as weighted polyhedral sets.
\end{prop}
Our second result is the already mentioned flatness of maps between skeletons. We assume that $\partial X = \emptyset$ and that $Y$ is a further $k$-analytic space, purely of some dimension $m$ with $\partial Y = \emptyset$. Recall that a point of a $k$-analytic space is called Abhyankar if its graded residue field has maximal possible transcendence degree.
\begin{thm}[cf. Thm. \ref{thm:skeleton_flat}]\label{thm:intro_flatness}
Let $π:X\to Y$ be a morphism that is locally on $X$ dominant in the sense that it takes Abhyankar points to Abhyankar points. Given toric coordinates $g:X\to \mbG_m^r$ and $x\in \Sigma(X,g)$, there exist an open neighborhood $U$ of $y = π(x)$ in $Y$, an open neighborhood $V$ of $x$ in $π^{-1}(U)$ and toric coordinates $f:U\to \mbG_m^s$ such that $π$ induces a flat map of skeletons
\begin{equation}\label{eq:intro_map_skeletons}
\Sigma(V, g\times (f\circ π))\lr \Sigma(U, f).
\end{equation}
\end{thm}
Flatness essentially means that every fiber is itself a tropical space in a natural way. In the case at hand, the fibers of \eqref{eq:intro_map_skeletons} will be the skeletons of $g$ in fibers of $π$.

The above two results play a fundamental role in the definition of pullbacks of $δ$-forms, cf. Thm. \ref{thm:intro_pull_back_delta} below.

\subsection{$δ$-Forms on Non-Archimedean Spaces}
We continue to assume that $X$ is purely of dimension $n$ with $\partial X = \emptyset$. In general, since the tropical cycle property of $T(X,f)$ and the tropical space property of $\Sigma(X,f)$ only hold away from $t_f(\partial X)$ (resp. $\partial X$), our whole theory of $δ$-forms only applies to spaces without boundary. (A similar restriction is made in \cite{CLD}*{§4}, where currents are defined on smooth $(p,q)$-forms with compact support away from the boundary.)

Motivated by a similar definition of Ducros \cite{Duc_local}, we call a subset $\Sigma\subseteq X$ a skeleton if it is locally a closed piecewise linear subspace of some $\Sigma(U, f)$. (Here, $U\subseteq X$ is open and $f:U\to \mbG_m^r$ a tuple of toric coordinates.) 

\begin{defn}[cf. Def. \ref{def:delta_form_non_arch}]\label{def:intro_delta_non_arch}
A $δ$-form on $X$ is the datum of a polyhedral current $ω_{\Sigma}$ on every skeleton $\Sigma \subseteq X$ such that
\begin{enumerate}[wide, labelindent=0pt, labelwidth=!, label=(\arabic*), topsep=4pt, itemsep=4pt]
\item Whenever $\Sigma'\subseteq \Sigma$, then $ω_{\Sigma}\vert_{\Sigma'} = ω_{\Sigma'}$.

\item Every point $x\in X$ has an open neighborhood $x\in U$ that admits toric coordinates $(f_1,\ldots,f_r):U\to \mbG_m^r$ and a $δ$-form $γ\in B(\mbR^r)$ such that for all open subsets $V\subseteq U$ and all toric coordinates $(g_1,\ldots,g_s):V\to \mbG_m^s$,
$$ω_{\Sigma(V,f\times g)} = t_f^\star(γ).$$
Here, $t_f\in L(\Sigma(V,f\times g))^r$ is viewed as a tuple of linear functions on a tropical space and the pullback $t_f^\star(γ)$ is the one from Def. \ref{def:intro_delta_tropical}.
\end{enumerate}
\end{defn}
We write $B$ or $B_X$ for the sheaf of $δ$-forms on $X$. They come with a $\wedge$-product and the same kind of differential operators $d',d'',d'_P,d''_P,\partial'$ and $\partial''$ as before. For example, the smooth differential forms of Chambert-Loir--Ducros \cite{CLD}, the weakly smooth forms of Gubler--Jell--Rabinoff \cite{GJR} and the $δ$-forms of Gubler--Künnemann \cite{GK} are all also $δ$-forms in our sense (\S4.2).

A further piece of structure that is common to $δ$-forms on $\mbR^r$, to $δ$-forms on tropical spaces and to $δ$-forms on analytic spaces is the existence of a trigrading $B = \bigoplus_{p,q,c} B^{p,q,c}$. In all three situations, a $δ$-form is of tridegree $(p,q,c)$ if it can be (locally) written as (a pullback of) a balanced polyhedral current $\sum_{i\in I} α_i \wedge [σ_i, μ_i]$ where the coefficients $α_i$ are of bidegree $(p,q)$ and the polyhedra $σ_i$ of codimension $c$.

We also require the notion of a piecewise smooth form on an analytic space. Those are defined in the same way as the smooth forms in \cite{CLD}, but allowing piecewise smooth coefficients. We denote the sheaf of piecewise smooth forms of bidegree $(p,q)$ on $X$ by $PS^{p,q}_X$.
\begin{prop}[cf. Prop. \ref{prop:pws_are_delta_of_codim_0}]\label{prop:intro_pws_is_delta}
There is an isomorphism $PS_X^{p,q}\iso B^{p,q,0}_X$.
\end{prop}
It is a general phenomenon that constructions by formal geometry often give rise to only piecewise smooth (as opposed to smooth) functions and differential forms. The significance of Prop. \ref{prop:intro_pws_is_delta} is that the formalism of $δ$-forms is sufficiently general to handle such functions and forms.

Finally, we mention the existence of pullbacks for $δ$-forms. Let $Y$ be a further $k$-analytic space, purely of some dimension $m$ and with $\partial Y = \emptyset$.
\begin{thm}[cf. Thm. \ref{thm:pull_backs_exist}]\label{thm:intro_pull_back_delta}
Given a morphism $π:X\to Y$, there is a well defined pullback map $π^*:π^{-1}B_Y\to B_X$ that satisfies the relation
$$π^*(t_f^\star(γ)) = t_{f\circ π}^\star(γ).$$
\end{thm}
For example, this defines the restriction $ω\vert_Z$ of a $δ$-form to a Zariski closed subspace $Z\subseteq X$ even though the skeletons in Def. \ref{def:intro_delta_non_arch} do not intersect Zariski closed subspaces of codimension $\geq 1$. Thm. \ref{thm:intro_pull_back_delta} also constitutes an improvement over the definition of $δ$-forms in \cite{GK}. There, the $δ$-forms $t_f^\star(γ)$ were considered up to an equivalence relation involving all possible maps $X'\to X$. The theorem here states that the much more concrete Def. \ref{def:intro_delta_non_arch} in terms of skeletons leads to an equivalent notion.


\subsection{Intersection Theory}

We still assume that $X$ is purely of dimension $n$ with $\partial X = \emptyset$. Further assume that $Z\subseteq X$ is a Zariski closed subspace that is a complete intersection in the sense that $Z = V(f_1, \ldots, f_r)$ for a regular sequence $f_1,\ldots,f_r\in \mcO_X(X)$. In particular, $Z$ is purely of dimension $n-r$. The function $ϕ = \min\{v(f_1), \ldots, v(f_r)\}$ on $X\setminus Z$ is piecewise linear and hence a $δ$-form by Prop. \ref{prop:intro_pws_is_delta}. Thus we may define a $δ$-form of bidegree $(r-1, r-1)$ on $X\setminus Z$ by
$$γ = (-1)^{r-1}ϕ\, (d'd'' ϕ)^{r-1}.$$
This is a non-archimedean analog of a construction in complex geometry, cf. Ex. \ref{ex:complex_geom}. Our main result on $γ$ is that it gives rise to the following Poincaré--Lelong equation for $Z$.
\begin{thm}[cf. Thm. \ref{thm:Green_form_complete_inter}]\label{thm:intro_Green}
For every $δ$-form $η\in B^{n-r, n-r}(X)$ with compact support,
$$\int_{X\setminus Z} γ \wedge d'd''η = -\int_Z η\vert_Z.$$
\end{thm}
Thm. \ref{thm:intro_Green} generalizes the Poincaré--Lelong formula for divisors of Chambert-Loir--Ducros \cite{CLD}*{Théorème 4.6.5} and Gubler--Künnemann \cite{GK}*{Theorem 7.2} to our notion of $δ$-forms and to higher codimension. It in particular states that $γ$ defines a Green current for $Z$. In the given situation, meaning for complete intersection cycles $Z$, Green currents were already constructed in \cites{CLD, GK} by writing $Z$ as a successive divisor intersection and then using $\star$-products. What is new here, however, is that the Green current is given by a $δ$-form. This potentially has applications to the definition of $\star$-products for cycles of higher codimension.

We come to our second main result. Assume for this that $v(k^\times)\neq \{0\}$ but drop the assumption $\partial X = \emptyset$. Let $\mcX$ be a special formal scheme over the ring of integers $\Spf \mcO_k$. (Special formal schemes are defined in \cite{Ber_vanishing_II}, we review the definition in \S5.2.) Assume that $X$ is the generic fiber of $\mcX$ and denote by $\mr{sp}:X\to \mcX$ the specialization map. Let $\mcZ\subseteq \mcX$ be a closed formal subscheme with generic fiber $Z$. Define a function $ϕ$ by
\begin{equation}
\label{eq:intro_distance_function}
ϕ: X\setminus Z \lr \mbR_{\geq 0},\ x\longmapsto \min\{v(f(x)),\ f\in \mcI_{\mr{sp}(x)}\}.
\end{equation}
Here $\mcI = \ker(\mcO_\mcX \to \mcO_\mcZ)$ denotes the ideal sheaf defining $\mcZ$ and $\mcI_{\mr{sp}(x)}$ its stalk in the specialization of $x$. Then $ϕ$ can be thought of as measuring the distance from $Z$ with respect to the given formal models. As before, $ϕ$ is piecewise linear and hence a $δ$-form, and we may again define
\begin{equation}
\label{eq:def_mysterious_delta_form}
γ = (-1)^{r-1}ϕ\, (d'd'' ϕ)^{r-1} \in B^{r-1, r-1}(X\setminus (Z\cup \partial X)).
\end{equation}
The following theorem is the promised analog of \eqref{eq:intro_intersection_identity} for higher codimension.
\begin{thm}[cf. Thm. \ref{thm:int_num_identity}]\label{thm:intro_intersection}
Assume that $\mcO_k$ is a DVR with uniformizer $π$ and that $v$ is normalized by $v(π) = 1$. Assume that $\mcX$ is flat over $\mcO_k$ and purely of formal dimension $n+1$. Assume further that $\mcZ\subseteq \mcX$ is a local complete intersection that is purely of codimension $r$. Let $\mcY\subseteq \mcX$ be a closed formal subscheme that is Cohen--Macaulay and purely of dimension $r$. Also assume that $\mcY$ is flat over $\mcO_k$ and that the intersection $\mcZ\cap \mcY$ is artinian. Let $Y$ denote its generic fiber. Then $γ\vert_{Y\setminus \partial Y}$ has compact support and
$$\int_{Y\setminus \partial Y} γ\vert_{Y\setminus \partial Y} = \mr{len}_{\mcO_k} \mcO_{\mcY\cap \mcZ}.$$
\end{thm}

We end this section with two questions that arise from Thms. \ref{thm:intro_Green} and \ref{thm:intro_intersection}. First, let $(E, v)$ be a metrized vector bundle of rank $r$ on $X$ and let $s\in E(X)$ be a regular section. Let $Z = V(s)$ be its vanishing locus which is of codimension $r$. Then it would be interesting to know if Thm. \ref{thm:intro_Green} extends to a construction of a Green $δ$-form for $Z$ from the datum $(E, v, s)$. In the archimedean setting, such a construction can be obtained from the work of Bismut--Gillet--Soulé \cite{BGS}. We refer to the work of Garcia--Sankaran \cite{Garcia_Sankaran} for a description and for applications to the arithmetic intersection theory on Shimura varieties.

Second, one might try to extend Thm. \ref{thm:intro_intersection} to improper intersections. This will require a new definition of $γ$. If $\mcZ = V(s)$ is the vanishing locus of a regular section $s \in \mcE(\mcX)$ of a vector bundle $\mcE$ on $\mcX$, then such a definition is probably related to the Koszul complex of $(\mcE, s)$ which would be analogous to the archimedean theory in \cite{Garcia_Sankaran}.

\subsection{$δ$-Forms on Lubin--Tate space}

In the last section, we apply the previous results in the context of the Lubin--Tate intersection problem of Q. Li \cite{Li_LT}. Let $F$ be a non-archimedean local field and $\breve F$ the completion of its maximal unramified extension. Given an unramified quadratic extension $E/F$ and integers $h\geq 1$, $n\geq 0$, we consider the deformation space $\mcM_n$ of a height $2h$, dimension $1$, strict formal $\mcO_F$-module with a level $π^n$-structure. It is a regular local formal scheme of dimension $2h$, flat over $\mcO_{\breve F}$. There is an analogous deformation space $\mcN_n$ of a height $h$, dimension $1$, strict formal $\mcO_E$-module with level $π^n$-structure. It is regular of dimension $h$ and there is a closed immersion $\mcN_n\hookrightarrow \mcM_n$. The intersection number in question takes the form
$$\langle \mcN_n,\ (γ,g)\mcN_n\rangle_{\mcM_n} = \mr{len}_{\mcO_k} \mcO_{\mcN_n\cap (γ,g)\mcN_n} \in \mbZ,$$
where $(γ,g)\mcN_n$ is a certain translate of $\mcN_n$. Since all involved formal schemes are regular and local, the situation is that of an artinian complete intersection like in Thm. \ref{thm:intro_intersection}. Hence
$$\langle \mcN_n,\ (γ,g)\mcN_n\rangle = \int_{N_n} (γ,g)^*h_n\vert_{N_n}$$
where $N_n$ is the generic fiber of $\mcN_n$ and $h_n$ the $δ$-form from \eqref{eq:def_mysterious_delta_form} for $\mcN_n$. We evaluate this integral for large $n$ in an essentially analytic way, cf. Prop. \ref{prop:evaluation_integrals}, and reprove the intersection number formula of Li \cite{Li_LT}*{Thm. 1.3} in the considered situation, cf. Thm. \ref{thm:LT_main}. The evaluation depends on the fact that the distance function \eqref{eq:intro_distance_function} for $\mcN_n\subset \mcM_n$ can be described explicitly in terms of Newton polygons of formal $\mcO_F$-modules.

\subsection{Acknowledgements}

I heartily thank J. Anschütz, A. Ducros, W. Gubler, K. Künnemann, Q. Li, P. Scholze and M. Temkin for valuable discussions and correspondence during the preparation of the present article. I am especially thankful for comments by W. Gubler and P. Scholze on earlier versions of this text. I am similarly grateful to the two referees for providing numerous valuable suggestions for improvement.

Furthermore, I would like to thank the Massachussetts Institute of Technology and its staff for offering such a hospitable working environment during my stay in 2019/20, where parts of this project were completed. I also thank the DFG (Deutsche Forschungsgemeinschaft) for making my stay possible through grant MI 2591/1-1.

\section{Tropical Spaces}
\subsection{Piecewise Linear Spaces}
\label{ss:pwl_spaces}

This section fixes some definitions and terminology related to piecewise linear spaces. The exposition will be similar to the ones in \cite{Ber4} and \cite{Duc_local}; in fact, our notion of piecewise linear space is a simplified version of the ones in \cite{Ber4} and \cite{Duc_local}. We will comment on the difference in Rmk. \ref{rmk:relation_with_Berkovich} below.

Given a topological space $X$, let $C^0_X$ denote the sheaf of $\mbR$-valued continuous functions on $X$.

\begin{enumerate}[wide, labelindent=0pt, labelwidth=!, label=(\arabic*), topsep=4pt, itemsep=6pt]

\item \label{item:def_space_with_functions} We define the category of \emph{spaces with functions} as follows. Its objects are pairs $(X, Λ_X)$ where $X$ is a topological space and $Λ_X\subseteq C^0_X$ an $\mbR$-vector space subsheaf. A morphism $f:(X, Λ_X) \to (Y, Λ_Y)$ in this category is a continuous map $f:X\to Y$ such that pullback by composition $f^{-1}C^0_Y \to C^0_X$ takes $f^{-1}Λ_Y$ to $Λ_X$.

\item \label{item:def_restriction_space_with_functions} Let $(Y, Λ_Y)$ be a space with functions and $f:X \to Y$ a continuous map. Then $\left(X, \mr{Im}(f^{-1}Λ_Y \to C^0_X)\right)$ is a space with functions. We will mostly use this construction for inclusions $i:X\hookrightarrow Y$ of subsets (endowed with the induced topology). In this case, we write $Λ_Y\vert_X := \mr{Im}\left(i^{-1}Λ_Y\to Λ_X\right)$ for the resulting sheaf on $X$. Put differently, $Λ_Y\vert_X$ is the sheaf of those continuous functions on $X$ that locally extend to a section of $Λ_Y$.
We use the notation $Λ_Y(X) := (Λ_Y\vert_X)(X)$ and, in this way, extend the domain of definition of $Λ_Y$ from open subsets of $Y$ to all subsets.

\item \label{item:def_affine_lin_polyhedron} By a \emph{polyhedron} in $\mbR^r$, we mean a subset $σ$ that is the intersection of finitely many (not necessarily rational) closed half-spaces. A function $φ:σ→\mbR$ on a polyhedron $σ\subseteq \mbR^r$ is called \emph{affine linear} if it is the restriction of an affine linear function $(x_1,\ldots,x_r)\mapsto a_1x_1+\ldots+a_rx_r + c$. We write $\partial σ$ and $σ^\circ = σ\setminus \partial σ$ for the relative boundary resp. interior of $σ$.

\item \label{item:def_pwl_polyhedron} Let $σ\subseteq \mbR^r$ be a polyhedron. A function $φ:σ\to \mbR$ is called \emph{piecewise linear} (pwl, for short) if $σ$ can be written as a locally finite union $σ = \bigcup_{i\in I} σ_i$ of polyhedra $σ_i\subseteq \mbR^r$ such that each restriction $φ\vert_{σ_i}$ is affine linear. Here and in the following, a set of polyhedra $\{σ_i\}_{i\in I}$ or a union thereof is called \emph{locally finite} if there exists an open covering $\mbR^r = \bigcup_{j\in J} U_j$ such that for each $j\in J$, we have $U_j\cap σ_i\neq \emptyset$ only for finitely many $i$. A pwl function is necessarily continuous.

\item Let $U \subseteq \mbR^r$ be open. A function $φ:U\to \mbR$ is called \emph{piecewise linear} (or pwl, as before) if the restriction $φ\vert_σ$ is pwl for every polyhedron $σ\subseteq U$. We denote by $Λ_{\mbR^r}$ the so defined sheaf of pwl functions on $\mbR^r$. We consider $(\mbR^r, Λ_{\mbR^r})$ as a space with functions.

\item A subset $C \subseteq \mbR^r$ is \emph{polyhedral} if it can be written as a locally finite (in the sense of \ref{item:def_pwl_polyhedron} above) union of polyhedra. Note that a polyhedral subset is necessarily closed. The restriction $Λ_{\mbR^r}\vert_C$ is called the sheaf of pwl functions on $C$. Note that $φ:C\to \mbR$ lies in $Λ_{\mbR^r}(C)$ if and only if $C$ can be written as a locally finite union $C = \bigcup_{i\in I}σ_i$ of polyhedra such that each restriction $φ\vert_{σ_i}$ is affine linear. In fact, there is the stronger statement that the restriction map from global sections $Λ_{\mbR^r}(\mbR^r) \to Λ_{\mbR^r}(C)$ is surjective, see Lem. \ref{lem:extend_pwl} below.
\end{enumerate}
%

\begin{enumerate}[wide, labelindent=0pt, labelwidth=!, label=(\arabic*), topsep=4pt, itemsep=6pt, resume]
\item A \emph{piecewise linear space} (pwl space, for short) is a space with functions $(X, Λ_X)$ with the following two properties. First, $X$ is a Hausdorff space. Second, every point $x\in X$ has a (not necessarily open) neighborhood $V$ such that $(V, Λ_X\vert_V)$ is isomorphic to $(C, Λ_{\mbR^r}\vert_C)$ for some $r$ and some polyhedral subset $C\subseteq \mbR^r$. The datum $Λ_X$ will usually be omitted from the notation.

\item Let $X$ and $Y$ be pwl spaces. A \emph{pwl map} $X\to Y$, or \emph{map of pwl spaces}, is simply a map of spaces with functions.

Given two polyhedral subsets $C\subseteq \mbR^r$ and $D\subseteq \mbR^s$ as well as a map $φ:C→D$, it is equivalent for $φ$ to be defined by an $s$-tuple of pwl functions on $C$ or to define a map of pwl spaces $(C, Λ_{\mbR^r}\vert_C) \to (D, Λ_{\mbR^s}\vert_D)$, cf. \cite{Ber4}*{Lem. 1.2.1}.

\item A pwl space is \emph{polyhedral} if it is isomorphic to the pwl space $(C, Λ_{\mbR^r}\vert_C)$ of a polyhedral set. (This is sometimes called abstract polyhedral set in the literature.)
\end{enumerate}

\begin{lem}\label{lem:nice_neighborhoods}
Let $X$ be a pwl space. Every point $x\in X$ has an open polyhedral neighborhood.
\end{lem}
\begin{proof}
There exist pwl isomorphisms $(-1,1) \overset{\sim}{\to} \mbR$. An example is given by the pwl function $f$ on $(-1,1)$ that satisfies $f(0) = 0$ and that is linear on the intervals $\pm (1-2^{-n}, 1 - 2^{-n-1})$ of slope $2^n$. Taking products, there also exist pwl isomorphisms $(-1,1)^r\overset{\sim}{\to} \mbR^r$.

Given a point $x\in X$, pick any compact polyhedral set neighborhood $x\in K$. Choose any realization $ι:K\hookrightarrow \mbR^r$ of $K$ as a polyhedral subset such that $ι(x) = 0$. Replacing $ι$ by $λι$ for some $λ>0$, we may arrange that $U = ι^{-1}((-1,1)^r)$ is an open neighborhood of $x$. Composing $ι\vert_U$ with any isomorphism as above, we see that $U$ is also a polyhedral set.
\end{proof}

\begin{rmk}\label{rmk:relation_with_Berkovich}
The notion of pwl space in \cite{Ber4} or \cite{Duc_local} also allows to restrict the slopes $(a_1,\ldots,a_r)$ and abscissae $c$ in the definition of pwl functions \ref{item:def_affine_lin_polyhedron} above. Pwl spaces are then defined as spaces with a suitable $G$-covering by polyhedral sets. In the present paper, we work with all polyhedra in which case there is no distinction between the two approaches, cf. \cite{Ber4}*{Prop. 1.4.1}. Then the required tropical intersection theory without rationality conditions has been developed in \cite{Mih_trop_inter}.
\end{rmk}

\begin{enumerate}[wide, labelindent=0pt, labelwidth=!, label=(\arabic*), topsep=2pt, itemsep=6pt, resume]
\item A \emph{polyhedral complex} in $\mbR^r$ is a locally finite set of polyhedra $\mcC$ which is stable under taking faces and such that $σ_1\cap σ_2$ is a face of both $σ_1$ and $σ_2$ for all $σ_1,σ_2\in \mcC$. Every polyhedral complex defines a polyhedral set by taking the union $|\mcC| := \bigcup_{σ\in \mcC} σ$, and every polyhedral set admits a polyhedral complex structure in this sense.
\end{enumerate}

\begin{lem}\label{lem:extend_pwl}
Let $C\subseteq \mbR^r$ be a polyhedral set and let $φ:C\to \mbR$ be a continuous function. Assume that $C$ can be written as a locally finite union $C = \bigcup_{i\in I} σ_i$ of polyhedra $σ_i\subseteq \mbR^r$ such that each restriction $φ\vert_{σ_i}$ is affine linear. Then there exists a pwl function $\wt{φ}:\mbR^r\to \mbR$ such that $φ = \wt{φ}\vert_C$.
\end{lem}
\begin{proof}
There exists a polyhedral complex structure $\mcD$ for $\mbR^r$ with the property that $\{σ\in \mcD\mid σ\subseteq C\}$ is a polyhedral complex structure for $C$. Subdividing the polyhedra of $\mcD$ if necessary, we may assume that $φ\vert_σ$ is affine linear for every $σ\in \mcD$ with $σ\subseteq C$. Subdividing further, we may also assume that every $σ\in \mcD$ is a simplex, i.e. the convex hull of $\dim(σ) + 1$ points. Then, given any function $ψ_0:\mcD_0\to \mbR$ on the vertices of $\mcD$, there is a unique pwl function $ψ:\mbR^r \to \mbR$ such that $ψ(x) = ψ_0(x)$ for every $x \in \mcD_0$ and such that $ψ\vert_σ$ is affine linear for every $σ\in \mcD$. The desired function $\wt{φ}$ can now be taken as the unique pwl function of this kind whose value $\wt{φ}(x)$ at a vertex $x \in \mcD_0$ is given by
$$\wt{φ}(x) = \begin{cases} φ(x) & \text{if $x\in C$}\\
0  & \text{otherwise}.\end{cases}$$
\end{proof}

\begin{enumerate}[wide, labelindent=0pt, labelwidth=!, label=(\arabic*), topsep=2pt, itemsep=6pt, resume]
\item Let $f:C\to D$ be a pwl map of polyhedral sets. We say that two polyhedral complex structures $\mcC$ and $\mcD$ for $C$ resp. $D$ are \emph{subordinate to $f$} if $f(\mcC)\subseteq \mcD$ and if $f\vert_σ:σ\to f(σ)$ is a linear map for all $σ\in \mcC$.
\end{enumerate}

\begin{lem}\label{lem:uniqueness_pwl_str_compact}
\begin{enumerate}[wide, labelindent=0pt, labelwidth=!, label=(\roman*), topsep=2pt, itemsep=2pt]
\item Let $C\subseteq \mbR^r$ and $D\subseteq \mbR^s$ be compact polyhedral subsets, and let $f:C\to D$ be a bijective pwl map. Then $f$ is a pwl isomorphism.

\item Let $C$ be a compact pwl space and $f:C\to \mbR^r$ an injective pwl map. Then $C$ is polyhedral and $f:C\to f(C)$ a pwl isomorphism.

\item Let $f:X\to Y$ be a pwl map of pwl spaces. If $f$ is a homeomorphism, then $f$ is a pwl isomorphism.
\end{enumerate}
\end{lem}
\begin{proof}
In (i), since $C$ is compact, it admits a finite polyhedral complex structure $C = |\mcC|$ that is subordinate to $f$. Then $\mcD = \{f(σ)\mid σ\in \mcC\}$ is a polyhedral complex structure for $\mcD$ that is subordinate to $f^{-1}$.

We come to (ii). Every point $x\in C$ has a compact polyhedral set neighborhood $C_x$. Since $C$ is compact, it is covered by finitely many such sets. It follows that $f(C)\subseteq \mbR^r$ is a finite union of polyhedral subsets and hence polyhedral itself. Since $C$ is compact and Hausdorff, the map $f:C\to f(C)$ is a homeomorphism. Given $x\in C$, the restriction $f:C_x\to f(C_x)$ is then a pwl isomorphism by (i). Thus $f^{-1}$ is pwl because this property can be checked locally. It follows that $f:C\iso f(C)$, so $C$ is polyhedral as claimed.

In (iii), we have to check that the inverse $f^{-1}$ as again pwl. This can be done locally on compact polyhedral neighborhoods and thus reduces to (ii).
\end{proof}

\begin{enumerate}[wide, labelindent=0pt, labelwidth=!, label=(\arabic*), topsep=2pt, itemsep=6pt, resume]
\item Let $X$ be a pwl space and $Y\subseteq X$ a subset. We call $Y$ a \emph{pwl subspace} if $(Y, Λ_X\vert_Y)$ is again a pwl space. We say that $Y$ is an open, closed etc. pwl subspace if $Y$ is an open, closed etc. subset of $X$.

Note that the properties for $Y$ being open, closed or a pwl subspace are all local on $X$ in the following sense: $Y$ has the required property if and only if there is an open covering $X = \bigcup_{i\in I}U_i$ such that each $Y\cap U_i \subseteq U_i$ has that property.

We also mention that a closed subset $Y\subseteq X$ is a closed pwl subspace if and only if the following condition holds. For every polyhedral pwl subspace $K\subseteq X$, say realized as a polyhedral subset $f:K\overset{\sim}{\to} f(K)\subseteq \mbR^r$, the image $f(K\cap Y)$ is also a polyhedral subset of $\mbR^r$. In particular, it is clear that a locally finite intersection or locally finite union of closed pwl subspaces is again a closed pwl subspace.
\end{enumerate}

\subsection{Piecewise Smooth Forms}
\label{ss:pws_polyehdra}

We next discuss differential forms on pwl spaces. These are always meant in the sense of Lagerberg \cite{Lag}. We recall their definitions in \ref{item:def_smooth_form_polyhedron} and \ref{item:def_smooth_forms_differential_operators} below, but refer to the summary in \cite{Mih_trop_inter}*{§2} for some additional details that we will use freely.

\begin{enumerate}[wide, labelindent=0pt, labelwidth=!, label=(\arabic*), topsep=4pt, itemsep=6pt, resume]
\item For a polyhedron $σ\subseteq \mbR^r$, we let $N_σ$ denote the linear space spanned by the differences $x-y,\ x,y\in σ$. We denote by $M_σ = \Hom(N_σ, \mbR)$ the dual vector space.

\item \label{item:def_smooth_form_polyhedron} A function $φ:σ→\mbR$ on a polyhedron is \emph{smooth} if it is the restriction of a smooth function from an open neighborhood of $σ$ in $\mbR^r$. Write $C^\infty(σ)$ for the ring of smooth functions on $σ$ and $\Omega^i(σ) = C^\infty(σ) \tensor_{\mbR} \bigwedge^i_{\mbR} M_σ$ for usual smooth differential $i$-forms on $σ$. A \emph{smooth form} on $σ$ is an element of the exterior algebra
\begin{equation}\label{eq:def_smooth_form}
A(σ) := \bigwedge\nolimits^{\!\bullet}_{\, C^\infty(σ)} \big(\Omega^1(σ)\oplus \Omega^1(σ)\big).
\end{equation}
A form is said to be \emph{of bidegree $(p,q)$} if it lies in the direct summand
$$A^{p,q}(σ) = \Omega^p(σ) \tensor_{C^\infty(σ)} \Omega^q(σ).$$
The $\wedge$-product on $A(σ)$ is bihomogeneous with respect to bidegree.

\item \label{item:def_smooth_forms_differential_operators} Let $d_{\mr{std}}:C^\infty(σ)\to \Omega^1(σ)$ denote the usual differential. Define from this the two differentials
$$d', d'': C^\infty(σ) \lr \Omega^1(σ) \oplus \Omega^1(σ),\ \ \ d'f = (d_{\mr{std}}f, 0),\ \ d''f = (0, d_{\mr{std}}(f)).$$
They admit unique extensions to differentials $d',d'':A(σ)\to A(σ)$ that satisfy the Leibniz rule with respect to the $\wedge$-product and $(d')^2 = (d'')^2 = 0$. These extensions are bihomogeneous of bidegree $(1,0)$ and $(0,1)$, respectively.

A concrete description of $A^{p,q}(σ)$ is now as follows. Let $n = \dim σ$ and let $(x_1,\ldots,x_n):σ\to \mbR^n$ be an injective affine linear map. For $I\subseteq \{1,\ldots,n\}$, say $I = \{i_1,\ldots,i_p\}$ with $i_1<\ldots < i_p$, write $d'x_I = d'x_{i_1}\wedge \ldots \wedge d'x_{i_p}$. Define $d''x_I$ in a similar way. Then every element of $A^{p,q}(σ)$ has a unique expression as
\begin{equation}\label{eq:std_expression_smooth_form}
\sum_{I,J \subseteq \{1,\ldots,n\},\ |I| = p,\ |J| = q} φ_{I,J} d'x_I \wedge d''x_J
\end{equation}
with smooth functions $φ_{I,J}\in C^\infty(σ)$.

\item \label{item:def_pws_form} A \emph{piecewise smooth $(p,q)$-form} (pws $(p,q)$-form, for short) on a polyhedral set $C\subseteq \mbR^r$ is the datum of a locally finite decomposition $C = \bigcup_{i\in I}σ_i$ of $C$ as a union of polyhedra together with $(p,q)$-forms $ω_i\in A^{p,q}(σ_i)$ such that $ω_i\vert_{σ_i\cap σ_j} = ω_j\vert_{σ_i\cap σ_j}$ for all $i,j\in I$; up to refinement of the decomposition. Note that the intersections $σ_i\cap σ_j$ are again polyhedra, so definition \ref{item:def_smooth_form_polyhedron} applies. We write $PS^{p,q}_C$ or $PS^{p,q}$ for the resulting sheaf on $C$ as well as $PS = PS_C = \bigoplus_{p,q}PS^{p,q}_C$.

\item \label{item:def_polyhedral_deriv} The $\wedge$-product and the differentials $d'$, $d''$ of smooth forms immediately extend to pws forms; these extensions are simply computed piecewisely. We write $\wedge$, $d'_P$ and $d''_P$ for the resulting operators. The index $P$ stands for ``polyhedral''; the operators $d'_P$ and $d''_P$ are called the \emph{polyhedral derivatives}. (This definition and terminology was introduced in \cite{GK}.)
Given a pwl map of polyhedral sets $f:C\to D$, there is a pullback map $f^*:f^{-1}PS_D \to PS_C$ with a similarly piecewise definition.

\item It is clear that definitions \ref{item:def_pws_form} and \ref{item:def_polyhedral_deriv} extend to all pwl spaces $X$. We write $PS_X = \bigoplus_{p,q}PS^{p,q}_X$ for the resulting sheaves of pws forms. We also continue to denote the operators introduced so far by $\wedge$, $d'_P$, $d''_P$ and $f^*$ (assuming $f$ is some pwl map).
\end{enumerate}

%

Recall that a Hausdorff topological space is called paracompact if every open cover has a locally finite refinement. A Hausdorff space that is locally compact is paracompact if and only if it is a disjoint union of topological spaces that are countable at infinity \cite{Bourbaki_top_general_1_to_4}*{Chapitre I, \S9, No. 10, Thm. 5}. This equivalent characterization in particular applies to pwl spaces.

\begin{prop}[\cite{CLD}*{Prop. 3.3.6}]\label{prop:part_of_1_pws}
Let $X$ be a paracompact pwl space. Given an open covering $X = \bigcup_{i\in I} U_i$, there exists a family $(ρ_i)_{i\in I}$ of nonnegative pws functions, with $\Supp ρ_i\subseteq U_i$ for every $i$, and such that the sum $\sum_{i\in I} ρ_i$ is locally finite and equals $1$.
\end{prop}

The cited reference \cite{CLD}*{Prop. 3.3.6} is for smooth functions on non-archimedean spaces, but its arguments carry over to our setting.

\begin{lem}\label{lem:paracompact_neighborhood}
Let $X$ be a pwl space. Every compact subset $S\subseteq X$ has a paracompact open neighborhood.
\end{lem}
\begin{proof}
Closed subsets of paracompact spaces are paracompact and $\mbR^n$ is paracompact, so every polyhedral set is paracompact. Thus every point $x\in S$ has a paracompact open neighborhood by Lem. \ref{lem:nice_neighborhoods}. Since $S$ is compact, it is covered by finitely many such open neighborhoods and this finite union is again paracompact.
\end{proof}

\begin{enumerate}[wide, labelindent=0pt, labelwidth=!, label=(\arabic*), topsep=2pt, itemsep=6pt, resume]
\item The \emph{dimension} of a polyhedron $σ$ is by definition the dimension of $N_σ$. The local dimension of a polyhedral set $C$ at a point $x\in C$ is
$$\dim_xC := \max_{x\in σ\subseteq C} \dim σ$$
where the maximum is taken over all polyhedra $σ\subseteq C$ containing $x$. We define the dimension of $C$ as $\dim C = \mr{sup}_{x\in C} \dim_x C$. We say that $C$ is purely of dimension $n$ if $\dim_xC = n$ for all $x\in C$. We similarly define $\dim_xX$ and $\dim X$ for any pwl space. For example, $PS^{p,q}_X = 0$ whenever $\max\{p,q\} > \dim X$.

\item Assume that $\mcC$ is a polyhedral complex in some real affine space. We write $\mcC_{d}$, $\mcC_{\leq d}$, etc. for the polyhedra of $\mcC$ of dimension $d$, $\leq d$, etc. If $|\mcC|$ is purely of dimension $n$, then we write $\mcC^c = \mcC_{n-c}$ for the polyhedra of codimension $c$ in $C$.

\item Assume that $\mcC$ is a polyhedral complex in some real affine space and $C = |\mcC|$. Let $η\in PS(C)$ be a pws form. We call $\mcC$ \emph{subordinate} to $η$ if $η\vert_σ$ is smooth for all $σ\in \mcC$.
\end{enumerate}

\subsection{Weights and Integrals}

Our next aim is to define the integral of top degree pws forms, say on a polyhedron or on a polyhedral set. This requires the notion of weights. We begin with a discussion of the case $\mbR^n$ which is the basis for everything that follows.

The integral of a compactly supported, top degree pws form $η\in PS_c^{n,n}(\mbR^n)$ is defined as follows. Let $x_1,\ldots,x_n$ denote the standard coordinate functions on $\mbR^n$. There is a unique way to write
$$η = φ\,d'x_1\wedge d''x_1 \wedge \ldots \wedge d'x_n \wedge d''x_n$$
with a pws function $φ$ which is necessarily compactly supported. Then put
\begin{equation}\label{eq:def_integral_R_n}
\int_{\mbR^n} η := \int_{\mbR^n} φ
\end{equation}
where the right hand side is defined with respect to the Lebesgue measure. Because of the doubling of the de Rham complex, it satisfies the following transformation rule under bijective affine linear maps $f:\mbR^n\to \mbR^n$,
\begin{equation}\label{eq:transformation_rule_integral}
\int_{\mbR^n} f^*(η) = |\det f| \int_{\mbR^n} η.
\end{equation}
(Here and in the following, the determinant of an affine linear map is defined as the determinant of its linear part.) In particular, such a map $f$ should be considered to be a covering map of degree $|\det f|$. This convention is standard in tropical geometry: Assume, for example, that $f$ restricts to a map $\mbZ^n\to \mbZ^n$. Then it has degree $[\mbZ^n:f(\mbZ^n)]$ when viewed as endomorphism of $\mbR^n$ as tropical variety.
\begin{defn}\label{def:weight_polyhedron}
\begin{enumerate}[wide, labelindent=0pt, labelwidth=!, label=(\arabic*), topsep=2pt, itemsep=6pt]
\item A \emph{weight} for a polyhedron $σ$ is the datum of a Haar measure $μ$ for $M_σ$. A weighted polyhedron is a pair $(σ, μ)$ consisting of a polyhedron $σ$ with a weight $μ$.

\item \label{item:def_weight_functoriality} Let $f:(σ, μ) \to (τ, ν)$ be a bijective affine linear map of weighted polyhedra. Let $f^\vee:M_τ\to M_σ$ be the natural dual map on dual linear spaces. The image and preimage weights of $μ$ resp. $ν$ are defined by
$$f^{-1}(ν) := (f^\vee)_*(ν),\ \ \ \ f(μ) := (f^{-1})^{-1}(μ).$$
Here, $(f^\vee)_*$ denotes the pushforward of measures. The degree of $f$ is defined as the ratio
$$\deg(f) := μ/f^{-1}(ν) = f(μ)/ν \in \mbR_{>0}.$$

\item The sum $μ_1 + μ_2$ of two weights $μ_1, μ_2$ (on some polyhedron $σ$, say) is the sum as Haar measures.
\end{enumerate}
\end{defn}
The passage to the dual space $M_σ$ is precisely to ensure that the transformation behavior comes out as $f(μ) = |\det f|\cdot μ$ in case $f:σ\to σ$ is an affine linear automorphism. We give an equivalent definition:

\begin{defn}[Variant of Def. \ref{def:weight_polyhedron}]\label{def:weight_polyhedron_variant}
\begin{enumerate}[wide, labelindent=0pt, labelwidth=!, label=(\arabic*), topsep=2pt, itemsep=6pt]
\item A \emph{weight} for a polyhedron $σ$ is the datum of a generator $μ \in \det N_σ$ up to sign. The convention in case $\dim σ = 0$ is that $\det N_σ = \mbR$.

\item \label{item:def_weight_functoriality_variant} Let $f:(σ, μ) \to (τ, ν)$ be a bijective affine linear map of weighted polyhedra. The image and preimage weights of $μ$ resp. $ν$ are defined as image and preimage along $\det f: \det N_σ\to \det N_τ$.

\item The sum $μ_1 + μ_2$ of two weights $μ_1, μ_2$ (on some polyhedron $σ$, say) is the class of the sum of representatives of the same sign.
\end{enumerate}
\end{defn}

The bijection with Def. \ref{def:weight_polyhedron} is as follows. If $n \geq 1$, then given a generator $μ\in \det(N_σ)$, pick a basis $e_1,\ldots,e_n \in N_σ$ such that $μ = \pm e_1\wedge\ldots \wedge e_n$. Let $x_1,\ldots,x_n \in M_σ$ be the dual basis. Then associate to $μ$ the Haar measure such that the parallelepiped spanned by $x_1,\ldots,x_n$ has volume $1$. If $n = 0$ however, then we associate to a scalar $μ\in \det(0) = \mbR$ the Dirac measure of volume $|μ|$.

Def. \ref{def:weight_polyhedron} seems more natural in that it does not require extra conventions for the $0$-dimensional case or for sums $μ_1+μ_2$. Def. \ref{def:weight_polyhedron_variant} however has the more obvious transformation behavior and will be our preferred viewpoint in the following.

We next recall from \cite{Mih_trop_inter}*{§2} the integral of pws forms on weighted polyhedra.

\begin{defn}\label{def:integral}
\begin{enumerate}[wide, labelindent=0pt, labelwidth=!, label=(\arabic*), topsep=2pt, itemsep=6pt]
\item Let $e_1,\ldots,e_n\in \mbR^n$ be the standard basis. The \emph{standard weight} $μ_{\mr{std}}$ on $\mbR^n$ is the Lebesgue measure on $(\mbR^n)^\vee = \mbR^n$, where the identification is with respect to the dual of the standard basis. Equivalently, it is the weight defined by $e_1\wedge \ldots \wedge e_n\in \det(\mbR^n)$. In particular, if $n = 0$, then the standard weight is the Dirac measure of volume $1$.

\item \label{item:integral_on_polyhedron_in_R_n} Let $τ\subseteq \mbR^n$ be an $n$-dimensional polyhedron and $η\in PS_c^{n,n}(τ)$. Then $N_τ = \mbR^n$, so the standard weight $μ_{\mr{std}}$ may be viewed as a weight for $τ$. Define $\int_{[τ, μ_{\mr{std}}]} η$ by \eqref{eq:def_integral_R_n} (resp. its straightforward extension to measurable coefficient functions.)

\item \label{item:integral_on_weighted_polyhedron} Let $(σ, μ)$ be an $n$-dimensional weighted polyhedron and $η\in PS_c^{n,n}(σ)$. Pick any $n$-dimensional polyhedron $τ\subseteq \mbR^n$ together with an affine linear bijection $f:τ \iso σ$. Define
$$\int_{[σ, μ]} η := (\deg f)^{-1} \int_{[τ, μ_{\mr{std}}]} f^*(η)$$
where the degree is with respect to $μ_{\mr{std}}$ on $τ$. This is independent of choices by \eqref{eq:transformation_rule_integral}.
\end{enumerate}
\end{defn}

The following was already used in Def. \ref{def:integral} \ref{item:integral_on_polyhedron_in_R_n}. If $τ\subseteq σ$ is an inclusion of polyhedra of the same dimension, then $N_τ = N_σ$. Hence a weight $µ_σ$ for $σ$ induces a weight $µ_σ\vert_τ$ for $τ$ (and conversely). This makes the next definition feasible.

\begin{defn}\label{def:weight_pwl_space}
A family of weights for a purely $n$-dimensional polyhedral subset $C\subseteq \mbR^r$ is the datum of a locally finite decomposition $C = \bigcup_{i\in I} σ_i$ into $n$-dimensional polyhedra together with weights $µ = (µ_i)_{i\in I}$ for all $σ_i$ such that $µ_i\vert_{σ_i\cap σ_j} = µ_j\vert_{σ_i\cap σ_j}$ whenever $\dim σ_i\cap σ_j = n$; up to subdivision of the presentation of $C$. The pair $(C,µ)$ is called a \emph{weighted polyhedral subset}. (In particular, all our weighted polyhedral subsets are pure-dimensional by convention.) Def. \ref{def:integral} \ref{item:integral_on_weighted_polyhedron} extends to an integral operator
$$\int_{[C,µ]}\colon PS^{n,n}_c(C) \to \mbR.$$
Given $(C,µ)$ as above, a polyhedral subset $D\subseteq \mbR^s$ and an isomorphism $f:C \iso D$, there is a unique way to define a weight $f(µ)$ for $D$ such that for all $η\in PS^{n,n}_c(D)$,
$$\int_{[D, f(µ)]} η = \int_{[C, µ]} f^*η.$$
Concretely, choosing the above decomposition $C = \bigcup_{i\in I}σ_i$ so fine that all $f\vert_{σ_i}$ are linear, the image weights are $f(µ) = (f(µ_i))_{i\in I}$. (The images $f(µ_i)$ were defined in Def. \ref{def:weight_polyhedron} above.) In this way, the definition of weight extends to pure-dimensional abstract polyhedral sets and then further to pure-dimensional pwl spaces.
\end{defn}

%

\begin{defn}\label{def:integral_weighted_pwl_space}
Let $(X,μ)$ be a purely $n$-dimensional weighted pwl space. The integral
$$\int_{[X,μ]}:PS^{n,n}_c(X) \lr \mbR$$
is defined as follows. Choose a finite open covering $\Supp η \subseteq \bigcup_{i\in I} U_i$ by polyhedral sets (Lem. \ref{lem:nice_neighborhoods}). Let $(ρ_i)_{i\in I}$ be a subordinate pws partition of unity for $\bigcup_{i\in I} U_i$ (Prop. \ref{prop:part_of_1_pws}). The support of each product $ρ_iη$ is contained in $U_i$ and compact by assumption on $η$. Then put
$$\int_{[X,μ]} := \sum_{i\in I} \int_{[U_i, μ\vert_{U_i}]} ρ_iη.$$
We omit the argument that this is well defined.
\end{defn}


Let $X$ be an arbitrary pwl space. For each open subset $U\subseteq X$, consider the vector space $E_{p,q}(U) := \Hom(PS^{p,q}_c(U), \mbR)$. Put $E(U) = \bigoplus_{p,q} E_{p,q}(U)$. Given an inclusion of open subsets $V\subseteq U$, there is an inclusion map $PS^{p,q}_c(V) \to PS^{p,q}_c(U)$ whose dual defines a restriction map
\begin{equation}
E_{p,q}(U)\to E_{p,q}(V),\ \ell\mapsto \ell\vert_V.
\end{equation}

\begin{lem}\label{lem:sheaf_property_E}
On every pwl space $X$, the so-defined presheaf $E$ is a sheaf.
\end{lem}
\begin{proof}
Let $U = \bigcup_{i\in I} U_i$ be an open covering of an open subset $U\subseteq X$ and let $\ell_i\in E(U_i)$ be a family of linear forms with $\ell_i\vert_{U_{ij}} = \ell_j\vert_{U_{ij}}$ for all $i,j\in I$. (We have put $U_{ij} = U_i \cap U_j$ here.) Our task is to show that there is a unique $\ell\in E(U)$ with $\ell\vert_{U_i} = \ell_i$. Given an element $η\in PS_c(U)$, let $V$ be a paracompact open neighborhood of $\Supp η$ (Lem. \ref{lem:paracompact_neighborhood}) and let $(ρ_i)_{i\in I}$ be a partition of unity that is subordinate to $(V\cap U_i)_{i\in I}$. Then $η = \sum_{i\in I} ρ_iη$, so there is at most one way to define $\ell(η)$, namely $\ell(η) = \sum_{i\in I} \ell_{i}(ρ_iη)$. We omit the argument that this is well defined.
\end{proof}

For later application, we also note here that a linear form $\ell \in E(X)$ extends naturally to the subspace of those $η\in PS(X)$ that have the property that $S := \Supp(\ell)\cap \Supp(η)$ is compact. In order to define this value $\ell(η)$, choose any compactly supported pws function $λ\in PS^{0,0}_c(X)$ that is constantly $1$ on a neighborhood of $S$ and set
\begin{equation}\label{eq:extension_ell}
\ell(η) := \ell(λη).
\end{equation}
This is independent of the choice of $λ$.

Coming back to Def. \ref{def:integral_weighted_pwl_space}, given a pwl space $X$, a purely $d$-dimensional weighted closed pwl subspace $(Y,μ)$ and a pws form $α\in PS^{p,q}(Y)$, we have defined an element of $E_{d-p,d-q}(X)$:
\begin{equation}\label{eq:def_polyhedral_current}
α\wedge [Y, µ]: PS^{d-p,d-q}_c(X) \lr \mbR,\ \ \ η \longmapsto \int_{[Y,µ]} α\wedge η\vert_Y.
\end{equation}
We call such linear forms \emph{integration currents} or \emph{currents of integration}. (A general notion of \emph{currents} will be introduced later, cf. Def. \ref{def:currents_pwl_space}.)

\begin{defn}\label{def:polyhedral_current}
\begin{enumerate}[wide, labelindent=0pt, labelwidth=!, label=(\arabic*), topsep=4pt, itemsep=4pt]
\item Let $X$ be a pwl space and $U\subseteq X$ an open subset. An element $\ell\in E(U)$ is called a \emph{polyhedral current} if it has the following property. There exists an open covering $U = \bigcup_{i\in I} U_i$ such that each $\ell\vert_{U_i}$ is a locally finite sum of integration currents. Polyhedral currents by definition form a sheaf which we denote by $P$ or $P_X$. There is a grading $P = \oplus_{p,q,d} P_d^{p,q}$, where $P_d^{p,q}$ is generated by those $α\wedge[Y,µ]$ with $Y$ of dimension $d$ and $α$ of bidegree $(p,q)$. If the ambient space $X$ is purely of dimension $n$, we prefer the grading $P^{p,q,c} := P^{p,q}_{n-c}$ by codimension.

\item For $ω\in PS^{p_1,q_1}(X)$ and $T = α\wedge [Y, μ] \in P^{p_2,q_2,c}(X)$, we define
\begin{equation}\label{eq:def_wedge_pws_form}
\begin{aligned}
ω\wedge T &\ := (-1)^{p_1+p_2+q_1+q_2} T \wedge ω\\
&\ := (ω\wedge α)\wedge [Y, μ] \in P^{p_1+p_2, q_1+q_2,c}(X).
\end{aligned}
\end{equation}
This definition extends to a $\wedge$-product $PS\times P \to P$ by linear extension.

\item The integral of a compactly supported polyhedral current $T \in P^{d,d}_{d}(X)$ is defined as
$$\int_X T := T(1)$$
where the right hand side is understood in the sense of \eqref{eq:extension_ell}. In other words, if $T = α\wedge [Y, μ]$ as in \eqref{eq:def_polyhedral_current} with $α$ of bidegree $(\dim Y, \dim Y)$ and $\Supp(α\vert_Y)$ compact, then
$$\int_X T = \int_{[Y, μ]} α.$$
\end{enumerate}
\end{defn}

Every locally finite sum of integration currents defines a polyhedral current and the above definition sheafifies the resulting sub-presheaf of $E$. This sheafification is however unnecessary for paracompact open subsets:

\begin{lem}\label{lem:paracompact_polyhedral_current}
Assume that $U\subseteq X$ is a paracompact open subset. Then $P(U)$ equals the set of all $\ell\in E(U)$ that are a locally finite sum of integration currents.
\end{lem}
\begin{proof}
Given $\ell\in P(U)$, choose an open covering $U = \bigcup_{i\in I}U_i$ such that each $\ell_i := \ell\vert_{U_i}$ is a locally finite sum of integration currents. (Such a covering exists by definition.) Say
$$\ell_i = \sum_{j\in J_i} α_j\wedge [Y_j, μ_j].$$
Let $(ρ_i)_{i\in I}$ be a pws partition of unity that is subordinate to the covering, cf. Prop. \ref{prop:part_of_1_pws}. Then
$$\ell = \sum_{i\in I} \sum_{j\in J_i} (ρ_iα_j)\wedge [Y_j, μ_j]$$
where the sum is locally finite as required.
\end{proof}

Let now $C\subseteq \mbR^r$ be a polyhedral subset. Given $T\in P(C)$, we say that a polyhedral complex structure $\mcC$ for $C$ is \emph{subordinate to $T$} if we can write
$$T = \sum_{σ\in \mcC} α_σ\wedge [σ, μ_σ]$$
for some choice of weights $μ_σ$ for each $σ\in \mcC$ and smooth forms $α_σ\in A(σ)$. Every $T\in P(C)$ admits a subordinate polyhedral complex structure. Moreove, having $\mcC$ and the $μ_σ$ fixed, the coefficients $α_σ$ are uniquely determined.

We also need the boundary integrals from \cite{Mih_trop_inter}*{§2.4}. Let $σ\subseteq \mbR^r$ be a polyhedron with $\dim σ = n$ and let $τ\subset σ$ be a facet. (Recall that this means that $τ\subset σ$ is a face of codimension $1$.) Let $μ\in \det N_σ$ be a weight for $σ$ and let $ν\in \det N_τ$ be a weight for $τ$. (Recall that $μ$ and $ν$ are only defined up to sign.) It is possible to write $μ = n_{σ,τ}\wedge ν$ where $n_{σ,τ}\in N_σ\setminus N_τ$ points in direction of $σ$. Such a vector $n_{σ,τ}$ is called a \emph{normal vector} for $(τ,ν)\subset (σ,μ)$. Its image in $N_σ/N_τ$ is uniquely determined.

For $w\in N_σ\oplus N_σ$, there is a contraction (interior derivative)
$$A(σ) \lr A(σ),\ \ \ α\longmapsto (α, w)$$
that is characterized by the following properties. It is $C^\infty(σ)$-linear in $α$ and $\mbR$-linear in $w$, it satisfies the Leibniz rule (assuming $α$ homogeneous)
\begin{equation}\label{eq:contraction_Leibniz}
(α\wedge β, w) = (α, w) + (-1)^{\deg α}α\wedge (β, w),
\end{equation}
and it satisfies the relations
\begin{equation}\label{eq:contraction_basic}
(d'φ, (v_1,v_2)) = \frac{\partial φ}{\partial v_1},\ \ \ (d''φ, (v_1,v_2)) = \frac{\partial φ}{\partial v_2},\ \ \ φ\in C^\infty(σ),\ v_1,v_2\in N_σ.
\end{equation}
We use the following shorthand from now on,
$$v' = (v,0),\ \ \ v'' = (0, v),\quad v\in N_σ.$$
Coming back to our facet inclusion $(τ, ν)\subset (σ, μ)$, for any $α\in A_c^{n-1,n}(σ)$, the $(n-1,n-1)$-form $(α, n_{σ, τ}'')\vert_τ$ depends only on $ν$, not on the chosen normal vector. We define
\begin{equation}\label{eq:def_boundary_integral_facet}
\int_{\partial'_τ[σ,µ]} α := -\int_{[τ,ν]} (α, n_{σ,τ}'')\vert_{τ}
\end{equation}
which does not even depend on $ν$ anymore. We put
\begin{equation}\label{eq:def_boundary_integral}
\int_{\partial' [σ,µ]} α := \sum_{τ\subset σ\ \text{facet}}\int_{\partial'_τ[σ,µ]} α.
\end{equation}
Given $β\in A_c^{n,n-1}(σ)$, we similarly define
\begin{equation}\label{eq:def_boundary_integral_smooth_II}
\int_{\partial''_τ[σ,µ]} β := \int_{[τ,ν]} (β, n_{σ, τ}')\vert_τ,\ \ \ \int_{\partial''[σ, μ]} β := \sum_{τ\subset σ\ \text{facet}}\int_{\partial''_τ[σ, μ]} β.
\end{equation} 
\begin{prop}[\cite{CLD}*{Lemme 1.5.7}, Stokes' Theorem]\label{prop:stokes_polyhedron}
Let $(σ,µ)$ be a weighted polyhedron of dimension $n$, and let $α\in A^{n-1,n}_c(σ)$ and $β\in A^{n,n-1}_c(σ)$. Then
$$\int_{[σ,µ]}d'α = \int_{\partial'[σ,µ]} α,\ \ \ \int_{[σ,µ]} d''β = \int_{\partial''[σ,µ]} β.$$
\end{prop}

The boundary integral extends to pwl spaces and pws forms in the following way. Let $(C,μ)\subseteq \mbR^r$ be a weighted purely $n$-dimensional polyhedral subset and $η\in PS_c^{n-1,n}(C)$. Pick any polyhedral complex structure $\mcC$ that is subordinate to $μ$ and $η$ and define
$$\int_{\partial'[C,μ]} η = \sum_{σ\in \mcC_n} \int_{\partial'[σ,μ\vert_σ]} η\vert_σ.$$
This is well defined, i.e. independent of the choice $\mcC$, and extends to a definition for arbitrary $n$-dimensional weighted pwl spaces,
\begin{equation}
\int_{\partial'[X,μ]}:PS^{n-1,n}_c(X) \to \mbR.
\end{equation}
A similar definition is made for $\int_{\partial''[X,μ]}$. Stokes' Theorem carries over in the form
$$\int_{[X,µ]}d'_Pα = \int_{\partial'[X,µ]} α,\ \ \ \int_{[X,µ]} d''_Pβ = \int_{\partial''[X,µ]} β$$
for $α,β$ of compact support and in suitable bidegrees.

There is a fiber integration formalism that works as follows. Assume $f:σ\to \mbR^r$ is an affine linear map from a polyhedron. Its image $f(σ)$ is a polyhedron as well. Let $K := \ker(f:N_σ\to N_{f(σ)})$, put $d = \dim K$, and let $δ$ be a weight for $K$. Then fiber integration is the unique map $f_{δ,*}:PS_c(σ)\to PS(f(σ))$ that has the following two properties. First, if $α\in PS^{d,d}_c(σ)$, then $f_{δ,*}α\in PS^{0,0}(f(σ))$ is the pws function on $f(σ)$ that is pointwise given by
\begin{equation}
\label{eq:def_fiber_int_polyhedron}
(f_{δ,*}α)(y) := \int_{[f^{-1}(y), δ]} α\vert_{f^{-1}(y)},\ \ \ y\in f(σ).
\end{equation}
Strictly speaking, this only makes sense if the fiber $f^{-1}(y)$ is $d$-dimensional which is the case at least for $y\in f(σ)^\circ$. If the dimension of $f^{-1}(y)$ is $<d$, then the expression \eqref{eq:def_fiber_int_polyhedron} is defined to vanish. The second property of fiber integration is the projection formula
\begin{equation}\label{eq:proj_formula_fiber_int_polyhedron}
f_{δ,*}(α\wedge f^*β) = (f_{δ,*}α)\wedge β,\ \ \ \ β\in PS(f(σ)).
\end{equation}
In particular, fiber integration is bihomogeneous of bidegree $(-d,-d)$. Its significance in the theory is that it provides a construction of the pushforward of polyhedral currents. To this end, we first say more about fiber weights.

\begin{construction}\label{constr:fiber_weight}
Let $f:σ→f(σ)$ and $K$ be as before. Assume we are given two out of the following three: A weight $μ \in \det N_σ$ for $σ$, a weight $ν \in \det N_{f(σ)}$ for $f(σ)$ and/or a weight $δ\in \det K$ for $K$. Then there is a unique way to define the third weight such that the three weights are related by $μ = \wt{ν}\wedge δ$ (up to sign), where $\wt{ν}\in \bigwedge^{\dim f(σ)} N_σ$ is any lift of $ν$. If this relation holds, then we also denote $δ$ by $μ/ν$.

When phrased in terms of Haar measures as in Def. \ref{def:weight_polyhedron}, then this means that the measure on $K^\vee = M_σ/M_{f(σ)}$ is the quotient Haar measure.
\end{construction}

Let $f:X\to Y$ be a map of pwl spaces and let $\ell \in E(X)$ be such that $\Supp(\ell)\to Y$ is proper. Then, for every test form $η\in PS_c(Y)$, the intersection $\Supp(\ell)\cap f^{-1}(\Supp(η))$ is compact. Thus the construction in \eqref{eq:extension_ell} applies and allows to define the pushforward $f_*\ell \in E(Y)$ by
\begin{equation}\label{eq:push_forward_current}
(f_*\ell)(η) := \ell(f^*η).
\end{equation}

\begin{prop}\label{prop:push_of_polyhedral}
Let $f:X→Y$ be a map of pwl spaces and let $T\in P^{p,q}(X)$ be a polyhedral current such that $\Supp T → Y$ is proper. Then $f_*T \in E(Y)$ is again a polyhedral current.
\end{prop}
\begin{proof}
The property of $f_*T$ being polyhedral can be checked on an open covering of $Y$. In light of Lem. \ref{lem:nice_neighborhoods}, we may assume $Y$ to be a polyhedral set. By its very definition, $ρf_*T = f_*(f^*(ρ)T)$ for every pws function $ρ$ on $Y$. Writing $1 = \sum_{i\in I} ρ_i$ (locally finite) on $Y$ with each $ρ_i$ compactly supported, we may thus assume $\Supp(T)$ compact. Then we may replace $X$ by a paracompact open neighborhood of $\Supp(T)$ (Lem. \ref{lem:paracompact_neighborhood}) and, after the choice of a suitable partition of unity on $X$, assume $X$ to be a polyhedral set (Lem. \ref{lem:nice_neighborhoods}).

Let $\mcX$ and $\mcY$ be polyhedral complex structures for $X$ and $Y$ that are subordinate to both $f$ and $T$, say
$$T = \sum_{σ\in \mcX} α_σ\wedge [σ, μ_σ].$$
For each $σ$, pick a weight $ν_σ$ for $f(σ)$. The projection formula \eqref{eq:proj_formula_fiber_int_polyhedron} implies that
\begin{equation}
\label{eq:push_forward_polyhedral_set}
f_*T = \sum_{σ\in \mcX} f_{μ_σ/ν_σ,*}(α_σ)\wedge [f(σ), ν_σ],
\end{equation}
the sum being finite by the compact support assumption on $T$. This completes the proof.
\end{proof}

We note in passing that the individual summands in \eqref{eq:push_forward_polyhedral_set} are independent of the choices $ν_σ$.

\subsection{Flatness}
\label{ss:flatness}

Let $(X,µ)$ and $(Y,ν)$ be weighted pwl spaces that are purely of dimensions $n$ and $m$, respectively, and $f:X→Y$ a pwl map. We have seen that pws forms admit pullbacks for $f$ while polyhedral currents admit pushforwards. As $X$ is weighted, every pws form $η\in PS(X)$ defines a polyhedral current, namely $η\wedge [X,µ]$. Assuming that $\Supp η$ is proper over $Y$, one may ask for the existence of a pws form $f_{μ/ν,*}η\in PS(Y)$ such that $f_*(η\wedge [X,µ]) = (f_{μ/ν,*}η)\wedge [Y, ν]$. Note that $f_{μ/ν,*}η$ is uniquely determined if it exists.
\begin{defn}\label{def:flat}
The map $f$ is \emph{flat} if the pws form $f_{μ/ν,*}η$ exists for every $η\in PS_c(X)$. The form $f_{μ/ν,*}η$ is then called the \emph{pushforward} or the \emph{fiber integral} of $η$ along $f$.
\end{defn}
We remark that if $f$ is flat, then $f_{μ/ν,*}η$ exists for every $η$ with support proper over $Y$. This can be shown by the same technique as during the proof of Prop. \ref{prop:push_of_polyhedral}.

We next give a different definition that will turn out to be equivalent.
\begin{defn}\label{def:fiber_weight}
Assume for a moment that $X$ and $Y$ are polyhedral sets. The map $f$ is said to admit \emph{well defined fiber weights} if $\dim f^{-1}(y) \leq n-m$ for all $y\in Y$ and if the following condition is satisfied. Choose polyhedral complex structures $\mcX$ and $\mcY$ for $X$ and $Y$ that are subordinate to $µ$, $ν$ and $f$. Given ${\ob{σ}}\in \mcY_m$ and $σ\in \mcX_n$ with $f(σ) = {\ob{σ}}$, Constr. \ref{constr:fiber_weight} defines a fiber weight $μ_σ/ν_{\ob{σ}}$ on $\ker(f:N_σ\to N_{\ob{σ}})$. The condition is now that for every $τ\in \mcX$ with $\dim τ - \dim f(τ) = n - m$, the following sum is independent of the choice of ${\ob{σ}}$, where ${\ob{σ}}\in \mcY_m$ denotes any maximal dimensional polyhedron with $f(τ)\subseteq {\ob{σ}}$,
\begin{equation}\label{eq:fiber_weight_summed}
(μ/ν)_{τ,{\ob{σ}}} := \sum_{σ\in \mcX_n,\ τ\subseteq σ,\ f(σ) = {\ob{σ}}} μ_σ/ν_{\ob{σ}}.
\end{equation}
Here, $\ker(N_τ\to N_{f(τ)}) = \ker(N_σ \to N_{\ob{σ}})$ for all occurring $σ$ because of our dimension assumptions and the above sum is understood as fiber weight for $f$ on $τ$. In case \eqref{eq:fiber_weight_summed} is independent of $\ob{σ}$, then we denote this weight by $(μ/ν)_τ$ or simply $μ/ν$.

One checks that the above condition is independent of the choice of $\mcX$ and $\mcY$. In fact, it depends only on $(X,µ)$ and $(Y,ν)$ up to pwl isomorphism. We may hence extend it to maps $f:X \to Y$ of general weighted pwl spaces $(X,µ)$ and $(Y,ν)$. More precisely, we say that $f$ admits \emph{well defined fiber weights} if $\dim f^{-1}(y) \leq n-m$ for all $y\in Y$ and if for every $x\in X$ the above condition holds on polyhedral open neighborhoods of $x$ and $f(x)$.
\end{defn}

\begin{ex}\label{ex:push_forward_does_not_preserve_pws}
\begin{enumerate}[wide, labelindent=0pt, labelwidth=!, label=(\arabic*), topsep=4pt, itemsep=4pt]
\item Assume that $\dim X = \dim Y$. A map of weighted pwl spaces $f:(X,µ)\to (Y,ν)$ is said to be \emph{generically of degree $λ \in \mbR_{\geq 0}$} if $f_*[X,µ] = λ [Y,ν]$ and we write $\deg f = λ$ in this case. (This is the usual concept of degree in tropical geometry, cf. \cite{ST}*{(1.2)}. Note that the definition does not imply fibers to be $0$-dimensional. In fact, if all fibers of $f$ are positive dimensional, then $f$ is generically of degree $0$ by definition.) With this terminology, $f$ has well defined fiber weights if and only if all its fibers are $0$-dimensional or empty and if for each $x\in X$, the induced map of weighted polyhedral fans $f_x:X_x→Y_{f(x)}$ is generically of some degree $\deg_x(f)$. In case $X$ and $Y$ are $1$-dimensional, this coincides with the definition of harmonic maps of metrized graphs in \cite{ABBR}*{Def. 2.6}.

\item Consider $\mbR^2$ with coordinates $x_1,x_2$ and the following map of pwl spaces of dimension $1$,
\begin{equation}\label{eq:key_example}
f: X = \{x_1 = 0\} \sqcup \{x_2 = 0\} → Y = \{x_1 \cdot x_2 = 0\}.
\end{equation}
All occurring sets here are endowed with their standard weights, denoted by $μ$ resp. $ν$ again. Then the fiber $f^{-1}\big((0,0)\big)$ has no well defined weight. The map is also not flat above the origin. Namely consider the pws function $ϕ$ which is $1$ on $\{x_1 = 0\}$ and $0$ on $\{x_2 = 0\}$. Then $f_*(ϕ\wedge [X, μ])$ is the coordinate axis $[\{x_1 = 0\}, ν\vert_{\{x_1 = 0\}}]$ and is not pws. Relatedly, there seems to be no natural way to pullback the Dirac measure $[\{(0,0)\}, \mr{std}]$ from $\{x_1\cdot x_2 = 0\}$. In particular, having only the relation $f_*[X,µ] = \deg(f)[Y,ν]$ in the first example is insufficient.

\item A flat map (resp. a map that has well defined fiber weights) need not have fibers of exactly dimension $n-m$. Consider for example the polyhedron $σ = \{0\leq x_2 \leq x_1\}\subseteq \mbR^2$ and the projection map $f:σ\to \mbR,\ (x_1,x_2)\mapsto x_1$. Then $f$ is flat and admits well defined fiber weights.
\end{enumerate}
\end{ex}

\begin{prop}\label{prop:flatness_characterization}
Let $(X, μ)$ and $(Y, ν)$ be weighted pwl spaces of dimensions $n$ and $m$. A pwl map $f:(X,µ)→(Y,ν)$ is flat if and only if it admits well defined fiber weights.
\end{prop}
\begin{proof}
Def. \ref{def:fiber_weight} is local on $X$ and $Y$. Similarly, for a compactly supported pws form $α\in PS_c(X)$, the pushforward $f_*(α\wedge [X, μ])$ may be computed locally with the help of partitions of unity (cf. the proof of Prop. \ref{prop:push_of_polyhedral}). Applying Lem. \ref{lem:nice_neighborhoods}, we are reduced to the case where $X$ and $Y$ are polyhedral sets.

Choose polyhedral complex structures $\mcX$ resp. $\mcY$ that are subordinate to $μ$, $ν$ and $f$. In particular, the datum $ν$ is equivalent to that of weights $ν_{\ob{σ}}$ for every ${\ob{σ}}\in \mcY_m$. As an auxiliary datum, we also choose weights $ν_ρ$ for all $ρ\in \mcY_{<m}$. We obtain the general formula
$$f_*(α\wedge [X,μ]) = \sum_{{\ob{σ}}\in \mcY}\ \sum_{σ\in \mcX_n,\ f(σ) = {\ob{σ}}} f_{μ_σ/ν_{\ob{σ}}, *}(α\vert_σ)\wedge [{\ob{σ}},ν_{\ob{σ}}].$$
Assume that there is some $σ\in \mcX_n$ with $\dim f(σ) < m$. Its interior $σ^\circ$ is open in $X$. There always exists some bump form $α\in A_c(σ^\circ) \subseteq PS_c(X)$ such that $f_{μ_σ/ν_{\ob{σ}},*}(α) \neq 0$. This shows that $f$ can only be flat if $\dim f^{-1}(y) \leq n-m$ for all $y\in Y$.

Suppose that this dimension condition holds. Then $$f_*(α\wedge [X, μ]) = \sum_{{\ob{σ}}\in \mcY_m} β_{\ob{σ}}\wedge [{\ob{σ}},ν_{\ob{σ}}]$$
for certain (uniquely determined) $β_{\ob{σ}}\in PS({\ob{σ}})$ and the question is for when these paste to a pws form, i.e. satisfy
$$β_{{\ob{σ}}_1}\vert_{\ob {τ}} = β_{{\ob{σ}}_2}\vert_{\ob {τ}}\quad \forall {\ob{σ}}_1,{\ob{σ}}_2\in \mcY_m,\ \ob {τ} = {\ob{σ}}_1\cap {\ob{σ}}_2.$$
These restrictions are determined as follows. If $σ \in \mcX_n$, $f(σ) = {\ob{σ}}$ and $\ob {τ}\subseteq {\ob{σ}}$ a face, then $τ = f^{-1}(\ob {τ}) \cap σ$ is a face of $σ$. One has $f_{μ_σ/ν_{\ob{σ}},*}(α\vert_σ)\vert_{\ob {τ}} = 0$ if $\dim τ - \dim \ob{τ}< n - m$ by definition of the fiber integral. If however $\dim τ - \dim \ob{τ}= n-m$, then
$$f_{μ_σ/ν_{\ob{σ}},*}(α\vert_σ)\vert_{\ob {τ}} = f_{μ_σ/ν_{\ob{σ}},*}(α\vert_{τ}).$$
Here, we have used that $μ_σ/ν_{\ob{σ}}$ defines a fiber weight for $τ\to \ob{τ}$ because $\ker(N_τ\to N_{\ob{τ}}) = \ker(N_σ\to N_{\ob{σ}})$ for dimension reasons. Hence
\begin{equation}
\label{eq:restriction_flatness}
β_{\ob{σ}}\vert_{\ob {τ}} = \sum_{τ\in f^{-1}(\ob {τ}),\ \dim τ - \dim \ob{τ} = n-m} f_{(μ/ν)_{τ, \ob{σ}},*}(α\vert_{τ})
\end{equation}
with the weights $(μ/ν)_{τ,\ob{σ}}$ from \eqref{eq:fiber_weight_summed}. Assuming that $f$ admits well defined fiber weights this is independent of ${\ob{σ}}$. Hence $f$ is flat.

Assume conversely that flatness is given. For each $τ\in \mcX$ with the property $\dim τ - \dim f(τ) = n - m$, pick some $y\in f(τ)^\circ$ and choose a compactly supported $(n-m,n-m)$-form $α$ with $\Supp (α)\cap f^{-1}(y)\subseteq τ^\circ$ and such that $\int_{τ\cap f^{-1}(y)} α \neq 0$. Then the independence of $(μ/ν)_{τ,{\ob{σ}}}$ from ${\ob{σ}}$ is implied by the independence of $β_{\ob{σ}}\vert_{\ob{τ}}$ from ${\ob{σ}}$ in \eqref{eq:restriction_flatness}.
\end{proof}

\begin{prop}\label{prop:pull_back_polycurrent_flat}
Let $f:(X,µ)→(Y,ν)$ be a flat map of pure-dimensional weighted pwl spaces and let $T\in P(Y)$ be a polyhedral current. Define a linear form $f^*T \in E(X)$ by
\begin{equation}\label{eq:def_pullback_flat}
(f^*T)(α) = T(f_{μ/ν,*}(α)),\ \ \ \ α\in PS_c(X).
\end{equation}
Then $f^*T$ is again a polyhedral current. Moreover, $f^*T$ satisfies the projection formula
\begin{equation}\label{eq:projection_formula_flat_pullback}
f_*(α\wedge f^*T) = f_{μ/ν,*}(α) \wedge T,\ \ \ \ α\in PS_c(X).
\end{equation}
\end{prop}
\begin{proof}
We first prove that $f^*T$ is polyhedral. After similar localization arguments as in previous proofs, we may assume that $X$ and $Y$ are (pure-dimensional, weighted) polyhedral sets. Let $\mcX$ and $\mcY$ be respective polyhedral complex structures that are subordinate to $μ$, $ν$, $f$ and $T$. Using the linearity in $T$, we may assume that $T = β\wedge [\ob {τ}, ε]$ for some $\ob {τ}\in \mcY$ and $β\in PS(\ob {τ})$. Let $\ob{σ}\in \mcY_m$ be any maximal dimensional polyhedron with $\ob {τ}\subseteq \ob{σ}$. By definition, for every $α\in PS_c(U)$,
\begin{equation}\label{eq:pull_back_description_flat}
T(f_{μ/ν,*}α) = \sum_{σ\in \mcX_n,\ f(σ) = \ob{σ}} \int_{[\ob {τ},ε]} β\wedge f_{μ_σ/ν_{\ob{σ}},*}(α\vert_σ)\vert_{\ob {τ}}.
\end{equation}
Fix one such $σ$ and put $τ = f^{-1}(\ob {τ})\cap σ$ which is a face of $σ$. If $τ$ satisfies $\dim τ - \dim \ob {τ} < n-m$, then $f_{μ_σ/ν_{\ob{σ}},*}(α\vert_σ)\vert_{\ob{τ}} = 0$. Otherwise we have that $\dim τ - \dim \ob {τ} = n- m$ and $\ker(N_τ\to N_{\ob{τ}}) = \ker(N_σ\to N_{\ob{σ}})$. In this case,
$$f_{μ_σ/ν_{\ob{σ}},*}(α\vert_σ)\vert_{\ob{τ}} = f_{μ_σ/ν_{\ob{σ}},*}(α\vert_τ).$$
Reordering the sum in \eqref{eq:pull_back_description_flat}, we obtain
\begin{equation}\label{eq:pull_back_description_flat_II}
\begin{aligned}
T(f_{μ/ν,*}α) &\ = \sum_{τ\in f^{-1}(\ob {τ}),\ \dim τ - \dim \ob {τ} = n-m} \int_{[\ob{τ}, ε]} β\wedge (f_{(μ/ν)_τ, *}α\vert_{τ})\\
&\ = \sum_{τ\in f^{-1}(\ob {τ}),\ \dim τ - \dim \ob {τ} = n-m} \int_{[τ, ε\wedge (μ/ν)_τ]} f^*β\wedge α\vert_{τ}.
\end{aligned}
\end{equation}
The second equality here is by the projection formula for fiber integration \eqref{eq:proj_formula_fiber_int_polyhedron}.
In other words, the pullback is polyhedral as claimed,
\begin{equation}\label{eq:pull_back_formula_flat}
f^*(T) = \sum_{τ\in f^{-1}(\ob {τ}),\ \dim τ - \dim \ob {τ} = n-m} f^*β\wedge [τ, ε\wedge (μ/ν)_τ].
\end{equation}
We turn to the projection formula which may similarly be proved under the assumption that $X$ and $Y$ are polyhedral sets. Continuing from \eqref{eq:pull_back_formula_flat}, we have
$$f_*(α\wedge f^*T) = \sum_{τ\in f^{-1}(\ob {τ}),\ \dim τ - \dim \ob {τ} = n-m} f_*(α \wedge f^*β\wedge [τ, ε\wedge (μ/ν)_τ]).$$
By \eqref{eq:push_forward_polyhedral_set}, this expression equals
\begin{equation}\label{eq:proof_projection_flat}
\sum_{τ\in f^{-1}(\ob {τ}),\ \dim τ - \dim \ob {τ} = n-m} f_{(μ/ν)_τ,*}(α \wedge f^*β\vert_τ) \wedge [\ob{τ}, ε].
\end{equation}
Fiber integration satisfies the projection formula by definition \eqref{eq:proj_formula_fiber_int_polyhedron}, meaning for each occurring $τ$ we have the following identity in $PS(\ob{τ})$:
$$f_{(μ/ν)_τ,*}(α \wedge f^*β\vert_τ) = f_{(μ/ν)_τ,*}(α\vert_τ) \wedge β.$$
Substituting this in \eqref{eq:proof_projection_flat} we obtain
\begin{equation}
\begin{split}
f_*(α\wedge f^*T) & = \left(\sum_{τ\in f^{-1}(\ob {τ}),\ \dim τ - \dim \ob {τ} = n-m} f_{(μ/ν)_τ,*}(α\vert_τ)\right) \wedge β \wedge [\ob{τ}, ε] \\
& = f_{μ/ν,*}(α)\wedge T
\end{split}
\end{equation}
which is what we wanted to prove.
\end{proof}

\begin{defn}\label{def:faithfully_flat}
A flat map $f:(X,μ)\to (Y,ν)$ of weighted pwl spaces is called \emph{faithfully flat} if no fiber is completely degenerate, i.e. no fiber satisfies $\dim f^{-1}(y) < n - m$. In particular, $f$ is surjective.
\end{defn}

\begin{prop}\label{prop:flat_maps_faithful}
Let $f:(X,μ)\to (Y,ν)$ be a faithfully flat map of weighted pwl spaces.
\begin{enumerate}[leftmargin=*, label=(\arabic*), topsep=4pt, itemsep=4pt]
\item The pullback map of polyhedral currents, $f^*:P(Y)\to P(X)$, is injective.

\item If $g:(Y,ν)\to (Z,δ)$ is a map of weighted pwl spaces such that $g\circ f$ is flat, then $g$ itself is flat.
\end{enumerate}
\end{prop}
\begin{proof}
The two statements hold for $Y$ if and only if they hold for all maps $f^{-1}(U_i)\to U_i$ from an open covering $Y = \bigcup_{i\in I} U_i$. Applying Lem. \ref{lem:nice_neighborhoods}, we may assume that $Y$ is a polyhedral set for the rest of the proof. We first show the following claim: There exists a pws $(n-m,n-m)$-form $ω$ on $X$, with proper support over $Y$, such that $f_{μ/ν,*}(ω) = 1$ (constant function).

Assume for a moment that, for each point $y\in Y$, we have constructed a pws $(n-m,n-m)$-form $ω_y$ on $X$ with proper support over $Y$ such that $(f_{μ/ν,*}ω_y)(y) > 0$. By continuity, $y$ has an open neighborhood $U_y$ on which $f_{μ/ν,*}ω_y$ is pointwise $>0$. Using that $Y$ is paracompact, let $(ρ_y)_{y\in Y}$ be a subordinate pws partition of unity (cf. Prop. \ref{prop:part_of_1_pws}). By the projection formula \eqref{eq:projection_formula_flat_pullback}, $\wt{ω} = \sum_{y\in Y} f^*ρ_y \cdot ω_y$ has the property that $f_{μ/ν,*}(\wt{ω})$ is pointwise $>0$. Then $ω = \wt{ω}/f^*f_{μ/ν,*}(\wt{ω})$ has the desired property $f_{μ/ν,*}(ω) = 1$. Thus we are reduced to constructing the $ω_y$.

By faithful flatness, there exists a point $x\in f^{-1}(y)$ with the property $\dim_x f^{-1}(y) = n-m$. Replacing $X$ by a polyhedral open neighborhood of $x$, we may assume $X$ to be a polyhedral set. Let $\mcX$ and $\mcY$ be polyhedral complex structures for $X$ and $Y$ that are subordinate to $μ$, $ν$ and $f$. Also assume $\{y\}\in \mcY$. Pick a polyhedron $τ\in f^{-1}(\{y\})$ of maximal dimension $n - m$. Let $ω_0 \in A^{n-m, n-m}_c(τ^\circ)$ be any smooth form with compact support contained in its interior and such that $\int_{[τ,μ/ν]} ω_0 > 0$. The set $\bigcup_{τ'\in f^{-1}(\{y\}),\ τ'\neq τ} τ'$ is closed because it is itself a polyhedral set. (This uses that $τ$ has maximal dimension $n-m$.) It moreover equals $f^{-1}(y) \setminus τ^\circ$. There exists a pws $(n-m, n-m)$-form $\wt{ω}_0$ on $X$ with compact support away from $(f^{-1}(y) \setminus τ^\circ)$ and such that $\wt{ω}_0\vert_{τ^\circ} = ω_0$. Put $ω_y := \wt{ω}_0\vert_X$, which has compact support and satisfies $ω_y\vert_{f^{-1}(y)} = ω_0$. By construction, $(f_{μ/ν,*}ω_y)(y) = \int_{[τ,μ/ν]} ω_0 > 0$ and our claim is proven.

We are now ready to prove the proposition. For statement (1), the form $ω$ and the projection formula \eqref{eq:projection_formula_flat_pullback} allow to define a section
$$P(X) \lr P(Y),\quad T\longmapsto f_{μ/ν,*}(ω\wedge T)$$
for the pullback $f^*$:
$$f_{μ/ν,*}(ω\wedge f^*T) = f_{μ/ν,*}(ω) \wedge T = T,$$
hence $f^*$ is injective. Statement (2) in turn is deduced as follows. For every $η\in PS_c(Y)$,
$$g_*(η\wedge [Y, ν]) = g_*(f_{μ/ν,*}(ω)\wedge η \wedge [Y, ν]) = (g\circ f)_*(ω\wedge f^*η \wedge [X,μ])$$
is pws by the flatness of $g\circ f$. Hence $g$ is flat.
\end{proof}

\subsection{Smooth forms}
\label{ss:smooth_forms}

Recall the following notation for spaces with functions from \S\ref{ss:pwl_spaces}. Given a subsheaf $Λ\subseteq C^0_X$ of the sheaf of continuous functions on some topological space $X$ and a subset $i:Y\hookrightarrow X$, we denote by $Λ\vert_Y$ the image $\mr{Im}(i^{-1}Λ\to C^0_Y)$. We also write $Λ(Y) = Λ\vert_Y(Y)$.

\begin{defn}\label{def:sheaf_of_linear_functions}
A \emph{sheaf of linear functions} on a pwl space $X$ is a subsheaf of $\mbR$-vector spaces $L\subseteq Λ_X$ with the following three properties.
\begin{enumerate}[wide, labelindent=0pt, labelwidth=!, label=(\arabic*), topsep=4pt, itemsep=4pt]
\item It contains the locally constant functions.

\item For every point $x\in X$, there exists a polyhedral neighborhood $K$ and $f= (f_1,\ldots,f_r)\in L(K)^r$ such that $f:K\to \mbR^r$ has finite fibers.

\item For every polyhedral set $K\subseteq X$, the restriction $L\vert_K$ is locally finitely generated in the following sense. There is a polyhedral complex structure $\mcK$ for $K$ such that for every open subset $U\subseteq K$, for every $ϕ\in L(U)$, for every $σ\in \mcK$ and every polyhedron $τ\subseteq σ\cap U$, the restriction $ϕ\vert_τ$ is affine linear. In this situation, we call $\mcK$ \emph{subordinate} to $L\vert_K$.
\end{enumerate}
\end{defn}

Property (1) is just for convenience. Property (2) is equivalent to demanding that there exists a locally finite covering $X = \bigcup_{i\in I}K_i$ by polyhedral sets such that, for each $i$, there is some a tuple $f = (f_1,\ldots,f_r)\in L(K_i)^r$ such that $f:K_i\to \mbR^r$ is injective. Property (3) is imposed to make sure that the boundary calculus of polyhedral currents works nicely, cf. Ex. \ref{ex:non_standard_R} below. Note that one consequence of (3) is that whenever $K\subseteq X$ is a compact polyhedral set, then $L(K)$ is finite dimensional.

A \emph{linear} map $f:(X,L_X)→(Y,L_Y)$ of pwl spaces with linear functions  is a pwl map $f:X→Y$ such that the natural map $f^{-1}(L_Y) \to Λ_X$ factors through $L_X$. In other words, we consider pwl spaces with linear functions as a full subcategory of spaces with functions.

\begin{ex}
\begin{enumerate}[wide, labelindent=0pt, labelwidth=!, label=(\arabic*), topsep=4pt, itemsep=4pt]
\item Let $\mr{Aff}$ be the sheaf of affine linear functions on $\mbR^r$ in the usual sense. Then $(\mbR^r, \mr{Aff})$ is a pwl space with linear functions. We always view $\mbR^r$ as endowed with this standard structure and suppress $\mr{Aff}$ in the notation. In particular, a linear map $x:(X,L)→\mbR^r$ from a pwl space with linear functions $(X,L)$ is the same as an $r$-tuple $(x_1,\ldots,x_r)\in L(X)^r$ of linear functions.

\item Let $(X, L)$ be a pwl space with linear functions and let $Y\subseteq X$ be a pwl subspace. Then $(Y, L\vert_Y)$ is a pwl space with linear functions.
\end{enumerate}
\end{ex}

\begin{defn}\label{def:smooth_forms_pwl}
Let $(X,L)$ be a pwl space with linear functions. A \emph{smooth $(p,q)$-form} on $X$ is a pws form that is locally of the form $x^*(α)$ for a linear map $x:U→\mbR^r$ and a smooth $(p,q)$-form $α$ on $\mbR^r$. We write $A^{p,q}$ or $A^{p,q}_X$ for the resulting sheaf of smooth $(p,q)$-forms on $X$. Smooth forms are stable under the $\wedge$-product and the polyhedral derivatives $d'_P,d''_P$. We denote these simply by $d'$ and $d''$ for smooth forms.
\end{defn}
Given a linear map $f:(X,L_X)→(Y,L_Y)$ and a smooth form $α\in A(Y)$, the pullback $f^*α$ is defined as the pullback of $α$ viewed as pws form. It is again smooth since $f^*(x^*α) = (x\circ f)^*α$ for any tuple $x\in L_Y(Y)^r$ and any smooth form $α$ on $\mbR^r$.

It is seen just as in the piecewise smooth case that there exist smooth partitions of unity on paracompact spaces.
\begin{prop}[cf. Prop. \ref{prop:part_of_1_pws}]\label{prop:part_of_1_smooth}
Let $(X, L)$ be a paracompact pwl space with linear functions. Given an open covering $X = \bigcup_{i\in I} U_i$, there exists a family $(ρ_i)_{i\in I}$ of nonnegative smooth functions, with $\mr{Supp}(ρ_i)\subseteq U_i$ for every $i\in I$, and such that the sum $\sum_{i\in I} ρ_i$ is locally finite and equals $1$.
\end{prop}

We endow each $A_c(U)$, $U\subseteq X$ open, with the topology from \cite{GK}*{§6.1}: A sequence (resp. net) $(α_i)_{i\in I}$ converges to $α$ if and only if the following conditions hold. There exist finitely many compact polyhedral sets $K_1,\ldots,K_r\subseteq X$ with presentations $K_j\iso \bigcup_{k\in K_j} σ_{jk} \subseteq \mbR^{r_j}$ as finite unions of polyhedra such that
\begin{enumerate}[leftmargin=*, label=(\arabic*), topsep=4pt, itemsep=4pt]
\item The supports $\Supp α_i,\ \Supp α$ are all contained in $K_1\cup\ldots\cup K_r$.
\item Each restriction $α_i\vert_{σ_{jk}},\ i\in I,\ j\in \{1,\ldots,r\},\ k\in K_j$ is smooth.
\item The sequences (nets) $(α_i\vert_{σ_{jk}})_{i\in I}$ converge to $α\vert_{σ_{jk}}$ in the Schwartz sense, i.e. all higher partial derivatives of all coefficients of $α_i\vert_{σ_{jk}}$ converge uniformly.
\end{enumerate}

\begin{defn}\label{def:currents_pwl_space}
Let $(X,L)$ be a purely $n$-dimensional pwl space with linear functions.

\begin{enumerate}[wide, labelindent=0pt, labelwidth=!, label=(\arabic*), topsep=4pt, itemsep=4pt]
\item The sheaf of \emph{$(p,q)$-currents} $D^{p,q}$ on $X$ is defined as
$$D^{p,q}(U) := \Hom_{\mr{cont}}(A^{n-p,n-q}_c(U),\mbR).$$
The restriction maps $D^{p,q}(U)\to D^{p,q}(V)$ here are given as the duals to the inclusions $A_c(V)\to A_c(U)$ for $V\subset U$. Using smooth partitions of unity (Prop. \ref{prop:part_of_1_smooth}), the sheaf property of $D^{p,q}$ can be shown just as in the case of $E$ (Lem. \ref{lem:sheaf_property_E}). We write $D = \bigoplus_{p,q} D^{p,q}$ for the sheaf of all currents. Elements of $D^{p,q}$ are also said to be of \emph{bidegree $(p,q)$}. Their degree is $p + q$.

\item It is endowed with partial derivatives $d'$ and $d''$ which are defined by duality as
\begin{equation}\label{eq:degree_convention_currents}
(dT)(η) = (-1)^{\deg T + 1}T(dη),\ \ \ d\in \{d',d''\}.
\end{equation}
Here, $T$ is assumed to be homogeneous. There is an associative pairing
$$A^{p_1,q_1}\times D^{p_2,q_2} → D^{p_1+p_2,q_1+q_2},\ (α\wedge T)(η) := (-1)^{\deg α\deg T} T(α\wedge η).$$
In light of the degree convention \eqref{eq:degree_convention_currents}, this wedge product satisfies the Leibniz rule (assuming $α$ and $T$ homogeneous)
\begin{equation}\label{eq:Leibniz_current}
d(α\wedge T) = dα \wedge T + (-1)^{\deg α} α\wedge dT,\ \ \ d\in \{d', d''\}.
\end{equation}

\item The support $\mr{Supp} (T)$ of a current $T\in D(U)$ is defined as the intersection of all closed subsets $Z\subseteq U$ with $T\vert_{U\setminus Z} = 0$. Given a linear map $f:(X, L_X)→(Y,L_Y)$ of pure-dimensional pwl spaces with linear functions and a current $T\in D(X)$ with $\mr{Supp}(T)$ proper over $Y$, there is a \emph{pushforward} $f_*(T)$ that is uniquely characterized by the projection formula
\begin{equation}\label{eq:projection_current}
(f_*T)(η) = T(f^*η),\ \ \ η\in A_c(Y).
\end{equation}
This pushforward satisfies $d(f_*T) = f_*(dT)$, $d\in \{d',d''\}$, because we have $d(f^*η) = f^*(dη)$ for all test forms $η$.
\end{enumerate}
\end{defn}

Every polyhedral current defines a current. It is easily checked that this realizes $P_X(U)$ as a subset of $D_X(U)$. In particular, this defines derivatives $d'T$ and $d''T$ of any $T\in P_X(U)$. On the other hand, one may also extend the polyhedral derivatives $d'P,\ d''_P$ (cf. \S\ref{ss:pws_polyehdra} (18)) from pws forms to polyhedral currents. Namely, one defines
\begin{equation}\label{eq:def_polyhedral_derivatives_current}
d_P(α\wedge [Z, ε]) = (d_Pα)\wedge [Z, ε],\quad d_p\in \{d_P', d_P''\}
\end{equation}
which extends in a well defined way to all polyhedral currents.

The derivatives $d'$ and $d''$ need not agree with $d'_PT$ and $d''_PT$, see e.g. \cite{Mih_trop_inter}*{Ex. 2.10}. We next study those polyhedral currents $T$ such that $d'T$ and $d''T$ are again polyhedral. Their characterization through the balancing condition carries over from $\mbR^n$, cf. \cite{Mih_trop_inter}.

\begin{defn}\label{def:balanced_polyhedral_current}
A polyhedral current $T\in P(X)$ is \emph{balanced} if the following condition is met. Assume first that $X$ is a polyhedral set with a polyhedral complex structure $\mcX$ that is subordinate to $T$ and $L$. This means we can write $T = \sum_{σ\in \mcX} α_σ\wedge [σ,μ_σ]$. Then for all $U\subseteq X$ open, $τ\in \mcX$ and $ϕ\in L(U)$ with $ϕ\vert_{τ\cap U}$ constant,
\begin{equation}\label{eq:balanced_polyhedral_current}
\sum_{τ\subset σ\in \mcX\ \text{a facet}} \frac{\partial ϕ\vert_{σ\cap U}}{\partial n_{σ,τ}}α_σ\vert_{τ\cap U} = 0.
\end{equation}
The vectors $n_{σ,τ}\in N_σ$ here are normal vectors with respect to the weights $μ_σ$ and $μ_τ$. That is, they have the two properties that they point from $N_τ$ in the direction of $σ$ and satisfy $μ_σ = μ_τ\wedge n_{σ,τ}$ (up to sign).

The balancing condition is independent of the choice of $\mcX$. For general $X$, we demand that $T$ is balanced on the interior of all open polyhedral sets contained in $X$ (Lem. \ref{lem:nice_neighborhoods}).
\end{defn}

\begin{ex}\label{ex:non_standard_R}
\begin{enumerate}[wide, labelindent=0pt, labelwidth=!, label=(\arabic*), topsep=4pt, itemsep=4pt]
\item The balanced polyhedral currents on $\mbR^n$ of bidgree $(p,q,0)$ are precisely the pws forms of bidegree $(p,q)$ (cf. \cite{Mih_trop_inter}*{Ex. 3.8}). Here, pws $(p,q)$-forms are identified with polyhedral $(p,q,0)$-currents by $α\mapsto α\wedge [\mbR^n, μ_{\mr{std}}]$.

Furthermore, polyhedral currents of bidegree $(n,q)$ or $(p, n)$ (in the sense of currents) are always balanced for degree reasons.

\item Let now $X = (\mbR,\ \mbR + \mbR\cdot x + \mbR\cdot \max\{0,x\})$ be the real line with a non-standard sheaf of linear functions ($x$ denotes the usual coordinate). Then a polyhedral current of the form $ψ\wedge [\mbR, μ_{\mr{std}}]$, where $ψ$ is a pws function on $\mbR$, is balanced if and only if $ψ(0) = 0$. The principle here is that there are more linear functions on $X$ than there are on $\mbR$, so there are more $ϕ$ in \eqref{eq:balanced_polyhedral_current} against which one has to check balancedness, so there are less balanced polyhedral currents. This principle also motivates condition (3) of Def. \ref{def:sheaf_of_linear_functions}. For example, only $0\in P^{0,0,0}(\mbR)$ satisfies condition \eqref{eq:balanced_polyhedral_current} with respect to all pwl functions on $\mbR$.
\end{enumerate}
\end{ex}

\begin{prop}\label{prop:balanced_equiv_polyhedral_deriv}
Let $(X,L)$ be a pwl space with linear functions and let $T\in P(X)$ be a polyhedral current. Then the following statements are equivalent.
\begin{enumerate}[leftmargin=*, label=(\arabic*), topsep=4pt, itemsep=4pt]
\item $T$ is balanced.
\item Both derivatives $d'T$ and $d''T$ are again polyhedral.
\item At least one of the derivatives $d'T$ and $d''T$ is again polyhedral.
\item For every compact polyhedral set $K\subseteq X$ and every tuple of linear functions $f\in L(K)^r$, the pushforward $f_*(T\vert_{K\setminus f^{-1}(f(\partial K))})$ is balanced.
\end{enumerate}
\end{prop}
We briefly explain the meaning of (4). First, $\partial K\subseteq K$ denotes the boundary of $K$ in $X$. It is a closed subset of $K$ and hence compact. Thus $f(\partial K)$ is compact and hence also closed. The restricted map
$$f:(K\setminus f^{-1}(f(\partial K))) \to \mbR^r\setminus f(\partial K)$$
is then a proper map of pwl spaces and it makes sense to take the pushforward of a polyhedral current on $K\setminus f^{-1}(f(\partial K))$. Moreover, $K\setminus f^{-1}(f(\partial K))$ is an open subset of $K\setminus \partial K$ and in particular open in $X$. Thus the restriction $T\vert_{K\setminus f^{-1}(f(\partial K))}$ is defined. In summary, the expression $f_*(T\vert_{K\setminus f^{-1}(f(\partial K))})$ in (4) is a polyhedral current on $\mbR^r\setminus f(\partial K)$.

The relevance of (4) is that the balanced polyhedral currents on an open subset $U\subseteq \mbR^r$ are precisely the $δ$-forms on $U$ from \cite{Mih_trop_inter}. This perspective will become relevant in \S\ref{ss:tropical_delta_forms} below.

\begin{proof}[Proof of the equivalence of (1), (2) and (3).]
All three properties are local on $X$. Using Lem. \ref{lem:nice_neighborhoods}, we may hence suppose that $X$ is a polyhedral set in the following. Using (2) of Def. \ref{def:sheaf_of_linear_functions}, we may furthermore suppose the existence of $x_1,\ldots,x_r\in L(X)$ such that the corresponding map $X\to \mbR^r$ has finite fibers.

First assume that $T$ is of tridegree $(p,q,c)$ and choose a presentation $T = \sum_{σ\in \mcX^c} α_σ\wedge [σ,μ_σ]$ in some polyhedral complex structure $\mcX$ on $X$. We also assume $\mcX$ to be subordinate to $L$ as in Def. \ref{def:sheaf_of_linear_functions} (3). Then every test form $η\in A_c^{n-p-c,n-q-c}(X)$ has the property that all restrictions $η\vert_{σ}$, $σ\in \mcX$ are smooth. Stokes' formula for polyhedra (Prop. \ref{prop:stokes_polyhedron}) implies
\begin{equation}\label{eq:current_to_check_polyhedralness}
(d'T - d'_PT)(η) = \sum_{τ\in \mcX^{c+1}}\int_{[τ,μ_τ]} \sum_{σ\in \mcX^c,\ τ\subset σ} (α_σ\wedge η, n_{σ,τ}'')\vert_τ.
\end{equation}
(Here, $(μ_τ)_{τ\in \mcX^{c+1}}$ is a fixed auxiliary choice of weights for the polyhedra of codimension $c+1$ and the $n_{σ,τ}$ are with respect to $μ_σ$ and $μ_τ$.) We now perform the same arguments as in \cite{Mih_trop_inter}*{around (3.8)}. Fix some $τ\in \mcX^{c+1}$, our aim being to show that $T$ is balanced in $τ$ if and only if the $τ$-contribution in \eqref{eq:current_to_check_polyhedralness} is polyhedral. We make some choices and introduce some notation before proving this. After reordering, we may assume that $x_1,\ldots,x_{\dim τ}$ embed $τ$ into $\mbR^{\dim τ}$. Let $z$ be a linear combination of $x_1,\ldots,x_r$ such that $z\vert_{τ}$ is constant and such that $z\vert_σ$ is nonconstant for all $τ\subset σ\in \mcX^c$. Then every $α_σ$ can be uniquely expressed as
$$α_σ = α_σ^{(1)} + d'z \wedge α_σ^{(2)} + d''z \wedge α_σ^{(3)} + d'z\wedge d''z \wedge α_σ^{(4)}$$
with the $α_σ^{(j)}$ all $C^\infty$-linear combinations of $d'x_I \wedge d''x_J$, $I,J\subseteq \{1,\ldots,\dim τ\}$. We can similarly expand any given test form $η$. Working locally near a point in $τ$, i.e. shrinking $X$ if necessary, we may assume $η$ to be of the form
$$η = (x_1,\ldots,x_{\dim τ}, y_1, \ldots, y_s)^*ω$$
where $y_1,\ldots,y_s \in L(X)$ are suitable further linear functions. We can (and will) furthermore assume that the restrictions $y_i\vert_τ$ are constant. (Modifying the $y_i$ by linear combinations of the $x_1,\ldots,x_{\dim τ}$ amounts to composition with a linear automorphism of $\mbR^{\dim τ + s}$.) Write
$$η = η^{(1)} + \sum_{i = 1}^s \big(d'y_i\wedge η_i^{(2)} + d''y_i \wedge η_i^{(3)}\big) + \sum_{i, j = 1}^s d'y_i \wedge d''y_j \wedge η_{ij}^{(4)}$$
where the $η^{(j)}$ are all $C^\infty(X)$-linear combinations of monomials $d'x_I \wedge d''x_J$ where $I, J\subseteq \{1,\ldots,\dim τ\}$.

The crucial observation in this setting is that
$$dz\vert_τ = dy_i\vert_τ = 0,\quad d\in\{d',d''\}$$
because we assumed $z\vert_τ$ and the $y_i\vert_τ$ to be constant. Using the properties \eqref{eq:contraction_Leibniz} and \eqref{eq:contraction_basic} for $(\ ,n''_{σ,τ})$, we see that for each $σ\in \mcX^c$ containing $τ$,
\begin{multline}\label{eq:wedge_explicit}
(α_σ\wedge η, n''_{σ,τ})\vert_τ = (α_σ^{(1)}\wedge η^{(1)}, n''_{σ,τ})\vert_τ\ + \\
\left.\left(\frac{\partial z\vert_σ}{\partial n_{σ,τ}} α^{(3)}_σ \wedge η^{(1)} + (-1)^{\deg α}\sum_j \frac{\partial y_j\vert_σ}{\partial n_{σ,τ}} α^{(1)}_σ \wedge η^{(3)}_j\right)\right\vert_τ.
\end{multline}
The products $α_σ^{(1)}\wedge η^{(1)}$ are $(\dim τ, \dim τ + 1)$-forms that are $C^\infty(σ)$-linear combinations of monomials $d'x_I\wedge d''x_J$, $I,J\subseteq \{1,\ldots,\dim τ\}$. Noting that there are only the $\dim τ$ many functions $x_1,\ldots,x_{\dim τ}$ available to form such monomials, we have $α_σ^{(1)}\wedge η^{(1)} = 0$.

Assume now that the balancing condition holds for $T$. Then summing \eqref{eq:wedge_explicit} over all $σ\in \mcX^c$ containing $τ$ and using the balancing condition for the $α_σ^{(1)}\wedge η_j^{(3)}$-summands shows that the $τ$-contribution to \eqref{eq:current_to_check_polyhedralness} is
\begin{equation}\label{eq:formula_deriv_balanced}
\int_{[τ,μ_τ]} \left(\sum_{τ\subset σ \in \mcX^c} \frac{\partial z\vert_σ}{\partial n_{σ,τ}} α_σ^{(3)}\right) \wedge η,
\end{equation}
which defines a polyhedral current in $η$. Assume conversely that the balancing condition does not hold in $τ$, say it fails for $y_1$ meaning
$$θ = \sum_{τ\subset σ\in \mcX^c}\frac{\partial y_1\vert_σ}{\partial n_{σ,τ}} α_σ\vert_τ \neq 0.$$
Pick any compactly supported smooth form $η^{(3)}_1$ as above with the additional property $\Supp η^{(3)}_1 \cap τ' = \emptyset$ for all $τ\neq τ'\in \mcX^{c+1}$ and such that furthermore $0 \neq \int_{[τ,μ_τ]} θ \wedge η^{(3)}_1$. Put $η = d''y_1\wedge \ob{η_1}$ and view it as compactly supported smooth form on $X$. Then
$$(d'T - d'_PT)(η) = (-1)^{\deg α}\int_{[τ,μ_τ]} θ \wedge η^{(3)}_1 \neq 0.$$
The difference $d'T - d'_PT$ has support in codimension $c+1$, but $η\vert_ρ = 0$ for every polyhedron $ρ\subseteq \Supp T$ of codimension $\geq c+1$ by construction, so the current $d'T - d'_PT$ cannot be polyhedral.

These arguments showed that a trihomogeneous polyhedral current $T$ is balanced if and only if $d'T$ is polyhedral. They apply symmetrically to $d''T$, proving the equivalence of (1), (2) and (3) for trihomogeneous $T$. The extension to not necessarily trihomogeneous $T$ is by a simple induction over dimensions of supports, we refer to Step (4) of the proof of \cite{Mih_trop_inter}*{Thm. 3.3}.
\end{proof}

\begin{proof}[Proof of the equivalence with (4).]
(4) follows immediately from (1) because $f_*(dT) = df_*(T)$, $d\in \{d',d''\}$, and because the pushforward of polyhedral currents is polyhedral.

For the converse, we first note that we may assume $T$ to be of some tridegree $(p,q,c)$ because $T$ is balanced if and only if each of its trihomogeneous components is. Let $τ\in \mcX^c$ as above, let $U\subseteq X$ be open and let $ϕ\in L(U)$ be a linear function with constant restriction $ϕ\vert_{τ\cap U}$ that we would like to check the balancing condition \eqref{eq:balanced_polyhedral_current} for. Working locally, we may assume $ϕ$ to be defined on all of $X$. Take $K = \bigcup_{τ\subseteq σ \in \mcX_n} σ$ and $f = (x_1,\ldots,x_r, ϕ)$. Then $τ^\circ\subset K\setminus \partial K$ and $τ = f^{-1}(f(τ))$ because each restriction $f\vert_σ$, $σ\in \mcX$, is injective. Comparing the balancing condition \eqref{eq:balanced_polyhedral_current} for $τ$ and $ϕ$ in $T$, and for $f(τ)$ and the last coordinate on $\mbR^{r+1}$ in $f_*T$ finishes the proof.
\end{proof}

\begin{rmk}\label{rmk:balanced_polyhedral_currents}
We would have liked to prove Prop. \ref{prop:balanced_equiv_polyhedral_deriv} by reduction to the case of $\mbR^r$ in \cite{Mih_trop_inter} instead of repeating an argument from there. However, one cannot check if a current is polyhedral by looking at all its linear pushforwards.

For example, consider the polyhedral set $X = σ_1 \cup σ_2 \cup σ_3\subseteq \mbR^2$ with
$$σ_1 = \{ (x, 0) \mid x \leq 0\},\quad σ_2 = \{(x,x) \mid x \geq 0\},\quad σ_3 = \{(x,-x) \mid x \geq 0\}$$
and endow it with the sheaf of linear functions $L$ that is generated by projection $p_1$ to the first coordinate. Let $φ$ be a smooth function on $\mbR_{>0}$ that is $L^1$ for the Lebesgue measure. Consider the following functional on pws $(1,1)$-forms on $X$:
\begin{equation}\label{eq:example_L_1_current}
T = (φ\circ p_1) \wedge [σ_2, μ] - (φ\circ p_1) \wedge [σ_3, μ]
\end{equation}
where $μ$ is the inverse image under $p_1$ of the standard weight on $\mbR$ and where the meaning of the notation is just as in \eqref{eq:def_polyhedral_current}. Restricting $T$ to $A^{1,1}_c(X)$ defines a current in the sense of \ref{def:currents_pwl_space}. Moreover, the restriction $T\vert_{X\setminus \{(0,0)\}}$ is polyhedral but $T$ itself is only polyhedral if $φ$ extends to a smooth function on $\mbR_{\geq 0}$.

The essential property of $T$ is that $p_{1,*}(T) = 0$ because of the opposite signs in \eqref{eq:example_L_1_current}. Since $p_1$ is (up to affine linear transformation) the only linear function defined near $(0,0)$, one deduces from this that $f_*(T\vert_{K\setminus f^{-1}(f(\partial K))})$ is polyhedral for every compact polyhedral set $K\subseteq X$ and every linear map $f \in L(K)^r$.
\end{rmk}

\subsection{Tropical Spaces}

We now combine the concepts of weighted spaces and of sheaves of linear functions.

\begin{defn}\label{def:trop_space}
A \emph{tropical space} is a triple $(X,µ,L)$ consisting of a weighted pwl space $(X,µ)$ with linear functions $L$ such that the fundamental cycle $[X,µ]$, viewed as current in the sense of Def. \ref{def:currents_pwl_space} (with respect to $L$), is closed,
$$d'[X,µ] = d''[X,µ] = 0.$$
\end{defn}

\begin{cor}[to Prop. \ref{prop:balanced_equiv_polyhedral_deriv}]\label{cor:equiv_char_trop_space}
Let $(X,µ,L)$ be a purely $n$-dimensional weighted pwl-space with linear functions. The following are equivalent.
\begin{enumerate}[leftmargin=*, label=(\arabic*), topsep=4pt, itemsep=4pt]
\item $(X,µ,L)$ is a tropical space.
\item $[X,μ]$ is balanced with respect to $L$.
\item For every purely $n$-dimensional compact polyhedral set $K\subseteq X$ and every linear map $f:K→\mbR^r$, the pushforward $f_*[K,µ]\vert_{\mbR^r\setminus f(\partial K)}$ is balanced.
\end{enumerate}
\end{cor}

Note that $f_*[K,µ]\vert_{\mbR^r\setminus f(\partial K)}$ lies in tridegree $(0,0,r-n)$. (For example, if $\dim f(K)<n$, then $f_*[K,μ] = 0$ by definition of the pushforward.) Requiring it to be balanced simply means that it is a tropical cycle of dimension $n$. Assume that $X$ is a polyhedral set and $\mcX$ a polyhedral complex structure that is subordinate to $μ$ and $L$. The balancing condition \eqref{eq:balanced_polyhedral_current} in this special case reads as
\begin{equation}\label{eq:balanced}
\sum_{τ\subset σ\in \mcX_n} \frac{\partial ϕ\vert_σ}{\partial n_{σ,τ}} = 0
\end{equation}
for every $τ\in \mcX_{n-1}$ and $ϕ\in L(U)$ with $ϕ\vert_{τ\cap U}$ constant.

\begin{proof}[Proof of Cor. \ref{cor:equiv_char_trop_space}.]
The equivalence of (2) and (3) is a special case of Prop. \ref{prop:balanced_equiv_polyhedral_deriv}. Moreover, condition (1) implies (2): If $d[X, μ] = 0$, then $d[X, μ]$ is in particular polyhedral, hence $[X,μ]$ is balanced again by Prop. \ref{prop:balanced_equiv_polyhedral_deriv}. (Here and in the following, $d\in \{d',d''\}$ can be any of the two differentials. Denote by $d_P\in \{d'_P, d''_P\}$ the polyhedral derivative of the same bidegree.)

It is left to show that (2) implies (1). The polyhedral derivative $d_P[X,μ]$ vanishes by definition. The difference $d[X,μ] - d_P[X,μ]$ has the form \eqref{eq:current_to_check_polyhedralness}. Assuming (2), i.e. assuming that $[X, μ]$ is balanced, this difference is again polyhedral and a sum of polyhedral currents of the form \eqref{eq:formula_deriv_balanced}. In our special situation, $α_σ = 1$ (constant function on $σ$) for all $σ$, and the vanshing of \eqref{eq:formula_deriv_balanced} for all $η$ is precisely the balancing condition \eqref{eq:balanced}.
\end{proof}

Let $α\in A(X)$ be a smooth form. Note that this notion depends only on $L$, not on $μ$. It defines a current $α\wedge [X,μ]$, which depends only on $α$ viewed as pws form. A crucial observation is now that the derivatives of $α$ as smooth form and of $α$ as current coincide,
\begin{equation}\label{eq:derivative_smooth_form_tropical_space}
d(α\wedge [X,μ]) = dα \wedge [X,μ],\ \ \ d\in \{d',d''\}.
\end{equation}
Namely, assuming $α$ homogeneous,
$$d(α\wedge [X,μ])(η) = (-1)^{\deg α + 1}\int_{[X,µ]} α\wedge d η = \int_{[X,µ]} d α\wedge η - d (α\wedge η)$$
and $\int_{[X,μ]} d(α\wedge η) = 0$ by the tropical space property. This observation essentially spells out that we have built Stokes' Theorem into our definition of tropical space.

\begin{ex}\label{ex:trop_spaces}
\begin{enumerate}[wide, labelindent=0pt, labelwidth=!, label=(\arabic*), topsep=4pt, itemsep=4pt]
\item Let $(X, μ)$ be a purely $1$-dimensional weighted pwl space, i.e. a metrized graph in the sense of \cite{ABBR}. This $1$-dimensional situation is special because every pwl function restricted to every $0$-dimensional polyhedron (i.e. to a point) is constant. Condition \eqref{eq:balanced} holds for a pwl function $ϕ$ on $X$ with respect to $μ$ if and only if it is harmonic in the sense of \cite{ABBR}. In particular, a sheaf of linear functions $L$ makes $(X,μ)$ into a tropical space if it consists only of such harmonic functions. Moreover, there is a unique maximal such choice, namely the sheaf of all harmonic functions.

\item \label{item:ex_L_bar} Let $(X,µ,L)$ be a tropical space. In analogy with the previous example, we call a pwl function $ϕ$ on $X$ harmonic if $(X,μ,L + \mbR ϕ)$ is still a tropical space. Enlarging $L$ by all such harmonic $ϕ$ defines a new sheaf of linear functions $\ob{L}$. It is the unique maximal sheaf $\ob{L}\subset Λ_X$ that contains $L$ and has the property that $(X,μ,\ob{L})$ is a tropical space. The next example shows that $\ob{L}$ continues to be an additional datum and is not determined by $(X,µ)$ in general.

\item The standard tropical space structure on $\mbR^n$ is $(\mbR^n, μ_{\mr{std}}, \mr{Aff})$, where $μ_{\mr{std}}$ is the standard weight and $\mr{Aff}$ the sheaf of affine linear functions. Note that $\mr{Aff}$ is already maximal in the sense of  \ref{item:ex_L_bar}, meaning a pwl function $φ$ on $\mbR^n$ is harmonic with respect to $\mr{Aff}$ if and only if it is itself affine linear. Assume now that $n\geq 2$. Then there exist pwl automorphisms $f:\mbR^n\to \mbR^n$ that are not affine linear and still satisfy $|\det(f)| = 1$ everywhere. The determinant condition implies $f^{-1}(μ_{\mr{std}}) = μ_{\mr{std}}$. By functoriality of all involved definitions, $(\mbR^n, μ_{\mr{std}}, f^{-1}(\mr{Aff}))$ is again a tropical space, but $f^{-1}(\mr{Aff}) \neq \mr{Aff}$.

\item Tropical spaces generalize tropical cycles in the following way. Consider a purely $n$-dimensional polyhedral subset $X\subseteq \mbR^r$ with weight $μ$. Then $(X, μ)$ is a tropical cycle in the classical sense if and only if $(X, μ, \mr{Aff}\vert_X)$ is a tropical space. In particular, all notions that are derived from that of tropical cycles, such as abstract tropical varieties \cite{AR} or the tropical spaces from \cites{Mikhalkin, Mikhalkin_Rau} are examples of tropical spaces in our sense.
\end{enumerate}
\end{ex}

\subsection{$δ$-Forms}
\label{ss:tropical_delta_forms}

We finally extend the formalism of $δ$-forms from $\mbR^n$ to tropical spaces.

\begin{lem}\label{lem:faithfulness_lemma}
Let $X$ be a pwl space, $T \in P(X)$ a polyhedral current and $f:X→\mbR^r$ a pwl map with finite fibers. Then $T$ is uniquely determined by all pushforwards $f_*(T\vert_{K\setminus f^{-1}(f(\partial K))})$, for $K\subseteq X$ a compact polyhedral set.
\end{lem}
\begin{proof}
We need to see that if $f_*(T\vert_{K\setminus f^{-1}(f(\partial K))}) = 0$ for all $K$ as in the lemma, then $T = 0$. This can be checked locally on $X$, so we may assume $X$ to be a polyhedral set (Lem. \ref{lem:nice_neighborhoods}). Let $\mcX$ and $\mcY$ be polyhedral complex structures for $X$ resp. $f(X)$ that are subordinate to $f$. Assume furthermore that $\mcX$ consists of compact polyhedra and is subordinate to $T$. Let $(μ_σ)_{σ\in \mcX}$ and $(ν_ρ)_{ρ\in \mcY}$ be choices of weights. Write $T = \sum_{σ\in \mcX} α_σ\wedge [σ,μ_σ]$, our aim being to show $α_σ = 0$ for all $σ\in \mcX$.

Given $σ$, consider
$$K = \bigcup_{τ\in \mcX,\ σ\subseteq τ} τ$$
which is a compact neighborhood of $σ^\circ$ in $X$. Since $\mcX$ is subordinate to $f$ and since $f$ has finite fibers, each restriction $f\vert_τ$ is injective and hence $K \cap f^{-1}(f(σ)) = σ$. In particular $σ^\circ \subseteq K\setminus f^{-1}(f(\partial K))$. Consider the polyhedral complex $\mcK = \{τ\in \mcX \mid τ\subseteq K\}$ which satisfies $|\mcK| = K$. Write $f_*(\sum_{τ\in \mcK} α_τ\wedge [τ, μ_τ]) = \sum_{ρ\in \mcY} β_ρ\wedge [ρ, ν_ρ]$. By assumption on $T$, we have $β_ρ\vert_{ρ\setminus f(\partial K)} = 0$ for all $ρ$. From our choice of $K$, we find that
$$α_σ = ν_{f(σ)}/f(μ_σ)\cdot f^*(β_{f(σ)}) = 0$$
as required.
\end{proof}

By \emph{refinement} of a linear map $f:X\to \mbR^r$, we mean a pair $(g,p)$ consisting of a linear map $g:X\to \mbR^s$ and an affine linear map $p:\mbR^s\to \mbR^r$ such that $f = p\circ g$. A typical example is obtained by adding further linear functions $f'$ and considering $(f\times f', p_1)$.

Recall from \cite{Mih_trop_inter} that the balanced polyhedral currents on an open subset $U\subseteq \mbR^r$ are called $δ$-forms. We denote them by $B(U)$. Recall further that $δ$-forms are endowed with a trihomogeneous $\wedge$-product and differential operators $d'$, $d''$, $d'_P$, $d''_P$, $\partial'$ and $\partial''$. Recall that the first two ($d'$ and $d''$) agree with the derivatives as currents. The next two ($d'_P$ and $d''_P$) are the polyhedral derivatives, and the last two are defined by
$$\partial' = d'_P - d',\quad \partial'' = d''_P - d''.$$
We refer to \cite{Mih_trop_inter} for more details.

\begin{prop}\label{prop:exist_realization}
Consider an $n$-dimensional tropical space $(X, μ, L)$, a linear map $f:X→\mbR^r$ and a $δ$-form $γ \in B(\mbR^r)$. Then there is a unique polyhedral current $f^\star γ\in P(X)$ such that for all compact polyhedral sets $K\subseteq X$ and all refinements $g:K→\mbR^s$ with finite fibers, $f = p \circ g$, the following identity holds,
\begin{equation}\label{eq:realization}
g_*\big((f^\star γ)\vert_{K\setminus g^{-1}(g(\partial K))}\big) = g_*(K) \wedge p^*γ,\ \ \ \text{away from }g(\partial K).
\end{equation}
\end{prop}
Here and in the following, for compact polyhedral sets $K\subseteq X$, we write $g_*(K)$ as shorthand for $g_*[K_n, μ\vert_{K_n}]$, where $K_n\subseteq K$ denotes the purely $n$-dimensional locus. Note that the complement $K\setminus g^{-1}(g(\partial K))$ is purely $n$-dimensional, so $g(K)\vert_{\mbR^r\setminus g(\partial K)}$ is a tropical cycle by Cor. \ref{cor:equiv_char_trop_space}. The product in \eqref{eq:realization} is then that of $δ$-forms on $\mbR^s\setminus g(\partial K)$.

The situation is somewhat analogous to that of pullback on the level of Chow groups for (not necessarily flat) morphisms of smooth projective varieties. Our choice of notation ($f^\star$ instead of $f^*$) is meant to emphasize this point. We will see later however that $f^\star(γ) = f^*(γ)$ whenever $f$ is flat.

\begin{proof}
Lem. \ref{lem:faithfulness_lemma} already implies that there is at most one such $f^\star γ$ and our task is to construct it. Let $(K,g,p)$ be as in the statement of the proposition. As an intermediate step, we first produce a polyhedral current $T_{K,g,p}$ on the open set $K\setminus \partial K$ that satisfies the characterizing identity \eqref{eq:realization} for all compact polyhedral sets $H\subseteq K$ and the refinement $(g\vert_H,p)$.

Pick any polyhedral complex structure $\mcK$ for $K$ with the following properties. It is subordinate to $g$, $μ$ and $\partial K$, and is such that $g(\mcK)$ is contained in a polyhedral complex structure subordinate to $p^*γ$. Fix weights $μ_σ$ for all polyhedra $σ\in \mcK$; the current $T_{K,g,p}$ will be of the form
$$T_{K,g,p} = \sum_{σ\in \mcK,\ σ\not\subset \partial K} α_σ\wedge [σ,μ_σ].$$

Given $τ\in \mcK$ with $τ\not\subset \partial K$, set $\mcK_τ = \{σ\in \mcK,\ τ\subseteq σ\}$ and $K_τ = |\mcK_τ|$. One may write (not necessarily uniquely)
$$g_*(K_τ)\wedge p^*γ = \sum_{σ \in \mcK_τ} β_σ \wedge [g(σ), g(μ_σ)]\ \ \text{away from}\ g(\partial K_τ).$$
(Use the fan displacement rule \cite{Mih_trop_inter}*{Prop. 4.21} to compute the $\wedge$-product, for example.) The coefficient $β_τ$ is unique, however, because $g^{-1}(g(τ))\cap K_τ = τ$ by construction. Then we put $α_τ := g^*β_τ$. Varying $τ$, this defines some $T_{K,g,p}\in P(K)$.

Checking our desired property is straightforward. Let $H\subseteq K$ be a compact polyhedral set. Let $\mcK'$ be a refinement of $\mcK$ and let $\mcY$ be a polyhedral complex structure for $g(K)$ such that the following properties hold: The pair $\mcK'$ and $\mcY$ is subordinate to $g$. Moreover, $\mcK'$ is subordinate to $H$ and $\partial H$ (boundary in $X$). Fixing auxiliary weights $(ν_τ)_{τ\in \mcY}$, we may write
$$g_*(H)\wedge p^*γ = \sum_{τ \in \mcY} β_τ \wedge [τ, ν_τ]\ \ \text{away from}\ g(\partial H).$$
Let $τ\in \mcY$, $τ\not\subset g(\partial H)$ be a polyhedron and let $f^{-1}(τ) = \{τ_1,\ldots,τ_m\}\subseteq \mcK'$. Then we have $τ_i\not\subset \partial H$ for all $i = 1,\ldots,m$. Put $K_i = \bigcup_{σ\in \mcK',\ τ_i\subseteq σ} σ$ which is a neighborhood of $τ_i^\circ$ in $H$. Then $g_*(H) = \sum_{i = 1}^m g_*(K_i)$ on a neighborhood of $τ^\circ$. Moreover, for each $τ_i$ there is some $τ_i'\in \mcK$ of the same dimension with $τ_i\subseteq τ_i'$, and $K_i$ coincides with the $K_{τ_i'}$ from above on an open neighborhood of $τ_i^\circ$. One deduces from this that
$$β_τ \wedge [τ,ν_τ] = g_*(T_{K, g, p}\vert_{H\setminus g^{-1}(g(\partial H))})$$
on a neighborhood of $τ^\circ$, which is the claimed property of $T_{K,g,p}$.

Our next claim is that $T$ satisfies \eqref{eq:realization} also for all other refinements $h$ of $f$ on $K$. We may assume that $h$ refines $g$, say $g = q\circ h$, by passing to the common refinement $(g,h)$ if necessary and using the projection formula \cite{Mih_trop_inter}*{Prop. 4.2}. Subdividing $\mcK$, we may assume it to be subordinate to $h$ as well. Note that $q(h(\partial K_τ)) = g(\partial K_τ)$ for every $τ$. The projection formula, \cite{Mih_trop_inter}*{Prop. 4.2} applies and yields
$$q_*(h_*K_τ \wedge q^*p^*γ) = g_*K_τ \wedge p^*γ,\ \ \ \text{away from $g(\partial K_τ)$}.$$
Using that $g\vert_{τ}$ is injective, we have that $q\vert_{h(τ)}$ is injective. This means that the multiplicities of $h(τ)$ in $h_*K_τ\wedge q^*p^*γ$ and $h_*(T_{K,g,p}\vert_{K_τ \setminus h^{-1}(h(\partial K_τ))})$ coincide. In other words, $T_{K,h,p\circ q} = T_{K,g,p}$.

Finally, the axioms for sheaves on linear functions ensure that every point $x\in X$ has a polyhedral neighborhood $K$ that admits a refinement $(g,p)$ as above. Lem. \ref{lem:faithfulness_lemma} ensures that the $T_{K,g,p}$ glue to a polyhedral current $T$. The proof moreover shows that $T$ may be computed in \emph{any} chart $(K,g,p)$ and then satisfies \eqref{eq:realization} for this chart. Thus we have given an explicit construction of $f^\star γ = T$ and the proof is complete. 
\end{proof}

\begin{prop}\label{prop:derivatives_delta_form}
Let $(X, μ, L)$ be a tropical space, $f:X\to \mbR^r$ a linear map and $γ\in B(\mbR^r)$ a $δ$-form. Then
$$d(f^\star(γ)) = f^\star(dγ),\ \ \ d_P(f^\star(γ)) = f^\star(d_Pγ),\ \ \ d\in \{d',d''\},\ d_P\in \{d'_P, d''_P\}.$$
\end{prop}
\begin{proof}
The characterizing identity \eqref{eq:realization} and Prop. \ref{prop:balanced_equiv_polyhedral_deriv} together imply that $f^\star(γ)$ is balanced. Again by Prop. \ref{prop:balanced_equiv_polyhedral_deriv}, the derivatives $d(f^\star(γ))$ and $d_P(f^\star(γ))$ are hence polyhedral. Our task is to show the identities (with $d\in \{d',d''\}$ and $d_P\in \{d'_P, d''_P\}$),
\begin{equation}\label{eq:push_forward_derivatives}
\begin{aligned}
g_*(d(f^\star(γ))\vert_{K\setminus g^{-1}(g(\partial K))}) &\ = g_*(K)\wedge p^*(dγ)\vert_{\mbR^r\setminus g(\partial K)}\\
g_*(d_P(f^\star(γ))\vert_{K\setminus g^{-1}(g(\partial K))}) &\ = g_*(K)\wedge p^*(d_Pγ)\vert_{\mbR^r\setminus g(\partial K)}
\end{aligned}
\end{equation}
for all compact polyhedral sets $K\subseteq X$ and all refinements $(g,p)$ of $f\vert_K$ with finite fibers.

For all linear maps of spaces with linear functions, $d'$ and $d''$ commute with pushforward, cf. Def. \ref{def:currents_pwl_space} (3). Moreover, for all pwl maps with finite fibers, $d'_P$ and $d''_P$ commute with pushforward. Thus the two left hand sides in \eqref{eq:push_forward_derivatives} equal
\begin{equation}\label{eq:push_forward_derivatives_II}
d(g_*f^\star(γ))\vert_{K\setminus g^{-1}(g(\partial K))}\quad \text{and}\quad d_P(g_*f^\star(γ))\vert_{K\setminus g^{-1}(g(\partial K))},
\end{equation}
respectively. Applying the definition of $f^\star(γ)$, these equal
\begin{equation}\label{eq:push_forward_derivatives_III}
d(g_*(K)\wedge p^*(γ))\vert_{K\setminus g^{-1}(g(\partial K))}\quad \text{and}\quad d_P(g_*(K)\wedge p^*γ)\vert_{K\setminus g^{-1}(g(\partial K))}.
\end{equation}
Finally, the four differential operators $d'$, $d''$, $d'_P$ and $d''_P$ for $δ$-forms on $\mbR^r$ all satisfy the Leibniz rule, all annihilate tropical cycles with constant coefficients such as $g_*(K)$, and all commute with $p^*$. Using these properties, \eqref{eq:push_forward_derivatives} follows from the equality with \eqref{eq:push_forward_derivatives_III}.
\end{proof}

\begin{prop}\label{prop:wedge_well_defined}
Let $(X, μ, L)$ be a tropical space and let $f_i:X\to \mbR^{r_i}$, $i = 1,\ldots,4$, be linear maps as well as $γ_i\in B(\mbR^{r_i})$. Let
$$(p_1,p_2):\mbR^{r_1+r_2}\lr \mbR^{r_1}\times \mbR^{r_2},\quad (p_3,p_4):\mbR^{r_3+r_4}\lr \mbR^{r_3}\times \mbR^{r_4}$$
denote the projection maps. Assume furthermore that $f_1^\star(γ_1) = f_3^\star(γ_3)$ and $f_2^\star(γ_2) = f_4^\star(γ_4)$. Then
\begin{equation}\label{eq:wedge_well_defined}
(f_1,f_2)^\star(p_1^*γ_1 \wedge p_2^*γ_2) = (f_3,f_4)^\star(p_3^*γ_3\wedge p_4^*γ_4).
\end{equation}
\end{prop}
\begin{proof}
By definition of the two sides in \eqref{eq:wedge_well_defined}, we are required to check the following. Whenever $K$ is a compact polyhedral subset of $X$ and $g:K\to \mbR^s$ a linear map together with four maps $q_i:\mbR^s\to \mbR^{r_i}$ such that $(g, q_i)$ refines $f_i\vert_K$ for $i = 1,\ldots,4$, then
$$g_*(K) \wedge (q_1,q_2)^*(p_1^*γ_1\wedge p_2^*γ_2) = g_*(K) \wedge (q_3,q_4)^*(p_3^*γ_3\wedge p_4^*γ_4)\quad \text{on $\mbR^s\setminus g(\partial K)$.}$$
We have the identities
$$\begin{aligned}
(q_1,q_2)^*(p_1^*γ_1\wedge p_2^*γ_2) &\ = q_1^*γ_1 \wedge q_2^*γ_2\\
(q_3,q_4)^*(p_3^*γ_3\wedge p_4^*γ_4) &\ = q_3^*γ_3 \wedge q_4^*γ_4.
\end{aligned}$$
Moreover, the assumption of the proposition implies that
$$g_*(K)\wedge q_1^*(γ_1) = g_*(K)\wedge q_3^*(γ_3),\quad g_*(K)\wedge q_2^*(γ_2) = g_*(K)\wedge q_4^*(γ_4)$$
away from $g(\partial K)$. Since all involved operations are trihomogeneous with respect to the trigrading on polyhedral currents, we may assume all $γ_i$ trihomogeneous in the following. Then associativity and graded-commutativity of the $\wedge$-product yield (away from $g(\partial K)$)
$$\begin{aligned}
(g_*(K)\wedge q_1^*(γ_1))\wedge q_2^*(γ_2) &\ = (g_*(K) \wedge q_3^*(γ_3))\wedge q_2^*(γ_2)\\
										   &\ = (-1)^{\deg(γ_3)\deg(γ_2)} (g_*(K) \wedge q_2^*(γ_2)) \wedge q_3^*(γ_3)\\
										   &\ = (-1)^{\deg(γ_3)\deg(γ_2)} (g_*(K) \wedge q_4^*(γ_4)) \wedge q_3^*(γ_3)\\
										   &\ = (g_*(K)\wedge q_3^*(γ_3))\wedge q_4^*(γ_4).
\end{aligned}$$
which finishes the proof.
\end{proof}

\begin{defn}\label{def:delta_form_tropical_space}
\begin{enumerate}[wide, labelindent=0pt, labelwidth=!, label=(\arabic*), topsep=4pt, itemsep=4pt]
\item A \emph{$δ$-form} on $(X,µ,L)$ is a polyhedral current on $X$ which is locally of the form $f^\star γ$ for a linear map $f$, mapping to $\mbR^r$ say, and a $δ$-form $γ\in B(\mbR^r)$. We write $B$ or $B_X$ for the sheaf of $δ$-forms on $X$.

\item A $δ$-form $ω$ has \emph{tridegree $(p,q,c)$} if it has this tridegree as polyhedral current. Since the product operation $g_*K\wedge -$ in \eqref{eq:realization} is trihomogeneous, this is equivalent to demanding that locally $ω = f^\star γ$ with $γ$ of tridegree $(p,q,c)$. We write $B^{p,q,c}$ or $B^{p,q,c}_X$ for the sheaf of $δ$-forms of tridegree $(p,q,c)$ and obtain the decomposition $B = \bigoplus_{p,q,c} B^{p,q,c}$.

\item The derivatives $d'ω$, $d''ω$ of a $δ$-form $ω\in B$ are the derivatives as current. The polyhedral derivatives $d'_Pω$, $d''_Pω$ are the derivatives as polyhedral current. The boundary derivatives $\partial'ω$, $\partial''ω$ are defined by the identities
$$\partial' = d'_P - d',\quad \partial'' = d''_P - d''.$$
By Prop. \ref{prop:derivatives_delta_form}, if $ω$ is (locally) given as $f^\star(γ)$, then its derivatives are (locally) given by
$$f^\star(d'γ),\ f^\star(d''γ),\ f^\star(d'_Pγ),\ f^\star(d''_Pγ),\ f^\star(\partial'γ)\ \text{and}\ f^\star(\partial''γ).$$

\item The wedge product $ω_1\wedge ω_2$ of two $δ$-forms on $X$ is defined as follows. If $ω_i$ is locally given as $f_i^\star(γ_i)$, then $ω_1\wedge ω_2$ is locally $(f_1,f_2)^\star(p_1^*γ_1\wedge p_2^*γ_2)$. This local construction glues and is well defined by Prop. \ref{prop:wedge_well_defined}.

\item Finally, the integral $\int_X ω$ of a compactly supported $δ$-form $ω$ on $(X, μ, L)$ of top bidegree $(\dim X, \dim X)$ is defined as its integral as polyhedral current.
\end{enumerate}
\end{defn}


\begin{prop}[Stokes' Theorem]\label{prop:stokes_delta_pwl}
Assume that $α\in B_c^{n-1,n}(X)$ and $β\in B_c^{n,n-1}(X)$ are compactly supported $δ$-forms of the indicated bidegree on a tropical space $(X, μ, L)$ that is purely of dimension $n$. Then
$$\int_{X} d'α = \int_X d''β = 0.$$
\end{prop}
\begin{proof}
This is just a formality after Prop. \ref{prop:derivatives_delta_form}: Let $λ\in A^{0,0}_c(X)$ be a smooth function that is constantly $1$ on an open neighborhood of $\Supp α$. Then
$$\int_X d'α = [d'α](λ) = -[α](d'λ) = 0$$
because the derivative $d'α$ is precisely the derivative as current and because $d'λ\vert_{\Supp(α)} = d'_P(λ)\vert_{\Supp(α)}\equiv 0$. A similar argument applies to $β$.
\end{proof}

We next address the existence of pullback maps for $δ$-forms. Given a linear map $f:X \to Y$ of tropical spaces, it is natural to expect a pullback through $f^*(φ^\star γ) := (φ\circ f)^\star γ$. This fails, however, because it not only depends on $φ^\star(γ)$ but on the whole datum $(φ, γ)$:
\begin{ex}\label{ex:pull_back_fail} 
Consider again the map $f$ from Equation \eqref{eq:key_example} and consider on $\mbR^2$ the $δ$-form (a tropical cycle in fact)
\begin{equation}\label{eq:def_example}
γ = \{ x_1 = 0\} - \{x_2 = 0\},
\end{equation}
where each axis is endowed with the standard weight. Because of the opposite signs in \eqref{eq:def_example}, $\{x_1\cdot x_2 = 0\} \wedge γ = 0$, but $f^\star γ = (- \{0\})\sqcup \{0\}$
is a difference of Dirac measures.
\end{ex}

\begin{prop}\label{prop:pullback_delta_form}
Let $f:(X, μ, L_X)\to (Y, ν, L_Y)$ be a flat linear map of tropical spaces. The pullback as polyhedral current $f^*ω$ of any $δ$-form $ω\in B(Y)$ is again a $δ$-form. More precisely, if $ω$ has the presentation $ω = φ^\star(γ)$, then $f^*(ω) = (φ\circ f)^\star(γ)$.
\end{prop}
\begin{proof}
Say $\dim X = n$ and $\dim Y = m$. The claim is local on $X$ and $Y$, so we may assume them to be polyhedral sets and further that $ω = φ^\star(γ)$ with $φ$ having finite fibers. We may also assume the existence of a linear map with finite fibers $ψ$ on $X$. The claim is also linear in $ω$, so we may assume it to be of tridegree $(p,q,c)$. Let $\mcX$ and $\mcY$ denote polyhedral complex structures for $X$ resp. $Y$ that are subordinate to their tropical space structures and to $f$. We also require them to be subordinate to $ψ$ resp. $φ$ and impose that $φ(\mcY)$ is part of a polyhedral complex subordinate to $γ$. Then $\mcX$ will also be subordinate to both $f^*(ω)$ and $(φ\circ f)^\star(γ)$ and the confirmation of their equality becomes a matter of comparing their coefficients for every $σ\in \mcX^c$.

Given such $σ$, set $τ = f(σ)$. The definition of flatness (Def. \ref{def:flat}) implies that $d = \dim σ - \dim τ \leq n - m$ and we treat the two cases $d < n-m$ and $d = n-m$ separately.

If $d < n-m$, then $\codim(τ) < c$, so $τ^\circ \cap \Supp(φ^\star(γ)) = \emptyset$. It follows that the coefficient of $σ$ in $f^*ω$ vanishes. (Alternatively, one can apply formula \eqref{eq:pull_back_formula_flat} to see this vanishing.)

Assume $d = n-m$ from now on. Put $X_σ = \bigcup_{σ\subseteq ρ\in \mcX} ρ$ and $Y_τ = \bigcup_{τ\subseteq ρ\in \mcY} ρ$, which are polyhedral neighborhoods of $σ^\circ$ and $τ^\circ$. Also choose a tuple of linear functions $s = (s_1, \ldots, s_d)\in L_X(X)^d$ such that $(s, φ\circ f)\vert_σ$ is injective. Write $(\ob{X_σ}, \ob{μ}) = (s, φ\circ f)_*(X_σ, μ)$ and $(\ob{Y_τ}, \ob{ν}) = φ_*(Y_τ, ν)$. Consider the commutative diagram
\begin{equation}\label{eq:diagram_flat_pull_back}
\xymatrix{
X_σ \ar[r]^{(s, φ\circ f)} \ar[d]_f & \ob{X_σ} \ar[d]^{p_2} & \subseteq \mbR^d\times \mbR^r\\
Y_τ \ar[r]^{φ} & \ob{Y_τ} & \subseteq \mbR^r.}
\end{equation}
For every $ρ\in \mcX$ that contains $σ$, we have $\dim ρ - \dim f(ρ) \leq d$ by flatness. But we also have $\dim ρ - \dim f(ρ) \geq \dim σ - \dim τ = d$, so we in fact have equality $\dim ρ - \dim f(ρ) = d$. It follows that the two polyhedral sets $\mbR^d \times \ob{Y_τ}$ and $\ob{X_σ}$ agree on a neighborhood of $(s, φ\circ f)(σ^\circ)$. Also taking weights into account and using the flatness, we find that
\begin{equation}\label{eq:local_structure_flat}
(\ob{X_σ}, \ob{μ})\cap U = \big((\mbR^d, s(μ/ν)) \times (\ob{Y_τ}, \ob{ν})\big)\cap U
\end{equation}
for some open neighborhood $U$ of $(s, φ\circ f)(σ^\circ)$.
We are now equipped to compare the $σ$-contributions to $f^*ω$ and $(φ\circ f)^\star(γ)$. Assume $α\wedge [φ(τ), ε]$ is the $φ(τ)$-contribution to $\ob{Y_τ}\wedge γ$ near $φ(τ^\circ)$. By \eqref{eq:local_structure_flat}, the contribution of $(s, φ\circ f)(σ)$ to $\ob{X_σ} \wedge p_2^*γ$ is then $p_2^*(α) \wedge [σ, s(μ/ν)\wedge ε]$. We obtain that the $τ$-contribution to $φ^\star(γ)$ is $φ^*(α)\wedge [τ, φ^{-1}(ε)]$ while the $σ$-contribution to $(s, φ\circ f)^\star(γ)$ is $(s, φ\circ f)^*(p_2^*α) \wedge [σ, μ/ν \wedge f^{-1}(ε)]$. This last expression equals $f^*(φ^*α)\wedge [σ, μ/ν\wedge f^{-1}(ε)]$, which is the $σ$-contribution to $f^*ω$. This is exactly what we needed to show.
\end{proof}

\begin{ex}\label{ex:delta_forms}
\begin{enumerate}[wide, labelindent=0pt, labelwidth=!, label=(\arabic*), topsep=4pt, itemsep=4pt]
\item The $δ$-forms on $\mbR^r$ of tridegree $(p,q,0)$ are precisely the currents of the form $η\wedge [\mbR^r,μ_{\mr{std}}]$, where $η\in PS^{p,q}(\mbR^r)$ is a piecewise smooth form. If $f:X\to \mbR^r$ is a linear map from a tropical space and $η$ as before, then $f^\star(η\wedge [\mbR^r, μ_{\mr{std}}]) = f^*η\wedge [X,μ]$. By definition, all $δ$-forms of tridegree $(p,q,0)$ on $X$ are locally of this form. In particular, we have a natural way to view $B^{p,q,0} \subseteq PS^{p,q}$ as subsheaf. Since not all pws forms on $X$ need to come via pullback along linear functions from some $\mbR^r$, not even locally, the inclusion is strict in general. Note that for every pws form $η$ on $X$, the polyhedral current $η\wedge [X, μ]$ is balanced by the tropical space property. In particular, being a $δ$-form is in general a strictly stronger property than being balanced.

\item Given a pwl function $ϕ\in Λ(X)$, one may define the associated \emph{Weil divisor} $\mr{div}(ϕ)\in P^{0,0,1}(X)$ in the usual way, see e.g. \cite{AR}*{\S6}: Working locally, assume that $X = |\mcX|$ for a polyhedral complex $\mcX$ that is subordinate to $ϕ$ and the tropical space structure. Then $\mr{div}(ϕ)$ will also be subordinate to $\mcX$. Given a face $τ\in \mcX^1$, pick an auxiliary weight $ν\in \det N_τ$ and (locally) a linear $ϕ_τ\in L$ with $ϕ_τ\vert_τ = ϕ\vert_τ$. The existence of such $ϕ$ is ensured by the axioms of sheaves of linear functions, cf. Def. \ref{def:sheaf_of_linear_functions}. The multiplicity of $τ$ in $\mr{div}(ϕ)$ is then
\begin{equation}\label{eq:corner_locus_defn}
\sum_{τ\subseteq σ\in \mcX_n}\frac{\partial (ϕ-ϕ_τ)\vert_σ}{\partial n_{σ,τ}} \wedge [τ, ν].
\end{equation}
In particular, $ϕ$ is harmonic if and only if $\mr{div}(ϕ) = 0$. Moreover, the Weil divisor may also be described as $\mr{div}(ϕ) = d'd''(ϕ\wedge [X,μ])$. Thus if $ϕ$ is a $δ$-form, then $\mr{div}(ϕ)$ is one as well.

\item It would be useful to also have a pushforward for $δ$-forms. Given a flat linear map $f:X\to Y$ of tropical spaces and some $γ\in B_c(X)$, one may consider $f_*(γ)$, which is a polyhedral current and satisfies the projection formula with respect to $f^*$. It is always balanced but need not be a $δ$-form. Simple examples for this can be constructed from the identity map $f:(X,μ,\ob L) \to (X,μ, L)$ when $\ob L \neq L$, but this is possibly not the only obstruction.
\end{enumerate}
\end{ex}

\section{Skeletons}

Fix a complete non-archimedean field $k$ together with an additive valuation $v:k\to \mbR\cup\{\infty\}$. We allow $v$ to be trivial. By \emph{analytic space over $k$} or simply \emph{analytic space}, we mean a Hausdorff $k$-analytic space in the sense of Berkovich \cite{Ber1}. (These are the \emph{good} analytic spaces from \cite{Ber2}.)

For an analytic space $X$ and $x\in X$, we denote by $\mcO_{X,x}$ the local ring in $x$. We write $v_x:\mcO_{X,x}\to \mbR\cup\{\infty\}$ for the valuation associated to $x$. Our notation for the boundary of $X$ is $\partial X$, cf. \cite{Ber2}*{§1.5.4} for its definition.

\subsection{Skeletons and Tropicalization}
We work with skeletons and tropicalizations in the sense of Ducros \cites{Duc_early, Duc_local} and Chambert-Loir--Ducros \cite{CLD}*{§2}. When we write $\mbG_m$ or $\mbG_m^r$, we always mean the analytic torus over $k$. An \emph{affine linear map} $\mbG_m^s\to \mbG_m^r$ is a morphism of the form
$$(t_1, \ldots, t_s) \mapsto (λ_i \prod_{j = 1}^s t_j^{n_{ij}})_{i = 1,\ldots,r}$$
where $(n_{ij})\in M_{r\times s}(\mbZ)$ and $(λ_1,\ldots,λ_r)\in (k^\times)^r$. The \emph{tropicalization map} for the torus is
$$t:\mbG_m^r\lr \mbR^r,\ \ \ x\longmapsto (v_x(t_1),\ldots,v_x(t_r)),$$
where $t_1,\ldots,t_r$ denote the standard coordinates. There is a standard way to define a continuous section for $t$, cf. \cite{CLD}*{(2.2.4)}, whose image is the \emph{standard skeleton} $\Sigma(\mbG_m^r)\subseteq \mbG_m^r$. We write $\Sigma(S)\subseteq \Sigma(\mbG^r_m)$ for the image of some subset $S\subseteq \mbR^r$ in the standard skeleton.

A \emph{moment map} or tuple of \emph{toric coordinates} for an analytic space $X$ is a morphism $f\colon X \to \mbG_m^r$ for some $r$. We call
$$t_f := t\circ f\colon X \lr \mbR^r$$
the \emph{tropicalization map} for $f$. We denote its image by $T'(X,f)$ and call it the \emph{tropicalization of $X$} with respect to $f$. A \emph{refinement} of a tuple of toric coordinates is a pair $(g,p)$ consisting of toric coordinates $g$, say with target $\mbG_m^s$, and an affine linear map $p:\mbG_m^s\to \mbG_m^r$ such that $f = p\circ g$.

\begin{thm}[\cite{Duc_local}*{Théorème 3.2}]\label{thm:trop_is_polyhedral}
Assume that $X$ is a compact $k$-analytic space of dimension $n$ and $f:X\to \mbG_m^r$ a tuple of toric coordinates. Then $T'(X,f)$ is a polyhedral set with $\dim T'(X,f)\leq n$. The image of the boundary $t_f(\partial X)$ is contained in a polyhedral set of dimension $\leq n - 1$. It is itself a polyhedral set if $X$ is affinoid.
\end{thm}

Ducros also studied the local geometry of $t_f$ in terms of the $\mbR$-graded residue field $\wt{\mcH(x)}$ of $\mcO_{X,x}$ defined by Temkin \cites{Tem1, Tem2}. We briefly recall that an $\mbR$-graded field is, by definition, an $\mbR$-graded ring in which every nonzero \emph{homogeneous} element is invertible. We refer to \cite{Duc_local}*{\S0} for more background and further related notions. Here, we only recall that the graded residue field $\wt{\mcH(x)}$ is defined as
\begin{equation}\label{eq:recap_graded_residue_field}
\wt{\mcH(x)} = \bigoplus_{r\in \mbR} \mcO_{X,x,v\geq r} / \mcO_{X,x,v > r}
\end{equation}
where $\mcO_{X,x,v \geq r}$ resp. $\mcO_{X,x,v >r}$ denotes the subgroup of those functions whose valuation in $x$ is $\geq r$ resp. $>r$. The graded field $\wt{\mcH(x)}$ extends the graded residue field $\wt k$ of $k$ and we write
$$d(x) := \trdeg(\wt{\mcH(x)} / \wt k).$$
Given $m\in \mcO_{X,x}$ with $v_x(m) \neq \infty$, we denote by $\wt{m} \in \wt{\mcH(x)}$ the graded reduction
$$\wt{m} := \big(m \text{ mod } \mcO_{X,x,v >v_x(m)}\big).$$
It is a homogeneous unit of degree $v_x(m)$. If $v_x(m) = \infty$, then we put $\wt{m} = 0$.

Coming back to our toric coordinates $f = (f_1,\ldots,f_r)$, we may now for each $x\in X$ consider the graded subfield $\wt k( \wt {f_x})\subseteq \wt {\mcH(x)}$ generated by the graded reductions $\wt{f_{1,x}},\ldots,\wt{f_{r,x}}$ of the germs of the $f_i$ in $x$.
\begin{thm}[\cite{Duc_local}*{Théorème 3.4}]\label{thm:trop_local_structure}
Assume $X$ is a $k$-analytic space and $f:X\to \mbG_m^r$ a tuple of toric coordinates. Let $x\in X\setminus \partial X$ be any point away from the boundary. Then $x$ has an affinoid neighborhood $U$ in $X$ such that for every affinoid neighborhood $V$ of $x$ in $U$, the tropicalization $T'(V, f)$ is pure-dimensional at $t_f(x)$ with
$$\dim_{t_f(x)} T'(V, f) = \trdeg(\wt k(\wt{f_x}) / \wt k).$$
\end{thm}

Assume now that $X$ is compact and purely $n$-dimensional. We denote by $T(X,f) \subseteq T'(X,f)$ the $n$-dimensional locus, defined as the union of all $n$-dimensional polyhedra contained in $T'(X,f)$. Write $T = T(X,f)$ in the following. Chambert-Loir--Ducros \cite{CLD} use degrees of maps from $X$ to tori to endow $T$ with weights, which relies on the following basic fact.
\begin{lem}[\cite{CLD}*{§2.4.3}]\label{lem:automatic_finite_flatness}
Let $q:X→\mbG_m^n$ be a map and $w\in \Sigma(\mbG_m^n)$ such that $\partial X \cap q^{-1}(w) = \emptyset$. Then there exists an open neighborhood $w\in U$ such that $q^{-1}(U)→U$ is finite flat.
\end{lem}

\begin{defn}[Weights on $T(X,f)$ after \cite{CLD}*{§3.5}]\label{def:weights_on_trop}
Choose a polyhedral complex structure $\mcT$ on $T$ such that $t_f(\partial X)$ is contained in its $(n-1)$-skeleton. For $σ\in \mcT_n$, set $X(σ^\circ):=t_f^{-1}(σ^\circ)$. Choose a surjective affine linear morphism $q:\mbG_m^r→\mbG_m^n$ such that $q\vert_σ$ is injective. Here, $q:\mbR^r\to \mbR^n$ also denotes the induced map on tropicalizations. The composite map
\begin{equation}\label{eq:fin_flat_trop}
q\circ f:X(σ^\circ)→\mbG_m^n(q(σ^\circ))
\end{equation}
is then finite flat of some degree $d_{σ,q}$ over a neighborhood of $\Sigma(q(σ^\circ))$ by Lem. \ref{lem:automatic_finite_flatness}. We endow $σ$ with weight
$µ_σ := d_{σ,q} \cdot q^{-1}(µ_{\mr{std}})$.
\end{defn}
The independence of $μ_σ$ from the choice of $q$ has been proved in \cite{GJR}. Namely, $μ_σ$ is invariant under extension of the base field, so we may assume $k$ to be nontrivially valued. Then the independence from $q$ is \cite{GJR}*{Prop. 4.6}, combined with \cite{GJR}*{(4.1.2)} and \cite{GJR}*{Rmk. 4.9}.

\begin{thm}[\cite{CLD}*{Thm. 3.6.1}]\label{thm:trop_is_trop}
$(T(X,f), µ)$ is an $n$-dimensional tropical cycle away from $t_f(\partial X)$.
\end{thm}

\begin{rmk}\label{rmk:comparison_of_weight_definitions}
The discussion so far extends to not necessarily compact $X$ if we instead assume that $t_f$ is a compact map. In particular, Chambert-Loir--Ducros' definition turns the tropicalization of the analytification of a closed subvariety $(X^{\mr{alg}})^{\mr{an}}\subseteq \mbG_m^r$ into a tropical cycle. It is \cite{Gub_forms_currents}*{Proposition 7.11} that their weights agree with the ones defined through initial degenerations by Speyer \cite{Speyer_thesis}, at least if $v(k^\times)\neq \{0\}$.
\end{rmk}

We next introduce the skeleton $\Sigma(X,f)$, whose definition is due to Ducros, cf. \cites{Duc_early, Duc_local}. The definition is moreover local on $X$, so we may drop the compactness assumption on $X$. First consider the closed subspace
\begin{equation}\label{eq:def_skeleton}
Σ'(X,f) = \bigcup_{q:\mbG_m^r→\mbG_m^n\ \text{affine linear}} (q\circ f)^{-1} (\Sigma(\mbG_m^n)).
\end{equation}
By \cite{Duc_local}*{(0.13)}, the standard skeleton $\Sigma(\mbG_m^n)$ has the description
\begin{equation}\label{eq:alternative_description_std_skeleton}
\Sigma(\mbG_m^n) = \{x\in \mbG_m^n \mid \trdeg(\wt k(\wt{t_1},\ldots,\wt{t_n}) / \wt k) = n\}.
\end{equation}
This identity immediately implies the alternative description of the skeleton as
\begin{equation}\label{eq:skeleton_characterization_residue_fields}
Σ'(X,f) = \{x\in X\mid \trdeg(\wt k(\wt {f_x})/\wt k) = n\}.
\end{equation}
Recall that points $x\in X$ with $d(x) = n$ are called \emph{Abhyankar} points. Thus, the set of Abhyankar points may be thought of as the union of all (local) skeletons.

Moreover, we see that $\Sigma'(X, f) = \emptyset$ if $r<n$. If $r\geq n$, then \eqref{eq:skeleton_characterization_residue_fields} shows that it suffices to take the union in \eqref{eq:def_skeleton} only over the finitely many projection maps.

Ducros' \cite{Duc_local}*{Théorème 5.1} states that $\Sigma' := \Sigma'(X,f)$ is naturally a pwl space of dimension $\leq n$. Its pwl structure is uniquely characterized by the property that for every compact analytic domain $U\subseteq X$ and $g\in \mcO^\times_X(U)$, the intersection $U\cap \Sigma'\subseteq \Sigma'$ is a closed pwl subspace and the value function $v(g):s\mapsto v_s(g)$ is pwl. More precisely, varying such $U$ and $g$ defines the sheaf of pwl functions on $\Sigma'$. In particular, if $t_f$ is a proper map, the restriction of $t_f$ to $\Sigma'(X, f)$ defines a pwl map
\begin{equation}\label{eq:map_skeleton_to_trop}
t_f: \Sigma'(X, f) \lr T'(X,f).
\end{equation}

\begin{prop}\label{prop:skeleton_pure_dim}
Assume that $X$ and $f$ are as before. The pwl space $\Sigma'(X, f)\setminus \partial X$ is purely of dimension $n$.
\end{prop}
\begin{proof}
By \eqref{eq:skeleton_characterization_residue_fields} and Thm. \ref{thm:trop_local_structure}, $\Sigma'(X, f)\setminus \partial X$ equals the set of those points $x\in X\setminus \partial X$ such that $T'(U, f)$ is $n$-dimensional for all affinoid neighborhoods $U$ of $x$. Moreover, we already mentioned that $\dim \Sigma'(X, f)\leq n$ (\cite{Duc_local}*{Théorème 5.1}).

We now prove the corollary by contradiction. If $\dim_x\Sigma'(X, f) < n$, then there exists an affinoid neighborhood $U$ of $x$ such that $\dim \Sigma'(U, f) < n$. Then $S = t_f(\Sigma'(U, f)) \cup t_f(\partial U)$ is a polyhedral set of dimension $\leq n-1$ and we may find an affinoid domain $V\subseteq U\setminus t_f^{-1}(S)$ that has the properties $\Sigma(V, f) = \emptyset$ and $\dim T'(V, f) = n$. We claim that this is impossible.

By assumption, every point $x\in V$ satisfies $\trdeg(\wt k(\wt{f_x}) / \wt k) < n$. Thus by \cite{Duc_local}*{Théorème 3.4}, any $x\in V$ has an affinoid neighborhood $V_x$ in $V$ with $\dim T'(V_x, f) < n$. Since $V$ is compact, finitely many such $V_x$ cover it and we obtain $\dim T'(V, f) < n$, in contradiction with our assumptions. This proves the claim and the proposition.
\end{proof}

We note two further consequences of \eqref{eq:skeleton_characterization_residue_fields} in conjunction with Thm. \ref{thm:trop_local_structure}. These can also be found in \cite{GJR}*{Prop. 3.12 (7) and Thm. 3.15}. First, the set $\Sigma'(X,f)\setminus \partial X$ depends only on the continuous map $t_f$. Second, if $q:\mbG_m^r\to \mbG_m^n$ is generic for $\Sigma'(X, f)$, then
\begin{equation}\label{eq:skeleton_away_from_boundary}
\Sigma'(X,f)\setminus \partial X = \Sigma'(X,q\circ f)\setminus \partial X.
\end{equation}
Here, generic means that $q\circ t_f$ has discrete fibers in $\Sigma'(X, f)$.

\begin{defn}\label{def:skeleton_as_tropical_space}
\begin{enumerate}[wide, labelindent=0pt, labelwidth=!, label=(\arabic*), topsep=4pt, itemsep=4pt]
\item The \emph{skeleton} of $f$ in $X$ is defined as the $n$-dimensional locus of $\Sigma'(X,f)$. It is denoted by $\Sigma(X,f)$. By Prop. \ref{prop:skeleton_pure_dim}, we always have
$$\Sigma'(X,f)\setminus \partial X = \Sigma(X,f)\setminus \partial X.$$

\item The skeleton becomes a weighted pwl space by the following definition. Working locally, we may assume $\Sigma(X,f)$ to be a polyhedral set. Choose a polyhedral complex structure for $\Sigma(X,f)$ such that $\Sigma(X,f) \cap\partial X$ is contained in its $(n-1)$-skeleton. For each $n$-dimensional polyhedron $σ$, pick a morphism $q:\mbG_m^r→\mbG_m^n$ such that $(q\circ f)\vert_σ$ is injective. Denoting by $d_{σ,q}$ the degree of $X$ over $\Sigma((q\circ f)(σ^\circ))\subseteq \Sigma(\mbG_m^n)$ (Lem. \ref{lem:automatic_finite_flatness}), endow $σ$ with the weight
$$μ_σ = d_{σ,q} \cdot t_{q\circ f}^{-1}(μ_{\mr{std}}).$$
This is independent of $q$ again by \cite{GJR}*{Prop. 4.6}.

\item Endow $\Sigma(X,f)$ with the sheaf of linear functions $L$ generated by all $t_{f_i},\ f =(f_1,\ldots,f_r)$.
\end{enumerate}
\end{defn}

\begin{rmk}\label{rmk:independence_weights}
The independence of $μ$ from $q$ in (2) also has the remarkable consequence that $(\Sigma(X,f), μ)$ does not depend on $f$ in the following sense. Whenever $g:X\to \mbG_m^s$ are further toric coordinates, then the weight of $\Sigma(X, f)$ as subspace of $\Sigma(X, f\times g)$ is also $μ$. Namely for any $n$-dimensional $σ\subseteq \Sigma(X,f)$ as above, the generic map $q:\mbG_m^{r+s}\to \mbG_m^n$ may be chosen to factor through $\mbG_m^r$. The argument more generally applies to any refinement $g$ of $f = p\circ g$. 
\end{rmk}

\begin{ex}\label{ex:dependence_on_functions}
We emphasize that the sheaf of linear functions in (3) is really an additional datum and cannot be constructed from the underlying weighted pwl space. We gave an abstract example in \ref{ex:trop_spaces} (3). Here is one of a different nature that involves an analytic space. Let $X = \mbA^1\setminus \{0, 1\}$ with coordinate $x$ and consider $f = x(x-1)$ and $g = (x, x-1)$. As pwl spaces, $\Sigma(X, f)$ and $\Sigma(X, g)$ are both the union of the three rays leading from the Gauss point of the unit disk around $0$ toward the three classical points $0$, $1$ and $\infty$. The weights on the three rays are $\ell^{-1}(μ_{\mr{std}})$, where $\ell = t_x$ or $t_{x-1}$ depending on the ray, and where $μ_{\mr{std}}$ is the standard weight on $\mbR$. There are strictly more linear functions on $\Sigma(X, g)$ than on $\Sigma(X, f)$.
\end{ex}

We finally discuss the relation of skeletons with tropicalizations. Assume again that $X$ is compact. One always has $t_f(\Sigma'(X,f))\subseteq T'(X,f)$ but the inclusion might be strict. For example, $\Sigma'(X,f) = \emptyset$ if $r < n$. For the $n$-dimensional loci however, we always have $t_f(\Sigma(X,f)) = T(X,f)$. The definitions of weights are compatible, providing
\begin{equation}\label{eq:rel_trop_skeleton}
T(X,f) = t_{f,*}\Sigma(X,f).
\end{equation}
\begin{cor}\label{cor:skeleton_is_tropical}
\begin{enumerate}[wide, labelindent=0pt, labelwidth=!, label=(\arabic*), topsep=4pt, itemsep=4pt]
\item Assume $X$ compact or, more generally, that $t_f$ is a compact map. Then $T(X,f)$ is a tropical cycle away from $t_f(\Sigma(X,f)\cap \partial X)$.

\item The triple $(\Sigma(X,f)\setminus \partial X, μ, L)$ is a tropical space.
\end{enumerate}
\end{cor}
\begin{proof}
The first statement is proved by Chambert-Loir--Ducros in their already cited \cite{CLD}*{Thm. 3.6.1}. (Their proof ultimately relies on Lem. \ref{lem:automatic_finite_flatness}, so only $\Sigma(X,f)\cap \partial X$ intervenes.) The statement on $\Sigma(X,f)$ then follows from Cor. \ref{cor:equiv_char_trop_space} (3). 
\end{proof}

For later reference, we end this section with a generalization of Lem. \ref{lem:automatic_finite_flatness}, which is proved the same way.
\begin{lem}[\cite{CLD}*{§2.4.3}]\label{lem:automatic_finite_flatness_II}
Let $π:X→Y$ be a map of compact $k$-analytic spaces with both $X$ and $Y$ purely of the same dimension $n$, let $y\in Y$ be Abhyankar. Assume that $Y$ is reduced and that $\partial X \cap π^{-1}(y) = \emptyset$. Then there is an open neighborhood $y\in U$ such that $π^{-1}(U)→U$ is finite flat.
\end{lem}

\subsection{Regular Immersions}
Let $X$ be a purely $n$-dimensional affinoid analytic space and $Z\subseteq X$ a Zariski closed subspace. We assume that $Z$ equals the vanishing locus $V(f_1,\ldots,f_r)$ of a regular sequence of length $r$. In particular, $\dim Z = n-r$. Denote by $U = X\setminus V(f_1\cdots f_r)$ the open subspace where all $f_i$ are invertible. For any constant $C\in \mbR$, we write $X_{\geq C}$ for the Weierstrass domain where $v(f_i)\geq C$ for all $i$. Similarly for $U_{\geq C}$.

\begin{prop}\label{prop:reg_sequence_statement}
Let $g: X\to \mbG_m^s$ be a set of toric coordinates and put
$$h = (f,g): X \lr \mbA^r\times \mbG_m^s.$$
There exists a constant $C$ such that the natural map
$$T'(U_{\geq C}, h)\to [C,\infty)^r\times \mbR^s$$
induces an isomorphism of weighted polyhedra
$$T(U_{\geq C}, h) \iso [C,\infty)^r \times T(Z, g).$$
Here, $[C, \infty)^r$ is endowed with the standard weight from $\mbR^r$.
\end{prop}

It might be possible that there is an induction argument that deduces Prop. \ref{prop:reg_sequence_statement} from the case $r = 1$, which is the case that has been proved by Chambert-Loir--Ducros \cite{CLD}*{Prop. 4.6.6}. However, the space $U_{\geq C}$ that intervenes is not affinoid anymore so such an induction does not seem to be straightforward. Instead, our proof will essentially repeat the arguments from \cite{CLD}*{Prop. 4.6.6}. It is also similar to \cite{Liu}*{Step 6 of Proof of Thm. 5.8} where the algebraic smooth situation was considered.

\begin{proof} Let $T' = T'(Z,g)$, $T = T(Z, g)$ and $V = t_g^{-1}(T')$. Then $V$ is an analytic domain in $X$ and contains $Z$. The map $Z\to X$ is a closed immersion and hence without boundary, cf. \cite{Duc_families}*{1.3.21 (1)}. Thus $V$ is a neighborhood of $Z$ by \cite{Duc_families}*{1.3.21 (3)}. The map $t_f:X \to (-\infty, \infty]^r$ is proper because $X$ is compact Hausdorff, so there is some $C_1 > 0$ with $X_{\geq C_1}\subseteq V$. It follows that
\begin{equation}\label{eq:polyhedra_containment}
t_h(X_{\geq C_1}) \subseteq [C_1,\infty]^r\times T'.
\end{equation}
Next, we claim that there exist a compact polyhedral set $\Delta \subseteq T'$ of dimension $<n-r$ and a constant $C_2 > 0$ such that
\begin{equation}\label{eq:polyhedra_containment_boundary}
t_h(\partial X \cap X_{\geq C_2}) \subseteq [C_2,\infty]^r\times \Delta.
\end{equation}
\begin{proof}[Proof of the claim.] The following arguments are taken from \cite{CLD}*{Proof of Prop. 4.6.6}. First, the boundary $\partial X$ is contained in a finite union of sets of the form $φ^{-1}(η_s)$, where $φ:X\to \mbA^1$ is some morphism and $η_s\in \mbA^1$ the Gauss point of the disk of radius $s\in \mbR$ around $0$, cf. \cite{Duc_local}*{Lemme 3.1}. It hence suffices to find, for each such pair $(φ,s)$, a polyhedral set $\Delta_{φ,s}$ and a constant $C_{φ,s}$ such that
\begin{equation}
\label{eq:polyhedra_containment_simplified}
t_h(φ^{-1}(η_s) \cap X_{\geq C_{φ,s}}) \subseteq [C_{φ,s},\infty]^r\times \Delta_{φ,s}.
\end{equation}
Indeed, we may then take $\Delta$ as a finite union of suitable $\Delta_{φ,s}$ and $C_2$ as the maximum of the corresponding $C_{φ,s}$. So fix $φ$ and $s$ as above from now on. If $z\in φ^{-1}(η_s) \cap Z$, then
$$\mr{trdeg}(\wt{\mcH(z)} / \wt{\mcH(η_s)}) \leq n-r-1$$
because $\trdeg(\wt{\mcH(η_s)} / \wt{k}) = 1$ and $\dim Z = n-r$. The fiber $φ^{-1}(η_s)\cap Z$ is thus an affinoid $\mcH(η_s)$-analytic space of dimension $\leq n-r-1$. It follows from Thm. \ref{thm:trop_is_polyhedral} that $\Delta_{φ,s} := t_g(φ^{-1}(η_s)\cap Z)$ is a polyhedral set of dimension $\leq n-r-1$. Then $V_{φ,s} := t_g^{-1}(\Delta_{φ,s})$ is an analytic domain of $X$ that contains $φ^{-1}(η_s)\cap Z$. Moreover, $φ^{-1}(η_s)\cap Z \to φ^{-1}(η_s)$ is a closed immersion, hence without boundary, so $V_{φ, s}$ is a neighborhood of $φ^{-1}(η_s)\cap Z$ by \cite{Duc_families}*{1.3.21 (3)} again. Since $φ^{-1}(η_s)$ is a closed subset of $X$, hence compact, the restricted map $t_f\vert_{φ^{-1}(η_s)}:φ^{-1}(η_s) \to (-\infty,\infty]^r$ is proper. Thus there exists $C_{φ,s}$ such that $X_{\geq C_{φ,s}} \cap φ^{-1}(η_s) \subseteq V_{φ,s}$ and our claim is proved.
\end{proof}

Let $\Delta$ and $C_2$ be as in the claim and put  $C = \max\{C_1, C_2\}$. We claim that there is an equality
\begin{equation}\label{eq:to_show_prop_reg_seq}
T(U_{\geq C}, h) = [C,\infty)^r\times T
\end{equation}
as weighted purely $n$-dimensional polyhedral sets. This claim is local over $T$ in the following sense. Choose a rational polyhedral complex structure $\mcT$ for $T'$ that is also subordinate to $\Delta$. For $σ\in \mcT_{n-r}$, the inverse image $X(σ) = t_g^{-1}(σ)\subseteq X$ is an affinoid domain and the same holds for the intersection $Z(σ) = Z \cap X(σ)$. Also, \eqref{eq:polyhedra_containment} and \eqref{eq:polyhedra_containment_boundary} hold with the same set $\Delta$ and the same constant $C$ for $Z(σ)\subseteq X(σ)$. We may hence assume that $T' = σ$ and that $σ^\circ \cap \Delta = \emptyset$. Then it is enough to show \eqref{eq:to_show_prop_reg_seq} after pushforward under any generic (for $σ$) surjection $q:\mbR^s\to \mbR^{n-r}$. Let $q:\mbG_m^s\to \mbG_m^{n-r}$ realize such a map and consider the product map
\begin{equation}\label{eq:flatness_reg_immersion}
\overbar h = (f, q\circ g): X → \mbA^r\times \mbG_m^{n-r}.
\end{equation}
We claim that $\overbar h$ is finite flat over a neighborhood of $0\times \Sigma(q(σ)^\circ)$, where $0 = (0,\ldots,0)\in \mbA^r$ and where $\Sigma(q(σ)^\circ) \subseteq \Sigma(\mbG_m^{n-r})$ via the standard section.

In order to prove the finiteness, we observe that $q\circ g\vert_Z :Z\to \mbG_m^{n-r}$ is a morphism of analytic spaces of the same dimension. Thus all fibers of $q\circ g\vert_Z$ over $\Sigma(q(σ))$ are finite. Moreover, $(q\circ g)(\partial Z) \cap \Sigma(q(σ)^\circ) = \emptyset$ by construction. Since $\partial Z = \partial X \cap Z$, we equivalently know $\overbar{h}(\partial X)\cap (0\times \Sigma(q(σ)^\circ) = \emptyset$. In this situation, \cite{Ber2}*{Prop. 3.1.4, equivalence of (a) and (c)} states that $\overbar{h}$ is finite over a neighborhood of $0\times \Sigma(q(σ)^\circ)$.

We turn to the proof of flatness. Every point $z\in (q\circ g)^{-1}(\Sigma(q(σ)^\circ))$ is an Abhyankar point of $Z$. In this situation, \cite{Duc_families}*{Example 3.2.10} states that the local ring $\mcO_{X,z}$ has Krull dimension $r = \dim_z X - \dim_z Z$. By assumption, the germs $f_{1,z},\ldots, f_{r,z}$ form a regular sequence in $\mcO_{X,z}$. It follows that $\mcO_{X, z}$ is Cohen--Macaulay. The local rings of $\mbA^r\times \mbG_m^{n-r}$ are regular and the local rings $\mcO_{\mbA^r\times \mbG_m^{n-r}, (0,y)}$, for $y\in \Sigma(\mbG_m^{n-r})$, are of the same Krull dimension $r$. Furthermore, we already noted that $\overbar h$ is finite over a neighborhood of $0\times \Sigma(q(σ)^\circ)$. It now follows from ``miracle flatness'' \cite{Stacks}*{Tag 00R4} that
$$\overbar{h}^*: \mcO_{\mbA^r\times \mbG_m^{n-r}, \overbar{h}(z)} \lr \mcO_{X, z}$$
is flat for every $z\in (q\circ g)^{-1}(\Sigma(q(σ)^\circ))$. By openness of flatness for finite maps \cite{Ber2}*{Prop. 3.2.8}, $\overbar{h}^{-1}$ is flat over an open neighborhood of $\Sigma(0\times q(σ)^\circ)$. Since $0\times q(σ)^\circ$ is connected, we obtain that $\bar h$ is finite flat of some uniform degree $e$ over an open neighborhood of $0\times \Sigma(q(σ)^\circ)$.

Now consider the restricted map $\overbar h: U_{\geq C} → \mbG_m^r\times \mbG_m^{n-r}$. It is finite and flat above $\Sigma((C,\infty)^r\times q(σ)^{\circ})$ of some degree $e'$ by Lem. \ref{lem:automatic_finite_flatness}. Since $0\times q(σ)^\circ$ lies in the closure of $(C,\infty)^r\times q(σ)^\circ$, we obtain $e = e'$. 

By the definition of tropicalization, the weight of the polyhedron $σ$ in $T(Z, g)$ is $e\cdot q^{-1}(μ_{\mbR^{n-r}})$. The weight on $[C,\infty)^r\times σ$ in $T(U_{\geq C}, f\times g)$ is $e\cdot (\mr{id}_{\mbR^r}\times q)^{-1}(μ_{\mbR^n}) = μ_{\mbR^r}\wedge e\cdot q^{-1}(μ_{\mbR^{n-r}})$. Thus
$$T(U_{\geq C}, f\times g) = [C, \infty)^r \times T(Z, g)$$
as weighted polyhedral sets, which is precisely what we needed to verify.
\end{proof}

\subsection{Dominant Morphisms}
The main result of this section is Thm. \ref{thm:skeleton_flat}, which states that for a (locally on $X$) \emph{dominant} map $π:X\to Y$ of pure-dimensional analytic spaces, the skeletons in $X$ lie flat over the skeletons of $Y$. \emph{Flatness} here is meant in the sense of weighted pwl spaces, cf. Def. \ref{def:flat}.

\begin{lem}[also cf. \cite{GJR}*{Lem. 4.4}]\label{lem:irred_comp_decomp_of_skeleton}
Let $X = \mcM(A)$ be an affinoid analytic space that is purely of dimension $n$. Let $\mfp_1,\ldots,\mfp_r$ denote the minimal prime ideals of $A$, put $X_i = \mcM(A/\mfp_i)$ and set $e_i = \mr{len}_{A_{\mfp_i}} A_{\mfp_i}.$ Let $μ_i$ denote the weight on $\Sigma(X_i, f)$. Then for any set of toric coordinates $f:X\to \mbG_m^r$,
\begin{equation}\label{eq:irred_comp_decomp_of_skeleton}
(\Sigma(X,f), μ) = \coprod_{i = 1}^r (\Sigma(X_i, f), e_i \cdot μ_i).
\end{equation}
Here, the right hand side is viewed as a subset of $X$ along the inclusions $X_i \to X$.
\end{lem}
\begin{proof}
The intersections $X_i \cap X_j$, $i\neq j$, are of dimension $\leq n-1$. It follows from the characterization \eqref{eq:skeleton_characterization_residue_fields} that no point of $X_i\cap X_j$ lies on $\Sigma'(X,f)$. This already implies the identity \eqref{eq:irred_comp_decomp_of_skeleton} for underlying sets. The claim on weights follows from the fact that $\mbG_m^n$ is reduced. Namely let $σ\subseteq \mbR^n$ be a compact polyhedron and let $\mcM(B)\subseteq \mbG_m^n$ be an affinoid domain with $\mcM(B)\cap \Sigma(\mbG_m^n) = \Sigma(σ)$. Assume that $B \to C$ is a finite flat map to a $k$-affinoid algebra $C$ that has a unique minimal prime ideal $\mfp\subseteq C$. Then
$$[C:B] = \dim_{\mr{Frac}(B)} C_\mfp = \mr{len}_{C_\mfp}C_\mfp \cdot \dim_{\mr{Frac}(B)} \mr{Frac}(C/\mfp)$$
and $\dim_{\mr{Frac}(B)} \mr{Frac}(C/\mfp)$ is the degree of $\mcM(C/\mfp) \to \mcM(B)$ above $σ$.
\end{proof}

\begin{lem}\label{lem:finite}
Let $π:X\to Y$ be a map of purely $n$-dimensional analytic spaces that is finite flat of degree $d$. Let $f:Y\to \mbG_m^r$ be a tuple of toric coordinates. Then $π$ restricts to a faithfully flat map of weighted pwl spaces
$$\Sigma(X, f\circ π) \lr \Sigma(Y, f)$$
that is of degree $d$ in the sense that $π_*\Sigma(X, f\circ π) = d\cdot \Sigma(Y, f)$.
\end{lem}
\begin{proof}
First note that $\Sigma'(X, f\circ π) = π^{-1}(\Sigma'(Y,f))$ by definition, cf. \eqref{eq:def_skeleton}. The restriction of $π$ to $\Sigma'(X, f\circ π)$ is pwl because $v(g\circ π) = π\circ v(g)$ for every invertible holomorphic function $g$ on (an affinoid domain of) $Y$. The fibers of $π$ are finite, so $π$ restricts to a map on the $n$-dimensional loci $\Sigma(X, f\circ π) \to \Sigma(Y,f)$. We apply the definition of weights, cf. Def. \ref{def:skeleton_as_tropical_space}: Working locally on $Y$, we may assume both skeletons to be polyhedral sets. Let $\mcX$ and $\mcY$ be respective polyhedral complex structures that are subordinate to $π$. Let $\ob{σ}\in \mcY_n$ and let $q:\mbG_m^r\to \mbG_m^n$ be a generic projection for $\ob{σ}$. Denote by $d_{\ob{σ},q}$ resp. $d_{σ,q}$ the degree of $q\circ f$ near $\ob{σ}$ (resp. of $q\circ f \circ π$ near $σ\in π^{-1}(σ)\subseteq \mcX_n$). Then, denoting the weight on a polyhedron $τ$ by $μ_τ$,
$$\begin{aligned}
μ_{\ob{σ}} &\ = d_{\ob{σ},q} \cdot t_{q\circ f}^{-1}(μ_{\mr{std}})\\
&\ = d^{-1} \sum_{σ\in \mcX_n,\ π(σ) = \ob{σ}} d_{σ, q} \cdot t_{q\circ f}^{-1}(μ_{\mr{std}})\\
&\ = d^{-1} \sum_{σ\in \mcX_n,\ π(σ) = \ob{σ}} π(μ_σ).
\end{aligned}$$
This shows $π_*\Sigma(X, f\circ π) = d\cdot \Sigma(Y, f)$.

We would like to apply Prop. \ref{prop:flatness_characterization} to show flatness. (Also see Ex. \ref{ex:push_forward_does_not_preserve_pws} (1) for the specialization to relative dimension $0$.) This criterion is local on $Y$, so we may assume $Y$ (and \emph{a forteriori} $X$) to be affinoid. Let $x\in \Sigma(X, f\circ π)$ be any point and put $y = π(x)$. Using the topological properness of $X\to Y$ and the fact that $π^{-1}(y)$ is finite, we find affinoid neighborhoods $y \in V$ and $x\in U$ such that $π$ restricts to a finite map $U\to V$ and such that $π^{-1}(y)\cap U = x$. By the previous arguments, $π_*\Sigma(U, f\circ π) = d_x\cdot \Sigma(V, f)$ for the local degree $d_x$. This means $\Sigma(X, f\circ π)\to \Sigma(Y, f)$ has well defined fiber weights in the sense of Def. \ref{def:fiber_weight} (also see Ex. \ref{ex:push_forward_does_not_preserve_pws} (1)) and is hence flat. It is moreover surjective, hence faithfully flat because $\Sigma(X, f\circ π)$ and $\Sigma(Y, f)$ have the same dimension. The proof is complete.
\end{proof}

\begin{lem}\label{lem:bc_for_skeletons}
Let $X$ be a purely $n$-dimensional analytic space and denote by $f:X\to \mbG_m^r$ a tuple of toric coordinates. Let $K/k$ be a complete valued extension and denote by $X_K$ the base change of $X$ to $K$. Then the projection $π:X_K\to X$ restricts to a faithfully flat map of weighted pwl spaces
\begin{equation}\label{eq:restriction_bc}
\Sigma(X_K, f) \lr \Sigma(X,f).
\end{equation}
\end{lem}
\begin{proof}
First note that $\Sigma(\mbG_{m,K}^n)$ maps to $\Sigma(\mbG_m^n)$ under the natural projection map, so $π$ restricts to a map $\Sigma'(X_K,f)\to \Sigma'(X,f)$ by definition of these sets as union of inverse images of $\Sigma(\mbG_{m,K}^n)$ resp. $\Sigma(\mbG_m^n)$, cf. \eqref{eq:def_skeleton}. This restriction is pwl by the same arguments as before.

Next, consider the case where $K/k$ is a finite extension. Then one may view $X_K$ also as $k$-analytic space and we write $\Sigma_k(X_K, f)$ for the skeleton as $k$-analytic space. The map $X_K\to X$ is finite flat of degree $[K:k]$. So by Lem. \ref{lem:finite}, $π$ restricts to a faithfully flat map $\Sigma_k(X_K, f)\to \Sigma(X, f)$.

For every point $x \in X_K$, we have
$$\trdeg(\wt k(\wt {f_x})/\wt k) = \trdeg(\wt K(\wt {f_x})/\wt K)$$
because $K/k$ is finite. Thus by \eqref{eq:skeleton_characterization_residue_fields}, $\Sigma(X_K, f)$ and $\Sigma_k(X_K, f)$ agree as sets. They then agree as pwl spaces with linear functions because this part of the definition does not depend on the ground field. Concerning weights, we have that $\mbG_{m,K}^n\to \mbG_m^n$ is finite flat of degree $[K:k]$. One deduces directly from the definition (cf. Def. \ref{def:skeleton_as_tropical_space} (2)) that the weights on $\Sigma_k(X_K, f)$ and $\Sigma(X_K, f)$ are related by
$$μ_{\Sigma_k(X_K, f)} = [K:k]\cdot μ_{\Sigma(X_K, f)}.$$
This proves the lemma for $K/k$ finite.

We come back to the general situation. The claim of the lemma is local on $X$, so we may assume $X$ compact. Then \cite{Duc_local}*{Théorème 5.1 (2)} states that there is a finite separable extension $k_0/k$ with the property that for all complete extensions $K_0/k_0$, the natural map $\Sigma(X_{K_0}, f)\to \Sigma(X_{k_0},f)$ is a homeomorphism. In this case, it is also an isomorphism of weighted pwl spaces with linear functions. Namely, the tropical weights are defined in terms of degrees of finite flat maps and those degrees are invariant under base field extension.

So pick $k_0$ as above and let $K_0$ be a finite extension of $K$ containing $k_0$. Consider the diagram
\begin{equation}\label{eq:diagram_flatness}
\xymatrix{
\Sigma(X_{K_0},f) \ar[r] \ar[d] & \Sigma(X_K,f) \ar[d]\\
\Sigma(X_{k_0},f) \ar[r] & \Sigma(X,f).\\
}
\end{equation}
The horizontal maps are faithfully flat because $k_0/k$ and $K_0/K$ are finite (our previous case). The left vertical map is an isomorphism, in particular faithfully flat. Flatness is stable under composition, so $\Sigma(X_{K_0}, f)$ is flat over $\Sigma(X, f)$. The final statement now follows from the cancellation law in Prop. \ref{prop:flat_maps_faithful} (2).
\end{proof}

\begin{rmk}\label{rmk:flatness_tropical}
Finite morphisms are inner by \cite{Ber1}*{Cor. 2.5.13}. Thus in the situation of Lem. \ref{lem:finite}, $π^{-1}(\partial Y) = \partial X$ and hence $π$ restricts to a faithfully flat map $π:\Sigma(X, f\circ π)\setminus \partial X \to \Sigma(Y, f) \setminus \partial Y$ of tropical spaces that is of degree $d$.
\end{rmk}

\begin{lem}\label{lem:fin_dominant_covering}
Let $π:X\to Y$ be a map of purely $n$-dimensional analytic spaces and let $f:Y\to \mbG_m^r$ be a tuple of toric coordinates. Then $π$ restricts to a flat map of tropical spaces
\begin{equation}\label{eq:restrict_dominant_general}
\Sigma(X,f\circ π)\setminus \partial X \lr \Sigma(Y,f)\setminus \partial Y.
\end{equation}
\end{lem}
\begin{proof}
Just like in the proof of Lem. \ref{lem:finite}, $π$ restricts to a pwl (even linear) map $\Sigma'(X, f\circ π) \to \Sigma'(Y, f)$. Every $y\in \Sigma'(Y, f)$ is Abhyankar, so the fiber of $π$ over such $y$ is discrete. It follows that $π$ restricts to a map on $n$-dimensional loci $\Sigma(X, f\circ π) \to \Sigma(Y,f)$. Moreover, $π(X\setminus \partial X)\subseteq Y\setminus \partial Y$ by \cite{Ber2}*{Prop. 1.5.5 (ii)}, so \eqref{eq:restrict_dominant_general} is defined. (Recall for the application of the cited result that our analytic spaces are good and Hausdorff.)

The claim on flatness is local on both $\Sigma(X, f\circ π)$ and $\Sigma(Y,f)$, so we may assume $X$ and $Y$ to be affinoid, say $X = \mcM(B)$ and $Y = \mcM(A)$. Let $x\in \Sigma(X, f\circ π) \setminus \partial X$ and put $y = π(x)$. Localizing on $X$ and $Y$ while using the topological properness of $X\to Y$, we may assume that $\{x\} = π^{-1}(y)$. Lem. \ref{lem:irred_comp_decomp_of_skeleton} allows to replace $X$ resp. $Y$ by the irreducible components containing $x$ resp. $y$ and to assume further that $A$ and $B$ are integral domains. With these assumptions, Lem. \ref{lem:automatic_finite_flatness_II} yields that $π$ is finite and flat of some degree $d_x$ over a neighborhood of $y$. Now Lem. \ref{lem:finite} applies and completes the proof.
\end{proof}

\begin{prop}\label{prop:fin_maps_flat_on_skeletons}
Let $π:X→Y$ be a finite map of analytic spaces that are purely of the same dimension $n$ and let $g:X\to \mbG_m^r$ denote toric coordinates on $X$. For every point $y\in Y$, there exist an open neighborhood $y\in U$ and toric coordinates $f$ on $U$ such that $π$ restricts to a flat map of tropical spaces
\begin{equation}\label{eq:fin_map_flat_skeleton}
\Sigma(π^{-1}(U), g \times (f\circ π))\setminus \partial X \lr \Sigma(U, f)\setminus \partial Y.
\end{equation}
In fact, there exist $f$ and $U$ such that for every open neighborhood $y\in U'\subset U$ and every refinement $f'$ of $f\vert_{U'}$, $π$ restricts to a flat map
\begin{equation}\label{eq:fin_map_flat_skeleton_refined}
\Sigma(π^{-1}(U'), g \times (f'\circ π))\setminus \partial X \lr \Sigma(U, f')\setminus \partial Y.
\end{equation}
\end{prop}
\begin{proof}
First note that if $d(y) < n$, then also $d(x) < n$ for every point $x$ of $π^{-1}(y)$ by the finiteness of $π$. So there is nothing to prove for such $y$ and we assume $y$ Abhyankar from now on.

Let $U = \mcM(A)$ be an affinoid neighborhood of $y$. Then also $V = π^{-1}(U)$ is affinoid, say $V = \mcM(B)$, and $B$ is a finite $A$-algebra. The point $y$ lies on a unique irreducible component of $U$ since its graded residue field has maximal transcendence degree. In light of Lem. \ref{lem:irred_comp_decomp_of_skeleton}, we may assume $A$ integral.

Consider for each component $g_i$ of $g$ some monic polynomial $p_i\in A[T]$ with $p_i(g_i) = 0$. Let $f:U\to \mbA^{s}$ be the tuple of nonzero coefficients of all $p_i$. Then all components $f_j$ of $f$ are invertible in $y$ because $\mcO_{Y,y}$ is a field (use that $A$ is an integral domain and \cite{Duc_families}*{Ex. 3.2.10}). Shrinking $U$ further, we may assume all $f_j$ to be invertible on $U$. Then, with $V = π^{-1}(U)$ as before,
\begin{equation}\label{eq:skeleton_equality_finite_case}
\Sigma(V, g \times (f\circ π)) = \Sigma(V,f\circ π)
\end{equation}
because all graded field extensions $\wt k(\wt{g_x}, \wt{(f\circ π)_x})/\wt k( \wt{f_{π(x)}})$ are finite. Equality \eqref{eq:skeleton_equality_finite_case} even holds in the sense of weighted pwl spaces, because weights are invariant under refinement of coordinates. In this situation, Lem. \ref{lem:fin_dominant_covering} to $(U, f)$ and all refinements $(U', f')$ as in the proposition and the proof is complete.
\end{proof}

\begin{rmk}
Equality \eqref{eq:skeleton_equality_finite_case} only holds as weighted pwl spaces (and not as tropical spaces) because the linear functions $v(g_i)$ may not admit an expression as linear combination of the $v(f_j)\circ π$.
\end{rmk}

\begin{prop}\label{prop:flatness_in_explicit_rel_dim_1}
Let $A$ be a purely $n$-dimensional $k$-affinoid algebra and $g_1,\ldots,g_r\in A[t]$ polynomials. Put $Y = \mcM(A)$ and let $X \subseteq \mbA^1_Y$ be the open subspace where all $g_i$ are invertible. Denote by $π:X\to Y$ the natural map. Put $g = (g_1,\ldots,g_r)$.

Then for every point $x\in \Sigma(X,g)\setminus \partial X$, there exist open neighborhoods $U$ of $y = π(x)$ and $V\subseteq π^{-1}(U)$ of $x$ as well as toric coordinates $f$ for $U$ such that $π$ restricts to a flat map of tropical spaces
$$\Sigma(V, g \times (f\circ π))\setminus \partial X \lr \Sigma(U, f)\setminus \partial Y.$$
In fact, $U$, $V$ and $f$ may be chosen such that for every open subset $U'\subseteq U$ and every refinement $f'$ of $f\vert_{U'}$, the map $π$ restricts to a flat map
$$\Sigma(V\cap π^{-1}(U'), g \times (f'\circ π))\setminus \partial X \lr \Sigma(U', f')\setminus \partial Y.$$
\end{prop}
\begin{proof}
Applying Lem. \ref{lem:irred_comp_decomp_of_skeleton} to both $X$ and $Y$, we may without loss of generality assume that $A$ is an integral domain.

Fix $x$ and $y=π(x)$ as in the statement of the proposition. Every pair $\wt{g_{i,x}}$, $\wt{g_{j,x}}$, $i \neq j$, is algebraically dependent over $\wt {\mcH(y)}$ because $X\to Y$ is of relative dimension $1$. For each pair $i\neq j$, pick a nontrivial polynomial $\wt {p_{i,j}}\in \wt{\mcH(y)} [R,S]$ such that $\wt {p_{i,j}}(\wt{g_{i,x}}, \wt{g_{j,x}}) = 0$. Let $p_{i,j}\in A[R, S]$ denote any choice of polynomial that coefficient by coefficient lifts $\wt {p_{i,j}}$. Let $f:Y\to \mbA^s$ be given by the tuple of all nonzero coefficients of all these $p_{i,j}$, ordered in some arbitrary fixed way.

It follows from $\mr{trdeg}(\wt{\mcH(x)}/\wt {\mcH(y)}) \leq 1$ and $\mr{trdeg}(\wt {\mcH(x)}/\wt k) = n+1$ that $\mr{trdeg}(\wt {\mcH(y)}/\wt k) = n$. In particular, $y$ is Abhyankar. By our assumption for $A$ to be an integral domain and by \cite{Duc_families}*{Ex. 3.2.10}, the local ring $\mcO_{Y,y}$ is a field so all $f_i\in \mcO_{Y,y}$ are invertible. Let $U$ be the open neighborhood of $y$ in $Y\setminus \partial Y$ defined by the condition that all $f_i$ are invertible.

\emph{Claim: There exists an open neighborhood $V\subseteq π^{-1}(U)$ of $x$ such that for all $z\in V$, every pair $(\wt{g_{i,z}}, \wt{g_{j,z}})$ is algebraically dependent over $\wt k(\wt {f_{π(z)}})$.}

Such a neighborhood $V$ may be constructed as follows. For every point $z\in π^{-1}(U)$ and every pair $(i,j)$, we define a notion of degree $\deg_{z,i,j}$ on $\mcO(U)^\times[R, S]$ (set of nonzero polynomials all of whose nonzero coefficients are invertible functions on $U$) by
\begin{equation}\label{eq:degree_z_i_j}
\deg_{z, i, j}\left(\sum_{a, b \in \mbZ} λ_{a,b} R^a S^b\right) := \min_{a,b\in \mbZ} \left\{v_z(λ_{a,b} \,g_i^a\, g_j^b)\right\} \in \mbR.
\end{equation}
Then we define $V$ as the set of all $z\in π^{-1}(U)$ such that for all $i,j$ we have
\begin{equation}\label{eq:def_V}
v_z(p_{i,j}(g_i, g_j)) > \deg_{z,i,j} (p_{i,j}).
\end{equation}
Observe first that $x\in V$ by construction and that \eqref{eq:def_V} is an open condition. Hence $V$ is an open neighborhood of $x$. Further observe that if \eqref{eq:def_V} holds for a point $z$, then at least two monomials of $p_{i,j}$ attain the minimal valuation $\deg_{z,i,j}(p_{i,j})$ in $z$. One of them has to be nonconstant and hence
\begin{equation}\label{eq:alg_dependence_V}
\sum_{a,b\in \mbZ,\ v_z(λ_{a,b} g_i^a g_j^b) = \deg_{z,i,j}(p_{i,j})} \wt{λ_{a,b,π(z)}} R^a S^b \in \wt k(\wt{f_{π(z)}})[R, S]
\end{equation}
defines a nontrivial algebraic relation of $(\wt{g_{i,z}}, \wt{g_{j,z}})$ over $\wt k(\wt{f_{π(z)}})$. This shows that $V$ fulfills the requirements of the above claim.

We claim that with these choices $π(\Sigma(V, g \times f)) \subseteq \Sigma(U, f).$ To prove this, first note that $\partial U = \partial V = \emptyset$, so Prop. \ref{prop:skeleton_pure_dim} applies. Consider any point $z\in \Sigma'(V, g\times f)$. Using \eqref{eq:skeleton_characterization_residue_fields}, we pick elements
$$g_{i_1},\ldots,g_{i_d},f_{j_1},\ldots,f_{j_m},\quad d+m = n+1,$$
whose graded reductions in $z$ form a transcendence basis of $\wt {\mcH(z)}$ over $\wt k$ and make this choice with maximal possible $m$. We need to see that $m = n$. But if $m < n$, the definition of $V$ ensures the existence of a relation
$$p_{i_1,i_2}(\wt{g_{i_1, z}}, \wt{g_{i_2, z}}) = 0\ \ \ \mr{in}\ \wt {\mcH(z)}$$
with coefficients from $\{\wt{f_{1,z}},\ldots, \wt{f_{s,z}}\}$. Since the graded reductions $\{\wt{g_{i_1,z}}, \wt{g_{i_2,z}}\}$ are algebraically independent of the reductions of $\{f_{j_1,z},\ldots,f_{j_m,z}\}$, there has to be some coefficient $f_k$ of $p_{i_1,i_2}$ such that $\wt{f_{k,z}}$ is algebraically independent of $\{\wt{f_{j_1,z}},\ldots,\wt{f_{j_m,z}}\}$, in violation of our maximality assumption. This proves the claim. Note that the same arguments work for every refinement $f'$ of $f\vert_{U'}$ on some open subset $U'\subseteq U$.

It is clear that the so-defined map $π:\Sigma(V, g\times f)\to \Sigma(U, f)$ is pwl, even linear. We claim that it is also flat.

Let $(K,v)$ be a complete valued extension field of $k$ with $v(K^\times) = \mbR$. Lem. \ref{lem:bc_for_skeletons} says that $\Sigma(U_K, f) \to \Sigma(U,f)$ and $\Sigma(V_K, g\times f)\to \Sigma(V, g\times f)$ are faithfully flat. If we can show that 
$$\Sigma(V_K, g\times f) \lr \Sigma(U_K, f)$$
is flat, then the cancellation law of Prop. \ref{prop:flat_maps_faithful} will imply flatness of the map $\Sigma(V, g\times f)\to \Sigma(U,f)$. In other words, we can and will henceforth assume that $k$ has value group $\mbR$.

The given point $x$ is an Abhyankar point of the fiber $V_y = π^{-1}(y)$, which is homeomorphic to an open subspace of $\mbA^1_{\mcH(y)}$. Since the value group of $k$ is $\mbR$, every such point is the Gauss point of a disk of some positive radius in the algebraic closure $\ob {\mcH(y)}$. Shrinking $U$ (and $V$) if necessary, there exists a finite flat map $U'\to U$ together with a point $y'\mapsto y$ such that this disk is rational over $\mcH(y')$. Applying Lem. \ref{lem:fin_dominant_covering} and Prop. \ref{prop:flat_maps_faithful} as above allows to replace $U$ by $U'$ without loss of generality. In other words, we assume $x$ to be the Gauss point of a rational disk in $\mbA^1_{\mcH(y)}$ in the following. By density of $\mcO_{Y,y}$ in $\mcH(y)$, there is a linear change of coordinates for $\mcO_{Y,y}[t]$ such that $x$ becomes the Gauss point for the disk of radius $1$ around $0$ in $\mbA^1_{\mcH(y)}$. Possibly shrinking $U$, we assume this coordinate change to be defined over all of $U$. With that choice of $t$, we have $v_x(t) = 0$ and $κ(x) = κ(y)(\wt t)$. Here, $κ(x)$ and $κ(y)$ denote the residue fields of $\mcH(x)$ resp. $\mcH(y)$.

Consider some polynomial $p = c_dt^d + c_{d-1}t^{d-1}+ \ldots + c_0\in k[t]$.
The Gauss point valuation $v_x$ has the concrete description
$$v_x(p) = \min\{v(c_d),\ldots,v(c_0)\}.$$
If $v_x(p) = 0$, then the graded reduction $\wt p\in κ(x)$ is
\begin{equation}\label{eq:graded_reduction_explicit}
\wt p = \sum_{\text{$i$ s.t. $v(c_i) = 0$}} \wt{c_i}\, \wt t.
\end{equation}
Coming back to our situation, after multiplying the $g_i$ by scalars from $k^\times$, which does not change $\Sigma(V, g\times f)$, we may assume that $v_x(g_i) = 0$ for all $i$ and obtain graded reductions $\wt{g_i} \in κ(x)$ as in \eqref{eq:graded_reduction_explicit}.

Again extending $\mcH(y)$ with the help of a finite flat map $U'\to U$ (possibly shrinking $U$), we may without loss of generality assume that the $\wt{g_i}$ factor as a product of linear polynomials. Thus we may write
\begin{equation}\label{eq:product_decomp_g_s}
g_i = λ_i \prod_{j = 1}^{e_i} (t - a_{ij})^{n_{ij}} + r_i\ \ \ \mr{in}\ \mcO_{V,x}
\end{equation}
with $λ_i\in \mcO^\times_{U,y},\ r_i\in \mcO_{V,x},\ a_{ij}\in \mcO_{U,y}$ and $n_{ij}\in \mbZ_{>0}$, subject to the conditions $v_y(λ_i) = 0$, $v_y(a_{ij}) \geq 0$ and $v_x(r_i) > 0$ for all $i$ and $j$. Shrinking $V$ and $U$, we may assume that all $a_{ij}$ and $λ_i$ are defined on $U$, that all $r_i$ are defined on $V$, and that
$$v(r_i) > v(λ_i) + \sum_{j = 1}^{e_i} n_{ij} \cdot v(t-a_{ij})$$
holds for all $i$ everywhere on $V$. Put $h_i := g_i - r_i$ and $h = (h_1,\ldots,h_r)$. Then there is an equality of tropical spaces
\begin{equation}\label{eq:ident_trop_sp}
\Sigma(V, g\times f) = \Sigma(V, h\times f).
\end{equation}
Namely, \eqref{eq:ident_trop_sp} holds as pwl spaces because the graded reductions of $h$ and $g$ agree on $V$. The weights on the two spaces agree because they do not depend on the tropical coordinates (Rmk. \ref{rmk:independence_weights}). The sheaves of linear functions agree because $v(g_i) = v(h_i)$ on $V$. 

We are now reduced to showing that $π:\Sigma(V, h\times f) \to \Sigma(U, f)$ is flat. The only properties of $U$, $V$ and $f$ that we will use for this will be that $π\big(\Sigma(V, h\times f)\big) \subseteq \Sigma(U,f)$. So the argument will also show the statement about refinements $f'$ of the restrictions $f\vert_{U'}$ to open subsets $U'\subseteq U$.

\emph{Convention on notation}: In the following, we simply write $\Sigma(\mbA^1_U, h\times f)$ or $\Sigma(\mbA^1_{\mcH(y)}, h\times f)$ instead of the more precise notation
$\Sigma(\mbA^1_U \setminus V(h), h\times f)$ and $\Sigma(\mbA^1_{\mcH(y)} \setminus V(h(y)), h\times f)$.

Let $z \in \mbA^1_{\mcH(y)}$ be a point in the fiber above some $y\in \Sigma(U, f)$. This point lies in $\Sigma(\mbA^1_U, h\times f)$ if and only if
$$\mr{trdeg}(\wt k (\wt {h_z}, \wt {f_z}) / \wt k) = n+1.$$
The graded reductions $\wt{f_z}$ are contained in $\wt{\mcH(y)}$ (viewed as subring of $\wt{\mcH(z)}$) and contain a transcendence basis of $\wt{\mcH(y)}$ over $\wt k$. The transcendence degree $\mr{trdeg}(\wt \mcH(y) / \wt k)$ is $n$ because $y\in \Sigma(U, f)$. Thus $z\in \Sigma(\mbA^1_U, h\times f)$ if and only if
$$\mr{trdeg}(\wt{\mcH(y)} ( \wt {h_z}) / \wt {\mcH(y)}) = 1.$$
In other words, we have seen that
\begin{equation}\label{eq:skeleton_comparison}
\mbA^1_{\mcH(y)} \cap \Sigma(\mbA^1_U, h\times f) = \Sigma(\mbA^1_{\mcH(y)}, h(y)).
\end{equation}

The description of this set in terms of the graph structure on $\mbP^1_{\mcH(y)}$ is well known; it is the union of all edges in the tree of Berkovich $\mbP^1_{\mcH(y)}$ on which some $h_i(y)$ has nonzero slope. We formulate this more precisely. Given $a\in \mcH(y)$ and $z\in \mbR$, there is the disk $D(a,z) = \{α\in \mcH(y)\mid v(α-a) \geq z\}$ of $v$-radius $z$, centered at $a$. Let $η_{a,z}\in \mbA^1_{\mcH(y)}$ denote its Gauss point. The path
$$ϕ_a:\mbR\lr \mbA^1_{\mcH(y)},\ \ z\longmapsto η_{a,z}$$
is a homeomorphism onto its image; it is the path connecting the two points $a$ and $\infty$ in $\mbP^1(\mcH(y))$. Then $v(h_i)\circ ϕ_a$ is pwl for each $i$. Its slope outside the finitely many break points is given by
\begin{equation}
(v(h_i)\circ ϕ_a)'(z) = \sum_{j = 1,\ v(a - a_{ij}(y)) \geq z}^{e_i} n_{ij}.
\end{equation}
The pwl space $\Sigma(\mbA^1_{\mcH(y)}, h(y))$ is the union over all $i$ and all $1\leq j\leq e_i$ of the paths $ϕ_{a_{ij}}(\mbR)$.

Our aim is to describe the union of all these $\Sigma(\mbA^1_{\mcH(y)}, h(y))$ when varying $y\in \Sigma(U,f)$. Note for this that the union $ϕ_{b_1}(\mbR)\cup\ldots\cup ϕ_{b_u}(\mbR)$ of finitely many paths is, as pwl space, the quotient $(\coprod_{i = 1}^u \mbR)/\sim$ with equivalence relation $(b_i, z_1) \sim (b_j, z_2)$ if $D(b_i,z_1) = D(b_j,z_2)$. (In particular, $z_1 = z_2$.) This construction is fully encoded by all pairwise distances $v(b_i - b_j)$. Note also that the weights on the right hand side of \eqref{eq:skeleton_comparison} agree with the standard weight on $\mbR$ via the $ϕ_a$. This motivates considering the weighted pwl space
\begin{equation}\label{eq:descript_skeleton_A1}
\Sigma' = \left(\coprod_{i,\ 1\leq j\leq e_i} \mbR\times \Sigma(U,f)\right)/\sim
\end{equation}
with the following definition. As a topological space, $\Sigma'$ is the indicated quotient for the following equivalence relation: $(z_1, y_1)_{i_1,j_1} \sim (z_2, y_2)_{i_2, j_2}$ if and only if $y_1 = y_2 = y$, $z_1 = z_2 = z$, and $v(a_{i_1,j_1}(y) - a_{i_2,j_2}(y)) \geq z$. This quotient can also be constructed inductively by gluing in one further copy of $\mbR\times \Sigma(X, f)$ in each step. For each step, the locus of gluing is a closed pwl subspace of the two spaces in question, because it can be described in terms of the pwl function $z$ (the vertical coordinate from $\mbR$) and the pwl functions $y\mapsto v(a_{i_1,j_1}(y) - a_{i_2, j_2}(y))$ on $\Sigma(U, f)$. The pwl structure on $\Sigma'$ is obtained by successive application of the following lemma.
\begin{lem}\label{lem:pwl_equiv_rel}
Let $X_1$ and $X_2$ be two pwl spaces, let $Y_i \subseteq X_i$ be two closed pwl subspaces, and let $φ:Y_1\iso Y_2$ be a pwl isomorphism. Let
$$ι:X_1\sqcup X_2 \lr X := X_1\sqcup_{φ} X_2$$ be the topological space obtained by identifying $Y_1$ and $Y_2$ along $φ$. Define a sheaf of continuous functions $Λ$ on $X$ by
\begin{equation}\label{eq:sheaf_gluing_pwl}
Λ(W) = \{φ \mid φ\circ ι:ι^{-1}(W)\lr \mbR \text{ is pwl}\}.
\end{equation}
Then $(X, Λ)$ is a pwl space and $ι$ a pwl map. If furthermore $X_1$ and $X_2$ are weighted and purely of the same dimension $n$, and if the restriction of $φ$ to the $n$-dimensional locus of $Y_1$ is an isometry, then there is a unique weight on $X$ such that $ι\vert_{X_1}$ and $ι\vert_{X_2}$ are locally isometric outside a closed pwl subspace of dimension $\leq n-1$.
\end{lem}
The proof of this lemma is obvious and thus omitted. Coming back to the definition of $\Sigma'$, we endow each copy $\mbR\times \Sigma(U, f)$ in \eqref{eq:descript_skeleton_A1} with the product weight. The isometry condition of the lemma is satisfied, making $\Sigma'$ into a weighted pwl space.
Let $ι_{i,j}:\mbR\times \Sigma(U,f)\to \Sigma'$ denote the quotient map from the $(i,j)$-th copy of $\mbR\times \Sigma(U, f)$. Each $ι_{i,j}$ is injective and there is a well defined map
\begin{equation}\label{eq:embedding_skeleton_A1}
ϕ:\Sigma'\lr \Sigma(\mbA^1_U,h\times f),\quad ι_{i,j}(z,y)\longmapsto η_{a_{ij}(y),z} \in \mbA^1_{\mcH(y)}.
\end{equation}

Note that $\mr{pr}_2:\Sigma' \to \Sigma(U, f)$ is flat. The well defined fiber weights from Def. \ref{def:fiber_weight} are simply the standard weights on the fibers $\mbR$ of each copy $ι_{i,j}:\mbR\times \Sigma(U,f) \to \Sigma(U,f)$ in \eqref{eq:descript_skeleton_A1}. Thus the proof of Prop. \ref{prop:flatness_in_explicit_rel_dim_1} is complete if we can prove the following final claim.

\emph{Claim: The map $ϕ$ from \eqref{eq:embedding_skeleton_A1} is an isomorphism of weighted pwl space.}

The previous fiberwise analysis already showed that $ϕ$ is a bijection of sets. Establishing that it is a pwl isomorphism is now easier for the set-theoretic inverse $ϕ^{-1}$ because we have an explicit pwl atlas (in the sense of \cite{Duc_local}*{(0.25.3)}) for $\Sigma'$ and only need to see that all its pwl functions pullback to pwl functions on $\Sigma(\mbA^1_U, h\times f)$.

The atlas we mentioned is the following. For each affinoid domain $K\subseteq U$, the intersection $K\cap \Sigma(U, f)$ is a closed pwl subspace and any $g\in \mcO_U(K)$ defines a pwl function $v(g)$ on $K\cap \Sigma(U, f)$. The pairs $(K, v(g))$ give an atlas of $\Sigma(U, f)$. An atlas of $\Sigma'$ is now given by the images $ι_{i,j}\big(\mbR\times (K\cap \Sigma(U, f))\big)$ together with the functions $\mr{pr}_1$ and $v(g)\circ \mr{pr}_2$.

For each of the pairs $(i,j)$, the image of $ι_{i,j}$ is the closed pwl subspace $\Sigma(\mbA^1_U, (t-a_{ij})\times f)$. Moreover, the two compositions $\mr{pr}_1\circ ϕ^{-1} = v(t-a_{ij})$ and $v(g\circ π)\circ \mr{pr_2} \circ ϕ^{-1} = v(g)$ are pwl functions on $\Sigma(\mbA^1_U, h\times f)$. So we see that $ϕ^{-1}$ is a bijective pwl map, in particular continuous. Since $ϕ^{-1}$ is also proper, it is a homeomorphism.

%
%

We next claim that $ϕ$ is an isomorphism of weighted pwl spaces. This claim is local on $U$, so we may assume $\Sigma(U,f)$ to be a polyhedral set. Let $\mcX$ and $\mcY$ be polyhedral complex structures for $\Sigma'$ (resp. $\Sigma(U,f)$) that are subordinate to $π$. Given a maximal dimensional polyhedron $σ\in \mcX_{n+1}$ and a point $x\in σ^\circ$, there are some $i,j$, a closed interval $I\subset \mbR$ and some $τ\in \mcY_n$ such that $ι_{i,j}$ maps $I\times τ$ isomorphically onto a neighborhood of $x$ in $σ^\circ$. 
Working locally near $y = π(x)\in \Sigma(U,f)$, let $q:\mbG_m^s\to \mbG_m^n$ be such that $q\circ f$ is finite flat near $y$, say of degree $e$. Then $(t-a_{ij})\times (q\circ f)$ is finite flat of the same degree $e$ near $x$, which means that the weight of $σ$ coincides with the product weight on $I\times τ$. This finishes the proof of the claim and of the proposition.
\end{proof}

Combining this proposition with the generically finite case (Prop. \ref{prop:fin_maps_flat_on_skeletons}) we obtain our most general result.

\begin{thm}\label{thm:skeleton_flat}
Let $π:X\to Y$ be a map of analytic spaces that are purely of dimensions $\dim Y = m$ and $\dim X = n$. Assume that $π$ is locally on $X$ dominant in the sense that it takes Abhyankar points to Abhyankar points. Let $g:X\to \mbG_m^r$ be a tuple of toric coordinates.

Then for every point $x\in \Sigma(X,g)\setminus \partial X$, there are open neighborhoods $U$ of $y = π(x)$ and $V\subseteq π^{-1}(U)$ of $x$ as well as toric coordinates $f:U\to \mbG_m^s$ such that $π$ restricts to a flat map of tropical spaces
$$\Sigma(V,g\times (f\circ π))\setminus \partial X \lr \Sigma(U,f)\setminus \partial Y.$$
In fact, $U$, $V$ and $f$ may be chosen such that for every open subset $U'\subseteq U$ and every refinement $f'$ of $f\vert_{U'}$, the map $π$ restricts to a flat map
$$\Sigma(V\cap π^{-1}(U'), g\times (f'\circ π)) \setminus\partial X \lr \Sigma(U', f') \setminus \partial Y.$$
\end{thm}

\begin{proof}
First note that $f(X\setminus \partial X)\subseteq (Y\setminus \partial Y)$ by \cite{Ber2}*{Prop. 1.5.5 (ii)}. So under the given assumptions, $y = π(x)$ is an Abhyankar point of $Y\setminus \partial Y$. The claimed flatness is local on both $X$ and $Y$, so we may assume $X$ and $Y$ to be $k$-affinoid, say $X = \mcM(B)$ and $Y = \mcM(A)$. By Lem. \ref{lem:irred_comp_decomp_of_skeleton}, we may also assume that both $A$ and $B$ are integral domains. By \cite{Duc_dimension}*{Théorème 4.6}, there exists a factorization
$$X\overset{π_0}{\lr} \mbA^{n-m}_Y \overset{π_1}{\lr} Y$$
such that $x$ is isolated in its fiber $F := π_0^{-1}(π_0(x))$. After shrinking $X$, we may assume that $F$ is finite. Since $\mbA^{n-m}_Y$ is purely of dimension $n$, the image $π_0(x)$ has to be an Abhyankar point of $\mbA^{n-m}_Y$. Since $x\notin \partial X$ by assumption, the intersection $F \cap \partial X$ does not contain $x$. After replacing $X$ by an affinoid neighborhood of $x$ that is disjoint from $F\cap \partial X$, we may assume that $F \subseteq X\setminus \partial X$. Then $π_0$ is finite over a neighborhood of $π_0(x)$ by Lem. \ref{lem:automatic_finite_flatness_II}. Applying Prop. \ref{prop:fin_maps_flat_on_skeletons}, we find an affinoid neighborhood $U$ of $π_0(x)$ and toric coordinates $f_0$ on $U$ such that
$$\Sigma(π_0^{-1}(U), g\times (f_0\circ π_0)) \lr \Sigma(U, f_0)$$
is flat, and such that the analogous flatness even holds for all refinements of $(U, f_0)$.

In this way, we are from now on reduced to the case of an affinoid domain $X\subset \mbA^{n-m}_Y$. By induction on $n-m$, we only need to treat the case $n = m + 1$. Since $\Sigma(X, g)\setminus \partial X$ as a tropical space depends only on the tropicalization map $t_g$ (see the comment before \eqref{eq:skeleton_away_from_boundary} as well as Rmk. \ref{rmk:independence_weights}) and since $X$ is compact, we may replace $g$ by an approximation by polynomials in $A[T]$. Then Prop. \ref{prop:flatness_in_explicit_rel_dim_1} applies and finishes the proof.
\end{proof}

\section{$δ$-Forms on Non-Archimedean Spaces}

We next define and study $δ$-forms on analytic spaces. Everything here relies on the tropical space property of skeletons, which only holds away from boundaries. For this reason, we mostly restrict to boundaryless spaces. 

\subsection{$δ$-Forms}
We use a blunt notion of restriction for polyhedral currents. Let $H \subseteq K\subseteq \mbR^m$ be two polyhedral subsets and let $T\in P(K)$ be a polyhedral current. Let $\mcK$ be a polyhedral complex structure on $K$ that is also subordinate to $H$ and $T$, say $T = \sum_{σ\in \mcK} α_σ\wedge [σ,μ_σ]$. The \emph{(polyhedral) restriction} of $T$ to $H$ is defined as
$$T\vert_H := \sum_{σ\in \mcK,\ σ\subseteq H} α_σ\wedge [σ,μ_σ] \in P(H).$$
It is independent of the choice of $\mcK$. Given any other presentation of $K$ as polyhedral set, $K \iso K' \subseteq \mbR^r$, the image of $H$ is again a polyhedral set. (Being a closed pwl subspace is a transitive property.) The definition of $T\vert_H$ is then independent of the presentation of $K$. Thus, whenever $Y\subseteq X$ is a subset of a pwl space $(X, Λ_X)$ such that $(Y, Λ_X\vert_Y)$ is also a pwl space, the above considerations apply locally on $Y$ and together define a restriction map
\begin{equation}
P(X) \to P(Y),\ T\mapsto T\vert_Y.
\end{equation}

Beware that the following confusion might arise. Assume $(X,μ)$ is a weighted pwl space. Then any pws form $α\in PS(X)$ defines a polyhedral current $α\wedge [X,μ]$. The restrictions of $α$ as pws form and as polyhedral current cannot be compared in general. For example, $(φ\wedge [\mbR, μ])\vert_{\{0\}} = 0$ for dimension reasons. We will take care that this ambiguity does not arise in the following and sometimes write $T\vert^P_H$ instead of $T\vert_H$, indicating the polyhedral character of the restriction.

A special case in which the two notions agree however is when $Y\subseteq X$ is purely of the same dimension and endowed with the restriction weights $μ\vert_Y$. In formulas: $(α\wedge [X, μ])\vert_Y = α\vert_Y \wedge [Y, μ\vert_Y]$. For example, $Y\subset X$ might be an open pwl subspace. Then the restriction $T\vert^P_Y$ also agrees with the restriction of $T$ as current.

\begin{lem}[Restriction Lemma]\label{lem:restriction_lemma}
\begin{enumerate}[wide, labelindent=0pt, labelwidth=!, label=(\arabic*), topsep=4pt, itemsep=4pt]
\item Let $(\Sigma, μ, L)$ be a tropical space of dimension $n$ and let $f\in L(\Sigma)^r$ be a tuple of linear functions. Let $\Sigma' \subseteq \Sigma$ be the union of all $n$-dimensional polyhedra $σ\subset \Sigma$ such that $f\vert_σ$ is injective. Then $(\Sigma', μ' = μ\vert_{\Sigma'}, L' = \mbR + \mr{Span}\{f_1,\ldots,f_r\})$ is a tropical space as well.

\item Let additionally $γ\in B(\mbR^r)$ be a $δ$-form. Then
\begin{equation}\label{eq:restriction_lemma}
f^\star(γ)\vert_{\Sigma'} = (f\vert_{\Sigma'})^\star(γ).
\end{equation}
\end{enumerate}
\end{lem}
\begin{proof}
(1) The pwl space $\Sigma'$ is purely of dimension $n$ by definition and $L'$ a sheaf of linear functions. Working locally with a fine enough polyhedral complex structure on $\Sigma'$, we need to see that for every polyhedron $τ\subseteq \Sigma'$ of codimension $1$ and every linear combination $ϕ$ of the $f_i$ with $ϕ\vert_τ$ constant, the balancing condition \eqref{eq:balanced} is satisfied.

Let $σ\subseteq \Sigma$ be $n$-dimensional and such that $τ\subset σ$ is a facet. We claim that if $ϕ\vert_σ$ is nonconstant, then $σ\subseteq \Sigma'$. Indeed, since $τ\subseteq \Sigma'$, there exist indices $i_1,\ldots,i_{n-1}$ such that $(f_{i_1},\ldots,f_{i_{n-1}})\vert_τ$ is injective. It follows that if $ϕ\vert_σ$ is nonconstant, then $(f_{i_1},\ldots,f_{i_{n-1}}, ϕ)\vert_σ$ is injective and hence $σ\subseteq \Sigma'$. (This argument used that $ϕ\vert_τ$ is constant.)

By assumption, the balancing condition \eqref{eq:balanced} is satisfied for $τ$ and $ϕ$ when $τ$ is viewed in $\Sigma$. But $(\partial ϕ\vert_σ)/(\partial n_{σ,τ}) \neq 0$ only for $σ\subseteq \Sigma'$ by the previous argument, so the balancing condition also holds in $\Sigma'$. This proves statement (1).

(2) The second statement follows from the projection formula. Namely, let $K\subseteq \Sigma$ be any $n$-dimensional compact polyhedral set, put $K' = K\cap \Sigma'$. Then there is an equality $f_*(K, μ\vert_K) = f_*(K', μ\vert_{K'})$. Moreover by (1), $f_*(K', μ\vert_{K'})$ is a tropical cycle away from $f(\partial K')$, the boundary being taken for $K'\subseteq \Sigma'$. Given $g\in L(K)^s$ with finite fibers, the projection formula reads
\begin{equation}\label{eq:projection_identity}
p_{1,*}\big( (f,g)_*(K,μ)\wedge p_1^*(γ)\big) = f_*(K', μ\vert_{K'}) \wedge γ
\end{equation}
away from $f(\partial K')$. Statement (2) is then obtained as follows. Assume $γ$ is of tridegree $(p,q,c)$. Working locally with fine enough polyhedral complex structures, let $τ\subseteq \Sigma'$ be of codimension $c$. Take $K$ such that $τ^\circ \cap \partial K = \emptyset$ and $f^{-1}(f(τ))\cap K' = τ$. (For example, take $K = \bigcup_{τ\subseteq σ} σ$, where $σ$ runs through the polyhedra of the complex structure.) Then the coefficients of $(f,g)(τ)$ and $f(τ)$ in \eqref{eq:projection_identity} being equal means that the coefficients of $τ$ in $(f,g)^\star(p_1^*(γ))$ and $f^\star(γ)$ agree, cf. the defining property \eqref{eq:realization}.
\end{proof} 

Let $U\subseteq X$ be an open subset of a purely $n$-dimensional analytic space with $\partial U = \emptyset$. Let $f$ and $g$ be tuples of toric coordinates on $U$ and assume that $g$ refines $f$. We claim that the inclusion of skeletons
$$\Sigma(U,f) \subseteq \Sigma(U, g)$$
is of the type considered in Lem. \ref{lem:restriction_lemma} (with respect to $t_f$). Indeed, assume $x\in \Sigma(U, f)$. Then for every affinoid neighborhood $V$ of $x$,
$$\dim t_f(\Sigma(V, g)) \geq \dim t_f(\Sigma(V, f)) = n$$
by \eqref{eq:rel_trop_skeleton}, so $x$ lies in the pwl subspace of $\Sigma(U, g)$ defined by $t_f$. Conversely, if $\dim t_f(\Sigma(V, g)) = n$ for all affinoid neighborhoods $V$ of $x$, then $x\in \Sigma(U, f)$ by Thm. \ref{thm:trop_local_structure} and \eqref{eq:skeleton_characterization_residue_fields}. This proves the claim.

Assume now that $g:U\to \mbG_m^r$ and that $γ\in B(\mbR^r)$. Then we define a polyhedral current $t_g^\star(γ)\vert_{\Sigma(U,f)}$ by restriction from $\Sigma(U,g)$. Lem. \ref{lem:restriction_lemma} implies that if $(h, g = p\circ h)$ is a refinement of $g$, then
$$t_h^\star(p^*(γ))\vert_{\Sigma(U,f)} = t_g^\star(γ)\vert_{\Sigma(U,f)}.$$
In particular, we obtain a well defined restriction $t_g^\star(γ)\vert_{\Sigma(U,f)}$ also in the case that $g$ does not refine $f$: Choose locally any refinement $(h,g = p\circ h)$ of $g$ that also refines $f$ and put locally
\begin{equation}\label{eq:def_realization_on_skeleton}
t_g^\star(γ)\vert_{\Sigma(U,f)} := t_h^\star(p^*γ)\vert_{\Sigma(U,f)}.
\end{equation}
Let $X$ be a purely $n$-dimensional analytic space. By \emph{skeleton} in $X$, we mean a locally closed subset $\Sigma \subseteq X$ which is locally a closed pwl subspace of some $\Sigma(U,f)$. All such are skeletons in the sense of Ducros' \cite{Duc_local}*{Def. (4.6)}. (We actually suspect that every skeleton in Ducros' sense is locally contained in some $\Sigma(U,f)$. This is true at least if the skeleton in Ducros' sense is purely $n$-dimensional.) There is no natural tropical space structure on a skeleton $\Sigma$, e.g. $\Sigma$ might not even be pure-dimensional. However, the above logic still applies and yields a well defined restriction $t_g^\star(γ)\vert_\Sigma$.

\begin{defn}\label{def:delta_form_non_arch}
Let $X$ be an analytic space with $\partial X = \emptyset$, purely of some dimension $n$.
\begin{enumerate}[wide, labelindent=0pt, labelwidth=!, label=(\arabic*), topsep=4pt, itemsep=4pt]
\item A \emph{$δ$-form} on $X$ is the datum of a polyhedral current $ω_\Sigma$ for every skeleton $\Sigma \subseteq X$ such that the following two conditions are satisfied. First, if $\Sigma \subseteq \Sigma'$, then $ω_{\Sigma} = (ω_{\Sigma'})\vert_\Sigma$. Second, every point $x\in X$ has an open neighborhood $U$, toric coordinates $f:U\to \mbG_m^r$ and a $δ$-form $γ\in B(\mbR^r)$ such that for all $\Sigma \subseteq U$,
$$ω_\Sigma = t_f^\star(γ)\vert_{\Sigma}.$$
Given a $δ$-form $ω = (ω_\Sigma)_\Sigma$, we usually write $ω\vert_\Sigma := ω_\Sigma$. We write $t_f^\star(γ)$ for the $δ$-form $\Sigma \mapsto t_f^\star(γ)\vert_\Sigma$.

\item By their local definition, $δ$-forms on $X$ form a sheaf which we denote by $B$ or $B_X$. There is a natural trigrading $B = \bigoplus_{p,q,c} B^{p,q,c}$ that comes from the trigrading of $δ$-forms on $\mbR^r$.

\item Recall that, for any tropical space $(\Sigma, μ, L)$, the operators $\wedge$, $d'$, $d''$, $d'_P$, $d''_P$, $\partial'$ and $\partial''$ can all be defined in terms of charts, cf. Def. \ref{def:delta_form_tropical_space} (3) and (4). It follows that they all commute with formation of the restriction as in \eqref{eq:restriction_lemma}. Thus we may endow $B$ with $\wedge$, $d'$, $d''$, $d'_P$, $d''_P$, $\partial'$ and $\partial''$ by
$$(t_f^\star α) \wedge (t_f^\star β) := t_f^\star (α\wedge β),\ \ \ d'(t_f^\star α) := t_f^\star (d'α),\ \ \ \partial'(t_f^\star α) := t_f^\star(\partial' α),\ \ldots$$
\end{enumerate}
\end{defn}

\begin{rmk}\label{rmk:def_delta_form}
\begin{enumerate}[wide, labelindent=0pt, labelwidth=!, label=(\arabic*), topsep=4pt, itemsep=4pt]
\item Among the above operators, at least $d'_P$ and $d''_P$ have a simple description: They coincide with the polyhedral derivative of polyhedral currents,
$$(d_P'ω)\vert_\Sigma = d'_P(ω\vert_\Sigma),\ \ \ (d_P''ω)\vert_\Sigma = d''_P(ω\vert_\Sigma).$$
This follows from an analogous description for $δ$-forms on $\mbR^r$. Another explicit situation is the $\wedge$-product with a $δ$-form $η$ of tridegree $(p,q,0)$. Then each $η\vert_{\Sigma(U,f)}$ has the form $η_{\Sigma(U,f)}\wedge [\Sigma(U,f),μ]$ for some pws $(p,q)$-form $η_{\Sigma(U,f)}$, cf. Ex. \ref{ex:delta_forms}, and
$$(η\wedge ω)\vert_{\Sigma(U,f)} = η_{\Sigma(U,f)} \wedge \left(ω\vert_{\Sigma(U,f)}\right).$$
Note that the pws forms $η_{\Sigma(U,f)}$ form a compatible family for varying $U$ and $f$. This will be discussed in detail below, cf. §\ref{ss:pws_forms}.

\item One might equivalently define $δ$-forms on $X$ by only considering open subsets $U\subseteq X$ and skeletons of the form $\Sigma(U,f)$. However, the above formalism also allows to consider e.g. $ω\vert_{\Sigma(K,f)}$, where $K\subseteq X$ is an affinoid domain and, in particular, has boundary. This will be convenient for integration theory later.
\end{enumerate}
\end{rmk}

\begin{lem}\label{lem:delta_vanishing}
Let $X$ be an analytic space with $\partial X = \emptyset$, purely of dimension $n$. Let $f:X\to \mbG_m^r$ denote toric coordinates and let $γ\in B(\mbR^r)$. The following are equivalent.
\begin{enumerate}[leftmargin=*]
\item The $δ$-form $ω = t_f^\star(γ)\in B(X)$ vanishes.
\item For every open subset $U\subseteq X$ and every $g:U\to \mbG_m^s$, the restriction $ω\vert_{\Sigma(U,g)}$ vanishes.
\item For every affinoid domain $K\subseteq X$ and every $g:K\to \mbG_m^s$, the restriction $ω\vert_{\Sigma(K,g)}$ vanishes.
\item Same as (2) or (3), but only for $g$ refining $f\vert_U$ resp. $f\vert_K$.
\item For every affinoid domain $K$ and every refinement $(g,f\vert_K = p\circ g)$ of $f\vert_K$, the intersection $T(K,g) \wedge p^*(γ)$ vanishes away from $t_g(\Sigma(K,g)\cap \partial K)$.
\end{enumerate}
\end{lem}
\begin{proof}
The equivalence of (1) -- (4) is immediate, because every skeleton is locally contained in a skeleton of the form $\Sigma(U,g)$ resp. $\Sigma(K,g)$ with $g$ refining $f\vert_U$ resp. $f\vert_K$.

Now assume $ω = 0$ and let $(K,(g,p))$ be as in (5). In particular we assume $ω\vert_{\Sigma(K,g)} = 0$. Since $g$ is assumed to refine $f$ along $p$, the restriction satisfies (by definition)
$$ω\vert_{\Sigma(K\setminus \partial K, g)} = (t_g\vert_{\Sigma(K\setminus \partial K, g)})^\star(p^*(γ)).$$
The defining property \eqref{eq:realization} of $t_g^\star$ on the right hand side implies now that
$$t_{g,*}(ω\vert_{\Sigma(K, g)}) = T(K,g) \wedge p^*(γ)\ \ \ \text{away from }t_g(\Sigma(K,g)\cap \partial K).$$
So (1) implies (5). 

Assume conversely that (5) holds. Our aim is to prove (1), meaning that for every skeleton $\Sigma \subseteq X$, the restriction $ω\vert_\Sigma$ vanishes. Fix any such $\Sigma$. Since $X$ is good, every point $x\in \Sigma$ has an affinoid neighborhood $K$ and toric coordinates $(g,p)$ refining $f$ such that $\Sigma \cap K \subseteq \Sigma(K, g)$. Thus we may directly assume $\Sigma = \Sigma(K\setminus \partial K,g)$ for such $K$ and $g$ and need to show
\begin{equation}\label{eq:to_show_vanishing}
ω\vert_{\Sigma(K\setminus \partial K, g)} = 0.
\end{equation}
Intersections of the form $H\cap \Sigma(K,g) = \Sigma(H,g)$, for varying affinoid domains $H\subseteq K\setminus \partial K$, contain a basis for the topology of $\Sigma(K\setminus \partial K,g)$. For every $H$, assumption (5) implies that
$$T(H, g) \wedge p^*(γ) = 0\quad\text{away from $t_g(\partial H)$}.$$
This form agrees with $t_{g,*}(ω\vert_{\Sigma(H, g)})$ (away from $t_g(\partial H)$) by definitions (compare \eqref{eq:realization}). Then Lem. \ref{lem:faithfulness_lemma} applies and yields $ω\vert_{\Sigma(K\setminus \partial K, g)} = 0$, which is the required statement \eqref{eq:to_show_vanishing}.
\end{proof}

\begin{thm}\label{thm:pull_backs_exist}
Let $π:X\to Y$ be a morphism of pure-dimensional analytic spaces without boundaries. There is a unique pullback map
$$π^*:π^{-1}B_Y \to B_X$$
that has the following property. For every open subset $U\subseteq Y$, every tuple of toric coordinates $f:U\to \mbG_m^r$ and every $δ$-form $γ\in B(\mbR^r)$,
\begin{equation}\label{eq:def_pullback}
π^*(t_f^\star(γ)) = t_{f\circ π} ^\star(γ) \in B_X(π^{-1}(U)).
\end{equation}
Put differently, for every triple $U$, $f$ and $γ$ as above, if $t_f^\star(γ) = 0$, then also $t_{f\circ π}^\star(γ) = 0$.
\end{thm}
\begin{proof}
The pullback may be defined through \eqref{eq:def_pullback} as soon as the independence of the choice of local presentation through $f$ and $γ$ is verified. This is precisely the second statement. Our proof is effectively by factoring $π$ locally around Abhyankar points into a dominant map followed by a closed immersion. The dominant map will be treated with Thm. \ref{thm:skeleton_flat} and Prop. \ref{prop:pullback_delta_form}, the closed immersion with Prop. \ref{prop:reg_sequence_statement}.

\emph{Step 1: Reduction to closed immersions.} Assume that $f:Y\to \mbG_m^r$ and $γ\in B(\mbR^r)$ are such that $t_f^\star(γ) = 0$. (We may and will assume $U = Y$ for the proof.) Let $\Sigma \subseteq X$ be any skeleton. We may check the vanishing of $t_{f\circ π}^\star(γ)\vert_\Sigma$ locally on $\Sigma$, so we may assume that $π$ is of the form $π = \ob{π}\vert_{\ob{X}\setminus \partial \ob{X}}$ where $\ob{π}:\ob{X}\to \ob{Y}$ is a morphism of $k$-affinoids. Lem. \ref{lem:irred_comp_decomp_of_skeleton} allows us to assume both $\ob X$ and $\ob Y$ to be integral.

\begin{lem}\label{lem:dominant_factorization}
Let $π:\ob X\to \ob Y$ be a morphism of integral $k$-affinoid spaces. Let $x\in \ob X\setminus \partial \ob X$ be an Abhyankar point in the interior of $\ob X$. Then there exist
\begin{itemize}[leftmargin=*]
\item an affinoid neighborhood $U \subseteq \ob Y$ of $y = π(x)$,
\item an affinoid neighborhood $V\subseteq π^{-1}(U)$ of $x$,
\item and a Zariski closed subset $Z\subseteq U$, which we endow with its reduced $k$-analytic space structure,
\end{itemize}
such that $π\vert_V$ factors through a map $V\to Z$ which takes Abhyankar points to Abhyankar points.
\end{lem}
\begin{proof}
By assumption, $\ob X$ is integral and $x$ Abhyankar, so the local ring $\mcO_{\ob X, x}$ is a field. The ring homomorphism $\mcO_{\ob Y,y}\to \mcO_{\ob X,x}$ is local, so factors over $\mcO_{\ob Y,y}/\mfm_y$. The local ring $\mcO_{\ob Y, y}$ is noetherian, see \cite{Duc_families}*{\S2.1.3 (4)}. Let $U'\subseteq \ob Y$ be an affinoid neighborhood of $y$ such that there exist functions $f_1,\ldots,f_r \in \mcO_{\ob Y}(U')$ with $\mfm_y = (f_1,\ldots, f_r)\mcO_{Y, y}$. Set $V' = π^{-1}(U')$ and $Z' = V(f_1,\ldots,f_r)$. 

By construction, $y\in Z'$ and $x\in V'$. Let $V'_x\subseteq V'$ be the unique (because $x$ is Abhyankar) irreducible component containing $x$. Then $V'_x$ is affinoid and the map $V'_x\to V'$ is an isomorphism over an open neighborhood of $x$. Moreover, $V'_x$ is integral and hence the natural map
$$\mcO_{V'_x}(V'_x) \lr \mcO_{V'_x, x} = \mcO_{\ob X, x}$$
is injective.

It follows from this injectivity that the functions obtained by pullback, $π^*(f_1),\ldots, π^*(f_r)\in \mcO_{V'_x}(V'_x)$, all vanish. Thus $V'_x \to U'$ factors through $Z'$. The image $π(V'_x)$ is again irreducible; let $Z''\subseteq U'$ be its Zariski closure which is contained in $Z'$. Thus we obtain a map of integral affinoid $k$-analytic spaces $V'_x\to Z''$ together with an Abhyankar point $x\in V'_x \setminus \partial V'_x$ such that $\mcO_{Z'', y}$ is a field where $y = π(x)$. In particular, $V'_x \to Z''$ is naively flat at $x$ in the sense of \cite{Duc_families}*{\S4.1}. Because $x$ is an inner point, by \cite{Duc_families}*{Thm. 8.3.4}, $V'_x\to Z''$ is flat at $x$.

By openness of the flat locus \cite{Duc_families}*{Thm. 10.3.2}, $V'_x\to Z''$ is flat on an open neighborhood of $x$. Moreover, since $x$ is Abhyankar and by \cite{Duc_families}*{\S1.4.14 (4)}, we have $δ := d(y) = \dim(\ob X) - \dim φ^{-1}(y)$. By semi-continuity of the fiber dimension \cite{Duc_families}*{\S1.4.13}, there is an open neighborhood of $x$ on which the fiber dimension of $π$ is $\leq δ$. It is also $\geq δ$ on an open neighborhood by that same semi-continuity because $\overline{\{x\}}{}^{\mr{Zar}} = V'_x$. Thus there exists an affinoid neighborhood $V''\subseteq V'_x$ of $x$ in $V$ such that $π\vert_{V''}$ is flat of constant fiber dimension $δ$. By \cite{Duc_families}*{Thm. 9.2.1} with $Γ = \mbR$ (viewed as additive group), $π(V'')\subseteq Z''$ is an analytic domain. It contains $y$ by construction and even is a neighborhood of $y$. (In other words, it satisfies $y\notin \partial (π(V''))$ which follows from $x\notin \partial V''$ by \cite{Ber1}*{Prop. 3.1.3 (ii)}.) In addition, the dimension relation
$$\dim(π(V'')) + δ = \dim(V'')$$
from \cite{Duc_families}*{\S1.4.14 (3)} implies that $\dim π(V'') = \dim V'' - δ$. Since $Z''$ is irreducible and $π(V'')\subseteq Z''$ an analytic domain, this also implies that $Z''$ is purely of dimension $\dim V'' - δ$.

Let $U\subseteq U'$ be an affinoid neighborhood of $y$ that satisfies the property $Z := U\cap Z'' \subseteq π(V'')$. Let $V\subseteq V'' \cap π^{-1}(U)$ be an affinoid neighborhood of $x$. Then $π\vert_V:V\to U$ factors through $Z$. Moreover, $V$ and $Z$ are pure-dimensional and the fiber dimension of $π\vert_V$ is constant and equal to $\dim V - \dim Z$. By \cite{Duc_families}*{\S1.4.14 (4)} as before, $π\vert_V$ takes Abhyankar points of $V$ to Abhyankar points of $Z$ as desired.
\end{proof}

We come back to the proof of Thm. \ref{thm:pull_backs_exist}. After applying Lemma \ref{lem:dominant_factorization} to the given map $π:\ob{X}\to \ob{Y}$ and the given Abhyankar point $x\in \ob X \setminus \partial \ob X$, as well as after replacing $(\ob X, \ob Y)$ by the resulting $(V, U)$, we may assume that there is a factorization of $π$ as $\ob X \to \ob Z\hookrightarrow \ob Y$ where $\ob{X}\to \ob{Z}$ takes Abhyankar points to Abhyankar points and where $\ob Z\hookrightarrow \ob Y$ is a closed immersion. In this setting, Thm. \ref{thm:skeleton_flat} applies: Every skeleton in $X$ is (locally) contained in a skeleton that is flat over a skeleton in $Z = \ob{Z}\setminus \partial \ob{Z}$. By Prop. \ref{prop:pullback_delta_form}, we are now reduced to showing $t_{f\vert Z}^\star(γ) = 0$.

\emph{Step 2: The closed immersion $Z \to Y$.} A further application of Lem. \ref{lem:irred_comp_decomp_of_skeleton} implies the vanishing of $t_{f\vert \ob{Y}_i\setminus \partial \ob{Y}_i}^\star(γ)$ for every irreducible component (with reduced structure sheaf) $\ob{Y}_i \subseteq \ob{Y}$. Replacing $\ob{Y}$ by any of the $\ob{Y}_i$, subject to $\ob{Z}\subseteq \ob{Y}_i$, allows to assume $\ob{Y}$ integral.

Assuming $\ob{Y} \neq \ob{Z}$ (otherwise we would be done), let $u\in \mfp$ be any nonzero element. Then $V(u)$ is a Cartier divisor containing $\ob{Z}$. Assume we can show $t_{f\vert V(u)}^\star(γ) = 0$. Then we may repeat the above process of replacing $V(u)$ by any of its irreducible components (with reduced structure sheaf) and conclude by induction on the codimension of $Z$. Thus we only need to treat the case $\ob{Z} = V(u)$.

Every skeleton on $Z$ is locally contained in a skeleton that has the form $\Sigma(K\cap \ob{Z}, g\vert_{K\cap \ob{Z}})$ with $K\subseteq \ob{X}$ an affinoid domain and $g:K\to \mbG_m^s$ toric coordinates. Restricting to such pairs $(K\cap \ob{Z}, g\vert_{K\cap \ob{Z}})$ in Lem. \ref{lem:delta_vanishing} (5) suffices to prove the vanishing of $t_{f\vert Z}^\star(γ)$. Thus we have to show that, for all $(K, g)$ as above,
\begin{equation}\label{eq:vanishing_in_trop_charts}
T(K \cap \ob{Z}, g\times f)\wedge p_2^*(γ) = 0
\end{equation}
away from $t_{g\times f}(\partial K \cap \Sigma(\ob{Z} \cap K, g\times f))$. By Prop. \ref{prop:reg_sequence_statement}, which is due to Chambert-Loir--Ducros in this divisorial case, there is $C > 0$ such that
$$T(K_{\infty > v(u) > C}, u\times g \times f) = (C,\infty)\times T(\ob{Z} \cap K, g\times f).$$
Our assumption on $f$ and $γ$ is the vanishing $ t_f^\star(γ) = 0$ in $B(Y\setminus Z)$, which implies
\begin{equation}\label{eq:vanishing_prod_str}
((C,\infty)\times T(\ob{Z}\cap K, g\times f)) \wedge p_3^*(γ) = 0
\end{equation}
away from $t_{u\times g\times f}(\partial (K_{\infty > v(u)> C}))$. Trivially, for every tropical cycle $T$ and every $δ$-form $γ$, one has
$$((C,\infty) \times T) \wedge p_2^*γ = (C,\infty)\times (T\wedge γ).$$
Identity \eqref{eq:vanishing_prod_str} is of this form. The vanishing in \eqref{eq:vanishing_in_trop_charts} follows once we note that if $t\in T(K\cap \ob{Z}, g\times f) \setminus t_{g\times f}(\partial K\cap \ob{Z})$, then there is $C_t>C$ such that
$$(C_t,\infty)\times \{t\} \cap t_{u \times g\times f}(\partial K_{\infty > v(u) >C'}) = \emptyset.$$
(Use the properness of the map
$$t_{u\times g\times f}: \partial K_{v(u) \geq C} \lr [C, \infty]\times \mbR^s \times \mbR^r$$
for this statement.) This concludes the proof of the theorem.
\end{proof}

It is clear that the pullback map $π^*:π^{-1}B_Y \to B_X$ from the theorem commutes with all the operators $\wedge$, $d'$, $d''$, $d'_P$, $d''_P$, $\partial'$ and $\partial''$ since these are all computed in charts. In the case of immersions $i:X\hookrightarrow Y$, we also write $ω\vert_X$ instead of $i^*ω$ and consider this as the restriction of $ω$ to $X$.

\subsection{Examples of $δ$-Forms}\label{ss:pws_forms}
We discuss smooth and piecewise smooth forms, as well as the $δ$-forms from \cite{GK}.

Recall first the definition of smooth forms from \cite{CLD}. These may be defined on any $k$-analytic space $X$, possibly with boundary and not necessarily pure-dimensional. First, Chambert-Loir--Ducros consider a presheaf $Q$ of presented forms. This is purely formal: $Q(U)$ is the set of pairs $(f,η)$ consisting of toric coordinates $f:U\to \mbG_m^r$ and a smooth form $η \in A(\mbR^r)$. They then introduce the following equivalence relation $\sim$ on $Q$. Pairs $(f_1,η_1)$ and $(f_2,η_2)$ are equivalent if
\begin{equation}\label{eq:equiv_rel_smooth}
p_1^*(η_1)\vert_{T'(K,f _1\times f_2)} = p_2^*(η_2)\vert_{T'(K, f_1\times f_2)}
\end{equation}
for every affinoid domain $K\subseteq U$. Note that in particular $(f,η)\sim (g, p^*η)$ for every refinement $(g, f = p\circ g)$. Sheafifying $Q/\!\sim$ defines the sheaf $A_X$ of smooth forms on $X$. The class of $(f,η)$ is denoted by $t_f^*η$.

Smooth forms are generalized by pws forms, which are defined by the same procedure: Consider again a presheaf of pairs $(f,η)$ as above, but this time allow $η\in PS(\mbR^r)$ to be pws. Endow this presheaf with literally the same equivalence relation \eqref{eq:equiv_rel_smooth}. The sheafification of the quotient is the sheaf of pws forms, denoted by $PS$ or $PS_X$. The class of $(f,η)$ is again denoted as $t_f^*η$.

A (possibly) different notion is obtained by G-sheafifying smooth forms. More precisely, one defines a G-pws form $ω$ on an analytic space $X$ to be the datum of a G-covering $X = \bigcup_{i\in I} X_i$ together with a tuple $ω = (ω_i)_{i\in I}$ of smooth forms $ω_i\in A(X_i)$ such that $ω_i\vert_{X_i\cap X_j} = ω_j\vert_{X_i \cap X_j}$ for all $i,j$; up to refinement of the datum. These again form a sheaf which we call $GPS$.

Every pws form defines a G-pws form and it follows from definitions that this provides an inclusion $PS\subseteq GPS$. Assuming that $v(k^\times)\neq \{0\}$ and $X$ strict, \cite{GK}*{Prop. 8.4} shows that $PS = GPS$. (The proposition is for pws functions, but its proof carries over to pws forms. Also recall here that if $v$ is nontrivial, then every $k$-analytic space $X$ with $\partial X = \emptyset$ is strict \cite{Temkin_lecture}*{Ex. 4.2.4.2}.)

Assume now that $\partial X = \emptyset$ and that $X$ is purely of dimension $n$. Recall that $PS^{p,q}(\mbR^r) = B^{p,q,0}(\mbR^r)$, cf. \cite{Mih_trop_inter}. Then there is a well defined map of sheaves $PS_X\to B_X,\ t_f^*η\mapsto t_f^\star(η)$. Namely, $η\vert_{T'(K,f)} = 0$ certainly implies $η\vert_{T(K,f)} = 0$ because $T(K,f)\subseteq T'(K,f)$ and the second kind of vanishing characterizes the vanishing of $δ$-forms, cf. Lem. \ref{lem:delta_vanishing}.
It follows from the mentioned identity $PS^{p,q}(\mbR^r) = B^{p,q,0}(\mbR^r)$ that the map $PS^{p,q}_X \to B_X^{p,q,0}$ is surjective.

\begin{prop}\label{prop:pws_are_delta_of_codim_0}
The pws forms $PS^{p,q}(X)$ are precisely the $δ$-forms $B^{p,q,0}(X)$. In other words, the map $PS^{p,q}_X\to B^{p,q,0}_X$ is also injective.
\end{prop}
\begin{proof}
Let $U\subseteq X$ be an open subset, let $f:U\to \mbG_m^r$ be toric coordinates and let $η\in PS(\mbR^r)$ be such that $t_f^\star(η) = 0$. Our task is to show $t_f^*(η) = 0$.

Assume first that $(p,q) = (n,n)$. The vanishing of $t_f^\star(η)$ implies that for every affinoid domain $K\subseteq U$, the restriction $η\vert_{T(K, f)\setminus t_f(\partial K)}$ vanishes (Lem. \ref{lem:delta_vanishing}). This means that the support of $η\vert_{T'(K, f)}$ is contained in the complement
$$(T'(K, f) \setminus T(K, f)) \cup t_f(\partial K)$$
which is contained in a polyhedral set of dimension $\leq n-1$. Since $η$ is of bidegree $(n,n)$ by assumption, this implies $η\vert_{T'(K,f)} = 0$. Since $K$ was arbitrary, we obtain $t_f^*(η) = 0$ and the proposition is proved for $(n,n)$-forms.

For a pws form $η$ of general bidegree $(p,q)$, the previous argument implies that $t_f^*(η)\wedge α = 0$ for every pws $(n-p, n-q)$-form $α$ on $U$. By \cite{CLD}*{Prop. 3.2.8}, whose proof extends without change to pws forms, this implies $t_f^*(η) = 0$.
\end{proof}

In summary, we obtain maps $A^{p,q}_X \hookrightarrow PS^{p,q}_X \overset{\sim}{\to} B^{p,q,0}_X$ where the second one is an isomorphism. In case $v(k^\times)\neq \{0\}$, there is also a sheaf $A^w$ of so-called weakly smooth forms \cite{GJR}. It has the property that $A_X\subseteq A^w_X \subseteq PS_X$. The embedding $A^w_X\to PS_X$ and the isomorphism $PS_X\overset{\sim}{\to} B^{\bullet, \bullet, 0}_X$ from Prop. \ref{prop:pws_are_delta_of_codim_0} are compatible with derivatives, $\wedge$-product and pullback. More precisely, the boundary differentials $\partial', \partial''$ vanish on smooth and weakly smooth forms while $\wedge, d'_P$ and $d''_P$ are computed polyhedron-by-polyhedron in all definitions.

Finally, we mention without formal proof that the sheaf of $δ$-forms defined by Gubler--Künnemann, cf. \cite{GK}, embeds into the sheaf of $δ$-forms we have defined here. Recall that it is assumed in \cite{GK} that $k$ is algebraically closed, that $v(k^\times) \neq \{0\}$ and that $X$ is the analytification of an algebraic variety over $\Spec k$. Let $B_X^{GK}$ denote the sheaf of $δ$-forms on $X$ defined there. The precise relation with our definition now is that there exists an injective map $B^{GK}_X\to B_X$ which is given on generators by $t_f^*(γ) \mapsto t_f^\star(γ)$.

The existence of the map $B_X^{GK}\to B_X$ relies on two statements. First, condition (5) of Lem. \ref{lem:delta_vanishing} implies that $δ$-forms may also be defined in terms of tropicalizations (instead of in terms of skeletons). Second, the complicated equivalence relation involving all morphisms $f:X'\to X$ from \cite{GK}*{Def. 4.4} that goes into the definition of $B_X^{GK}$ is already implied by an (a priori weaker) equivalence relation in the style of Lem. \ref{lem:delta_vanishing}. This is the content of Thm. \ref{thm:pull_backs_exist}.

\subsection{Integration}

In this section, $X$ denotes a purely $n$-dimensional analytic space over $k$ with $\partial X = \emptyset$.

\begin{lem}\label{lem:support_lemma}
Let $(\Sigma, μ, L)$ be an $n$-dimensional tropical space, let $f \in L(\Sigma)^r$ be a linear map, let $γ\in B^{p,q}(\mbR^r)$ be of bidegree $(p,q)$ and let $ω = t_f^\star(γ)$ be the corresponding $δ$-form on $\Sigma$.
\begin{enumerate}[leftmargin=*]
\item If $\dim f(K) < \max\{p,q\}$ for every compact polyhedral set $K\subseteq \Sigma$, then $ω = 0$.

\item Assume $\max\{p,q\} = n$ and let $i:\Sigma' \subseteq \Sigma$ be the tropical space defined by $f$ as in Lem. \ref{lem:restriction_lemma}. Then $\Supp(ω)\subseteq \Sigma'$. More precisely,
$$ω = i_*(t_{f\vert \Sigma'}^\star(γ)).$$
\end{enumerate}
\end{lem}
\begin{proof}
(1) By the characterizing identity \eqref{eq:realization}, our task is to show that $(f,g)_*(K,μ\vert_K) \wedge p_1^*(γ) = 0$ away from $(f,g)(\partial K)$ for every compact polyhedral set $K\subseteq \Sigma$ and every linear map $g:K\to \mbR^s$ with finite fibers. By the fan displacement rule \cite{Mih_trop_inter}*{Prop. 4.18}, it is equivalent to show the vanishing of
\begin{equation}\label{eq:vanishing_to_show_aoe}
(f,g)_*(K, μ\vert_K) \wedge (εv + p_1^*(γ))\ \ \ \text{(away from $(f,g)(\partial K)$)}
\end{equation}
for all generic directions $v$ and sufficiently small $ε>0$. (The notation here means to translate $p_1^*(γ)$ by $εv$.) We may restrict attention to all translations $p_1^*(εw + γ)$ because $p_1^*(γ)$ is translation invariant for $\ker(p_1)$. Assume now that $γ$ is purely of codimension $c$, i.e. of tridegree $(\wt p, \wt q, c)$ with $p = \wt p + c$, $q = \wt q + c$. Then a generic translation $εw+γ$ intersects the polyhedral set $f(K)$ in dimension $<\max\{p, q\} - c = \max\{\wt p, \wt q\}$. Write $γ = \sum_{τ\in \mcC^c} α_τ\wedge [τ, μ_τ]$ for some polyhedral complex $\mcC$ and with $α_τ$ of bidegree $(\wt p, \wt q)$. We obtain that if $σ \subseteq (f,g)(K) \setminus (f,g)(\partial K)$ is an $n$-dimensional polyhedron such that $N_σ$ and $N_{f^{-1}(τ)}$ are in generic position for some $τ\in \mcC^c$, then
$$p_1^*(α_τ)\vert_{σ\cap f^{-1}(τ)} = (p_1\vert_{σ\cap f^{-1}(τ)})^*(α_σ\vert_{f(σ) \cap τ}) = 0$$
for dimension reasons. This implies the vanishing of \eqref{eq:vanishing_to_show_aoe}.

(2) By definition of $\Sigma'$, one has $\dim f(K) < n$ for all $K\subseteq \Sigma\setminus \Sigma'$. Then the statement
$$t_{f\vert \Sigma'}^\star(γ) = t_f^\star(γ)\vert_{\Sigma'}$$
from the Restriction Lemma \ref{lem:restriction_lemma} (2) specializes to $ω = i_*(t_{f\vert_{\Sigma'}}^\star)(γ)$ as claimed.
\end{proof}

\begin{cor}\label{cor:support_delta_form}
Let $ω\in B^{p,q}(X)$ be a $δ$-form.

\begin{enumerate}[leftmargin=*]
\item The support $\mr{Supp}(ω)$ is disjoint from all $x\in X$ with $d(x) < \max\{p,q\}$. In particular, if $Z\subseteq X$ is Zariski closed with $\dim(Z) < \max\{p,q\}$, then $Z\cap \mr{Supp}(ω) = \emptyset$.

\item Assume $\max\{p,q\} = n$ and that $ω = t_f^\star(γ)$. Then $\Supp ω \subseteq \Sigma(X, f)$ and for every refinement $g$ of $f$,
$$ω\vert_{\Sigma(X,g)} = i_*(ω\vert_{\Sigma(X,f)}).$$
Here $i:\Sigma(X,f)\to \Sigma(X,g)$ denotes the inclusion map.
\end{enumerate}
\end{cor}
\begin{proof}
Assume that $g:U\to \mbG_m^s$ is any set of toric coordinates near $x$. Then, by \cite{CLD}*{(2.3.3)}, there is an affinoid neighborhood $x\in K \subseteq U$ with $\dim T'(K,g) \leq d(x)$. Both statement (1) and (2) now follow from Lem. \ref{lem:support_lemma}.
\end{proof}

\begin{enumerate}[wide, labelindent=0pt, labelwidth=!, label=(\arabic*), topsep=4pt, itemsep=4pt]
\item Our next aim is to give a definition of the integral of a compactly supported $(n,n)$-form. Consider first $ω\in B^{p,q}_c(X)$ with $p=n$ or $q=n$. Let $\Sigma \subseteq X$ be any skeleton such that $\Supp ω \subseteq \Sigma$. Then Cor. \ref{cor:support_delta_form} implies that for every containing skeleton $i:\Sigma \subseteq \Sigma'$ one has
\begin{equation}\label{eq:well_def_integral}
i_*(ω\vert_{\Sigma}) = ω\vert_{\Sigma'}.
\end{equation}
Note that such a $\Sigma$ always exists. Namely pick a covering $\Supp ω \subseteq \bigcup_{i\in I} K_i^\circ$ by a finite union of the interiors of affinoids $K_i$ such that one may write $ω\vert_{K_i^\circ} = t_{f_i}^\star(γ_i)$. Then $\Supp ω \subseteq \Sigma = \bigcup_{i\in I} \Sigma(K_i, f_i).$

\begin{defn}\label{def:integral_delta_form}
Given $ω\in B^{n,n}_c(X)$, set
$$\int_{X} ω := \int_{\Sigma} ω\vert_\Sigma$$
where $\Sigma\subseteq X$ is any skeleton with $\Supp ω \subseteq \Sigma.$ This is well defined by Cor. \ref{cor:support_delta_form}.
\end{defn}

\item Next comes an equivalent definition through partitions of unity. For a pwl space $\Sigma$, a pws partition of unity $1 = \sum_{i\in I} λ_i$ on $\Sigma$ and a compactly supported polyhedral current $T$ of integrable degree on $\Sigma$, meaning that $T = \sum_{σ\in \mcS} α_σ\wedge [σ,μ_σ]$ with $α_σ\in A^{\dim σ,\dim σ}(σ)$, it is clear that
\begin{equation}\label{eq:integral_partition_of_unity}
\int_{\Sigma} T = \sum_{i\in I} \int_{\Sigma} λ_iT.
\end{equation}
One obtains that for $ω\in B_c^{n,n}(X)$ as before and any smooth or pws partition of unity $(λ_j)_{j\in J}$,
$$\int_X ω = \sum_{j\in J} \int_{X} λ_j ω.$$
Indeed, if $\Sigma$ is any skeleton, then $(λ_j)_{j\in J}\vert_\Sigma$ is a pws partition of unity and \eqref{eq:integral_partition_of_unity} applies. The usual application is, of course, with $(λ_j)_{j\in J}$ subordinate to an open covering $X = \bigcup_{i\in I} U_i$ such that $ω\vert_{U_i} = t_{f_i}^\star(γ_i)$ for suitable $(f_i,γ_i)$. Note that even if the $λ_i$ are smooth, the restriction $λ_i\vert_{\Sigma(U_i, f_i)}$ need only be piecewise smooth.

\item There is also an easy inclusion-exclusion type formula. Let again $\Sigma$ and $T$ be as in (2). Consider an expression $\Sigma = \bigcup_{i\in I} \Sigma_i$ as a \emph{finite} union of closed pwl subspaces $\Sigma_i$. Then
\begin{equation}\label{eq:inclusion_exclusion}
\int_{\Sigma} T = \sum_{\emptyset \neq J\subseteq I} (-1)^{|J|+1} \int_{\bigcap_{j\in J} \Sigma_j} T\vert_{\bigcap_{j\in J} \Sigma_j}
\end{equation}
by a simple counting argument. (The restrictions of $T$ are as polyhedral currents.) Coming back to our analytic space $X$, let $\Supp ω \subseteq \bigcup_{i\in I} K_i$ be a $G$-covering and let $\Sigma_i\subseteq K_i,\ i\in I$ be skeletons with $\Supp ω \cap K_i\subseteq \Sigma_i$. Then
$$\int_X ω = \sum_{J\subseteq I} (-1)^{|J| + 1} \int_{\bigcap_{i\in J} \Sigma_i} ω\vert_{\bigcap_{i\in J} \Sigma_i}.$$
Note, however, that there is still no definition of ``$ω\vert_{K_i}$'', since our theory of $δ$-forms only works in the absence of boundary. In other words, the definition of $ω\vert_{\bigcap_{i\in J} \Sigma_i}$ still relies on data from an open neighborhood of $\bigcap_{i\in J} K_i$ in $X$.

\item Finally, if $f:\Sigma \to \Sigma'$ is any pwl map and $T\in P_c(\Sigma)$ can be integrated (i.e. lies in $\bigoplus_{i\geq 0} P^{d,d}_d(\Sigma)$), then $\int_\Sigma T = \int_{\Sigma'} f_*T$. Thus integrals of $δ$-forms may be computed in tropical charts. The following is a special case. Assume that $K$ is an $n$-dimensional affinoid space, that $f:K\to \mbG_m^s$ is a tuple of toric coordinates and that $ω = t_f^\star(γ)\in B_c^{n,n}(K\setminus \partial K)$ has the additional property that $\Supp γ \subset \mbR^r \setminus t_f(\partial K \cap \Sigma(K, f))$. Then
$$\int_K ω = \int_{\mbR^r} T(K,f)\wedge γ.$$
\end{enumerate}

Chambert-Loir--Ducros \cite{CLD} endow the vector spaces $A_c(U)$, $U\subseteq X$ open, with an analog of the Schwartz topology. (The definition is recalled at the beginning of the proof of Prop. \ref{prop:delta_form_as_current} below.) They then define the sheaf of currents
$$D(U) := \Hom_{\mr{cont}}(A_c(U), \mbR).$$
It comes with differential operators $d',d''$ by duality,
\begin{equation}\label{eq:derivative_currents_CLD}
(d'T)(η) = (-1)^{\deg T + 1}(d'η),\ \ \ (d''T)(η) = (-1)^{\deg T + 1}(d''η).
\end{equation}
($T$ and $η$ are assumed to be homogeneous here.) By the next proposition, $δ$-forms may be understood as a specific kind of currents.

\begin{prop}\label{prop:delta_form_as_current}
Associate to $ω\in B^{p,q}(X)$ the linear functional
$$[ω]:A_c^{n-p,n-q}(X) \lr \mbR,\ η\longmapsto \int_X ω\wedge η.$$
Then $[ω]$ is a current, i.e. continuous for the Schwartz space topology on its source. The resulting map
$$B(X) \lr D(X),\ ω\longmapsto [ω]$$
is injective and commutes with $d'$, $d''$. In particular, Stokes' Theorem holds: For all $α\in B_c^{n-1,n}(X)$ and $β\in B_c^{n,n-1}(X)$,
$$\int_X d'α = \int_X d''β = 0.$$
\end{prop}
\begin{proof}
Let $(η_i)_{i\in I} \to η$ be a convergent net of compactly supported smooth $(n-p,n-q)$-forms on $X$ in the sense of \cite{CLD}*{(4.1.1)}. By definition this means that there are finitely many affinoid domains $K_j \subseteq X$ such that
\begin{enumerate}[leftmargin=*]
\item The supports of $η$ and all $η_i$ are covered by $\bigcup_{j\in J} K_j$.
\item There are toric coordinates $f_j$ for $K_j$ such that all restrictions $η_i\vert_{K_j}$ as well as $η\vert_{K_j}$ are presentable in $f_j$, say
$$η_i\vert_{K_j} = t_{f_j}^*(α_{ij}),\ \ \ η\vert_{K_j} = t_{f_j}^*(α_j).$$
\item For each $j$, the $α_{ij}$ and all their higher derivatives converge to the respective derivative of $α_j$ on each polyhedron in $T'(K_j,f_j)$.
\end{enumerate}
Note that the third condition is intrinsic to the family $(η_i)_{i\in I}$ as soon as the first two conditions are satisfied. Namely given a smooth form $t_f^*(α)$ on any (compact) analytic space $K$, the restriction $α\vert_{T'(K,f)}$ is uniquely determined.

We need to see that $[ω](η_i) \to [ω](η)$. Pick any finite covering
$$\Supp η \cup \bigcup_{i\in I} \Supp η_i \subseteq \bigcup_{l\in L} (H_l\setminus \partial H_l)$$
by the interiors of affinoid domains $H_l$ that tropicalize $ω$, say $ω\vert_{H_l\setminus \partial H_l} = t_{g_l}^\star(γ_l)$, where $g_l$ is defined on $H_l$. Then $ω\wedge η_i$ and $ω\wedge η$ are all supported on
$$\Sigma := \bigcup_{j\in J,\ l\in L} \Sigma(K_j\cap H_l, f_j \times g_l).$$
Moreover, the $η_i\vert_\Sigma$ and $η\vert_\Sigma$ are, polyhedron-by-polyhedron, simply the pullbacks of the $α_{ij}$ along the tropicalization maps. Hence $(η_i\vert_\Sigma)_{i\in I} \to η\vert_\Sigma$ in the sense of (1)--(3) above. Since simply $(ω\wedge η_i)\vert_\Sigma  = ω\vert_\Sigma \wedge η_i\vert_\Sigma$, where $ω\vert_\Sigma$ is a polyhedral current and $η_i\vert_\Sigma$ a pws form, the convergence $[ω](η_i) \to [ω](η)$ follows.

Having established that $[ω]$ is a current, we turn to the properties of the map $B(X)\to D(X)$. Injectivity is clear: If $ω\neq ω'$, then by definition there exists some skeleton $\Sigma$ with $ω\vert_\Sigma \neq ω'\vert_\Sigma$. One then finds some test form $η$ with $[ω](η) \neq [ω'](η)$; we omit the details.

Finally, we show $d[ω] = [dω]$ for $d\in \{d',d''\}$. By definition of the derivative of currents, $(d[ω])(η) = [dω](η) - \int_X d(ω\wedge η)$. So we need to prove that $\int_X d'α = \int_X d''β = 0$ for $α\in B^{n-1,n}_c(X)$ and $β\in B^{n,n-1}_c(X)$. We give the argument for $α$, the case of $β$ being identical.

Let $U\subseteq X$ be an open subset such that one may write $α\vert_U = t_f^\star(γ)$. In particular $\Supp α\vert_U \subseteq \Sigma := \Sigma(U, f)$ by Cor. \ref{cor:support_delta_form}. By a partition of unity argument, the statement to prove is that $\int_U d'(λα) = 0$ for every smooth function $λ$ on $U$ with compact support. Even though $\Supp λα \subseteq \Sigma$, the restriction $λα\vert_\Sigma$ need not lie in $B(\Sigma)$ because $λ$ need not be presentable in the toric coordinates $t_f$. However, it is still a balanced polyhedral current of bidegree $(n-1,n)$ because the balancing condition \eqref{eq:balanced_polyhedral_current} is trivially satisfied for polyhedral currents of bidegree $(p,n)$ or $(n,q)$. Its derivative $d'(λα\vert_\Sigma)$ as current on $\Sigma$ is a polyhedral current by Prop. \ref{prop:balanced_equiv_polyhedral_deriv} and we have (by definition of the derivative as current)
$$\int_\Sigma d'(λα\vert_\Sigma) = - (λα\vert_\Sigma)(d'1) = 0.$$
Our claim is now that $(d'(λα))\vert_\Sigma = d'(λα\vert_\Sigma)$. This is a local statement which does not require the condition for $λ$ to have compact support. Let $V\subseteq U$ be open and $g$ a refinement of $f\vert_V$ such that $λ\vert_{\Sigma(V,g)}$ is smooth. Then the statement
$$(d'(λα))\vert_{\Sigma(V,g)} = d'(λα\vert_{\Sigma(V,g)})$$
is known by Prop. \ref{prop:stokes_delta_pwl}. The left hand side is known to agree with the current $i_*((d'(λα))\vert_{\Sigma(V, f)})$ where the pushforward is along the inclusion $i:\Sigma(V, f)\to \Sigma(V, g)$, cf. Cor. \ref{cor:support_delta_form}. Similarly, on the right hand side, $λα\vert_{\Sigma(V,g)} = i_*(λα\vert_{\Sigma(V,f)})$. In this situation, the explicit description of the derivative of balanced polyhedral currents \eqref{eq:formula_deriv_balanced} shows that also
$$i_*(d'(λα\vert_{\Sigma(V,f)})) = d'(λα\vert_{\Sigma(V,g)}).$$
Using that $U$ is covered by such $V$, the proof is complete.
\end{proof}

\section{Intersection Theory}
\subsection{Green Forms}
\label{ss:Green_currents}

Throughout this section, $X$ continues to denote a purely $n$-dimensional $k$-analytic space with $\partial X = \emptyset$.

\begin{lem}\label{lem:Green_form_L1_property}
Let $Z\subseteq X$ be Zariski closed and purely of codimension $r$. Set $U = X\setminus Z$ and let $g\in B^{p,q}(U)$ be a $δ$-form with $\min\{p,q\} < r$. Then $A_c^{n-p,n-q}(X) = A_c^{n-p,n-q}(U)$ and the map
$$[g](η) := \int_{U} g\wedge η$$
defines a current of bidegree $(p,q)$ on $X$.
\end{lem}
\begin{proof}
By Cor. \ref{cor:support_delta_form}, every $η\in A^{n-p,n-q}(X)$ has $\Supp η \subseteq U$, hence the claimed equality $A_c^{n-p,n-q}(X) = A_c^{n-p, n-q}(U)$. Let now $(η_i)_{i\in I} \to η$ be a convergent net of $η_i\in A_c^{n-p,n-q}(X)$. By definition, cf. Proof of Prop. \ref{prop:delta_form_as_current}, there exist finitely many affinoid domains $K_j\subseteq X$ with toric coordinates $f_j$ that tropicalize all $η_i\vert_{K_j}$ and $η\vert_{K_j}$. Just as in the proof of Cor. \ref{cor:support_delta_form}, every point $z\in Z$ has an affinoid neighborhood $H_z$ with $\dim t_{f_j}(H_z\cap K_j) \leq r$. It follows from Lem. \ref{lem:support_lemma} that also the closure in $X$ of $S = \Supp η \cup \bigcup_{i\in I} \Supp η_i$ is contained in $U$. From here on, the proof of the lemma is exactly as in the proof of Prop. \ref{prop:delta_form_as_current}.
\end{proof}

\begin{defn}\label{def:Green}
\begin{enumerate}[wide, labelindent=0pt, labelwidth=!, label=(\arabic*), topsep=4pt, itemsep=4pt]
\item Let $Z\subseteq X$ be a Zariski closed subspace that is purely of codimension $r$. Its functional of integration is
$$δ_Z(η) := \int_Z η,\quad η\in B_c^{n-r,n-r}(X).$$

\item A \emph{Green current for $Z$} is an $(r-1,r-1)$-current $g$ on $X$ such that the following residual term is a $δ$-form,
$$ω(Z,g):=δ_Z + d'd''g \in B^{r,r}(X).$$

\item A \emph{Green $δ$-form for $Z$} is a Green current for $Z$ of the form $g = [γ]$ for some $δ$-form $γ\in B(X\setminus Z)$ of bidegree $(r-1,r-1)$. For divisors, Green forms are also called \emph{Green functions}.
\end{enumerate}
\end{defn}

The motivation for considering Green $δ$-forms instead of merely Green currents is the wish to define a $\star$-product. This requires the possibility to restrict the current to a Zariski closed subspace, see \eqref{eq:star_prod_Cartier}.

Recall that the existence of Green currents for cycles in complex geometry follows (essentially) from the fact that the de Rham complex is quasi-isomorphic to the complex of currents. Such an argument is unavailable for Berkovich spaces currently and the question for existence of Green currents is open in general. We next collect some known results which will all be related to Cartier divisors (Poincaré--Lelong formula).

A continuous metric on a line bundle $L$ over an analytic space $X$ is the datum of a map of sheaves
\begin{equation}\label{eq:def_metric}
v:L\lr C(-, \mbR\cup\{\infty\})
\end{equation}
such that $v(fs) = v(f) + v(s)$ for every open subset $U\subseteq X$, every $f\in \mcO_X(U)$ and every $s\in L(U)$. This is the same notion as in \cite{CLD}*{Def. 6.2.2} except that we write our valuations additively. Given a continuous function $ϕ$, Chambert-Loir--Ducros show that $ϕ$ defines an element of $D^{0,0}(X)$. Thus it makes sense to take the derivative $d'd''(ϕ)$ as current \eqref{eq:derivative_currents_CLD}. The curvature $c_1(L, v)\in D^{1,1}(X)$ is then defined by the property $c_1(L, v)\vert_U = d'd''(v(s))$ whenever $U\subseteq X$ is an open subset and $s\in L(U)$ a generator \cite{CLD}*{(6.4.1)}.

Gubler--Künnemann \cite{GK} made an analogous definition but view $ϕ$, $d'd''(ϕ)$ and $c_1(L, v)$ as $δ$-currents, i.e. as elements of $\Hom_{\mr{cont}}(B_c^{\mr{GK},n-1,n-1}(X), \mbR)$, where $B^{\mr{GK}}(X)$ denotes their space of $δ$-forms \cite{GK}*{\S6}.

\begin{thm}[Chambert-Loir--Ducros \cite{CLD}*{Thm. 4.6.5}, Gubler--Künnemann \cite{GK}*{Thm. 7.2}]\label{thm:Poincare_Lelong_GK}
Let $(L,v)$ be a line bundle with continuous metric on $X$ and $s$ be an invertible meromorphic section of $L$ with divisor $D = \mr{div}(s)$. Then
\begin{equation}\label{eq:PL_divisors}
δ_{D} + d'd''[v(s)] = c_1(L,v).
\end{equation}
In particular, if $c_1(L,v)$ is a $δ$-form, then $v(s)$ is a Green function for $D$.
\end{thm}

In order to construct Green functions for $D$ in the sense of Def. \ref{def:Green} from this, one needs to choose $v$ such that $c_1(L, v)$ is a $δ$-form. To this end, we call a metric $v$ on $L$ smooth, pws, G-pws, resp. G-locally linear (G-pwl) if the function $v(s)$ has that property whenever $s\in L(U)$ is a local generator of $L$. The precise definition of G-pwl will not play a role here, we refer to \cite{GJR}*{\S5} and \cite{CLD}*{(6.2.7)}. (The property G-pwl is simply called piecewise linear there. Our differing terminology was chosen to preserve the parallel with pws and G-pws.)

By a partition of unity argument, every line bundle $L$ on a paracompact $k$-analytic space has a smooth metric $v$, cf. \cite{CLD}*{Prop. 6.2.6}. Then its curvature satisfies $c_1(L,v) \in A^{1,1}(X)$, i.e. is a smooth $(1,1)$-form. In particular, Thm. \ref{thm:Poincare_Lelong_GK} establishes the existence of smooth Green functions for Cartier divisors on paracompact spaces.

If $k$ is nontrivially valued and $X$ paracompact, then every line bundle also has a G-pwl metric \cite{CLD}*{6.2.13}. Moreover, by \cite{GK}*{Prop. 8.4} any such metric is pws. The curvature current $c_1(L,v)$ then lies in $B^{0,0,1}(X)$, i.e. is a $δ$-form of tridegree $(0, 0, 1)$ in our sense.

In order to obtain Green currents for higher codimension cycles, Gubler--Künnemann also define a $\star$-product against metrized divisors: Assume that $(L, v)$ is a metrized line bundle such that $c_1(L,v) \in B^{GK, 1,1}(X)$ is a $δ$-form in their sense. Then, if $(L,s,D)$ is as in Thm. \ref{thm:Poincare_Lelong_GK} and if $(Z,g)$ is a cycle with Green current that intersects $D$ properly,
\begin{equation}\label{eq:star_prod_Cartier}
v(s)\vert_Z + c_1(L,v)\wedge g
\end{equation}
is a Green current for $D\cap Z$ in their sense \cite{GK}*{Prop. 11.4}. These definitions are made whenever $k$ is algebraically closed, $v(k^\times)\neq \{0\}$ and $X$ is algebraic. Applying Thm. \ref{thm:Poincare_Lelong_GK} and this $\star$-product construction provides Green currents for complete intersection cycles in such settings. They moreover show that the $\star$-product is associative and commutative \cite{GK}*{§11}.

We now come to our main result. It is an analog of Thm. \ref{thm:Poincare_Lelong_GK} for higher codimension cycles and motivated by the following example from complex geometry.

\begin{ex}[\cite{GS}*{Ex. 1.4}]\label{ex:complex_geom}
Consider $\mbC^r$ with coordinates $z_1,\ldots,z_g$ and the following function on $\mbC^r\setminus \{0\}$:
$$ϕ(z_1,\ldots,z_r) := \log ( |z_1|^2+\ldots+|z_r|^2).$$
Then $(-1)^rϕ(\partial \bar \partial ϕ)^{r-1}$ is a Green current for $\{0\}\subset \mbC^r$.
\end{ex}

The analog of $|z_1|^2 + \ldots + |z_r|^2$ in the non-archimedean setting is given by taking minimums. Let $f_1,\ldots,f_r \in \mcO_X(X)^r$ be a tuple of functions and put $U = X \setminus V(f_1\cdots f_r)$ as well as $Z = V(f_1,\ldots,f_r)$. Consider the pws function $ϕ_U := \min \{v(f_1),\ldots,v(f_r)\} \in PS(U)$. We claim that it extends to a pws function $ϕ$ on all of $X\setminus Z$. To this end, given some $x \in X\setminus Z$, we consider the set
$$I = \{1\leq i \leq r \mid f_i(x) \neq 0\}.$$
Then $x$ has an open neighborhood $W$ such that
$$ϕ_U\vert_{U\cap W} = \min \{v(f_i)\vert_{U\cap W} \mid i\in I\}.$$
The function $ϕ_{I, W} := \min \{v(f_i)\vert_W \mid i\in I\}$ is thus a pws function on $W$ that extends $ϕ_U$; it has the tropical presentation
\begin{equation}\label{eq:extending_phi}
ϕ_{I,W} = t_{f_I}^*(\min \{x_i \mid i\in I\})
\end{equation}
where $x_i$, $i\in I$, denote the coordinate functions on $\mbR^I$. Performing this construction at all $x\in X\setminus Z$ defines the desired extension $ϕ\in PS(X\setminus Z)$. Note that $ϕ$ is, in particular, a $δ$-form.

\begin{thm}\label{thm:Green_form_complete_inter}
Let $X$ be a purely $n$-dimensional $k$-analytic space such that $\partial X = \emptyset$ and let $Z\subseteq X$ be a complete intersection of codimension $r$, say equal to $V(f_1,\ldots,f_r)$. Set $ϕ:= \min\{v(f_1),\ldots,v(f_r)\}$ and define
\begin{equation}\label{eq:def_gamma}
γ := (-1)^{r-1}ϕ(d'd'' ϕ)^{r-1} \in B^{r-1,r-1}(X\setminus Z).
\end{equation}
Then for every $δ$-form $η\in B_c^{n-r,n-r}(X)$,
\begin{equation}\label{eq:poincare_lelong_higher_codim}
\int_{X\setminus Z} γ\wedge d'd''η = - \int_{Z} η\vert_Z.
\end{equation}
In particular, $γ$ satisfies the following identity of currents,
$$δ_Z + d'd'' [γ] = 0,$$
and is hence a Green $δ$-form for $Z$.
\end{thm}

Specializing to the case $r = 1$, this theorem implies a Poincaré--Lelong formula for metrized line bundles in the sense of Thm. \ref{thm:Poincare_Lelong_GK}. We formulate it here for pws metrics which is the situation where all occurring objects have been defined. The result is more general than the pws metric case of Thm. \ref{thm:Poincare_Lelong_GK} in that it extends the validity of \eqref{eq:PL_divisors} to all $δ$-forms.

\begin{cor}\label{cor:Poincare_Lelong_general}
Let $X$ be a purely $n$-dimensional $k$-analytic space such that $\partial X = \emptyset$. Let $(L, v)$ be a line bundle with pws metric on $X$ and let $s$ be an invertible meromorphic section of $L$. Let $D = \mr{div}(s)$ be its divisor and let $ϕ = v(s\vert_{X\setminus \Supp(D)})$ be its valuation. Then $c_1(L, v)$ is a $δ$-form and for every $η \in B_c^{n-1,n-1}(X)$,
\begin{equation}\label{eq:PL_extended}
\int_{X\setminus Z} ϕ \wedge d'd''η + δ_D(η) = \int_X c_1(L, v) \wedge η.
\end{equation}
\end{cor}
\begin{proof}
Identity \eqref{eq:PL_extended} can be checked locally on $X$. We can thus assume that $L$ is trivial and that there are regular sections $f,g\in \mcO_X(X)$ such that $D = \mr{div}(f) - \mr{div}(g)$. In this case, $ϕ = (v(f) - v(g) + ψ)\vert_{X\setminus \Supp(D)}$ for a pws function $ψ$ on $X$, and $c_1(L, v) = d' d'' ψ$. The corollary now follows directly from definitions and from \eqref{eq:poincare_lelong_higher_codim}.
\end{proof}

We come to the proof of Thm. \ref{thm:Green_form_complete_inter}. We first explain how to obtain \eqref{eq:poincare_lelong_higher_codim} for smooth test forms $η$. The full proof is more technical and will be given afterwards; its central idea however will be the same. The next lemma explains why the form $γ = (-1)^{r-1}ϕ(d'd'' ϕ)^{r-1}$ occurs.

\begin{lem}\label{lem:special_form}
Consider the pwl function $φ := \min\{x_1,\ldots,x_r\} \in B^{0,0,0}(\mbR^r)$. Then
$$(-1)^{r-1}φ (d'd''φ)^{r-1} = x\cdot \Delta,$$
where $\Delta \subseteq \mbR^r$ is the diagonal $\{x_1 = \ldots = x_r\}$ (with standard weight) and $x$ any of the coordinates functions.
\end{lem}
\begin{proof}
For each $\emptyset \neq I\subseteq \{1,\ldots,r\}$, consider the polyhedron
\begin{equation}\label{eq:min_special}
σ_I = \{ x\in \mbR^r \mid φ(x) = x_i\text{ for all }i\in I\}.
\end{equation}
These polyhedra satisfy $σ_I\cap σ_J = σ_{I\cup J}$ and hence make up a polyhedral complex $\mcC$. Furthermore, $φ\vert_{σ_I}$ is linear for every $I$, so $\mcC$ is subordinate to $φ$. Finally, $\dim σ_I = r + 1 - |I|$. 

Then $\mcC^{\geq 1}$ is subordinate to the corner locus $d'd''φ$, see e.g. \cite{GK}*{Def. 1.10} for the formulas. Applying the fan displacement rule (see e.g. \cite{Mih_trop_inter}*{Prop. 4.21}) successively shows that $\mcC^{\geq k}$ is subordinate to $(d'd''φ)^k$. Specialized to $k = r-1$, we see that $\Supp (d'd''φ)^{r-1} \subseteq \Delta$. It is now only left to show that its coefficient is $(-1)^{r-1}$, which is an easy computation.
\end{proof}

Thus we see that $\Supp γ \subseteq \{ x \mid v(f_i(x)) = v(f_j(x)) \text{ for all }i,j\}$ is contained in $U = X\setminus V(f_1\cdots f_r)$ and has the presentation
\begin{equation}\label{eq:presentation_gamma}
γ = t_f^\star(x\cdot \Delta).
\end{equation}

Denote by $ρ:\mbR\to \mbR$ any smooth function with $ρ(x) = 0$ for $x < 2C$ and $ρ(x) = 1$ for $x > 3C$ for some constant $C>0$. Let $T$ be some pwl space of dimension $n-r$ and $η\in PS^{n-r,n-r}_c(T)$. Then
\begin{equation}\label{eq:PL_key_identity}
\int_{T} η = - \int_{Δ\times T} p_1^*(x) \wedge d_P'd_P''(p_1^*ρ \wedge p_2^*η).
\end{equation}
Namely $d_P'η = d_P''η = 0$ for dimension reasons, so the integrand on the right hand side equals $p_1^*(x) \wedge d_P'd_P''(p_1^*ρ) \wedge p_2^*η$. Applying Fubini's Theorem and the identity $\int_\mbR x ρ''(x) dx = -1$ shows equality of both sides.

Our aim is to reduce \eqref{eq:poincare_lelong_higher_codim} (for a smooth compactly supported $(n-r,n-r)$-form $η$) to \eqref{eq:PL_key_identity} which will rely on Prop. \ref{prop:reg_sequence_statement}. Pick any finite covering by the interiors of affinoid domains, $\Supp η \subseteq \bigcup_{i\in I} K_i^\circ$, such that each restriction $η\vert_{K_i} = t_{g_i}^\star(ω_i)$ is presentable in toric coordinates. The inclusion-exclusion formula \eqref{eq:inclusion_exclusion} states that
\begin{equation}\label{eq:integral_decomp_Z}
\int_Z η\vert_Z = \sum_{J\subseteq I} (-1)^{|J|+1} \int_{\Sigma(Z \cap K_J, g_{i(J)})} t_{g_{i(J)}}^*(ω_{i(J)})
\end{equation}
where $K_J = \bigcap_{j\in J} K_j$ and $i(J)\in J$ a fixed arbitrary choice. Let $C > 0$ be a constant such that for all $J$,
\begin{equation}\label{eq:factorization}
T(U \cap K_J \cap  \{ ϕ > C\}, g_{i(J)}) = (C,\infty)^r \times T(Z\cap K_J, g_{i(J)})
\end{equation}
as in Prop. \ref{prop:reg_sequence_statement}. Choose a smooth function $ρ:\mbR\to \mbR$ with $ρ(x) = 0$ for $x<2C$ and $ρ(x) = 1$ for $x>3C$ and consider the pws function $ψ = ρ \circ ϕ$ on $U$. It extends as a pws function to all of $X$ and we obtain a decomposition $η = ψη + (1-ψ)η$. Then $(1-ψ)η$ lies in $B_c^{n-r,n-r}(X\setminus Z)$. A twofold application of Stokes' Theorem (Prop. \ref{prop:delta_form_as_current}) together with the vanishing $d'd''γ = t_f^\star(d'd''(x\Delta)) = 0$ and the relation $\Supp(γ)\subseteq U$ provides
\begin{equation}\label{eq:Green_application_Stokes}
\int_X γ \wedge d'd''\big((1-ψ)η\big) = \int_X d'd''γ \wedge (1-ψ)η = 0.
\end{equation}
Furthermore, $δ_Z\big((1-ψ)η\big) = 0$ because $Z \cap \Supp ((1-ψ)η) = \emptyset$. This proves \eqref{eq:poincare_lelong_higher_codim} for the test form $(1-ψ)η$ and it remains to consider the case of $ψη$. By definitions and Cor. \ref{cor:support_delta_form} (2),
\begin{equation}\label{eq:simple_proof}
\int_U γ\wedge d'd''(ψη)
= \sum_{J\subseteq I} (-1)^{|J|+1} \int_{\Sigma(U\cap K_J, f\times g_{i(J)})} γ\wedge d'd''(ψη)\vert_{\Sigma(U\cap K_J, f\times g_{i(J)})}.
\end{equation}

The choice of $C$ has the property that, for every $j$ and $ε >0$, there is the upper bound $\dim t_{g_j}(K_{j, \geq C+ε}) \leq \dim Z = n-r$. Moreover, $η$ is an $(n-r,n-r)$-form. It follows that
$$d'η\vert_{K_j \cap \{ϕ > C \}} = d''η\vert_{K_j\cap \{ϕ >C\}} = 0.$$
Since $\Supp ψ \subseteq \{ϕ >C\}$, we obtain that $d'd''(ψη) = (d'd''ψ)\wedge η$. The right hand side of \eqref{eq:simple_proof} is hence
\begin{equation}\label{eq:simple_proof_almost}
\sum_{J\subseteq I} (-1)^{|J|+1} \int_{\Sigma(U\cap K_J, f\times g_{i(J)})} γ \wedge d'd''(ψ) \wedge t_{g_{i(J)}}^*ω_{i(J)}.
\end{equation}
The space $U\cap K_J$ has boundary, so the restriction $γ\vert_{\Sigma(U\cap K_J, f\times g_{i(J)})}$ cannot be computed directly. However, after possibly enlarging $C$, we may assume that $t_{f\times g_{i(J)}}(U\cap \partial K_J) \subseteq \mbR^r \times P$ for a polyhedral set $P$ of dimension $\leq n-r-1$, see \eqref{eq:polyhedra_containment_boundary}. The restriction $ω_{i(J)}\vert_P$ vanishes for degree reasons. One may then deduce that the $J$-summand of \eqref{eq:simple_proof_almost} equals
\begin{equation}\label{eq:subtle_for_delta}
(-1)^{|J|+1}\int_{\Delta \times T(Z\cap K_J, f\times g_{i(J)})} p_1^*(xρ''(x)) \wedge p_2^*(ω_{i(J)}).
\end{equation}
By \eqref{eq:PL_key_identity}, this is
$$(-)^{|J|}\int_{T(Z \cap K_J, g_{i(J)})} ω_{i(J)},$$
which equals the negative of the $J$-summand of \eqref{eq:integral_decomp_Z}. The proof is complete.

The remainder of this section is devoted to the full proof of Thm. \ref{thm:Green_form_complete_inter}. The complication in the general case is that identity \eqref{eq:integral_decomp_Z} cannot be formulated so easily. Namely it was silently used for \eqref{eq:integral_decomp_Z} that the restriction of $η$ to the skeleton $\Sigma'(Z\cap K_J, g_{i(J)})$ is contained in the purely $(n-r)$-dimensional locus and moreover given by the pullback of $ω_{i(J)}$ as pws form along the pwl map $t_{g_{i(J)}}$. In particular, only the affinoid $Z\cap K_J$ and its toric coordinates $g_{i(J)}$ intervene. For a general $δ$-form however, we would have to compute $η\vert_{\Sigma'(Z\cap K_J, g_{i(J)})}$ by also using all the other charts $(K_j^\circ, g_j)$, and this gets more complicated.

In a similar vein, the vanishing argument that brought us from \eqref{eq:simple_proof_almost} to \eqref{eq:subtle_for_delta} does not apply if $ω_{i(J)}$ is not pws.

So instead of the inclusion-exclusion formula, we approach the problem with a partition of unity argument that keeps us in the interiors $K_j^\circ$.

\begin{proof}[Proof of Thm. \ref{thm:Green_form_complete_inter}.]
We begin with a simple observation. Let $K$ be an affinoid space, $g:K\to \mbG_m^d$ toric coordinates and $z\in \Sigma(K \setminus \partial K, g)$. Then $z$ has an affinoid neighborhood $K'$ in $K$ such that $t_g^{-1}(t_g(z)) \cap \Sigma(K',g) = \{z\}$. In particular, since $z\notin \partial K'$, also $t_g(z)\notin t_g(\Sigma(K', g) \cap \partial K')$. Indeed, the set $B = t_g^{-1}(t_g(z)) \cap \Sigma(K, g) \setminus \{z\}$ is finite and hence closed in $K$. Then any affinoid neighborhood $K'$ of $z$ in $K \setminus B$ has the required property.

Let $η\in B_c^{n-r,n-r}(X)$ and consider some $z\in \Supp (η\vert_Z)$. Since $η$ has bidegree $(n-r, n-r)$, this point is Abhyankar in $Z$ by Cor. \ref{cor:support_delta_form} (1). Let $K$ be an affinoid neighborhood of $z$ in $X$ such that there are toric coordinates $g:K\to \mbG_m^d$ and a $δ$-form $ω\in B^{n-r, n-r}(\mbR^d)$ with $η\vert_{K\setminus \partial K} = t_g^\star(ω)$. In particular, $z$ lies on $\Sigma(K\cap Z,g)$ by Cor. \ref{cor:support_delta_form} (2). By the above observation, we may choose $K$ such that $t_g^{-1}(t_g(z)) \cap \Sigma(K\cap Z, g) = \{z\}$.

As $\Supp η$ is compact, we only need to apply this construction to finitely many $z$ to obtain a finite family $(K_i, d_i, g_i, ω_i)_{i\in I}$ of affinoid domains $K_i$ in $X$, toric coordinates $g_i:K_i \to \mbG_m^{d_i}$ and $δ$-forms $ω_i\in B^{n-r, n-r}(\mbR^{d_i})$ with the properties
\begin{enumerate}[leftmargin=*, label=(\arabic*), topsep=4pt, itemsep=4pt]
\item $η\vert_{K_i\setminus \partial K_i} = t_{g_i}^\star(ω_i)$
\item $\Supp η\vert_Z \subseteq \bigcup_{i\in I} V_i$ where $V_i = K^\circ_i\setminus t_{g_i}^{-1}(t_{g_i}(\Sigma(K_i\cap Z, g_i) \cap \partial K_i)).$
\end{enumerate}
Let $(λ_i)_{i\in I}$ be a smooth partition of unity for a neighborhood (in $X$) of $\Supp η\vert_Z$ that is subordinate to all $V_i$. By this we mean that $\Supp λ_i \subseteq V_i$ for all $i\in I$ and that $ε = \sum_{i\in I} λ_i$ satisfies $ε \equiv 1$ on a neighborhood of $\Supp (η\vert_Z)$. Then the $λ_iη$ have the crucial property that $\Supp (λ_iη\vert_Z) \subseteq \Sigma(K_i\cap Z, g_i)$ lies above $T(K_i\cap Z, g_i) \setminus t_{g_i}(\Sigma(K_i \cap Z) \cap \partial K_i)$. This is the locus in $T(K_i\cap Z, g_i)$ which is a tropical cycle (Cor. \ref{cor:skeleton_is_tropical}).

\begin{lem}[Key local computation]\label{lem:key_Green_current}
Write $(K,d,g,ω,λ)$ in place of any of the above quintuples $(K_i, d_i, g_i, ω_i, λ_i)$. There exists a constant $C > 0$ such that, for every $δ$-form $α\in B(\mbR^r)$, there is the identity of polyhedral currents on $(C,\infty)^r\times \mbR^d$
\begin{equation}\label{eq:key_push_forward_identity}
t_{f\times g, *}\big( (t_f^\star(α)\wedge λη)\vert_{\Sigma(K\cap U\cap \{ϕ > C\}, f\times g)}\big) = \big(α\vert_{(C,\infty)^r}\big) \times t_{g,*}\big(λη\vert_{\Sigma(K\cap Z, g)}\big).
\end{equation}
\end{lem}

Assume for a moment that $λ$ can be presented in the tropical coordinates $(g_1,\ldots, g_d)$, i.e. assume that there exists a smooth function $λ'$ on $\mbR^d$ such that $λ = t_g^*(λ')$. Then \eqref{eq:key_push_forward_identity} follows directly from Prop. \ref{prop:reg_sequence_statement} and the projection formula. The technical point of Lem. \ref{lem:key_Green_current} is that $λ$ admits such a presentation only locally.

Our proof below is of the following kind. For each point $s\in \Sigma(K\cap Z, f\times g)$, we pick a tropical presentation of $λ$ near $s$. By Prop. \ref{prop:reg_sequence_statement}, there is a $C>0$ such that this presentation is ``homogeneous in direction of each function $v(f_i)$'' as long as $v(f_i)>C$. Taking suitable pushforwards and sums then gives statement \eqref{eq:key_push_forward_identity}.

\begin{proof}
Both sides of \eqref{eq:key_push_forward_identity} are polyhedral currents on $(C,\infty)^r\times \mbR^d$, so the claimed identity may be obtained locally in the following sense. Write $T = T(K\cap Z, g)$ which agrees with $t_{g,*}(\Sigma(K\cap Z, g))$ by \eqref{eq:rel_trop_skeleton}. Fix a point $t$ in the support of $t_{g,*}(λη\vert_{\Sigma(K\cap Z,g)})$, which is some polyhedral current on $T$. Our aim is to find $C_t> 0$ and a neighborhood $W_t$ of $t$ such that \eqref{eq:key_push_forward_identity} holds on $(C_t, \infty)^r \times W_t$. Because $t_g(\Supp(λη\vert_{\Sigma(K\cap Z,g)}))$ is compact, we can then cover it by finitely many such $W_t$, take $C$ as the maximum of the corresponding $C_t$, and the proof is complete.

The fiber $F_t = t_g^{-1}(t) \cap \Sigma(K\cap Z, g)$ is a finite set that is disjoint from $\partial K$ by choice of $(K,g,λ)$. Each $s\in F_t$ has an affinoid neighborhood $H_s$ in $K$ such that $λ\vert_{H_s}$ has a presentation $t_{h_s}^*(λ_s)$ for some toric coordinates $h_s:H_s\to \mbG_m^{e_s}$. We choose a family of such triples $(H_s, h_s, λ_s)_{s\in F_t}$ subject to the following two conditions.
\begin{enumerate}[leftmargin=*, label=(\arabic*), topsep=4pt, itemsep=4pt]
\item $H_s \cap H_{s'} = \emptyset$ for all pairs $s,s'\in F_t$, $s\neq s'$ and
\item $t_{h_s\times g}^{-1}(t_{h_s\times g}(s)) \cap \Sigma(H_s\cap Z, h_s \times g) = \{s\}$.
\end{enumerate}
Put $\Sigma_s = \Sigma(H_s\cap Z, h_s\times g)$ and $T_s = T(H_s\cap Z, h_s\times g)$.
Condition (2) implies that $t_{h_s\times g}(s)\notin t_{h_s\times g}(\Sigma_s \cap \partial H_s)$. Then the defining identity of $δ$-forms \eqref{eq:realization} states that on a neighborhood $W_s\subseteq \mbR^{e_s + d}$ of the image $t_{h_s\times g}(s)$,
$$t_{h_s\times g,*}\big(λη\vert_{\Sigma_s}\big) = p_1^*λ_s \wedge p_2^*ω \wedge T_s.$$
Recall that $U = X\setminus V(f_1\cdots f_r)$. Prop. \ref{prop:reg_sequence_statement} states that there exists some $C_s > 0$ such that
\begin{equation}\label{eq:factorization_III}
T(H_s \cap U \cap \{ϕ >C_s\}, f\times h_s\times g) = (C_s, \infty)^r \times T_s.
\end{equation}
\emph{Claim 1: We may choose $W_s$ and $C_s$ to also satisfy
$$t_{f\times h_s\times g}\big(\Sigma(H_s\cap U \cap \{ϕ> C_s\}, f\times h_s\times g) \cap \partial H_s) \cap \big((C_s,\infty)^r\times W_s\big) = \emptyset.$$}
\begin{proof}[Proof of Claim 1.]
Let $q:\mbG_m^{e_s}\times \mbG_m^d\to \mbG_m^{n-r}$ be a projection that is generic for $T_s$ and satisfies
$$t_{q\circ (h_s, g)}(s) \notin t_{q\circ(h_s,g)}(\Sigma_s \cap \partial H_s).$$
It was established in \eqref{eq:flatness_reg_immersion} that $(f,q\circ (h_s,g)):H_s \to \mbA^r\times \mbG_m^{n-r}$ is finite flat over an open neighborhood $Q$ of $0\times q(h_s(s),g(s))$. For any such $Q$, the set $(f,q\circ (h_s,g))(\partial H_s)\cap Q$ is closed in $Q$ and does not meet $0\times q(h_s(s),g(s))$ because finite maps are closed and inner. Choose a compact neighborhood $P$ of $t_{q\circ (h_s,g)}(s)\in \mbR^{n-r}$ such that $0\times \Sigma(P) \subseteq Q$ and such that moreover $(0\times \Sigma(P)) \cap (f, q\circ (h_s,g))(\partial H_s) = \emptyset$. Here, $\Sigma(P) = \mr{trop}^{-1}(P)\cap \Sigma(\mbG_m^{n-r})$ denotes the subset of the standard skeleton of $\mbG_m^{n-r}$ that is obtained from the standard tropicalization map $\mr{trop}:\mbG_m^{n-r}\to \mbR^{n-r}$, and $0\times \Sigma(P)\subseteq \mbA^r\times \mbG_m^{n-r}$ denotes the product.

Due to the compactness of $P$, there exists a constant $C_s > 0$ such that also
$$\Sigma((C_s,\infty)^r \times P) \cap (f,q\circ (h_s,g))(\partial H_s) = \emptyset.$$
Replacing $W_s$ by $W_s \cap q^{-1}(P^\circ)$ proves the claim.
\end{proof}

Let $C_s$ and $W_s$ be chosen as in Claim 1 from now on. Identity \eqref{eq:realization} then applies to $t_{f\times h_s\times g}$ over $(C_s,\infty)^r\times W_s$ and we conclude that for any $α\in B(\mbR^r)$, the following identity of $δ$-forms on $(C_s,\infty)^r\times W_s$ holds:
\begin{multline}\label{eq:factorization_IV}
t_{f\times h_s\times g,*}\big(t_f^\star(α)\wedge λη\vert_{\Sigma(H_s\cap U\cap \{ϕ>C_s\}, f\times h_s \times g)}\big) \\ = \big(α\vert_{(C_s,\infty)^r}\big) \times \big(t_{h_s\times g,*}\left(λη\vert_{\Sigma_s}\right)\vert_{W_s}\big).
\end{multline}
Our original claim \eqref{eq:key_push_forward_identity} is an identity on an open subset of $\mbR^r\times \mbR^d$ while \eqref{eq:factorization_IV} is an identity on an open subset of $\mbR^r\times \mbR^{e_s}\times \mbR^d$. (The middle factor resulted from our need to present $λ\vert_{H_s}$.) Our next aim is to pushforward along the projection
\begin{equation}\label{eq:projection_s}
p_{13}:\mbR^r\times \mbR^{e_s} \times \mbR^d \lr \mbR^r\times \mbR^d.
\end{equation}
The next statement concerns the second projection $p_2:\mbR^{e_s}\times \mbR^d\to \mbR^d$.

\emph{Claim 2: There exists an open neighborhood $W_{t,s} \subseteq \mbR^d$ of our initially fixed point $t\in T$ such that the restriction of \eqref{eq:factorization_IV} to $(C_s, \infty)^r \times (W_s \cap p_{2}^{-1}(W_{t,s}))$ has proper support over $(C_s, \infty)^r\times W_{t,s}$.}
\begin{proof}[Proof of the claim.]
This claim only concerns the rightmost factor of \eqref{eq:factorization_IV}. The form $λη\vert_{\Sigma_s}$ actually has its support on $H_s \cap \Sigma(K\cap Z, g)$ by Cor. \ref{cor:support_delta_form} (2). Under $t_g$, this intersection maps with finite fibers to $\mbR^d$ because already $t_g\vert_{\Sigma(K\cap Z, g)}$ has finite fibers. This means that $t_{h_s\times g}(H_s\cap \Sigma(K\cap Z, g))$ is proper over $\mbR^d$ with respect to $p_2:\mbR^{e_s}\times \mbR^d \to \mbR^d$. A fortiori, the same holds for $\Supp t_{h_s\times g,*}(λη\vert_{\Sigma(H_s \cap Z, h_s \times g)})$.

Now we use our assumption (1) on $H_s$ for the first time. It implies that
$$t_{h_s\times g}(H_s \cap \Sigma(K\cap Z, g)) \cap p_2^{-1}(t) = t_{h_s\times g}(s).$$
By the established properness, this means we find a neighborhood $W_{t,s}$ of $t$ in $t_g(\Sigma(K\cap Z, g))$ such that
$$p_2^{-1}(W_{t,s}) \cap t_{h_s\times g}(H_s \cap \Sigma(K\cap Z, g)) \subseteq W_s.$$
Any such $W_{t,s}$ has the property of the claim, finishing the proof.
\end{proof}

Let $W_{t,s} \subseteq \mbR^d$ be chosen as in Claim 2 from now on. Let $A_s = B_s$ denote the identity of polyhedral currents on $(C_s,\infty)^r\times W_{t,s}$ obtained by pushforward of \eqref{eq:factorization_IV} along the projection from \eqref{eq:projection_s}
$$p_{13}:(C_s,\infty)^r\times (W_s\cap p_2^{-1}(W_{t,s}))\lr (C_s,\infty)^r \times W_{t,s}.$$
Recall that we defined $F_t = t_g^{-1}(t) \cap \Sigma(K\cap Z, g)$. Put $W'_t = \bigcap_{s\in F_t} W_{t,s}$ and $C_t = \max\{C_s,\ s\in F_t\}$. Taking sums and restrictions, we obtain an identity of polyhedral currents on $(C_t,\infty)^r\times W'_t$,
\begin{equation}\label{eq:AB}
\sum_{s\in F_t} A_s = \sum_{s\in F_t} B_s.
\end{equation}
By construction, the union $H_t = \bigsqcup_{s\in F_t} H_s$ is an affinoid neighborhood of $F_t$ and the (restriction of the) tropicalization map $t_g:H_t \cap \Sigma(K\cap Z, g) \to \mbR^d$ is a proper map with fiber $F_t$ over $t$. There hence exists an open neighborhood $W_t\subseteq W'_t$ of $t$ with the property $t_g^{-1}(W_t) \cap \Sigma(K\cap Z) \subseteq H_t$. The restriction of \eqref{eq:AB} to $(C_t, \infty)^r\times W_t$ is precisely \eqref{eq:key_push_forward_identity} restricted to $(C_t,\infty)^r\times W_t$.

Since $\Supp λη$ is compact, finitely many $W_t$ cover $\Supp t_{g,*}(λη\vert_{\Sigma(K\cap Z, g)})$. Taking the maximum of the corresponding $C_t$ provides the constant $C$ in Lemma \ref{lem:key_Green_current} and finishes its proof.
\end{proof}

We resume the proof of Thm. \ref{thm:Green_form_complete_inter}. For each $i\in I$, let $C_i$ be a constant for $(K_i, d_i, g_i, ω_i, λ_i)$ as in Lem. \ref{lem:key_Green_current}, put $C' = \max\{C_i,\ i\in I\}$. Recall that we have defined $ε = \sum_{i\in I} λ_i$. Choose $C> C'$ such that $ε\vert_{\Supp(η) \cap \{ϕ > C \}} = 1$ and, using Cor. \ref{cor:support_delta_form} (1), such that
\begin{equation}\label{eq:C_eta}
d'η\vert_{U\cap \{ϕ>C\}} = d''η\vert_{U\cap \{ϕ> C\}} = 0.
\end{equation}
We can now start to write out $\int_Z η\vert_Z$. Recall for this that $(λ_i)_{i\in I}$ has the property that $\Supp(λ_iη\vert_Z)\subseteq \Sigma(K^\circ_i\cap Z, g_i) \subseteq K^\circ_i\cap Z$.
\begin{equation}\label{eq:awesome_highlight_I}
\int_Z η\vert_Z = \int_Z εη\vert_Z = \sum_{i\in I} \int_{K_i^\circ\cap Z} λ_iη\vert_{K_i^\circ\cap Z}.
\end{equation}
We may apply the definition of the integral (Def. \ref{def:integral_delta_form}) to rewrite \eqref{eq:awesome_highlight_I} as
\begin{equation}\label{eq:awesome_highlight_II}
\begin{aligned}
\int_Z η\vert_Z &\ = \sum_{i\in I} \int_{\Sigma(K_i\cap Z, g_i)} λ_iη\vert_{\Sigma(K_i\cap Z, g_i)}\\
&\ = \sum_{i\in I} \int_{T(K_i\cap Z, g_i)} t_{g_i,*}\big(λ_iη \vert_{\Sigma(K_i\cap Z,g_i)} \big).
\end{aligned}
\end{equation}
Each term here is the integral over some compactly supported polyhedral current on $\mbR^{d_i}$. Pick a smooth function $ρ:\mbR\to \mbR$ with $ρ(x) = 0$ for $x<2C$ and $ρ(x) = 1$ for $x>3C$. Also define $ψ = ρ\circ \min\{x_1,\ldots,x_r\}$, which is a pws function (hence $δ$-form) on $\mbR^r$. By Fubini applied to each integral, the sum in \eqref{eq:awesome_highlight_II} equals
\begin{equation}\label{eq:awesome_highlight_III}
- \sum_{i\in I} \int_{(C, \infty)^r \times T(K_i\cap Z, g_i)}
\big(x ρ''(x)\cdot \Delta\big) \times \big(t_{g_i,*}\left(λ_iη \vert_{\Sigma(K_i\cap Z, g_i)} \right)\big)
\end{equation}
We next observe that there is an identity
\begin{equation}\label{eq:rewrite_rho}
ρ''(x)\Delta = d'd''(ψ)\wedge \Delta.
\end{equation}
Indeed, if $\mcC$ is a polyhedral complex that is subordinate to $\min\{x_1,\ldots,x_r\}$, then $\mcC$ is also subordinate to $ψ$. Due to its translation homogeneity, there exists a translation invariant polyhedral complex that is subordinate to $\min\{x_1,\ldots,x_r\}$, see \eqref{eq:min_special} for a possible choice. Then the tridegree $(0,0,1)$-part of $d'd''(ψ)$ is also supported on a translation invariant complex, hence has trivial $\wedge$-product with $\Delta$ by the fan displacement rule. The tridegree $(1,1,0)$-part of $d'd''(ψ)$ is a pws form however whose restriction to $\Delta$ is $ρ''(x)$. This shows \eqref{eq:rewrite_rho}.

We now apply Lem. \ref{lem:key_Green_current} with $α = xρ''(x)\Delta = x\Delta \wedge d'd''(ψ)$ to each summand of \eqref{eq:awesome_highlight_III}. Writing $ψ$ also for $t_f^\star(ψ)$, this implies that \eqref{eq:awesome_highlight_III} equals
\begin{equation}\label{eq:awesome_highlight_IV}
\begin{aligned}
- \sum_{i\in I}& \int_{\Sigma(K_i \cap U \cap \{ϕ > C\}, f\times g_i)} t_f^\star(x \Delta)\wedge d'd''(ψ)\wedge λ_iη\vert_{\Sigma(K_i\cap U\cap \{ϕ> C\}, f\times  g_i)}\\
&= -\sum_{i\in I} \int_{K^\circ_i\cap U \cap \{ϕ> C\}} λ_i\wedge t_f^\star(x\Delta)\wedge d'd''(ψ)\wedge η.
\end{aligned}
\end{equation}
Restricting the domain of integration to $K_i^\circ \cap \{ϕ>C\}$ is superfluous because $\Supp(λ_iη)\subseteq K_i^\circ$ and $\Supp(ψ) \subseteq \{ϕ>C\}$. Moreover, recall that $t_f^\star(x\Delta) = γ$ by \eqref{eq:presentation_gamma}. Hence \eqref{eq:awesome_highlight_IV} equals
\begin{equation}\label{eq:awesome_highlight_V}
\begin{aligned}
&-\sum_{i\in I}\int_{U} λ_i\wedge γ \wedge d'd''(ψ)\wedge η\\
&= - \int_{U} εγ \wedge d'd''(ψ)\wedge η.
\end{aligned}
\end{equation}
Recall next that $d'η = d''η = 0$ on $U\cap \{ϕ>C\}$ by choice of $C$, see \eqref{eq:C_eta}. So $d'd''(ψ)\wedge η = d'd''(ψη)$. Finally, recall that $ε$ has the property $ε\vert_{\Supp(ψη)} = 1$. The equality of \eqref{eq:awesome_highlight_I} with \eqref{eq:awesome_highlight_V} thus becomes
\begin{equation}\label{eq:awesome_highlight_VI}
\begin{aligned}
\int_Z η\vert_Z&\ = - \int_{U} εγ\wedge d'd''(ψη)\\
&\ = -\int_U γ\wedge d'd''(ψη).
\end{aligned}
\end{equation}
In this way, we have shown that the Poincaré--Lelong identity \eqref{eq:poincare_lelong_higher_codim} holds for the $δ$-form $ψη$. This leaves us to consider the complementary $δ$-form $η^c = (1-ψ)η$. But $\Supp η^c \subseteq X\setminus Z$ and we conclude with Stokes' Theorem as in \eqref{eq:Green_application_Stokes}.
\end{proof}

\subsection{Relation with Formal Models}
Assume from now on that $k$ is discretely and nontrivially valued with ring of integers $\mcO_k$. Let $π\in \mcO_k$ be a uniformizer and assume that $v$ is normalized with $v(π) = 1$.

A formal scheme $\mcX\to \Spf \mcO_k$ over $\mcO_k$ is called special (cf. \cite{Ber_vanishing_II}*{\S1}) if for every affine open formal subscheme $\Spf A \subseteq \mcX$, the $O_k$-algebra $A$ has the following property. There exists a finitely generated ideal of definition $\mfa \subseteq A$ such that $A/\mfa$ is of finite type over $\mcO_k$. (Note that every such ideal $\mfa$ will contain a power of $π$ because we have already assumed $\mcX$ to lie above $\Spf \mcO_k$.) Equivalently, $A$ is isomorphic to a quotient of some $\mcO_k\{T_1,\ldots,T_k\}[\![S_1,\ldots,S_\ell]\!]$ where the curly brackets denote restricted power series, cf. \cite{Ber_vanishing_II}*{Lem. 1.2}.

Let $\mcX\to \Spf \mcO_k$ be special. Berkovich \cite{Ber_vanishing_II}*{\S1} assigns to such $\mcX$ a generic fiber $X = \mcX^{\mr{an}}$ (denoted $\mcX_η$ in \emph{loc. cit.}.) Assuming $\mcX$ is separated, this is a paracompact, Hausdorff $k$-analytic space. The construction is moreover functorial in $\mcX$ and there is a specialization map on topological spaces that we denote by $\mr{sp}:|X| \to |\mcX|$. It is anticontinuous in the following sense. If $\mcZ\subseteq \mcX$ is closed, then $\mr{sp}^{-1}(\mcZ)$ is open. If $\mcU\subseteq \mcX$ is open, then $\mr{sp}^{-1}(\mcU)$ is a closed analytic domain and there is a natural (in $\mcU$) pullback of functions
$$\mcO_{\mcX}(\mcU)\lr \mcO_X(\mr{sp}^{-1}(\mcU)).$$
Assume $X = \mcX^{\mr{an}}$ in the following and let $\mcZ = V(\mcI)\subseteq \mcX$ be the closed formal subscheme defined by an ideal sheaf $\mcI$. Denote by $Z = \mcZ^{\mr{an}}$ its generic fiber. From the datum $(\mcX, \mcI)$, we define the distance function
\begin{equation}\label{eq:def_distance_function}
ϕ := ϕ_{\mcZ}:X \lr [0,\infty],\ x\longmapsto \min\big\{v(f(x)),\ f\in \mcI_{\mr{sp}(x)}\big\}.
\end{equation}
The restriction $ϕ\vert_{X\setminus Z}$ is pws (G-pwl in the sense of \cite{CLD}*{(6.2.7)} in fact) and $ϕ^{-1}(\infty) = Z$. But note that there is some loss of information, meaning that $\mcZ$ cannot be fully reconstructed from $ϕ$, because the definition of $ϕ$ only involves rank $1$ valuation rings. For example, $V(x^2,π^2)$ and $V(x^2,xπ,π^2)$ define different closed formal subschemes of $\Spf \mcO_k[\![x]\!]$ but have the same distance function.

We say that a special formal scheme $\mcX$ is purely of formal dimension $n$ if for every affine open formal subscheme $\Spf A\subseteq \mcX$, the spectrum $\Spec A$ is of Krull dimension $n$.

Motivated by Thm. \ref{thm:Green_form_complete_inter} on complete intersections, we make the following conjecture.

\begin{conj}\label{conj:Green_current}
Let $\mcX$ be a separated and flat special formal scheme over $\Spf\mcO_k$ that is purely of formal dimension $n+1$. Let $\mcZ \subseteq \mcX$ be a local complete intersection of codimension $r$ and set $ϕ = ϕ_{\mcZ}$. Then $γ := (-1)^{r-1}ϕ (d'd''ϕ)^{r-1}$ is a Green $δ$-form for $Z\setminus \partial Z$ on $X\setminus \partial X$.
\end{conj}
The previous section's result settles the case of a complete intersection:
\begin{cor}[of Thm. \ref{thm:Green_form_complete_inter}]\label{cor:Green_current_formal_models}
Let $\mcX$ be as in Conj. \ref{conj:Green_current} and assume that $\mcZ = V(f_1,\ldots,f_r)$, for certain $f_i\in \mcO_{\mcX}(\mcX)$, is a complete intersection. Then $γ$ is a Green $δ$-form for $Z\setminus \partial Z$ on $X\setminus \partial X$.
\end{cor}
Indeed, $ϕ = \min\{v(f_1),\ldots,v(f_r)\}$ and the result is precisely Thm. \ref{thm:Green_form_complete_inter}. Our addition in this section is the following intersection number identity.
\begin{thm}\label{thm:int_num_identity}
Let $\mcX$ be as in Conj. \ref{conj:Green_current} and assume that $\mcZ \subseteq \mcX$ is a local complete intersection, purely of codimension $r$. Let $\mcY\subseteq \mcX$ be a Cohen--Macaulay closed formal subscheme, purely of formal dimension $r$ and flat over $\Spf \mcO_k$. Set $Y = \mcY^{\mr{an}}$. Assume that $\mcY\cap \mcZ$ is artinian. Then $\Supp(γ\vert_{Y\setminus \partial Y})$ is compact and
\begin{equation}\label{eq:int_number_identity}
\int_{Y\setminus \partial Y} γ = \len_{\mcO_k} \mcO_{\mcY \cap \mcZ}.
\end{equation}
\end{thm}
The flatness of $\mcY/\mcO_k$ is necessary. Otherwise it could for example happen that $Y$ is empty even though $\mcY\cap \mcZ \neq \emptyset$.

\begin{proof}
The intersection $\mcY\cap \mcZ \subseteq \mcY$ is artinian and hence purely of codimension $r$ in $\mcY$. Moreover, it is locally defined by the vanishing of $r$ functions because $\mcZ \subseteq \mcX$ is a local complete intersection. Using that $\mcY$ is Cohen--Macaulay, this implies that $\mcZ\cap \mcY \subseteq \mcY$ is a local complete intersection \cite{Stacks}*{Tag 02JN}. Moreover, the formation of $ϕ$ and $γ$ commutes with restriction to $Y\setminus \partial Y$, so we may assume that $\mcY = \mcX$ and that $r = n+1$. In particular, $\mcX$ is Cohen--Macaulay from now on, $\mcZ$ is artinian with empty generic fiber, and $ϕ$ is a function on all of $X$.

Our first aim is to reduce further to the case $\mcX = \Spf A$ for a special complete local $\mcO_k$-algebra $A$. Write $X = D \sqcup D^c$ with $D = \mr{sp}^{-1}(\mcZ)$. Then $D$ is an open subset and $D^c$ an analytic domain of $X$ by the anticontinuity of $\mr{sp}$. The open subset $D$ agrees with the generic fiber $(\widehat{\mcX}_{/\mcZ})^{\mr{an}}$ of the formal completion of $\mcX$ along $\mcZ$ \cite{Ber_vanishing_II}*{Beginning of \S3}. The formal completion $\widehat{\mcX}_{/\mcZ}$ is a disjoint union of formal spectra of special complete local $\mcO_k$-algebras. Let $A$ be such an algebra. By the construction of $(-)^{\mr{an}}$ in \cite{Ber_vanishing_II}*{\S1}, $(\Spf A)^{\mr{an}}$ is a Zariski closed subspace of a (not necessarily $k$-rational) open polydisk. It is, in particular, boundary free. Thus we have $D\subseteq X\setminus \partial X$.

\emph{Claim: Assume that $\Supp(γ\vert_D)$ is compact. Then already $γ = γ\vert_D$, meaning that $\Supp(γ) \subset D$.}

To prove this, let $V = \Supp(γ\vert_D)^c$ be the complement in $X\setminus \partial X$. The assumption that $\Supp(γ\vert_D)$ is compact implies that $V$ is an open neighborhood of $D^c$. In order to show that $\Supp(γ) \subset D$, it suffices to show that for every skeleton $\Sigma \subset V$, the restriction $γ\vert_\Sigma$ vanishes.

Given such $\Sigma$, this is a statement about a polyhedral current on a pwl space. Moreover, we are given a decomposition
$$\Sigma = (\Sigma \cap D) \sqcup (\Sigma\cap D^c)$$
into an open and a complementary closed pwl subspace. Here, the intersection $\Sigma \cap D^c$ is a closed pwl subspace because it equals the locus where $ϕ = 0$. By assumption $γ\vert_{\Sigma \cap D} = 0$. By definition, $γ$ is a multiple of $ϕ$ and hence $γ\vert_{\Sigma \cap D^c} = 0$. Together, these two statements imply $γ\vert_\Sigma = 0$, which proves the claim.

%

In this way, we have reduced to the case that $\mcX = \Spf A$ for a special complete local $\mcO_k$-algebra $A$. Namely note that in this case $\partial X = \emptyset$, so the formulation of Thm. \ref{thm:int_num_identity} comprises the statement that $\Supp(γ)$ is compact and this was precisely the condition of the claim. In fact, this compactness is the very first thing we prove. Assume $\mcX = \Spf A$ and $\mcZ = V(f_0, \ldots, f_n)$ from now on.

\emph{Claim: $\Supp γ$ is compact.} By Cor. \ref{cor:support_delta_form}, $γ$ is supported on $\Sigma(X, f)$. Moreover, under $t_f$, this skeleton maps properly to $\mbR_{>0}^{n+1}$. Indeed, for every $C> 0$, the closed subset $K_C = t_f^{-1}([C, \infty)^{n+1})$ is an affinoid domain in $X$ and the map $\Sigma(K_C, f)\to \mbR_{>0}^{n+1}$ is proper. This allows us to consider $T(X, f)$ as tropical cycle on $\mbR_{>0}^{n+1}$ in the following. Our aim is to find some $C$ with $\Supp(γ)\subseteq \Sigma(K_C, f)$.

By definition of the restriction \eqref{eq:def_realization_on_skeleton}, $γ\vert_{\Sigma(X, f)}$ equals $t_f^\star(x\Delta)$ in the sense of \eqref{eq:realization}. Now for every compact polyhedral set $H\subseteq \Sigma(X, f)$, the pushforward $t_{f,*}(H \setminus t_f^{-1}(t_f(\partial H)))$ is a tropical cycle with nonnegative coefficients. The same property holds for $\Delta$. Thus $γ\vert_{\Sigma(X, f)}$ is a sum of Dirac measures with nonnegative coefficients. We then have $γ\vert_{\Sigma(X \setminus K_C, f)} = 0$ if and only if 
\begin{equation}\label{eq:vanishing_for_support_I}
t_{f,*}(γ\vert_{\Sigma(X \setminus K_C, f)}) = x\Delta \wedge T(X, f)\vert_{\mbR_{>0}^{n+1}\setminus [C,\infty)^{n+1}} = 0.
\end{equation}
Our aim is to find some $C>0$ such that \eqref{eq:vanishing_for_support_I} holds. That is, we would like to see that the set $\Supp(\Delta\wedge T(X, f)) \subset \Delta$ does not have the origin as accumulation point.

To this end, we establish a finiteness property of $T(X, f)$ along $\Delta$. For $s \geq 2$, denote by $C_s \subset \mbR^{n+1}_{>0}$ the open cone where $x_i < \sum_{j = 0}^n m_j x_j $ holds for all $i$ and all $(m_0,\ldots,m_n)\in \mbZ_{\geq 0}^{n+1}$ with $\sum_{j=0}^n m_j = s$.
\begin{lem}\label{lem:support_lemma_artinian}
For each $s\geq 2$, there is a polyhedral set $T_s\subseteq \mbR^{n+1}$ such that
$$T(X,f)\cap C_s = T_s \cap C_s.$$
\end{lem}
\begin{proof}
Pick a $π$-adic, flat $\mcO_k$-algebra of formally finite type $B$ such that $A$ is the completion of $B$ along a maximal ideal $\mfm$. Write $Y = \mcM(B[π^{-1}])$ for the generic fiber of $\Spf B$, which is an affinoid analytic space. The ideal $I = (f_0,\ldots,f_n)\subset A$ is open because the $f_i$ form a regular sequence of maximal length. Choose $g_0,\ldots,g_n\in B$ with $g_i - f_i \in I^s$. Then $g_0,\ldots,g_n$ generate $I/I^2$, and hence $I$, by Nakayama's Lemma. It follows that they form a regular sequence in $A$ as well because $A$ is Cohen--Macaulay \cite{Stacks}*{Tag 02JN}. Localizing $B$, we may assume that $V(g_0,\ldots,g_n) = \{\mfm\}$. Then
$$T(X,g) = T(Y, g) \cap \mbR^{n+1}_{> 0}$$
for the simple reason that $\mr{sp}(y)\neq \mfm$, $y\in Y$, implies that $v(g_i(y)) = 0$ for at least one $i$.
Let $x\in X$ and assume that $v(f_i(x)) < \sum_{j = 0}^n m_jv(f_j(x))$ holds for all $i$ and all $m_j\geq 0$ with $\sum_{j = 0}^n m_j = s$. Then $v(g_i(x)) = v(f_i(x))$ because $g_i - f_i\in I^s$ for all $i$ and, consequently, $v(g_i(x)) < \sum_{j = 0}^n m_jv(g_j(x))$ for all such $i$ and $m_j$. The roles of the $f_i$ and $g_i$ may be exchanged in this argument, so we have
$$T(Y,g) \cap C_s = T(X, g)\cap C_s = T(X,f) \cap C_s.$$
Since $T(Y,g)$ is the tropicalization of an affinoid, it is a polyhedral set (Thm. \ref{thm:trop_is_polyhedral}) and the lemma is proved.
\end{proof}

For now, we only apply Lem. \ref{lem:support_lemma_artinian} with a single $s$, say $s = 2$. We obtain that $T(X,f)\cap C_s$ is a polyhedral fan (minus the origin) near the origin. The intersection $\Delta \wedge T(X,f)$ then has the same property. It is also $0$-dimensional, however, so has to vanish near the origin. This concludes the proof of the compact support property of $γ$.

\begin{proof}[Proof of the intersection number identity in case $n=0$.] Then $A$ is a finite flat $\mcO_k$-algebra and the proposition concerns $\len_{\mcO_k}A/fA$ for some regular element $f\in A$. Set $L:=A[π^{-1}]$ and write it as a product $L = \prod_i L_i$ with each $L_{i,\red}$ a field. Let $r_i := \len_{L_i} L_i$ and $d_i := [L_{i,\red}:k]$. Then $X$ equals $\coprod_{i\in I} \mcM(L_i)$ and $γ = v(f)$ is a function on this set. The skeleton $\Sigma(X,f)$ has underlying set $X$ itself where the point $\mcM(L_i)$ has weight equal to $\deg(\mcM(L_i)\to \mcM(k)) = r_id_i$. It follows that
\begin{equation}\label{eq:0_dim_intersection}
\int_X v(f) = \sum_i r_id_i v_i(f^{(i)})
\end{equation}
where $v_i$ is the unique extension of $v$ to $L_i$ and $f^{(i)}$ the $i$-component of $f$. (The extension $v_i$ takes values in $e_i^{-1}\mbZ$, where $e_i$ is the ramification index of $L_i/k$.) To compute the intersection-theoretic side, first note the following fact.
\begin{lem}\label{lem:lengths}
Let $Λ\subset L$ be any $f$-stable $\mcO_k$-lattice. Then
$$\len_{\mcO_k} Λ/fΛ = \len_{\mcO_k} A/fA.$$
\end{lem}
\begin{proof}
Both lengths equal $v(\det_k f)$, where $f$ is viewed as endomorphism of the finite $k$-vector space $L$.
\end{proof}
It follows that we may assume $A_{\red}$ to be normal, i.e. the maximal order in $L_{\mr{red}}$. Namely, every element of $L_\red$ that is integral over $\mcO_k$ lifts by Hensel's Lemma to an integral element of $L$ . Those may be adjoined to $A$ by Lem. \ref{lem:lengths} without changing the length in question. Then $A = \prod_iA_i$ for certain $\mcO_k$-orders $A_i\subseteq L_i$ and
\begin{equation}\label{eq:0_dim_case}
\len_{\mcO_k} A/fA = \sum_i \len_{\mcO_k} A_i/f^{(i)}A_i.
\end{equation}
Filtering each $A_i$ by ideals
$$0=N_{i,s(i)}\subseteq \ldots \subseteq N_{i,1}\subseteq N_{i,0} = A_i$$
such that each graded piece is a free $A_{i,\red}$-module shows that
$$\len_{\mcO_k} A/f = \sum_i \sum_{j=0}^{s(i)-1} \rk_{A_{i,\mr{red}}} N_{i,j}/N_{i,j+1} \cdot \len_{\mcO_k} A_{i,\red}/f^{(i)}A_{i,\red}$$
which is identical with the right hand side of \eqref{eq:0_dim_case} since
$$\sum_{j = 0}^{s(i)-1} \rk_{A_{i,\mr{red}}} N_{i,j}/N_{i, j+1} = r_i.$$
This proves Thm. \ref{thm:int_num_identity} for $n = 0$.
\end{proof}

\begin{proof}[Proof for general n.]
The idea is to reduce to the case $n=0$ by passing to $A/(f_1,\ldots,f_n)$. The argument requires a choice of the $f_0,\ldots,f_n$ that allows to control $T(X,f)$ to some degree. (Note that the two sides of \eqref{eq:int_number_identity} depend only on $\mcZ\subseteq \mcX$, not on the sequence $f_0,\ldots,f_n$.)

\begin{lem}\label{lem:choice_of_f_i}
There are generators $I = (f_0, \ldots, f_n)$ such that $A/(f_1,\ldots,f_n)$ is $π$-torsion free.
\end{lem}
\begin{proof}
We assume $n\geq 1$, otherwise there is nothing to prove. Let $η_1,\ldots,η_s$ be the finitely many generic points of $\Spec A/(π)$ and let $f'_0,\ldots,f'_n$ be any regular sequence generating $I$. Because $A/I$ is artinian, we in particular also have $η_j\notin \Spec A/I$ for all $j$. It follows that there exists for each $j$ some $0\leq i(j)\leq n$ with $f'_{i(j),η_j}\in A_{η_j}^\times$. By prime ideal avoidance, there also exists for each $j$ some $u_j \in A$ with $u_j \in A_{η_j}^\times$ but $u_j\notin A_{η_k}^\times$, $k\neq j$.

Consider the indices $J = \{ 1 \leq j \leq s \mid f_n' \notin A_{η_j}^\times\}$ and put
$$f_n := f_n' + \sum_{j\in J} u_j f_{i(j)}.$$
Then $(f_0',\ldots,f_{n-1}', f_n) = I$ as well, so $f_0', \ldots, f_{n-1}', f_n$ is a regular sequence because $A$ is Cohen--Macaulay \cite{Stacks}*{Tag 02JN}. By construction, the quotient $A/(π, f_n)$ has Krull dimension $n-1$, so $π, f_n$ is a regular sequence in $A$, again by \cite{Stacks}*{Tag 02JN}. Thus $π$ is a nonzero divisor in $A/(f_n)$, meaning $A/(f_n)$ is flat over $\mcO_k$. Furthermore, $A/(f_n)$ is Cohen--Macaulay again \cite{Stacks}*{Tag 02JN}, so we may replace $A$ by $A/(f_n)$ and conclude by induction.
\end{proof}

Let $f_0,\ldots,f_n$ be generators as in Lem \ref{lem:choice_of_f_i}. We claim that the map
\begin{equation}\label{eq:map_psi}
ψ^*:\mcO_k[\![z_1,\ldots,z_n]\!]\lr A,\ z_i\longmapsto f_i
\end{equation}
is finite and flat. Indeed, the $π$-torsion freeness of $A/(f_1,\ldots,f_n)$ implies that $(π,f_1,\ldots,f_n)$ form a regular sequence in $A$. Then $A/(π,f_1,\ldots,f_n)$ is artinian for dimension reasons and, because $A$ is special over $\mcO_k$, finite over $\mcO_k/(π)$. The topology on $A$ agrees with the $(π,f_1,\ldots,f_n)$-adic one, so the finiteness of $ψ^*$ follows from Nakayama. The flatness now follows from the regularity of $\mcO_k[\![z_1,\ldots,z_n]\!]$ and ``Miracle Flatness'' \cite{Stacks}*{Tag 00R4}.

Let $B = (\Spf \mcO_k[\![z_1,\ldots,z_n]\!])^{\mr{an}}$ denote the open unit ball.
\begin{lem}\label{lem:finite_flat}
The map $ψ:X\to B$ obtained from \eqref{eq:map_psi} is finite and flat of degree $\deg(ψ^*)$. More precisely, if $K\subseteq B$ is an affinoid domain, then
\begin{equation}\label{eq:inv_image_psi}
ψ^{-1}(K) = \mcM(A\tensor_{\mcO_k[\![z_1,\ldots,z_n]\!]} \mcO_K(K)).
\end{equation}
\end{lem}
\begin{proof}
The property of $ψ$ being finite and flat of degree $\deg(ψ)$ follows once we show \eqref{eq:inv_image_psi}. This identity requires neither the flatness nor the fact that the base ring is $\mcO_k[\![z_1,\ldots,z_n]\!]$. We obtain it by induction from the following statement.

Let $A_0\to A$ be a finite map of special complete local $\mcO_k$-algebras. Assume $A$ is generated by one element over $A_0$, say $A \iso A_0[T]/(h(T))$ with some monic polynomial $h(T)\in A_0[T]$. Then the canonical map
\begin{equation}\label{eq:map_to_convergent_ps}
A_0[T]/(h(T)) \lr A_0\{T\}/(h(T))
\end{equation}
is an isomorphism. The generic fiber functor in \cite{Ber_vanishing_II} is then such that the inverse image $ψ^{-1}(K)$ of an affinoid domain $K\subseteq (\Spf A_0)^{\mr{an}}$ is the Zariski closed subspace
$$ψ^{-1}(K) = V(h(T)) \subset \mcM(k\{T\}) \times_{\mcM(k)} K$$
which equals $\mcM\big( \mcO_K(K)\{T\}/(h(T))\big)$ and the lemma is proved.
\end{proof}
The morphism $ψ:X \to B$ is compatible with tropicalizations in the sense that $p\circ t_f = t_z \circ ψ$, where $p:\mbR^{n+1}\to \mbR^n$ denotes projection to the last $n$ coordinates. Let $P$ be the plane of $(y,x,x,\ldots,x)\in \mbR^{n+1}$.

\emph{Claim: There exist $r,c \in \mbR_{> 0}$ such that
\begin{equation}\label{eq:bound_support}
(y,x,\ldots,x)\in \Supp(T(X,f)) \cap P\quad \Longrightarrow\quad y > \min\{rx, c\}.
\end{equation}}
\begin{proof}[Proof of the claim.]
This is a Newton polygon argument. The finiteness of $ψ$ implies that $f_0$ satisfies a relation of the form
$$Q(f_1,\ldots,f_n; S) := S^d + \sum_{i = 0}^{d-1} q_i(f_1,\ldots,f_n)S^i = 0,\ \ \ q_i\in \mcO_k[\![z_1,\ldots,z_n]\!].$$
For every $z\in X$, the Newton polygon of $Q(f_1(z),\ldots,f_n(z); S)$ has finitely many nonpositive slopes. Moreover, since $v(f_i(z)) > 0$ for all $i$ and all $z$, the slope is zero precisely over the interval $[i_0, d]$, where $i_0$ is the minimal index such that $q_{i_0} \in \mcO_k[\![z_1,\ldots,z_n]\!]^\times$. The valuation $v(f_0(z))$ is a priori known to be positive because $f_0\notin A^\times$. It follows that $-v(f_0(z))$ is one of the finitely many negative slopes of the mentioned Newton polygon.

Now restrict to the case $v(f_1(z)) = \ldots = v(f_n(z)) = x$ and let $0\leq i < i_0$. Let $c_i = v(q_i(0,\ldots,0))$ be the valuation of the constant term of $q_i$. Then there is the bound
\begin{equation}\label{eq:bound_valuation_coefficient}
v(q_i(f_1(z),\ldots,f_n(z))) \geq \min\{x, c_i\}.
\end{equation}
There are only finitely many possible shapes for the Newton polygon over the interval $[0,i_0]$, so \eqref{eq:bound_valuation_coefficient} implies the existence of a bound on the slopes as claimed.
\end{proof}

We now come to the main arguments. The idea is to reduce from general $n$ to $n = 0$ by an application of Stokes' Theorem for the triangle in $\mbR^{n+1}$ bounded by the following three lines in $P$. (The scalar $ρ>0$ will be specified later, a sketch of the situation may be found below.)
\begin{enumerate}[leftmargin=*, label=(\arabic*), topsep=4pt, itemsep=4pt]
\item $D = [0,ρ] \cdot (1,\ldots,1)$, the diagonal,
\item $H =[0,ρ]\cdot (0, 1,\ldots,1)$, a horizontal base line and
\item $V =(0,ρ,\ldots,ρ) + [0,ρ]\cdot (1,0,\ldots,0)$, a vertical line.
\end{enumerate}
The boundary term for $D$ will turn out to be  $\int_X γ$, the one for $H$ will vanish and the boundary term for $V$ will be $\int_{V(f_1,\ldots,f_n)} v(f_0)$, which we know from the case $n=0$ to agree with $\len_{\mcO_k} A/I$.
We consider a subdivision of the cone spanned by $D$ and $H$. Put
\begin{equation}\label{eq:cone_subdivision}
\begin{aligned}
\Theta^- & = \{(y,x,\ldots,x),\ 0 \leq x\leq ρ,\ 0 \leq y \leq x\},\\
\Theta^0 & = \{(y,x,\ldots,x),\ ρ\leq x,\ 0\leq y \leq ρ\},\\
\Theta^+ & = \{(y,x,\ldots,x),\ ρ\leq x,\ ρ\leq y \leq x\}.
\end{aligned}
\end{equation}
%
%
%
%

We write $T = T(X, f)$ from now on. Let $\Theta \in \{\Theta^-,\Theta^0,\Theta^+\}$ be any of the three regions. Combining the bound \eqref{eq:bound_support} on $\Supp(T)\cap P$ with Lem. \ref{lem:support_lemma_artinian}, the closure in $\mbR_{\geq 0}^{n+1}$ of $\Theta \cap \Supp(P\wedge T)$ (originally a subset of $\mbR_{>0}^{n+1}$) is a polyhedral set. In fact, this is immediate for $\Theta^+\cap \Supp(P\wedge T)$ which is evidently closed in $\mbR_{\geq 0}^{n+1}$. The intersection $\Theta^0 \cap \Supp(P\wedge T)$ is similarly closed in $\mbR^{n+1}_{\geq 0}$ which follows from the existence of the constant $c$ in \eqref{eq:bound_support}. For $\Theta^-\cap \Supp(P\wedge T)$, we use the existence of the linear term in \eqref{eq:bound_support} together with Lem. \ref{lem:support_lemma_artinian}. (Choose $s$ in that lemma large in relation to $1/r$.)

We next specify $ρ$. Note for this that $V(f_1,\ldots,f_n) \subseteq X$ is a finite set of points. We thus may choose $ρ>0$ large enough to satisfy two conditions. First, Prop. \ref{prop:reg_sequence_statement} implies we may assume
\begin{equation}\label{eq:choice_of_rho}
T\vert_{\mbR \times (ρ,\infty)^n} = T(V(f_1,\ldots,f_n), f_0)\times (ρ,\infty)^n.
\end{equation}
Second, we also pick $ρ$ such that $ρ > \max\{v(f_0)\vert_{V(f_1,\ldots,f_n)}\}$. Endow $P$ with weight $x\wedge y$. We consider the $δ$-form $ω = ω_0\wedge P$ where $ω_0$ is the pws form on $P$ given by
\begin{equation}
\begin{aligned}
ω_0\vert_{\Theta^-} &= yd''x - x d''y,\\
ω_0\vert_{\Theta^0} &= -ρd''y,\\
ω_0\vert_{\Theta^+} &= 0
\end{aligned}
\end{equation}
and $0$ outside of these domains. We claim that the support of $ω\wedge T$ is compact. First, concerning $\Theta^-$, the fan property of $T$ near the origin, Lem. \ref{lem:support_lemma_artinian}, implies an analogous fan property for $P\wedge T\vert_{C_s}$, $s\geq 2$. Since the form $xd''y - yd''x$ restricts to $0$ on any line through the origin, we obtain that $ω\wedge T\vert_{C_s}$ vanishes near the origin for every $s\geq 2$. Combining this with the bound $y > \min\{rx, c\}$ for all $(y,x,\ldots,x)\in \Supp(T)\cap P$, we obtain that $ω\wedge T$ has support away from $H$. Concerning $\Theta^0$, the form $-ρd''y$ restricts to $0$ on the polyhedral set from \eqref{eq:choice_of_rho}. This establishes the compact support of $ω\wedge T$.

Now we may apply Stokes' Theorem to obtain
$$0 = \int_{\mbR^{n+1}} d'(ω \wedge T) = \int_{\mbR^{n+1}} d'ω \wedge T.$$
Recall that $d'ω = d'_Pω - \partial'ω$. Note that in the case at hand $d'_Pω \wedge T = 0$. Namely the form $xd''y - yd''x$ on $\Theta^-$ has the peculiar property that its derivative $d'xd''y - d'yd''x$ restricts to $0$ on every line in $P$, while $-ρd''y$ from $\Theta^0$ restricts to $0$ on the polyhedral set from \eqref{eq:choice_of_rho}. It follows that
\begin{equation}\label{eq:stokes_vanishing}
0 = \int_{\mbR^{n+1}} \partial'ω \wedge T.
\end{equation}
We briefly recall how to compute the boundary $\partial'ω$ of a pws form $ω$ on $\mbR^r$. This formula follows e.g. from \eqref{eq:formula_deriv_balanced}.
\begin{ex}\label{ex:boundary_of_pws}
Let $\mcC$ be a polyhedral complex structure on $\mbR^r$ that is subordinate to $ω$. Endow each $σ\in \mcC_r$ with the standard weight of $\mbR^r$. For each $τ\in \mcC_{r-1}$, fix a weight $μ_τ$. Every such $τ$ is a facet of precisely two $σ_1, σ_2 \in \mcC_r$ and we choose the normal vectors $n_{σ_1,τ}$ and $n_{σ_2,τ}$ such that $n_{σ_2,τ} = - n_{σ_1,τ}$. Then
\begin{equation}\label{eq:boundary_of_pws_form}
-\partial'ω = \sum_{σ\in \mcC_r} (ω_σ, n_{σ,τ}'')\vert_τ \wedge [τ, μ_τ].
\end{equation}
\end{ex}
Applying this in our situation, we obtain
\begin{equation}\label{eq:boundary_of_pws_int_number}
\begin{aligned}
-\partial'ω & = xD - x \cdot H - y\cdot V\\
& + ρ \cdot [(ρ, \ldots, ρ) + \mbR_{\geq 0}\cdot (0, 1,\ldots,1)]\\
& - ρ\cdot [(0,ρ,\ldots,ρ) + \mbR_{\geq 0}\cdot (0, 1,\ldots,1)].
\end{aligned}
\end{equation}

All occurring lines here are endowed with the standard weights given by the coordinate functions.
The (supports of the) second and third summand in \eqref{eq:boundary_of_pws_int_number} do not intersect $\Supp T$ by choice of $ρ$ and by \eqref{eq:bound_support}. We have also already established that $H$ does not meet $\Supp (P\wedge T)$.

We may now combine our various statements to compute $\int_X γ$. We first have
\begin{equation}
\int_X γ = \int_{\mbR^{n+1}_{>0}} (x\cdot \Delta) \wedge T = \int_{\mbR^{n+1}_{>0}} (x\cdot D)\wedge T.
\end{equation}
The first equality here is by definition of $γ$, the second because $\Supp T$ does not meet the line $\mbR_{\geq ρ}\cdot (1,\ldots,1)$ by the choice of $ρ$ in \eqref{eq:choice_of_rho}.
Now we apply the vanishing \eqref{eq:stokes_vanishing} and the computation \eqref{eq:boundary_of_pws_int_number}. We obtain
\begin{equation}\label{eq:int_number_intermediary}
\int_{\mbR^{n+1}_{>0}} (x\cdot D)\wedge T = \int_{\mbR^{n+1}_{>0}} (y \cdot V) \wedge T.
\end{equation}
By the choice of $ρ$ in \eqref{eq:choice_of_rho} again, the $\wedge$-product $V\wedge T$ equals the product $T(V(f_1,\ldots,f_n), f_0) \times \{(ρ,\ldots,ρ)\}$. Thus
$$\int_{\mbR^{n+1}_{>0}} y \cdot (V\wedge T) = \int_{V(f_1,\ldots,f_n)} v(f_0)$$
which is $\mr{len}_{\mcO_k} (A/I)$ by our computation in the case $n = 0$. The proof of Thm. \ref{thm:int_num_identity} is now complete.
\end{proof}
\phantom{\qedhere}
\end{proof}

\section{$δ$-Forms on Lubin--Tate Spaces}

\subsection{Formal $\mcO_F$-modules}
Let $F$ be a non-archimedean local field with ring of integers $\mcO_F$, uniformizer $π$ and residue field $\mbF_q$.

\begin{defn}\label{def:formal_O_F}
Let $R$ be a $π$-adically complete $\mcO_F$-algebra. By \emph{formal $\mcO_F$-module} over $R$ we mean a pair $(X, ι)$ of the following kind.
\begin{enumerate}[leftmargin=*]
\item $X$ is a $1$-dimensional, commutative formal group over $R$.
\item $ι:\mcO_F\to \End(X)$ is a \emph{strict} $\mcO_F$-action, meaning that $ι(a)$ acts on $\Lie(X)$ by multiplication with $a$, for every $a\in \mcO_F$.
\end{enumerate}
\end{defn}

We often write $X$ instead of $(X, ι)$. A homomorphism of $φ:X\to Y$ of formal $\mcO_F$-modules $X$ and $Y$ is an $\mcO_F$-linear homomorphism of formal groups. It is called an \emph{isogeny} if it is finite and locally free. In this case, $\ker(φ)$ is a finite locally free group scheme over $R$ with $\mcO_F$-action. Its degree will be of the form $\deg(φ) = q^{\mr{ht}(φ)}$ for a unique integer $\mr{ht}(φ)$ called the \emph{height of $φ$}, cf. \cite{ARGOS_7}*{Lem. 2.1}. We say that $X$ is of height $h$ if $ι(π)$ is an isogeny of height $h$. Our interest will always lie with formal $\mcO_F$-modules that are of some finite height in this sense.

We next explain an equivalent description of formal $\mcO_F$-modules in terms of $π$-divisible groups. Our only motivation for this is that we will sometimes need to cite results from the literature on $π$-divisible groups.

Given a formal $\mcO_F$-module $X$ over $R$ of height $h$, we denote by $X[π^n]$ the $π^n$-torsion in $X$. This is a finite and locally free group scheme over $R$ of degree $q^{nh}$ together with an $\mcO_F$-action. Taking the colimit over $n$ defines an equivalence of categories
\begin{equation}\label{eq:equiv_formal_pi_divisible}
\begin{aligned}
\left\{\text{\parbox{3.5cm}{\center formal $\mcO_F$-modules\\over $R$ of height $h$}}\right\} \overset{\iso}{\lr} &\ 
\left\{\text{\begin{varwidth}{\textwidth} \center formal $π$-divisible groups\\with strict $\mcO_F$-action\\
of height $h$ and dimension $1$\end{varwidth}}\right\}\\
(X,ι) \longmapsto &\ \left(X[π^n], ι\vert_{X[π^n]}\right)_{n\geq 1}.
\end{aligned}
\end{equation}
We explain the right hand. First assume $F$ is $p$-adic. Then by \emph{$π$-divisible} group, we simply mean a $p$-divisible group in the usual sense. The notion of dimension and being formal for a $p$-divisible group also have their usual meanings. Moreover, strictness of an $\mcO_F$-action has been defined in Def. \ref{def:formal_O_F} (2). Finally, we take the height relative to $\mcO_F$ so as to be consistent with our terminology for formal $\mcO_F$-modules. That is, $\mr{ht}_{\mbZ_p} = [F:\mbQ_p] \cdot \mr{ht}$ where $\mr{ht}_{\mbZ_p}$ is the usual height for $p$-divisible groups. The fact that \eqref{eq:equiv_formal_pi_divisible} is an equivalence of categories now follows from the results in \cite{Messing}*{\S2}.

Consider now the case that $F$ is a local function field. In this case, $\mcO_F \iso \mbF_q[\![π]\!]$ and $π$-divisible groups are meant in the sense of \cite{Hartl_z_div}*{Def. 7.1}. Strictness means that the integer $d$ in \cite{Hartl_z_div}*{Def. 7.1 (iv)} equals $1$. Dimension and height have their usual meanings, while being formal is meant in the sense of \cite{Hartl_z_div}*{Def. 1.1}. The fact that \eqref{eq:equiv_formal_pi_divisible} is an equivalence is \cite{Hartl_z_div}*{Thm. 8.3}.

\begin{prop}[Rigidity, \cite{RZ_book}*{(2.1)} and \cite{Hartl_z_div}*{Thm. 9.8}]\label{prop:rigidity}
Let $R$ be an $\mcO_F$-algebra in which $π$ is nilpotent. Let $I\subseteq R$ be a nilpotent ideal and let $X$,  $Y$ be two formal $\mcO_F$-modules of height $h$ over $R$. Then
$$F\tensor_{\mcO_F}\Hom(X, Y) \overset{\iso}{\lr} F\tensor _{\mcO_F}\Hom(R/I\tensor_R X, R/I\tensor_R Y).$$
\end{prop}

Let $R$ be an $\mcO_F$-algebra in which $π$ is nilpotent and let $X$, $Y$ be two formal $\mcO_F$-modules of height $h$ over $R$. By \emph{quasi-homomorphism} from $X$ to $Y$ (resp. \emph{quasi-endomorphism} of $X$), we mean an element of
$$\Hom^0(X, Y) := F\tensor_{\mcO_F} \Hom(X, Y),\quad \End^0(X) := F\tensor_{\mcO_F}\End(X).$$
A quasi-homomorphism or quasi-endomorphism $φ$ is called a quasi-isogeny if some multiple $π^nφ$ is an isogeny. In this situation, we define the height of $φ$ as $\mr{ht}(φ) = \mr{ht}(π^nφ) - nh$. In our situation (formal $\mcO_F$-modules of height $h$) and if $R$ is noetherian, being a quasi-isogeny is equivalent to $φ\neq 0$ at every maximal ideal of $R$.

\subsection{An Intersection Problem}
Denote by $\breve F$ the completion of a maximal unramified extension of $F$ and let $\mcO_{\breve F}$ and $\mbF$ denote its ring of integers and residue field, respectively.  Fix $h\geq 1$ and let $\mbX$ be a formal $\mcO_F$-module of height $2h$ and dimension $1$ over $\mbF$. For each $n\geq 0$, denote by $\mcM_n$ the Lubin--Tate formal deformation spaces for level $n$. It represents the functor
$$(\Sch/\Spf \mcO_{\breve F})^{\mr{op}} \lr (\mr{Set}),\ S \longmapsto \left\{ (X,ρ,α) \right\}/\iso$$
of isomorphism classes of the following kind of quadruples:
\begin{enumerate}[leftmargin=*]
\item $X/S$ denotes a formal $\mcO_F$-module of height $2h$.
\item $ρ$ is a quasi-isogeny of height $0$ over the special fiber $\ob S = \mbF\tensor_{\mcO_{\breve F}} S$ to the fixed formal $\mcO_F$-module,
$$ρ:\ob S \times_S X \lr \ob S \times _{\Spec \mbF} \mbX.$$
\item $α$ is a Drinfeld level structure, cf. \cite{Drinfeld_elliptic}*{\S4},
$$α:(\mcO_F/π^n\mcO_F)^{\oplus 2h}→X[π^n].$$
\end{enumerate}
By a result of Drinfeld \cite{Drinfeld_elliptic}*{Prop. 4.3}, $\mcM_n$ is a regular local formal scheme of dimension $2h$ and flat over $\mcO_{\breve F}$. For example,
$$\mcM_0 \iso \Spf \mcO_{\breve F}[\![z_1,\ldots,z_{2h-1}]\!].$$
Moreover, also by \cite{Drinfeld_elliptic}*{Prop. 4.3},
\begin{equation}\label{eq:proj_maps}
π_{n+k,n}:\mcM_{n+k} \lr \mcM_n,\ (X,ρ,α)\longmapsto (X,ρ,π^kα).
\end{equation}
defines a finite flat morphism of degree
$$\big\vert\ker\big(GL_{2h}(\mcO_F/(π^{k+n})) \lr GL_{2h}(\mcO_F/(π^n))\big)\big\vert.$$
Let $D_F = \End^0_F(\mbX)$, which is a central division algebra of Hasse invariant $1/2h$ over $F$. Denote its ring of integers by $\mcO_{D_F}$. Then there is an action
$$\mcO_{D_F}^\times \times GL_{2h}(\mcO_F) \circlearrowright \mcM_n,\quad (γ,g)\cdot (X,ρ,α) = (X, γ\circ ρ, α\circ g).$$
The cycles we are interested in are constructed as follows. Let $E/F$ be an unramified quadratic extension with ring of integers $\mcO_E$. Choose an embedding $E \to \breve F$ and fix an action $ι_{\mbX}:\mcO_E→\End_{\mcO_F}(\mbX)$ that makes $\mbX$ into a strict formal $\mcO_E$-module. (Its height as $\mcO_E$-module is then $h$.) Let $\mcN_n$ denote the space of isomorphism classes of $(Y, ρ, α)$, where $Y$ is a strict formal $\mcO_E$-module of dimension $1$ and $\mcO_E$-height $h$, where
$$ρ:\ob S \times _S Y \lr \ob S\times _{\Spec \mbF} \mbX$$
is an $\mcO_E$-linear quasi-isogeny of height $0$ and where
$$α:\big(\mcO_E/π^n \mcO_E\big)^{\oplus h} \lr Y[π^n]$$
is a Drinfeld level structure as $\mcO_E$-module. There is a closed immersion (by \cite{RZ_book}*{Prop. 2.9}, which also holds if $F$ is a local function field) given by the forgetful map
\begin{equation}\label{eq:forget_map_LT}
\mcN_0\longhookrightarrow \mcM_0,\ (Y,ρ) \longmapsto (Y,ρ).
\end{equation}
From now on, we also fix an $\mcO_F$-linear identification $τ:\mcO_F^{2h}\iso \mcO_E^h$. Then \eqref{eq:forget_map_LT} lifts to closed immersions
$$\mcN_n\longhookrightarrow \mcM_n,\ (Y,ρ,α)\longmapsto (Y, ρ, α\circ τ).$$
\begin{defn}\label{def:int_number_LT}
For $(γ,g)\in \mcO_{D_F}^\times \times GL_{2h}(\mcO_F)$, we define
\begin{equation}\label{eq:int_number_LT}
\mr{Int}_n(γ,g) := \mr{len}_{\mcO_{\breve F}} \mcO_{\mcN_n \cap (γ,g)\cdot \mcN_n} \in \mbZ_{\geq 1} \cup \{\infty\}.
\end{equation}
\end{defn}
Note that $\mcN_n$ and $\mcM_n$ are both regular, local, and of formal dimensions $2 \dim \mcN_n = \dim \mcM_n$. We will later impose a condition that implies that the intersection $\mcN_n\cap (γ, g)\mcN_n$ is artinian, in which case \eqref{eq:int_number_LT} is the natural definition of the intersection number of $\mcN_n$ and $(γ,g)^*\mcN_n$ inside $\mcM_n$.

We recall some of the relevant group theory from Li's paper \cite{Li_LT}.
Let $D_E = \End^0_E(\mbX)$, which is a central simple algebra of Hasse invariant $1/h$ over $E$. The units of its ring of integers $\mcO_{D_E}^\times$ act on $\mcN_n$, as does $GL_h(\mcO_E)$. Naturally $D_E\subset D_F$, while an embedding $GL_h(E)\to GL_{2h}(F)$ is given through the choice of $τ$. Then $\mcN_n \to \mcM_n$ is $\mcO_{D_E}^\times \times GL_h(\mcO_E)$-equivariant. It follows that $\mr{Int}_n(γ,g)$ depends only on the class of $(γ,g)$ in
$$\mcO_{D_E}^\times \backslash \mcO_{D_F}^\times /\mcO_{D_E}^\times\ \times\ GL_h(\mcO_E)\backslash GL_{2h}(\mcO_F)/GL_h(\mcO_E).$$
To each element of either double quotient, one associates a group-theoretic invariant, the so-called \emph{invariant polynomial}, cf. \cite{Li_LT}*{Def. 1.1}. These are monic polynomials in $F[t]$ of degree $h$. We denote by $P_γ,P_g\in F[t]$ the invariant polynomial of $γ$ (resp. $g$). The element $γ$ is called \emph{regular semisimple} if $P_γ$ is irreducible and if $P_γ(0)P_γ(1) \neq 0$. (The second condition is implied by irreducibility if $h \geq 2$.) If $γ$ is regular semisimple, then the intersection $\mcN_0 \cap γ\mcN_0$ is artinian and hence $\mr{Int}_n(γ,g) \in \mbZ_{>0}$ for all $n$ and $g$. This follows a posteriori from Li's arguments, but can also be shown directly, cf. \cite{LM}*{Lem. 3.6}.

Write $K_n = \ker (GL_{2h}(\mcO_F)\to GL_{2h}(\mcO_F/π^n\mcO_F))$ and endow $GL_{2h}(F)$ with the Haar measure such that $\mr{vol}(K_n) = 1$. We are interested in the following statement.
\begin{thm}[Li \cite{Li_LT}*{Thm. 1.3}]\label{thm:LT_main}
There is a constant $c$, depending only on $F$ and $n$, such that for all $(γ,g)$ with $γ$ regular semisimple,
$$\mr{Int}_n(γ,g) = c \int_{K_n} |\mr{Res}(P_γ, P_{xg})|^{-1}_F dx.$$
Here, $\mr{Res}(P_γ,P_{xg})$ is the resultant of the two invariant polynomials in question.
\end{thm}

For the proof of Thm. \ref{thm:LT_main}, Li first performs the following reduction steps. Recall that $π_{n+k,n}$ is our notation for the projection map in \eqref{eq:proj_maps}. We have already mentioned that they are flat with degrees given by
$$\begin{aligned}
[\mcN_{n+k}:\mcN_n]\ &= |GL_h(\mcO_E/π^{n+k}\mcO_E)| \cdot |GL_h(\mcO_E/π^n\mcO_E)|^{-1}\\
[\mcM_{n+k}:\mcM_n]\ &= |GL_{2h}(\mcO_F/π^{n+k}\mcO_F)| \cdot |GL_{2h}(\mcO_F/π^n\mcO_F)|^{-1}.
\end{aligned}$$
Li's idea is to first use the projection formula for the following two identities of cycles on $\mcM_{n+k}$ and $\mcM_n$:
\begin{equation}\label{eq:projection_formula_drinfeld_projection}
\begin{aligned}
π_{n+k,n,*}(\mcN_{n+k}) = & [\mcN_{n+k}:\mcN_n] \mcN_n,\\
π_{n+k,n}^*(\mcN_n) = & [\mcN_{n+k}:\mcN_n]^{-1}\sum_{x\in K_n/K_{n+k}} (1,x) \mcN_{n+k}.\end{aligned}
\end{equation}
Thus for all $k\geq 0$,
\begin{equation}\label{eq:proj_formula_application}
\begin{aligned}
\Int_n(γ,g) & = [\mcN_{n+k}:\mcN_n]^{-1}\langle π_{n+k,n,*}\mcN_{n+k},\ (γ,g)\mcN_n\rangle_{\mcM_n}\\
&= [\mcN_{n+k}:\mcN_n]^{-1} \langle \mcN_{n+k},\  (γ,g)π_{n+k,n}^*\mcN_n\rangle_{\mcM_{n+k}}\\
&= \frac{[\mcM_{n+k}:\mcM_n]}{[\mcN_{n+k}:\mcN_n]^2} \int_{K_n} \mr{Int}_{n+k}(γ,xg)\ dx.
\end{aligned}
\end{equation}
The second equality here merely rewrote the sum obtained from \eqref{eq:projection_formula_drinfeld_projection} as an integral, using $[K_n:K_{n+k}] = [\mcM_{n+k}:\mcM_n]$. The factor $[\mcM_{n+k}:\mcM_n]\cdot [\mcN_{n+k}:\mcN_n]^{-2}$ is independent of $k$ as long as $k\geq 1$ and will be equal to the constant $c$ in Thm. \ref{thm:LT_main}. In this way, Thm. \ref{thm:LT_main} is reduced to the following statement.

\begin{thm}[\cite{Li_LT}*{Thm. 4.1 combined with Thm. 5.20}]\label{thm:LT_main_reformulation}
Given regular semisimple $γ$, there exists $n_0$ such that for all $n\geq n_0$ and all $g$,
$$\mr{Int}_n(γ,g) = |\mr{Res}(P_γ,P_g)|^{-1}_F.$$
\end{thm}

The remainder of the article is devoted to the proof of Thm. \ref{thm:LT_main_reformulation}.

\begin{rmk}
The setting we introduced here is a simplified version of \cite{Li_LT}. Li also proves a generalization of Thm. \ref{thm:LT_main} to ramified quadratic extensions $E/F$ and to Hecke translates of $\mcN_n$.
\end{rmk}

\subsection{Analytic Reformulation}
Let $M_n = \mcM_n^{\mr{an}}$ and $N_n = \mcN_n^{\mr{an}}$ denote the generic fibers in the sense of \S5.2. By regularity and the fact that $\mcM_n$ is local, $\mcN_n$ is a complete intersection \cite{Stacks}*{Tag 0E9J}. Let
$$ϕ_n:= ϕ_{\mcN_n}:|M_n\setminus N_n|→ \mbR$$
be the distance function from \eqref{eq:def_distance_function} and let $h_n:=(-1)^{h-1}ϕ_n(d'd''ϕ_n)^{h-1}$ be the Green current for $N_n$ from Thm. \ref{thm:Green_form_complete_inter}. By Thm. \ref{thm:int_num_identity}, we have
\begin{equation}\label{eq:analytic_reformulation}
\Int_n(γ,g) = \int_{N_n}(γ,g)^*h_{n}\vert_{N_n}
\end{equation}
and our aim in the following is to evaluate the right hand side. (Strictly speaking, \eqref{eq:analytic_reformulation} computes the length of $\mcN_n\cap (γ,g)^{-1}\mcN_n$ but this agrees with $\mr{Int}_n(γ,g)$.) We always assume that $γ$ is regular semisimple, which is the case of interest. Let $C$ be any non-archimedean extension of $\breve F$ with ring of integers $\mcO_C$. The $C$-points of $M_n$ (resp. $N_n$) are then the $\mcO_C$-points of $\mcM_n$ (resp. $\mcN_n$).
\begin{lem}\label{lem:expressing_phi}
Let $x\in N_n(C)$ be a $C$-point and let $(Y, ρ, α)/\Spf \mcO_C$ be the corresponding point in $\mcN_n(\mcO_C)$. The value $(γ,g)^*ϕ_n(x) = ϕ_n((γ,g)(x))$ equals
\begin{equation}\label{eq:expression_phi_n_I}
\max\left\{ε > 0\ \left|\ \text{\begin{varwidth}{\textwidth}$(γρ)^{-1}ι_{\mbX}(\mcO_E)γρ$ acts by endomorphisms
on $\mcO_C/(π)^ε\tensor_{\mcO_C}Y$,\\ $α g τ^{-1}$ is $\mcO_E$-linear w.r.t. this $(γρ)^{-1}ι_{\mbX}(\mcO_E)(γρ)$-action\end{varwidth}}\right\}.\right.
\end{equation}
\end{lem}
\begin{proof}
For every $\mcO_F$-algebra $R$ in which $π$ is nilpotent, the set $\mcN_n(R)$ equals all those $(X, ρ, α)\in\mcM_n(R)$ that satisfy the following two conditions.
\begin{enumerate}[wide, labelindent=0pt, labelwidth=!, label=(\arabic*), topsep=4pt, itemsep=4pt]
\item The fixed $\mcO_E$-action $ι_{\mbX}$ on $\mbX$ can always be pulled back to an action by quasi-endomorphisms
$$ρ^{-1}ι_{\mbX}(\mcO_E)ρ \subseteq \End^0(R/π\tensor_R X) = \End^0(X).$$
(The equality on the right hand side here is by Prop. \ref{prop:rigidity}.) The first condition is now that the stronger inclusion
\begin{equation}\label{eq:pullback_O_E_action}
ρ^{-1}ι_{\mbX}(\mcO_E)ρ \subseteq \End(X)
\end{equation}
holds. In fact, this is the defining property of $\mcN_0 \subset \mcM_0$, compare \eqref{eq:forget_map_LT}. Namely, the strictness condition for the resulting action of $\mcO_E$ on $X$ will be automatic because $E/F$ is unramified.
\item The level structure $α$ as formal $\mcO_F$-module defines a level structure as formal $\mcO_E$-module. That is, assuming condition (1), $α\circ τ^{-1}$ is $\mcO_E$-linear with respect to \eqref{eq:pullback_O_E_action}.
\end{enumerate}
Now $ϕ_n((γ, g)(x))$ equals the maximal valuation $ε \in v(C^\times)\subseteq \mbR$ such that the composition
\begin{equation}\label{eq:point_mod_epsilon}
\mcO_C/(π)^ε\tensor_{\mcO_C} (Y, γρ, αg): \Spec \mcO_C/(π)^ε \lr \Spec \mcO_C \lr \mcM_n
\end{equation}
factors through $\mcN_n$. (We can take a maximum here because $\mcN_n$ is defined by the vanishing of finitely many functions in $\mcM_n$.) Specializing the above description of $\mcN_n(\mcO_C/(π)^ε)$ to points of the form \eqref{eq:point_mod_epsilon} proves \eqref{eq:expression_phi_n_I}.
\end{proof}
From now on, we simply say that $γ^{-1}\mcO_Eγ$ acts on a point $(Y, ρ, α)$ instead of the more precise formulation
\begin{equation}\label{eq:terminology_O_E_acts}
(γρ)^{-1} ι_{\mbX}(\mcO_E) (γρ) \subseteq \End^0(R/πR\tensor_R Y) \cap \End(Y).
\end{equation}
The expression in \eqref{eq:expression_phi_n_I} is very complicated because already the first condition ($γ^{-1}\mcO_Eγ$ acting mod $(π)^ε$) is described by $h$ equations and then the evaluation of the second condition is dependent on the first since it assumes the existence of the $γ^{-1}\mcO_Eγ$-action. Our first aim is to simplify this for $n\gg 0$.

Let $a_n:N_n→\mbR$ be the distance function for the inclusion of the closed point $\Spec \mbF→\mcN_n$. Then $a_0$ is simply
\begin{equation}\label{eq:expression_phi_0}
a_0(Y,ρ) = \max\left\{0<ε\leq 1\ \left|\
\text{\begin{varwidth}{\textwidth} \center $ρ:\mcO_C/(π)^ε\tensor_{\mcO_C}X\to \mcO_C/(π)^ε\tensor_\mbF \mbX$\\
is an isomorphism\end{varwidth}}\right\}\right..
\end{equation}
Indeed, a morphism $\Spec \mcO_C/(π)^ε$ can only factor through $\Spec \mbF$ if $ε \leq 1$ because $\mbF = \mcO_{\breve F}/(π)$. Then \eqref{eq:expression_phi_0} just expresses the moduli description of $\Spec \mbF \to \mcN_n$. For $n\geq 1$, $a_n$ is described as follows. By \cite{Drinfeld_elliptic}*{Prop. 4.3 (3)}, the closed point $\Spec \mbF \to \mcN_n$ agrees with the locus where the level structure sections vanish. Choose any coordinate $\mcY \iso \Spf \mcO_{\mcN_n}[\![t]\!]$ for the universal formal $\mcO_E$-module $\mcY/\mcN_n$ to view the universal level structure $α_n = (α_{n,1},\ldots,α_{n,h})$ as a set of sections of $\mcO_{\mcN_n}$. Then $a_n = \min\{a_{n,i}\}$ with $a_{n,i} = v(α_{n,i})$. In fact, this definition does not require formal models. By the following lemma, one may take any choice (or local choices) of coordinates $(Y^h,0)\iso (\mbD^h_{N_n},0)$, where $Y$ is the generic fiber of $\mcY$ and $\mbD_{N_n}→N_n$ the relative open unit disk. The composition $N_n\overset{α_n}{\to} Y^h \to \mbD^h_{N_n}$ may be viewed as a tuple of functions to $\mbD$ and the minimum of their valuations is precisely $a_n$.
\begin{lem}\label{lem:disk_invariance}
Let $F:(\mbD,0)^h\iso (\mbD,0)^h$ be an automorphism of the open unit polydisk in $h$ variables over a non-archimedean field $C$ that fixes the origin. Let $a := \min\{v(t_1),\ldots,v(t_h)\}$, where the $t_i$ are the coordinate functions. Then $F^*a = a.$
\end{lem}
\begin{proof}
Any such $F$ is given by a tuple $(F_1,\ldots,F_h)$ of integral power series $F_i\in \mcO_C[\![X_1,\ldots,X_h]\!]$. Since $F(0) = 0$, these have vanishing constant coefficient. Then $F^*a \geq a$. Applying the same argument to $F^{-1}$ finishes the proof.
\end{proof}

The ring $\mcO_{D_F} \tensor_{\mcO_F} M_{2h}(\mcO_F)^{\mr{op}}$ acts from the left on $\Hom_{\mcO_F}(\mcO_F^{2h}, \mbX)$. It contains $\mcO_E\tensor_{\mcO_F}\mcO_E$ as subring: The left copy stems from the inclusion $\mcO_E\subseteq \End_{\mcO_F}(\mbX)$, the right one from the fixed choice $τ:\mcO_F^{2h}\iso \mcO_E^h$. Write
$$\mcO_{D_F} \tensor_{\mcO_F} M_{2h}(\mcO_F)^{\mr{op}} = A \oplus A σ$$
for the decomposition into $\mcO_E$-linear and $\mcO_E$-Galois linear elements. Here, $σ\in M_{2h}(\mcO_F)^{\mr{op}}$ denotes coordinate wise Galois conjugation on $\mcO_E^h$. We obtain a decomposition
\begin{equation}\label{eq:galois_decomposition}
γ\tensor g = (γ,g)^+ + (γ,g)^- σ,\ \ \ (γ,g)^+,(γ,g)^-\in A.
\end{equation}
Assume that $(Y,ρ,α) \in \mcN_n(R)$, with $π$ nilpotent in $R$, has the additional property that $ρ^{-1}γρ \in \End^0(R/πR\tensor_R Y) \cap \End(Y)$. (In this case we say that $γ$ acts on $Y$.) Then $γ^{-1}\mcO_Eγ$ acts on $Y$ in the sense of \eqref{eq:terminology_O_E_acts}, and $αg$ is $\mcO_E$-linear for the $γ^{-1}\mcO_Eγ$-action if and only if $γαg$ is $\mcO_E$-linear for the standard $\mcO_E$-action. The latter can now be reformulated as $(γ,g)^- α = 0$. The condition that $γ$ acts (in the above sense) holds for the closed point $\Spec \mbF \to \mcN_0$ because $\End_{\mcO_F}(\mbX) = \mcO_{D_F}$, so we obtain
\begin{multline}\label{eq:expression_phi_n_II}
\min\{π_{n,0}^*(a_0),\ (γ,g)^*ϕ_n\}(Y,ρ,α)\\
= \max \left\{0 < ε\leq 1\ \left|\ 
\text{\begin{varwidth}{\textwidth} \center
$ρ:\mcO_C/(π)^ε\tensor_{\mcO_C} Y \to \mcO_C/(π)^ε \tensor_\mbF\mbX$ is an\\ isomorphism and $(γ,g)^-α = 0$ mod $(π)^ε$\end{varwidth}}\right\}.\right.
\end{multline}
Fix a coordinate $\mcY \iso \Spf \mcO_{\mcN_n}[\![t]\!]$ for the universal $\mcO_E$-module over $\mcN_n$. This provides a coordinate $\mbX \iso \Spf \mbF[\![t]\!]$ over the closed point also. The action of $(γ,g)^-$ on $\Hom_{\mcO_E}(\mcO_E^h, \mbX) \iso \mbX^h$ is then expressed through an $h$-tuple of power series $(γ,g)^-\in \mbF[\![t_1,\ldots,t_h]\!]^h$ with vanishing constant coefficients. We denote by $(γ,g)^- \in \mcO_{\breve F}[\![t_1,\ldots,t_h]\!]^h$ any choice of lift with vanishing constant coefficients. Then it makes sense to apply $(γ,g)^-$ to the universal level structure, providing a tuple of functions $(γ,g)^-(α_n) \in \mcO_{N_n}(N_n)^h$.

\begin{prop}\label{prop:function_comparison_LT}
Given regular semisimple $γ$, there exists $n_0$ such that for all $n\geq n_0$ and all $g$, there is an identity of continuous functions on $N_n$,
\begin{equation}\label{eq:function_identity_LT}
(γ,g)^*ϕ_n = v((γ,g)^-(α_n)).
\end{equation}
\end{prop}
\begin{proof}
Equation \eqref{eq:expression_phi_n_II} and the fact that $(γ, g)^-$ was chosen to lift the action of $(γ,g)^-$ on $\Hom_{\mcO_E}(\mcO_E^h, \mbX)$ imply that for every $n\geq 0$,
$$\min\{a_0,\ (γ, g)^*ϕ_n\} = \min\{a_0,\ v((γ,g)^-(α_n))\}.$$
Thus the proposition follows once we show that there exists an $n_0$ such that for all $n \geq n_0$ and all $g$,
$$(γ, g)^*ϕ_n,\ v((γ,g)^-(α_n)) \leq a_0.$$
Our argument proves both inequalities at the same time. Namely, we prove the existence of some $n_0$ such that the following holds for all $n\geq n_0$ and all $g$. Let $x = (Y, ρ, α_n) \in \mcN_n(\mcO_C)$ be a $C$-point of $N_n$, where $C/\breve F$ is a non-archimedean extension as before. Let $ε_0 = a_0(x)\in v(C^\times)$ be the maximal exponent such that $ρ\mod (π)^{ε_0}$ is an isomorphism. Then, the maximal $0<ε\leq ε_0$ such that $(γ,g)^-(α_n) = 0 \mod (π)^ε$ has the property $ε < ε_0$.

\emph{Step 1. We have $a_1 \leq 1/(q^{2h}-1)$. Furthermore, for every $n\geq 2$, the inequality $a_n \leq q^{-2h}a_{n-1}$ holds.} Let $x = (Y, ρ, α_n) \in \mcN_n(\mcO_C)$ be as before and $n\geq 1$. After a choice of coordinate, multiplication by $π$ is described by a power series of the form
\begin{equation}\label{eq:mult_by_pi}
ι(π)(t) = πt + c_2t^2 + \ldots + c_{q^{2h}-1}t^{q^{2h}-1} + ut^{q^{2h}} + \mr{higher} \in \mcO_C [\![t]\!]
\end{equation}
with $u$ a unit and $c_2,\ldots, c_{q^{2h}-1}$ topologically nilpotent. The valuations of the $π$-torsion points $Y[π] \subset \mbD_C$ are the negatives of the slopes of the Newton polygon of \eqref{eq:mult_by_pi}, which is the polygon through
$$(0,\infty),\ (1,1),\ (2,v(c_2)),\ \ldots,\ (q^{2h}-1, v(c_{q^{2h}-1})),\ (q^{2h}, 0).$$
Since $α_1$ gives a trivialization of $Y[π]$, we find that $a_1(x)$ is the minimal of the negative slopes of that polygon. We in particular have the bound $a_1(x) \leq 1/(q^{2h}-1)$, which is the negative of the slope from $(1,1)$ to $(q^{2h}, 0)$. This proves the first claim.

The second claim is by the same logic. Given a point $\Spf \mcO_C \to Y$, say $t \mapsto y$, the valuations of the $π$-division points $ι(π)^{-1}(y)$ are the negatives of the nonzero slopes of the Newton polygon of $ι(π)(t) = y$ which is the polygon through
$$(0, v(y)),\ (1,1),\ (2,v(c_2)),\ \ldots,\ (q^{2h}-1, v(c_{q^{2h}-1})),\ (q^{2h}, 0).$$
The minimal possible valuation is thus $v(y)/q^{2h}$. Applying this argument to the basis of $Y[π^{n-1}]$ given by $α_{n-1}$, we obtain $a_n(x) \leq q^{-2h}a_{n-1}(x)$ as claimed.

\emph{Step 2. Deducing that $ε < ε_0$ (see above) for a given $g$ for $n\gg 0$.} The element $(γ,g)^-$ is invertible in $A[π^{-1}]$ because $γ$ is regular semisimple, cf. \cite{Li_LT}*{Prop. 5.22}. The ring $A$ moreover preserves the subspace $\Hom_{\mcO_E}(\mcO_E^h,\mbX)$ of $\Hom_{\mcO_F}(\mcO_F^{2h}, \mbX)$, providing a map $A\to M_h(\mcO_{D_F})$. We view $(γ,g)^{-}$ in this matrix ring. By the elementary divisor theorem, we may write
\begin{equation}\label{eq:decomp_element}
(γ,g)^- = VSW,\ \ \ V, W \in GL_h(\mcO_{D_F}),\ S = \diag(S_1,\ldots,S_h) \in M_h(\mcO_{D_F})
\end{equation}
with $S_1\cdots S_h \neq 0$. We fix lifts $\wt V,\wt W$ and $\wt S$ of $V, W$ and $S$ to $\mcO_{\breve F}[\![t_1,\ldots,t_h]\!]^h$. We assume that these lifts have vanishing constant coefficient and that $\wt S$ is again diagonal, $\wt S = \diag(\wt S_1, \ldots, \wt S_h)$. Note that  $\wt V$ and $\wt W$ are again invertible. By Lem. \ref{lem:disk_invariance}, neither $\wt V$ nor $\wt W$ change the minimal valuation of an $h$-tuple of topologically nilpotent functions. Writing $β_n = \wt W(α_n)$, we need to see that $\wt Sβ_n = 0$ modulo $(π)^ε$ implies $ε<ε_0$ for $n\gg 0$.

Let $ν:D_F^\times \to \mbZ$ be the division algebra valuation with normalisation $ν(π) = 2h$. Put $r_i = q^{ν(S_i)}$. Then
\begin{equation}\label{eq:ps_expansion_T}
S_i(t) \in \mbF^\times\cdot t^{r_i} + \mr{higher} \mod π,
\end{equation}
so for $v(y) < r_i^{-1}$, one has $v(\wt S_i(y)) = r_i v(y)$. Put $r = \max\{r_1,\ldots,r_h\}$. Then \eqref{eq:ps_expansion_T} implies that assuming $v(α_n) < r^{-1}$, we obtain
\begin{equation}\label{eq:generic_inequality}
v((γ,g)^-(α_n)) = v(\wt S β_n) \leq r v(β_n) = r v(α_n).
\end{equation}
If we choose $n\geq 2$ with $n-2 \geq \max\{ν(S_1),\ldots,ν(S_h)\}/2h$, then the bound from Step 1 gives
$$v(α_n) \leq q^{-2h(n-2)}v(α_2) \leq r^{-1} v(α_2) < r^{-1},$$
so \eqref{eq:generic_inequality} can be applied and specializes to
$$v((γ,g)^-(α_n)) \leq r v(α_n) \leq v(α_2) < v(α_1) \leq ε_0.$$
Any such $n$ is large enough in the sense of Step 2.

\emph{Step 3.} Finally, to obtain the statement uniformly in $g$, we observe that the valuations of the elementary divisors $ν(S_i)$ that occur during Step 2 are locally constant in $g$. The group $GL_{2h}(\mcO_F)$ is compact which then allows to choose $n_0$ uniformly. This completes the proof.
\end{proof}

\subsection{Evaluation of Integrals}

\begin{prop}\label{prop:evaluation_integrals}
(1) For all $n\geq 0$, the following integral identity holds,
\begin{equation}\label{eq:int_ident_1}
\int_{N_n} (-1)^{h-1}a_n (d'd''a_n)^{h-1} = 1.
\end{equation}
(2) Let $γ$ be regular semisimple and $n_0$ as in Prop. \ref{prop:function_comparison_LT}. Then for any $g$ and any $n\geq n_0$,
\begin{equation}\label{eq:int_ident_2}
\int_{N_n} (γ,g)^*h_n\vert_{N_n} = |\mr{Res}(P_γ,P_g)|_F^{-1}.
\end{equation}
\end{prop}
\begin{proof}
By Thm. \ref{thm:int_num_identity}, the identity  \eqref{eq:int_ident_1} is equivalent to the statement that the closed point $\Spec \mbF \to \mcN_n$ agrees with the vanishing locus $V(α_{n,1},\ldots,α_{n,h})$, which is \cite{Drinfeld_elliptic}*{Prop. 4.3}.

For the proof of \eqref{eq:int_ident_2}, we fix some $(γ,g)$ and let $\wt V$, $\wt W$, $\wt S$, $(r_1,\ldots,r_h)$ and $β_n = \wt W α_n$ be as during Step 2 before. By \cite{Li_LT}*{Prop. 5.21}, the resultant term $|\mr{Res}(P_γ,P_g)|^{-1}$ equals the reduced norm $r_1\cdots r_h$ of $(γ,g)^-$. (Note for the application of \cite{Li_LT}*{Prop. 5.21} that, in our case, $|\det g| = |\mr{nrd}(g)| = 1$. Moreover, the element $\Delta (φ, τ)^{-1} \cdot γ \cdot g \cdot φ \cdot M_τ$ that occurs there is the element $γ\tensor g$ from \eqref{eq:galois_decomposition}, while the application of $(0\ I_h)$ projects to $(γ,g)^-$.) We have seen that, for $n\geq n_0$,
$$(γ,g)^*ϕ_n = v((γ,g)^-(α_n)) = v(\wt S β_n).$$
By \eqref{eq:analytic_reformulation}, the task is to show
\begin{equation}\label{eq:beta_I}
\int_{N_n} (-1)^{h-1} v(\wt S β_n) (d'd''v(\wt S β_n))^{h-1} = r_1\cdots r_h
\end{equation}
while we know from \eqref{eq:int_ident_1} and $v(β_n) = v(α_n)$ that
\begin{equation}\label{eq:beta_II}
\int_{N_n} (-1)^{h-1} v(β_n) (d'd''v(β_n))^{h-1} = 1.
\end{equation}
Recall that, by definition, $β_n = \wt W(α_n)$ and that $\wt W$ is an invertible family of power series with vanishing constant coefficients. Thus $β_n = (β_{n,1},\ldots,β_{n,h})$ forms a regular sequence on $\mcN_n$. To simplify further, we assume that the lifts $\wt S_i\in \mcO_{\breve F}[\![t]\!]$ have leading term in $\mcO_{\breve F}^\times \cdot t^{r_i}$, cf. \eqref{eq:ps_expansion_T}.
Then also $\wt S(β_n) = (\wt S_i(β_{n,i}))_{i = 1,\ldots,h}$ form a regular sequence and $v(\wt S_i(β_{n,i})) = r_i v(β_{n,i})$. At this point, one could wrap up the proof by first passing to formal model intersection numbers with Thm. \ref{thm:int_num_identity} for both \eqref{eq:beta_I} and \eqref{eq:beta_II} and then resorting to multilinearity of Cartier divisor intersection numbers. The following is a tropical variant of this argument which seems to be interesting in itself, cf. Prop. \ref{prop:bezout} below.

We view $β_n$ and $\wt S β_n$ as tuples of toric coordinates $N_n\to \mbG_m^h$ in the following. The tropicalization maps $t_{β_n}$ and $t_{\wt Sβ_n}$ are proper (same situation as before \eqref{eq:vanishing_for_support_I}) and the two integrals in \eqref{eq:beta_I} and \eqref{eq:beta_II} equal
$$\int_{\mbR^{h+1}_{>0}} (x\Delta) \wedge T(N_n, β_n)\quad \text{and}\quad \int_{\mbR^{h+1}_{>0}} (x\Delta)\wedge T(N_n, \wt S (β_n)).$$
Comparing \eqref{eq:beta_I} and \eqref{eq:beta_II}, our task is to establish the following identity:
\begin{equation}\label{eq:multi_linearity}
\int (x\cdot \Delta) \wedge T(N_n, \wt S (β_n)) = r_1\cdots r_h \int (x\cdot \Delta) \wedge T(N_n, β_n).
\end{equation}
With our simplified choice of $\wt S_i$,
$$T(N_n, \wt S (β_n)) = T(N_n, (β_{n,i}^{r_i})_{i = 1,\ldots,h}) = F_*T(N_n, β_n)$$
with $F = \diag(r_1,\ldots,r_h): \mbR^h\to \mbR^h$. By the projection formula \cite{Mih_trop_inter}*{Thm. 4.1 (3)},
$$\int (x\cdot \Delta)\wedge T(N_n, \wt S (β_n)) = \int F^*(x\cdot \Delta) \wedge T(N_n, β_n).$$
Put $u = (1,\ldots,1)$ and $w = F^{-1}(u) = (r_1^{-1},\ldots,r_h^{-1})$. Furthermore endow the line $\Delta_w = \mbR_{\geq 0}\cdot w$ with weight $w$ (view $w\in \det (N_{\Delta w}) = \mbR\cdot w$). Then $F^*(x\cdot \Delta) = ψ \cdot \Delta_w$ for some linear function $ψ$. It follows from the relation $F_*\circ F^* = \deg(F) = r_1\cdots r_h$ that $ψ(w) = r_1\cdots r_h$.

On the other hand, we may consider the plane $H$ spanned by $u$ and $w$. We endow it with weight $u\wedge w$, i.e. we identify it with $\mbR^2$ via $u$ and $w$. It supports a $δ$-form $ω = ω_0\wedge H$ where $ω_0$ is the pws form $ud''w - wd''u$ on the first quadrant $Q = \mbR_{\geq 0}\cdot u \times \mbR_{\geq 0}\cdot w$ and $0$ elsewhere. The derivative of $ω$ is
$$d'ω = d'_Pω - \partial'ω = (d'ud''w - d'wd''u)\vert_{Q} + u\Delta - w \Delta_w.$$
Just as in the proof of Thm. \ref{thm:int_num_identity}, $d'_Pω\vert_C = 0$ for every line $C\subseteq H$. We conclude that for every tropical hypersurface $T$ on an open subset $U\subseteq \mbR^h$ with $\Supp (T\wedge ω)$ compact,
\begin{equation}\label{eq:tropical_multilin}
0 = \int T\wedge d'ω = \int T\wedge \Delta - (r_1\cdots r_h )^{-1} \int T\wedge F^*\Delta.
\end{equation}
Our proof of \eqref{eq:multi_linearity} is complete if we can apply \eqref{eq:tropical_multilin} with $T = T(N_n, β_n)$. This requires us checking the compact support of $T(N_n, β_n) \wedge ω$. We know that the intersection of $\Supp T(N_n, β_n)$ with the cone generated by $u$ and $w$ is bounded because $v(β_n) \leq 1$. Thus we only need to check compactness near the origin. By Lem. \ref{lem:support_lemma_artinian}, $T(N_n, β_n)$ is a polyhedral fan near $0$ near the cone spanned by $u$ and $w$. Moreover, the form $ud''w - wd''u$ vanishes when restricted to any line through the origin. Thus, $\Supp(T(N_n, β_n) \wedge ω)$ is compact, \eqref{eq:tropical_multilin} applies, and the proof of Prop. \ref{prop:evaluation_integrals} is complete.
\end{proof}

This finishes the proof of both Theorems \ref{thm:LT_main} and \ref{thm:LT_main_reformulation}.

\begin{rmk}\label{rmk:conclude_LT}
The given proof relies on a comparison of functions on $M_n$ and $N_n$. Only rank $1$ valued fields intervene here. Li's proof is along the same lines, but always considers the intersection of formal schemes $\mcN_n\cap (γ,g)\mcN_n$. For example, he shows that, for regular semisimple $γ$, the intersection $\mcN_n \cap (γ,g)\mcN_n$ lies above the closed point $\Spec \mbF \subseteq \mcM_0$ for $n \gg 0$, cf. \cite{Li_LT}*{Prop. 3.11}. This is a stronger statement and also more difficult to prove than our comparison of functions in Prop. \ref{prop:function_comparison_LT}.
\end{rmk}

Lastly, we separately note the Bézout-like identity that underlies the above argument.

\begin{prop}\label{prop:bezout}
Let $C \subseteq \mbR^2$ be a tropical curve that has compact intersection with the first quadrant. Let $(u,w)$ denote coordinates on $\mbR^2$ and endow the rays $\Delta_u = \mbR_{\geq 0} (1,0)$ and $\Delta_w = \mbR_{\geq 0} (0,1)$ with their standard weight. Then
$$\int_{\mbR^2} (u\cdot \Delta_u) \wedge C = \int_{\mbR^2} (w\cdot \Delta_w)\wedge C.$$
\end{prop}

\end{document}